\documentclass[12pt,a4paper,twoside]{book}

\usepackage{mathrsfs}          
\usepackage[all]{xy}    
\usepackage[latin1]{inputenc}
\usepackage{amscd}
\usepackage{latexsym}
\usepackage{multicol}
\usepackage{amsmath}
\usepackage{amssymb}
\usepackage{amsthm}
\usepackage{makeidx}
\makeindex

\theoremstyle{plain}
\newtheorem{Theorem}{Theorem}[section]
\newtheorem{lemma}[Theorem]{Lemma}
\newtheorem{corollary}[Theorem]{Corollary}
\newtheorem{proposition}[Theorem]{Proposition}
\newtheorem{claim}[Theorem]{Claim}
\newtheorem{conjecture}[Theorem]{Conjecture}

\newtheorem{conjecture2*}{Conjecture}
\newtheorem{remark2*}[conjecture2*]{Remark}
\newtheorem{Theorem2*}[conjecture2*]{Theorem}
\newtheorem{proposition2*}[conjecture2*]{Proposition}
\newtheorem{claim2*}[conjecture2*]{Claim}
\newtheorem{example2*}[conjecture2*]{Example}
\newtheorem{lemma2*}[conjecture2*]{Lemma}

\theoremstyle{definition}

\newtheorem{example}[Theorem]{Example}

\newtheorem{definition}[Theorem]{Definition}
\newtheorem{remark}[Theorem]{Remark}

\newtheorem{pkt}[Theorem]{}
\newtheorem{construction}[Theorem]{Construction}

\newcommand{\sA}{{\mathcal A}}
\newcommand{\sB}{{\mathcal B}}
\newcommand{\sC}{{\mathcal C}}
\newcommand{\sD}{{\mathcal D}}
\newcommand{\sE}{{\mathcal E}}
\newcommand{\sF}{{\mathcal F}}

\newcommand{\sH}{{\mathcal H}}
\newcommand{\sI}{{\mathcal I}}

\newcommand{\sL}{{\mathcal L}}
\newcommand{\sM}{{\mathcal M}}
\newcommand{\sN}{{\mathcal N}}
\newcommand{\sO}{{\mathcal O}}
\newcommand{\sP}{{\mathcal P}}
\newcommand{\sQ}{{\mathcal Q}}
\newcommand{\sR}{{\mathcal R}}

\newcommand{\sV}{{\mathcal V}}
\newcommand{\sW}{{\mathcal W}}

\newcommand{\sX}{{\mathcal X}}
\newcommand{\sY}{{\mathcal Y}}
\newcommand{\sZ}{{\mathcal Z}}
\newcommand{\A}{{\mathbb A}}
\newcommand{\B}{{\mathbb B}}
\newcommand{\C}{{\mathbb C}}

\newcommand{\E}{{\mathbb E}}
\newcommand{\F}{{\mathbb F}}
\newcommand{\G}{{\mathbb G}}

\newcommand{\KK}{{\mathbb K}}
\newcommand{\LL}{{\mathbb L}}

\newcommand{\N}{{\mathbb N}}
\newcommand{\bP}{{\mathbb P}}
\newcommand{\Q}{{\mathbb Q}}
\newcommand{\R}{{\mathbb R}}
\newcommand{\BS}{{\mathbb S}}

\newcommand{\V}{{\mathbb V}}
\newcommand{\W}{{\mathbb W}}

\newcommand{\Z}{{\mathbb Z}}
\newcommand{\id}{{\rm id}}

\newcommand{\Hg}{{\rm Hg}}
\newcommand{\MT}{{\rm MT}}
\newcommand{\SL}{{\rm SL}}
\newcommand{\GL}{{\rm GL}}
\newcommand{\PGL}{{\rm PGL}}
\newcommand{\Mon}{{\rm Mon}}
\newcommand{\Sp}{{\rm Sp}}
\newcommand{\tr}{{\rm tr}}
\newcommand{\lcm}{{\rm lcm}}
\newcommand{\SU}{{\rm SU}}
\newcommand{\SO}{{\rm SO}}
\newcommand{\PU}{{\rm PU}}
\newcommand{\U}{{\rm U}}
\newcommand{\GSp}{{\rm GSp}}
\newcommand{\Hol}{{\rm Hol}}
\newcommand{\ad}{{\rm ad}}
\newcommand{\der}{{\rm der}}
\newcommand{\Gal}{{\rm Gal}}
\newcommand{\Eig}{{\rm Eig}}
\newcommand{\diag}{{\rm diag}}
\newcommand{\fg}{\mathfrak{g}}
\newcommand{\fh}{\mathfrak{h}}
\newcommand{\fsu}{\mathfrak{su}}
\newcommand{\fsl}{\mathfrak{sl}}
\newcommand{\fsp}{\mathfrak{sp}}
\newcommand{\fso}{\mathfrak{so}}
\newcommand{\fX}{\mathfrak{X}}
\begin{document}

\thispagestyle{empty}
\vspace*{2.5cm}
\begin{flushleft}
{\Huge \underline{Cyclic coverings},\\[0.5cm]
\underline{Calabi-Yau manifolds and}\\[0.5cm]
\underline{Complex multiplication}}\\[1.0cm]

{\Large with concrete examples of Calabi-Yau 3-manifolds with complex multiplication}\\[9.0cm]

{\LARGE Jan Christian Rohde}\\[0.3cm]
{\large Universit\"at Duisburg-Essen\\
Fachbereich Mathematik\\
jan.rohde@uni-due.de}
\end{flushleft}

\newpage

\chapter*{Acknowledgments}

This paper bases on my doctoral thesis, which has been supported by the funds of the DFG
(``Graduiertenkolleg mathematische und ingenieurs wissenschaftliche Methoden f\"ur
sichere Daten\"ubertragung und Informationsvermittlung'' and the
Leibniz-Preis of H\'el\`ene Esnault and Eckart Viehweg).

I am very grateful to Eckart Viehweg for giving me the subject of my thesis and 
his excellent guidance and support, not only during the work on this topic, but from the
beginning of my mathemathical studies.

Moreover, I would like to thank especially Martin M\"oller for many fruitful and stimulating hints
and discussions, and for reading this thesis and pointing out several mistakes.

I want to thank Kang Zuo and Stefan M\"uller-Stach for the hint to the
essay \cite{Voi2} of C. Voisin, which provided the main idea of the construction of many examples
of $CMCY$ families of 3-manifolds and finally of the construction of the Borcea-Voisin tower.

Many thanks to Ulrich G\"ortz for his instructive course about Shimura varieties at Mainz and
to Gebhard B\"ockle for his interesting reading course about Shimura curves.

Moreover, I would like to thank Bernd Siebert and Claire Voisin for stimulating discussions about
maximality and the Voisin mirror families in Oberwolfach during the conference ``Komplexe Algebraische Geometrie''
at the end of the summer 2007.

I wish to thank Irene Bouw, Juan Cervino, Andre Chatzistamatiou, Jochen Heinloth, Stefan Kukulies
and Kay R\"ulling for many fruitful discussions, and Pooja Singla for the careful reading of the
introduction.

Finally, I want to thank my father very much for his support during the last years and making
untroubled studies possible.

\tableofcontents

\chapter*{Introduction}\addcontentsline{toc}{chapter}{Introduction}
\markboth{Introduction}{Introduction}
We search for examples of families of Calabi-Yau manifolds with dense set of
complex multiplication fibers and for examples of families of curves with dense set of complex
multiplication fibers. This essay bases on the authors doctoral thesis. In addition we will
give some concrete examples of fibers with complex multiplication.

By string theoretical considerations, one is interested in Calabi-Yau manifolds, since Calabi-Yau
3-manifolds provide conformal field theories (CFT). One is especially interested in Calabi-Yau
3-manifolds with complex multiplication, since such a manifold has many symmetries and mirror
pairs of Calabi-Yau 3-manifolds with complex multiplication yield rational conformal field
theories (RCFT) (see \cite{GV}). Moreover S. Gukov and C. Vafa \cite{GV} ask for the existence of
infinitely many Calabi-Yau manifolds with complex multiplication of fixed dimension $n$.

For a Calabi-Yau manifold $X$ of dimension $n$ with $n \leq 3$, the condition of complex
multiplication is equivalent to the property that for all $k$ the Hodge group of $H^k(X,\C)$ is
commutative. We will call any family of Calabi-Yau $n$-manifolds, which has a dense set of
fibers satisfying the latter property with respect to the Hodge groups, a $CMCY$ family
of $n$-manifolds. \index{$CMCY$ family of $n$-manifolds} The author uses this condition for
technical reasons and hopes that such
a $CMCY$ family of $n$-manifolds in an arbitrary dimension may be interesting for its mathematical
beauty, too. Here we will give some examples of $CMCY$ families of $3$-manifolds and explain
how to construct $CMCY$ families of $n$-manifolds in an arbitrarily high dimension. Moreover we will explicitely
determine some fibers with complex multiplication (see Example $\ref{3cm3}$, Section $7.4$, Remark $\ref{cm3}$,
Remark $\ref{4fall}$ and Remark $\ref{cm20}$).

Starting with a family of cyclic covers of $\bP^1$
with a dense set of $CM$ fibers, E. Viehweg and K. Zuo \cite{VZ5} have constructed a $CMCY$ family of
$3$-manifolds. This construction of E. Viehweg and K. Zuo \cite{VZ5} is given by a tower of cyclic
coverings, which will be explained in Section $7.3$. In Chapter 8 we will give a modified version
of a Viehweg-Zuo tower for one of our new examples.

Hence we are interested in the examples of families of curves with a dense set of $CM$ fibers by our
search for $CMCY$ families of $n$-manifolds. But there is an other motivation given by an open
question in the theory of curves, too. In \cite{co} R. Coleman formulated the following conjecture:

\begin{conjecture2*}
Fix an integer $g \geq 4$. Then there are only finitely many complex
algebraic curves $C$ of genus $g$ such that ${\rm Jac}(C)$ is of
$CM$ type.
\end{conjecture2*}

Let $\sP_n$ denote the configuration space of $n+3$ points in $\bP^1$. One can endow these $n+3$
points in $\bP^1$ with local monodromy data and use these data for the construction of a family
$\sC \to \sP_n$ of cyclic covers onto
$\bP^1$ (see Construction $\ref{hiercr}$).

The action of $\PGL_2(\C)$ on $\bP^1$ yields a quotient $\sM_n = \sP_n/\PGL_2(\C)$. By fixing 3
points on $\bP^1$, the quotient $\sM_n$ can also be considered as a subspace of $\sP_n$.

\begin{remark2*}
In \cite{DeJN} J. de Jong and R. Noot gave counterexamples for $g =4$ and $g = 6$ to
the conjecture above. In
\cite{VZ5} E. Viehweg and K. Zuo gave an additional counterexample for $g = 6$. The
counterexamples are given by families $\sC \to \sP_n$ of cyclic covers of $\bP^1$ with infinitely
many $CM$ fibers. Here we will find additional families $\sC \to \sP_n$ of cyclic genus 5 and
genus 7 covers of $\bP^1$ with dense sets of complex multiplication fibers, too.
\end{remark2*}

All new examples $\sC \to \sP_n$ of the preceding remark have a variation $\sV$ of Hodge
structures similar to the examples of J. de Jong and R. Noot \cite{DeJN},  and of E. Viehweg and
K. Zuo \cite{VZ5}, which we call pure $(1,n)-VHS$. Let $\Hg(\sV)$ denote the generic Hodge group
of $\sV$ and let $K$ denote an arbitrary maximal compact subgroup of $\Hg^{\ad}(\sV)(\R)$.
In Section $4.4$ we prove that a pure $(1,n)-VHS$ induces an open (multivalued) period map to the
symmetric domain associated with ${\Hg}^{\ad}(\sV)(\R)/K$, which yields the dense sets of complex
multiplication fibers. We obtain the following result in Chapter 6:

\begin{Theorem2*}
There are exactly 19 families $\sC \to \sP_n$  of cyclic covers of $\bP^1$, which have a pure
$(1,n)-VHS$ (including all known and new examples).
\end{Theorem2*}

We will use the fact that the monodromy group $\Mon^0(\sV)$ is a subgroup of the derived group
$\Hg^{\der}(\sV)$ and we will study $\Mon^0(\sV)$. Let $\psi$ be a generator of the Galois group
of $\sC \to \sP_n$ and $C(\psi)$ be the centralizer of $\psi$ in the symplectic group with respect
to the intersection pairing on an arbitrary fiber of $\sC$. In Chapter 4 we obtain the result,
which will be useful for our study of $\Hg^{\der}(\sV)$ and $\Mon^0(\sV)$:

\begin{lemma2*}
The monodromy group $\Mon^0(\sV)$ and the derived Hodge group $\Hg^{\der}(\sV)$ are contained in
$C(\psi)$.
\end{lemma2*}

Unfortunely we will not be able to determine $\Mon^0(\sV)$ for all families $\sC \to \sP_n$ of
cyclic covers onto $\bP^1$. But we obtain for example the following results in Chapter 5:

\begin{proposition2*}
Let $\sC \to \sP_n$ be a family of cyclic covers of degree $m$ onto $\bP^1$. Then one has:
\begin{enumerate}
\item If the degree $m$ is a prime number $\geq 3$, the algebraic groups $C^{\der}(\psi)$,
$\Mon^0(\sV)$ and $\Hg^{\der}(\sV)$ coincide.
\item If $\sC \to \sP_{2g+2}$ is a family of hyperelliptic curves, one obtains
$$\Mon^0(\sV) = \Hg(\sV) \cong \Sp_{\Q}(2g).$$
\item In the case of a family of covers onto $\bP^1$ with 4 branch points, we need a pure
$(1,1)-VHS$ to obtain an open period map to the symmetric domain associated with
${\Hg}^{\ad}(\sV)(\R)/K$.
\end{enumerate}
\end{proposition2*}

By our new examples of Viehweg-Zuo towers, we will only obtain
$CMCY$ families of 2-manifolds. C. Voisin \cite{Voi2} has
described a method to obtain Calabi-Yau 3-manifolds by using
involutions on $K3$ surfaces. C. Borcea \cite{Bc2} has independently
arrived at a more general version of the latter method, which allows to
construct Calabi-Yau manifolds in arbitrary
dimension. By using this method, we obtain in Section $7.2$: 

\begin{proposition2*}
For $i = 1,2$ assume that $\sC^{(i)} \to V_i$ is a $CMCY$ family of $n_i$-manifolds endowed with
the $V_i$-involution $\iota_i$ such that for all $p \in V_i$ the
ramification locus $(R_i)_p$ of $\sC^{(i)}_p \to \sC^{(i)}_p/\iota_i$ consists of smooth disjoint
hypersurfaces. In addition assume that $V_i$ has a dense set of points $p \in V_i$ such that for
all $k$ the Hodge groups $\Hg(H^k(\sC^{(i)}_p,\Q))$ and $\Hg(H^k((R_i)_p,\Q))$ are commutative.
By blowing up the singular locus of $\sC^{(1)} \times \sC^{(2)}/\langle(\iota_1, \iota_2)\rangle$,
one obtains a $CMCY$ family of $n_1+n_2$-manifolds over $V_1 \times V_2$ endowed with an
involution satisfying the same assumptions as $\iota_1$ and $\iota_2$.
\end{proposition2*}

\begin{remark2*}
By the preceding proposition, one can apply the construction of C. Borcea and C. Voisin for families
to obtain an infinite tower of $CMCY$ families of $n$-manifolds, which we call a Borcea-Voisin
tower.
\end{remark2*}

\begin{example2*} \label{exa1}
The family $\sC \to \sM_1$ given by
$$\bP^2 \supset V(y_1^4- x_1(x_1 - x_0)(x_1 - \lambda x_0)x_0) \to \lambda \in \sM_1$$
has a pure $(1,1)-VHS$. Hence by the construction of Viehweg and Zuo \cite{VZ5}, one concludes
that the family $\sC_2$ given by
\begin{equation} \label{exa}
\bP^3 \supset V(y_2^4+y_1^4- x_1(x_1 - x_0)(x_1 - \lambda x_0)x_0) \to \lambda \in \sM_1
\end{equation}
is a $CMCY$ family of 2-manifolds.

This family has many $\sM_1$-automorphisms. The quotients by some of these automorphisms yield new
examples of $CMCY$ families of 2-manifolds.
Moreover there are some involutions on $\sC_2$, which make this family and its quotient families
of $K3$-surfaces suitable for the construction of a Borcea-Voisin tower (see Section $7.4$ for
the construction of $\sC_2$, and for the automorphism group and the quotient families of $\sC_2$
see Section $9.3$, Section $9.4$ and Section $9.5$).
\end{example2*}

\begin{example2*}
The family $\sC \to \sM_3$ given by
$$\bP(2,1,1) \supset V(y_1^3- x_1(x_1 - x_0) (x_1 - a x_0) (x_1 - b x_0)(x_1 - c x_0)x_0) \to
\lambda \in \sM_1$$
has a pure $(1,3)-VHS$. The desingularisation $\tilde \bP(2,2,1,1)$ of the weighted projective
space $\bP(2,2,1,1)$ is given by
blowing up the singular locus. By a modification of the construction of Viehweg and Zuo,
the family $\sW$ given by
\begin{equation}\label{nis}
\tilde \bP(2,2,1,1) \supset \tilde V(y_2^3+y_1^3- x_1(x_1 - x_0)
(x_1 - a x_0) (x_1 - b x_0)(x_1 - c x_0)x_0) \to \lambda \in \sM_3
\end{equation}
is a $CMCY$ family of 2-manifolds. The family $\sW$ has a degree 3 quotient, which yields a $CMCY$
family of 2-manifolds. Moreover it has an involution, which makes it and its degree 3 quotient
suitable for the construction of a Borcea-Voisin tower (see Chapter 8 for the construction of
$\sW$ and Section $9.1$ for its degree 3 quotient).
\end{example2*}

By using the preceding example, we will obtain (see Section $9.2$ for the construction and Section
$10.3$ for the maximality):

\begin{Theorem2*} \label{nunuu}
Let $\alpha_{\F_3}$ denote a generator of the Galois group of a degree 3 cover
$\F_3 \to \bP^1$. The family $\sW$ has an $\sM_3$-automorphism $\alpha'$ of order 3 such that
the quotient $\sW \times \F_3/\langle(\alpha',\alpha_{\F_3})\rangle$ has a desingularisation,
which is a $CMCY$ family $\sQ \to \sM_3$ of 3-manifolds. Moreover the family $\sQ$ is maximal.
\end{Theorem2*}

By using the V. V. Nikulins classification of involutions on $K3$ surfaces \cite{Niku}
and the construction of C. Voisin \cite{Voi2}, we obtain in Chapter 11:

\begin{Theorem2*}
For each integer $1 \leq r \leq 11$ there exists a maximal holomorphic $CMCY$ family of algebraic
3-manifolds with Hodge number $h^{2,1} = r$.
\end{Theorem2*}

The first three chapters explain well-known facts and yield an introduction
of the notations. Chapter 1 is an introduction to Hodge Theory with a special view towards complex
multiplication. We consider cyclic covers of $\bP^1$ in Chapter 2. Moreover Chapter 3 introduces
everything that we need to describe families of cyclic covers of $\bP^1$ and their variations of
Hodge structures.

In Chapter 4 we consider the Galois group action of a cyclic cover onto $\bP^1$ and we state
first results for the generic Hodge group of a family $\sC \to \sP_n$. Moreover we will give a sufficient criterion for the
existence of a dense set of $CM$ fibers given by the pure $(1,n)-VHS$. In Chapter 5 we compute
$\Mon^0(\sV)$, which provides many information about $\Hg(\sV)$. We will see that $\Mon^0(\sV)$
coincides with $C^{\der}(\psi)$ in infinitely many cases. In Chapter 6 we classify the
examples of families of cyclic covers onto $\bP^1$ providing a pure $(1,n)-VHS$.

The basic methods of the construction of $CMCY$-families in higher
dimension will be explained in Chapter 7.
We introduce the Borcea-Voisin tower and the Viehweg-Zuo tower and realize that only a small number
of families of cyclic covers of $\bP^1$ are suitable to start the construction of a Viehweg-Zuo
tower. We will also discuss some methods to find concrete $CM$ fibers at the end of this chapter. In Chapter 8
we will give a modified version of the method of E. Viehweg and K. Zuo to
construct the $CMCY$ family of 2-manifolds given by $\eqref{nis}$. We consider the automorphism
groups of our examples given by $\eqref{exa}$ and $\eqref{nis}$ in Chapter 9. This yields the
further quotients of the families given by $\eqref{exa}$ and $\eqref{nis}$, which are $CMCY$
families of 2-manifolds. We will see that these quotients are endowed with involutions, which
make them suitable for the construction of a Borcea-Voisin tower. Moreover we will construct the
family $\sQ$ of Theorem $\ref{nunuu}$ in Chapter 9. The next chapter is devoted to
the {\em length} of the Yukawa couplings of our examples families (motivated by the question of rigidity)
and the Hodge numbers of their fibers. We finish this chapter with an outlook
onto the possibilities to construct $CMCY$ families of 3-manifolds by quotients of higher order.
In Chapter 11 we use directly the mirror construction of C. Voisin to obtain maximal holomorphic
$CMCY$ families of 2-manifolds, which are suitable for the construction of a holomorphic
Borcea-Voisin tower.

\newpage

\chapter{An introduction to Hodge structures and Shimura varieties}
In this chapter we recall the general facts about Hodge structures and Shimura varieties,
which are needed in the sequel.

\section{The basic definitions}

\begin{definition} \index{Hodge structure}
Let $R$ be a Ring such that $\Z \subseteq R \subseteq \R$. An $R$-Hodge structure is given by an
$R$-module $V$ and a decomposition
$$V\otimes_R \C = \bigoplus\limits_{p,q \in \Z} V^{p,q}$$
such that $\overline{ V^{p,q}} = V^{q,p}$.
\end{definition}

Now let $\BS := Res_{\C/\R}\G_{m,\C}$ be the Deligne torus \index{$\BS$} given by the Weil
restriction of $\G_{m,\C}$.

\begin{proposition} \label{delig}
Let $V$ be an $\R$-vector space. Each real Hodge structure on $V$
defines by
$$z\cdot \alpha^{p,q} = z^{p}\bar z^{q} \alpha^{p,q}$$
for all $\alpha^{p,q} \in V^{p,q}$ an action of $\BS$ on $V \otimes \C$
such that one has an $\R$-algebraic homomorphism
$h: \BS \to \GL(V)$. Moreover by the eigenspace decomposition of $V_{\C}$ with respect to the
characters of $\BS$, any representation given by an algebraic homomorphism $h: \BS \to \GL(V)$
corresponds to a real Hodge structure on $V$.
\end{proposition} 
\begin{proof}
(see \cite{Dat}, 1.1.1\footnote{Note that P. Deligne writes
$$z\cdot \alpha^{p,q} = z^{-p}\bar z^{-q} \alpha^{p,q} \  \  \mbox{instead of} \  \
z\cdot \alpha^{p,q} = z^{p}\bar z^{q} \alpha^{p,q}$$ in \cite{Dat}. But this is only a matter of
the chosen conventions.} )
\end{proof} 

From now on let $V$ be a $\Q$-vector space and let
$$h: \BS \to \GL(V_{\R})$$
be the algebraic homomorphism corresponding to a Hodge structure on $V$. Note that $\BS$ is given
by ${\rm Spec}(\R[x,y,t]/(t(x^2 + y^2) = 1))$ and $S^1$ \index{$S^1$} is the algebraic subgroup
given by
${\rm Spec}(\R[x,y]/(x^2 + y^2 = 1))$. This yields
$$S^1(\R) = \{z \in \C : z \bar z = 1\} \subset \C^*.$$
We consider the exact sequence
$$0 \to \R^* \stackrel{\rm id}{\hookrightarrow} \C^*
\stackrel{z \to z/\bar z}{\longrightarrow} S^1(\R) \to 0,$$
which can be obtained by an exact sequence
\begin{equation} \label{hss}
0 \to \G_{m,\R} \stackrel{w}{\hookrightarrow} \BS \rightarrow S^1
\to 0
\end{equation}
of $\R$-algebraic groups.

\begin{remark}
The homomorphism given by $h \circ w$ is called weight homomorphism.  There exists a
$k \in \Z$ such that $V^{p,q} = 0$ for all $p+q \neq k$, if and only if $h \circ w$ is given by
$r \to r^{k}$. A Hodge structure is of weight $k$, if $h \circ w$ is given by
$r \to r^{k}$.
\end{remark}

\begin{remark}
By Proposition $\ref{delig}$, any (real) Hodge structure on $V_{\R}$
of weight $k$ determines a unique morphism $h_1 :S^1 \to GL(V_{\R})$ given by
$$S^1 \hookrightarrow \BS \stackrel{h}{\to} \GL(V_{\R}).$$
Since $\BS = \G_{m,\R} \cdot S^1$, one can reconstruct $h$ from $h|_{S^1}$ and the weight
homomorphism. By using Proposition $\ref{delig}$ again, one can easily see that there is a
correspondence between Hodge structures of weight $k$ on $V_{\R}$ and representations
$h_1: S^1 \to \GL(V_{\R})$ given by
$$z\cdot \alpha^{p,q} = z^{p}\bar z^{q} \alpha^{p,q}$$
for all $\alpha^{p,q} \in V^{p,q}$, which must satisfy $p+q = k$ for all $V^{p,q} \neq 0$.
\end{remark}

\begin{example} An integral Hodge structure of weight $k$ is given by
$$H^k(X,\C) = H^k(X,\Z) \otimes \C = \bigoplus\limits_{p + q = r} H^{p,q}(X) \ \ \mbox{with} \  \ 
H^{p,q}(X) = H^q(X, \Omega^p_X)$$
for any compact K\"ahler manifold $X$.
\end{example}

\begin{definition} \index{Hodge structure!polarized}
A polarized $R$-Hodge structure of weight $k$ is given by an $R$-Hodge
structure of weight $k$ on an $R$-module $V$ and a bilinear
form $Q: V \times V \to R$, which is symmetric, if $k$ is even, alternating otherwise, and whose
extension on $V\otimes_R \C$ satisfies:
\begin{enumerate}
\item The Hodge decomposition is orthogonal for the Hermitian form $i^kQ(\cdot,\bar \cdot)$.
\item For all $ \alpha \in V^{p,q} \setminus \{0\}$ one has
$$i^{p-q}(-1)^{\frac{k(k-1)}{2}}Q(\alpha ,\bar \alpha ) > 0.$$
\end{enumerate}
\end{definition}

\begin{example}
Let $X$ be a compact K\"ahler manifold. Recall that for $k \leq dim(X)$ the primitive cohomology
$H^{k}(X,\R)_{\rm prim}$ is the kernel of the Lefschetz operator
$H^{k}(X,\R) \to H^{2n-k+2}(X,\R)$ given by
$$\alpha \to \wedge^{n-k+1}(\omega)\wedge \alpha,$$
where $n := dim(X)$, $\omega$ denotes the chosen K\"ahler form and $\alpha \in H^{k}(X,\R)$. By
$$(\alpha ,\beta ) := \int_{X} \wedge^{n-k}(\omega)\wedge \alpha \wedge \beta,$$
one obtains a polarization on $H^{k}(X,\Z)_{\rm prim}$ and hencefore a polarized integral Hodge
structure on $H^{k}(X,\Z)_{\rm prim}$, if $[\omega] \in H^2(X, \Z)$
(see \cite{Voi}, 7.1.2)\footnote{
There is a more general definition of a polarized Hodge structure (see \cite{Dat}, 1.1.10). But
here we will mainly consider Hodge structures given by the primitive cohomology on a
K\"ahler manifold. Moreover we obtain $H^{n}(X,\R)_{\rm prim} = H^{n}(X,\R)$, if $X$ is a curve or
if $X$ is a Calabi-Yau 3-manifold. Hence in these two
cases of our main interest $H^{n}(X,\R_{\rm prim})$ is independent by the chosen K\"ahler form.
Moreover by its definition, the corresponding polarization is independent of the K\"ahler form, if
$k = n$. Thus in these two cases the integral polarized Hodge structure depends only on
the isomorphism class of $X$.}.
\end{example}

\begin{definition}
Let $\Q \subseteq K \subseteq \R$ be a field and $V$ be a $K$-vector space. The Hodge group
$\Hg_{K}(V,h)$ \index{Hodge group} of a $K$ Hodge structure $(V,h)$ is the
smallest $K$-algebraic subgroup $G$ of $\GL(V)$ such that
$$h(S^1) \subset G\times_{K}\R.$$
The Mumford-Tate group $\MT_{K}(V,h)$ \index{Mumford-Tate group} of a $K$ Hodge structure
$(V,h)$ is the smallest $K$-algebraic subgroup $G$ of $\GL(V)$ such that
$$h(\BS) \subset G\times_{K}\R.$$

For simplicity we will write $\Hg(V,h)$ instead of $\Hg_{\Q}(V,h)$ and $\MT(V,h)$ instead of
$\MT_{\Q}(V,h)$.
\end{definition}

\begin{definition}
Let $F$ be a number field. A compact K\"ahler manifold $X$ of dimension $n$ has complex
multiplication ($CM$) over $F$, \index{complex multiplication} if the Hodge group of the $F$ Hodge
structure on $H^{n}(X,F)$ is a torus. We say that $X$ has complex multiplication, if it has
complex multiplication over $\Q$.
\end{definition}

There is another concept of complex multiplication: An
Abelian variety $A$ is of $CM$ type, if it is isogenous to a
fiberproduct of simple Abelian varieties $X_i$ ($i = 1,
\ldots, n$) such that there are fields $K_i \subset
{\rm End}(X_i)\otimes_{\Z} \Q$, which satisfy
$$[K_i:\Q] \geq 2 \cdot dim(X_i).$$

\begin{remark}
If the Abelian variety $A$ is of $CM$ type, the fields $K_i$ are $CM$ fields (i.e. a totally
imaginary quadratic extension of a totally real number field) and satisfy
$$[K_i:\Q] = 2 \cdot dim(X_i).$$
\end{remark}
\begin{proof}(see \cite{Lang}, Theorem $3.1$ and Lemma $3.2$.)\end{proof}

\begin{lemma}
An Abelian variety $A$ is of $CM$ type, if and only if $\Hg(H^1(A,\Q))$ is a torus algebraic group. 
\end{lemma}
\begin{proof}
(see \cite{Mum1})
\end{proof}

Since the Hodge structures on $H^1(C,\Q)$ and $H^1({\rm Jac}(C),\Q)$ are isomorphic, the relation
of our two concepts of complex multiplication is obvious:

\begin{proposition}
A curve $C$ has complex multiplication, if and only if ${\rm Jac}(C)$ is of $CM$ type.
\end{proposition}

\section{Jacobians, Polarizations and Riemann's Theorem}

Let $X$ be a K\"ahler manifold. Consider the following exact sequence:
$$0 \to \Z \to \sO_X \to \sO_X^* \to 0$$
This yields the complex torus
$${\rm Pic}^0(X) = H^1(X,\sO_X)/H^1(X,\Z),$$
which isomorphic to the Jacobian ${\rm Jac}(C)$, if $X$ is a curve $C$. The theory of Abelian
varieties, their Hodge structures and their parameterizing spaces contains several features, which
we will need in the sequel.

\begin{pkt} \label{pucky}
On the homology $H_1(C, \Z)$ of a curve $C$ one can define an intersection pairing.
It is compatible with the polarization on$H^1(C,\Z)$ by a canonical monomorphism $\sigma$, which
assigns to each $\gamma \in H_1(C, \Z)$ the $\alpha \in H^1(C, \C)$, which has the property that
$$\int_C \alpha \wedge \beta = \int_{\gamma}\beta$$
for all $\beta \in H^1(C, \C)$. Thus the homology group $H_1(C, \Z)$ is the dual of
$$H^1(C, \Z)=\sigma(H_1(C, \Z)).\footnote{Note that ${\rm Jac}(C)$ is defined by the quotient of
$H^0(\omega_C)^*$ by the period lattice induced by integration over paths in $H_1(C, \Z)$. Thus
the statement that $H^1(C, \Z) =\sigma(H_1(C, \Z))$ is equivalent to the well-known fact that
${\rm Pic}^0(C) \cong {\rm Jac}(C)$.}$$

By integration over $\C$-valued paths in $H_1(C, \C) := H_1(C, \Z) \otimes_{\Z} \C$, the
$\C$-valued homology $H_1(C, \C)$ is a canonical dual of $H^1(C,\C)$. On $H_1(C, \C)$ the dual
Hodge structure of weight $-1$ is given by the Hodge filtration
$$ 0 \subset H^{0,-1}(C) \subset H_1(C,\C) \  \ \mbox{such that} \  \
H^{0,-1}(C) = H^{0,1}(C)^* \  \ \mbox{and} \  \ H^{-1,0}(C) = H^{1,0}(C)^*$$
with $H^{-1,0}(C) = H_{1}(C, \C)/H^{0,-1}(C)$. Moreover one has 
$$\sigma \circ  h_{-1}(z) = h_1(z) \circ \sigma \  \ \mbox{for all}  \  \ z \in S^1(\R),$$
where $h_{-1}$ and $h_{1}$ denote the corresponding embeddings
$$h_{-1}: S^1 \to \GL(H_1(C, \R)) \  \  \mbox{and} \  \ h_{1}: S^1 \to \GL(H^1(C, \R)).$$
Thus the Hodge groups of these Hodge structures on $H_1(C, \Z)$ and $H^1(C, \Z)$ are isomorphic.
Hence for a study of the Hodge structure on $H^1(C, \Z)$, it is sufficient to consider the
corresponding dual Hodge structure on $H_1(C,\Z)$.
\end{pkt}

Next we consider polarizations on Abelian varieties:

\begin{remark} \label{willwill}
Let $A = W/L$ be a complex $g$-dimensional torus.
There is a canonical isomorphism between $H^2(A,\Z)$ and
$\Z$-valued alternating forms on $L = H_1(A, \Z)$. Moreover for an alternating integral form $E$
on $L$, there is a line bundle $\sL$ on $A$ with $c^1(\sL) =E$, if and only if $E(i\cdot,i\cdot) =
E(\cdot,\cdot)$.
By
$$H(u,v) = E(iu,v)+iE(u,v),$$
we get the corresponding Hermitian form $H$ from $E$ and conversely, given $H$ we obtain $E$ by
$E = \Im H$. (See\cite{LB}, Proposition 2.1.6 and Lemma 2.1.7)

A polarization on an Abelian variety is given by a line bundle $\sL$,
whose Hermitian form $H$, which corresponds to its first Chern class $E$, is positive
definite. The alternating form $E$ of the polarization can be given by the matrix

$$\left(\begin{array}{cc}
0 & D_g\\
 -D_g & 0

\end{array} \right)$$
with respect to a symplectic basis of $L$, where $D_g = \diag(d_1, \ldots, d_g)$ with
$d_i|d_{i+1}$ (see \cite{LB}, $3.\S 1$). The matrix $D_g$ depends on the polarization,
and it is called the type of the polarization. The polarization $E$ on $A$ is principal, if
$D_g = E_g$.

A positive definite Hermitian form $H$ on $W$, which has the property that
$\Im H$ is an integral alternating form on $L$, satisfies that $\Im H(i\cdot,i\cdot) =
\Im H(\cdot ,\cdot )$
resp., is a polarization. Since the Chern class of a line bundle $\sL$ is a polarization, if and
only if $\sL$ is ample
(see \cite{LB}, Proposition 4.5.2.), $H$ yields an ample line bundle. By the Theorem of
Chow, $A$ is algebraic in this case. Moreover if $A$ is an Abelian variety, there is a positive definite
Hermitian form $H$ on $W$ such that $\Im H$ is integral on $L$ (see \cite{Mum3}, $\S1$, too). 
\end{remark}

Now let $V$ denote a $\Q$-vector space of dimension $2g$, $Q$ be a rational alternating bilinear
form on $V$, and $J$ be a complex structure on $V_{\R}$ (i.e. an automorphism $J$ with
$J^2 = -{\rm id}$).

\begin{remark} \label{cplxstr}
It is a well-known fact that there is a correspondence between Hodge structures $h$ on $V$ of
type $(1,0),(0,1)$ and complex structures $J$ on $V_{\R}$ via $h(i) = J$.
\end{remark}

\begin{lemma} \label{jupp}
The complex structure $J$ on $V_{\R}$ corresponds to a polarized Hodge structure $(V,h,Q)$ of type
$(1,0),(0,1)$, if and only if it satisfies
$$Q(J\cdot,J\cdot) = Q(\cdot,\cdot) \  \ \mbox{and} \  \ Q(J\tilde v,\tilde v) > 0$$
for all $\tilde v \in V_{\R}$.
\end{lemma}
\begin{proof}
Let the complex structure $J$ on $V_{\R}$ be given by a polarized Hodge structure of type
$(1,0),(0,1)$ on $V$. Any $\tilde v, \tilde w \in V_{\R}$ can be given by
$$\tilde v = v + \bar v \  \ and \   \ \tilde w = w + \bar w$$
for some $v, w \in H^{1,0}$, where $H^{1,0}$ and $H^{0,1}$ are totally isotropic with respect to
$Q$. Hence:
$$Q(J\tilde v,J\tilde w) = Q(i v,-i \bar w) + Q(-i \bar v,i w)= Q(v, \bar w) + Q(\bar v, w)
=Q(\tilde  v,\tilde  w)$$
Since the Hermitian form given by
$iQ(v,\bar v)$ is positive definite on $H^{1,0}$, one concludes:
\begin{equation} \label{polly}
Q(J\tilde v,\tilde v)= Q( iv -i \bar v, v+\bar v) = Q(iv, \bar v) + Q(-i \bar v, v)
= 2iQ(v, \bar v) > 0
\end{equation}

Conversely assume that $Q(J\cdot,\cdot) > 0$ and $Q(\cdot,\cdot) = Q(J\cdot,J\cdot)$. Thus one has 
$$Q(v,v) = Q(Jv,Jv) = Q(iv,iv) =-Q(v,v)$$
$$\mbox{resp.,} \  \ Q(v,v) = Q(Jv,Jv) = Q(-iv,-iv) =-Q(v,v)$$
for all $v \in H^{1,0}:= Eig(J,i)$ resp., for all $v \in H^{0,1}:= Eig(J,-i)$. Hence $H^{1,0}$
resp., $H^{0,1}$ is isotropic with respect to $Q$. The same calculation as in
$\eqref{polly}$ implies that $iQ(\cdot,\bar \cdot)$ is positive definite on $H^{1,0}$ and
negative definite on $H^{0,1}$. Hence one gets a polarized Hodge structure of type $(1,0),(0,1)$ by
Remark $\ref{cplxstr}$.
\end{proof}

By the preceding lemma and an easy calculation using that $z = a+ib \in S^1(\R)$ implies
$a^2 + b^2 = 1$,\footnote{Let $v, w \in V_{\R}$. The calculation is given by:
$$Q(zv,zw) = a^2Q(v,w) + b^2Q(v,w) + ab(Q(Jv,w) + Q(v,Jw))=$$
$$= Q(v,w) + ab(Q(Jv,w) + Q(Jv,J(Jw)))= Q(v,w) + ab(Q(Jv,w) + Q(Jv,-w)) = Q(v,w)$$} we obtain:

\begin{proposition} 
A polarized Hodge structure of type $(1,0),(0,1)$ on $V$ (where $Q$ denotes the polarization)
induces a faithful symplectic representation
$$h:S^1 \to \Sp(V_{\R},Q).$$
\end{proposition}

\begin{corollary} \label{sympemb}
Let $(V,h,Q)$ be a polarized Hodge structure of type $(1,0),(0,1)$. Then
$$\Hg(V,h) \subset \Sp(V,Q), \  \ \mbox{and} \  \ \MT(V,h) \subset \GSp(V,Q).$$
\end{corollary}

\begin{Theorem} [Riemann] \label{ralfsiegel}
There is a correspondence between polarized
Abelian varieties of dimension $g$ and polarized Hodge structures $(L,h,Q)$
of type $(1,0),(0,1)$ on a torsion-free lattice $L$ of rank $2g$.
\end{Theorem}
\begin{proof}
Let $(L,h,Q)$ be a polarized Hodge structure on a torsion-free lattice $L$ of rank $2g$. By 
$$L \otimes \R \hookrightarrow  L \otimes \C \to H^{0,1},$$
one has an isomorphism $f$ of $\R$-vector spaces. The complex structure of the Hodge structure
turns $L_{\R}$ into a $\C$-vector space. One has $f(\lambda v) = \bar \lambda f(v)$. By $f$, $Q$
may be considered as (real) alternating form on $H^{0,1}$. But it satisfies $Q(iv,v) <0$ for all
$v \in H^{0,1}$. Hence let $E = -Q$. Lemma $\ref{jupp}$ implies that $E(i\cdot,i\cdot) =
E(\cdot,\cdot)$ and
$E(iv,v) >0$ for all $v \in H^{0,1}$. Thus the corresponding Hermitian form is positive definite
(see Remark $\ref{willwill}$) and we have a polarization on the complex torus
$H^{0,1}/L$ and hencefore an Abelian variety.

Conversely take a polarized Abelian variety $(A,E)$, where $A = W/L$. Let $Q :=-E$. By $J = -i$,
one has similar to Lemma $\ref{jupp}$ a complex structure corresponding to a polarized Hodge
structure of type $(1,0),(0,1)$ on $L$. Thus we have obviously obtained the desired
correspondence.
\end{proof}

Since a polarized rational  Hodge structure can be considered as polarized integral Hodge
structure with respect to a fixed lattice, if the polarization on this lattice is integral, one
concludes by Lemma $\ref{jupp}$ and Theorem $\ref{ralfsiegel}$:

\begin{corollary} \label{hmg}
There is a bijection between the sets of polarized
Abelian varieties $A = W/L$ and complex structures on $L\otimes \R$ satisfying
$$Q(J\cdot,J\cdot) = Q(\cdot,\cdot) \  \ \mbox{and} \  \ Q(J v, v) > 0$$
for all $ v \in L \otimes \R$ with respect to an integral alternating form $Q$ on $L$.
\end{corollary}

\begin{remark}
The Jacobian ${\rm Jac}(C)$ of a curve $C$ is isomorphic to
$${\rm Pic}^0(C) = H^{0,1}(C)/H^1(C,\Z).$$
As in the proof of Riemann's Theorem, the polarization of the integral Hodge structure on
$H^1(C,\Z)$ can be identified with a polarization on ${\rm Jac}(C)$. Since the corresponding
intersection form on $H_1(C,\Z)$ can be given by the matrix
$$\left(\begin{array}{cc}
0 & E_g\\
 -E_g & 0
\end{array} \right)$$
with respect to a fixed symplectic basis (follows by \cite{LB}, Chapter 11, $\S1$ for
example), one concludes that this polarization on ${\rm Jac}(C)$ is principal.\footnote{Following
\cite{LB}, Chapter 11 the principal polarization is $c_1(\Theta)$, where the Theta divisor
$\Theta$ is obtained as the image of the Abel-Jacobi map $C^{g-1} \to {\rm Jac}(C)$ via
$(p_1,\ldots,p_{g-1}) \to \sO_C(p_1 + \ldots + p_{g-1} - (g-1)p_0)$ for an arbitrary $p_0 \in C$.}
\end{remark}

\begin{remark}

Two curves are isomorphic, if their Jacobians are isomorphic as principally
polarized Abelian varieties (see \cite{LB}, Torelli's Theorem 11.1.7).
\end{remark}

\section{Shimura data and Siegel's upper half plane}
From now on let $(L,h,Q)$ be a polarized integral Hodge structure of type $(1,0),(0,1)$ on a
torsion-free lattice $L$ of rank $2g$ and $V := L \otimes \Q$. For simplicity we assume that $Q$
is given by
\begin{equation} \label{einfach} J_0 =\left(\begin{array}{cc}
0 & E_g \\
-E_g & 0
\end{array} \right)\end{equation}
with respect to a symplectic basis of $L$.

Now we construct Siegel's upper half plane $\fh_g$ (at present as homogeneous
space):

\begin{construction} \label{1.25} \index{Siegel's upper half plane}
An embedding $h:S^1 \to \Sp(V,Q)_{\R}$ obtained by a polarized integral Hodge structure $(L,h,Q)$
of type $(1,0),(0,1)$ corresponds via $h(i)$ to a positive complex structure (i.e. a complex
structure $J$ such that $Q(Jv,v) > 0$) $J \in \Sp(V,Q)_{\R}$ for all $v \in V_{\R}$. By conjugation,
$\Sp(V,Q)_{\R}$ acts transitively on the positive complex structures $J \in \Sp(V,Q)_{\R}$ (see
\cite{Milne}, page 67\footnote{A positive complex structure in the sense of our
notation is a negative complex structure in the sense of the notation of \cite{Milne}, and vice
versa (Here $J$ is negative, if $Q(Jv,v) > 0$ for all $v \in V_{\R}$.). But via
$J \leftrightarrow -J$ we have a correspondence between
negative and positive complex structures commuting with the actions of $\Sp(V,Q)_{\R}$ on
positive and negative complex structures.}) and hencefore it acts transitively on the set
of polarized integral Hodge structures $(L,h,Q)$ of type $(1,0),(0,1)$.
Let $K$ be the subgroup of $\Sp(V,Q)(\R)$, which leaves a
fixed $h(S^1)$ stable by conjugation. Then Corollary $\ref{hmg}$ allows to identify the set of
points of the homogeneous space $\mathfrak{h}_g := \Sp(V,Q)(\R)/K$ with the set of principally
polarized Abelian varieties of dimension $g$ with symplectic basis.
\end{construction}

We want to endow $\mathfrak{h}_g$ with the structure of a Hermitian symmetric domain.
But first let us recall some needed facts about groups:

\begin{definition}
A Lie algebra $\mathfrak{g}$ is simple, if $dim(\mathfrak{g}) >1 $ and $\mathfrak{g}$ contains no non-trivial
ideals. A connected Lie group $G$ is simple, if its Lie algebra is simple.
\end{definition}

\begin{remark}
Let $G$ be an algebraic group. The quotient $G^{\ad}$ is the image of the adjoint representation
of $G$ on its Lie algebra $\mathfrak{g}$. It is a well-known fact that $G$ has the following algebraic
subgroups:

The derived group $G^{\der}$ of $G$ is the
subgroup of $G$ generated by its commutators. By $Z(G)$, we denote the center of $G$. The Radical
$R(G)$ is the maximal connected normal solvable subgroup of $G$. Its unipotent radical
$R_u(G)$ is given by
$$R_u(G) :=\{g \in R(G)|g \mbox{ is unipotent}\}.$$
\end{remark}

\begin{definition}
Let $G$ be an algebraic group. The group $G$ is a reductive, if $R_u(G) = \{e\}$, and semisimple,
if $R(G) = \{e\}$.
\end{definition}

\begin{proposition} \label{prodar}
Let $G$ be a connected algebraic group. It is reductive, if and only if it is the almost direct
product of a torus and a semisimple group. These groups can be given by $Z(G)$ and $G^{\der}$.
\end{proposition}
\begin{proof}
(see \cite{Sat}, Chapter {\bf I}. $\S3$ for the first statement and \cite{Bor} {\bf IV}. 14.2 
for the second statement)
\end{proof}

\begin{remark}

It is a well-known fact that the Lie algebras of an $\R$-algebraic group $G$ and the Lie group
$G(\R)$ coincide. Moreover $G$ is semisimple, if and only if its Lie algebra $\mathfrak{g}$ is a
direct sum of simple Lie algebras.
\end{remark}

\begin{remark}
\begin{enumerate}
\item Let $G$ be a reductive $\Q$-algebraic group with largest
commutative quotient $T$. One has the obvious exact sequences:
$$1 \to G^{\der} \to G \to T \to 1$$
$$1 \to Z(G) \to G \to G^{\ad} \to 1 $$
$$1 \to Z(G^{\der}) \to Z(G) \to T \to 1$$
\item The exact sequences induce a natural isogeny $G^{\der} \to  G^{\ad}$ with kernel
$Z(G^{\der})$ (see \cite{De2}, 1.1.)
\end{enumerate}
\end{remark}

\begin{lemma}
If $G$ is a semisimple connected Lie group with trivial center, then it is isomorphic to a direct
product of simple adjoint groups.
\end{lemma}
\begin{proof}
By the assumptions and \cite{Helga}, {\bf II}. Corollary $5.2$, $G$ coincides with its adjoint
group $G^{\ad} \cong G/Z(G)$. Since the Lie algebra $\mathfrak{g}$ of $G$ is the direct sum of
simple Lie algebras, $\mathfrak{g}$ is the Lie algebra of a certain direct product of simple groups,
too. Without loss of generality one can assume that these simple Lie groups have trivial centers.
Recall that the adjoint group depends only on the Lie algebra. Thus this product of
simple groups is isomorphic to its adjoint, which is the adjoint of
$G$ coinciding with $G$.
\end{proof}

Let $\mathfrak{g}$ be a complex Lie algebra. By
$\R \hookrightarrow \C$, $\mathfrak{g}$ can be considered as a
real Lie algebra $\mathfrak{g}^{\R}$ with complex structure $J$ given by the scalar multiplication
with $i$. A real form of $\mathfrak{g}^{\R}$ is a
subalgebra $\mathfrak{g}_0$ of $\mathfrak{g}^{\R}$ such that
$$\mathfrak{g}^{\R} = \mathfrak{g}_0 \oplus J\mathfrak{g}_0.$$
A real form is called compact, if its (real) adjoint group is a compact Lie group.

Now let $\mathfrak{g}$ be a semisimple real Lie algebra. Any involution
$\iota$ on $\fg$ endows $\mathfrak{g}$ with a decomposition into two
eigenspaces $\mathfrak{t} = \Eig(\iota,1)$ and $\mathfrak{p} = \Eig(\iota,-1)$. The involution $\iota$
is called a Cartan involution, if $\mathfrak{u} := \mathfrak{t} + i \mathfrak{p}$ is a compact
real form of the complexified semisimple Lie algebra $\mathfrak{g}_{\C}$. An involutive algebraic
automorphism $\tilde \iota$ on a connected $\R$-algebraic group $G$ is a Cartan involution on $G$,
if the (real) Lie subgroup of $G(\C)$ corresponding to $\mathfrak{u} := \mathfrak{t} + i \mathfrak{p}$ is
compact, where again $\mathfrak{t} = \Eig(\tilde \iota,1), \mathfrak{p} = \Eig(\tilde \iota,-1) \subset
Lie(G)$.

\begin{proposition} \label{dxi}
A connected $\R$-algebraic group is reductive, if and only if it has a Cartan involution. Any two
Cartan involutions are conjugate by an inner automorphism.
\end{proposition}
\begin{proof}
(See \cite{Sat}, ${\bf I}$. 4.3)
\end{proof}

\begin{example} \label{exi}
The group $\Sp_{\R}(V,Q) \cong \Sp_{\R}(2g)$ is reductive. The inner automorphism of
$\Sp_{\R}(2g)$ given by $J_0$ (see $\eqref{einfach}$) is a Cartan involution (see \cite{Helga},
{\bf VIII}. Exercise $B.2$).

Since $\Sp(V,Q)$ is defined for the alternating form
$Q$ given by the matrix $J_0$, one can easily calculate that $J_0$ satisfies that
$Q(J_0v,v) > 0$ for all $v \in V_{\R}$. Hence by Construction $\ref{1.25}$, the complex structure
$J_0$ corresponds to a point of $\mathfrak{h}_g$. Moreover one can
easily see that the Cartan involution of $J_0$ fixes exactly its isotropy subgroup $K$ with
respect to the action of $\Sp(V,Q)(\R)$ on $\mathfrak{h}_g$. Since all points
of $\mathfrak{h}_g$ correspond to complex structures conjugate to $J_0$, the corresponding involutions,
which are conjugate by inner automorphisms, are the Cartan involutions fixing the respective
isotropy subgroups.
\end{example}

\begin{definition}
Let $M$ be a $\sC^{\infty}$ manifold. A Riemannian structure on $M$ is a symmetric tensor field
$Q$ of type $(0,2)$, which yields a positive definite non-degenerate bilinear form on $T_p(M)$ for
all $p \in M$.
\end{definition}

\begin{definition}
Let $M$ be a connected $\sC^{\infty}$ manifold with an almost complex structure $J$. A Riemannian
structure $g$ on $M$ is a Hermitian structure, if $g(J\cdot,J\cdot) = g(\cdot,\cdot)$.
\end{definition}

\begin{definition}
Let $D$ be a connected complex manifold with a Hermitian structure. It is a Hermitian symmetric
space, if each point is an isolated fixed point of an involutive holomorphic isometry on $D$.
Let $\Hol(D,g)$ denote the Lie group of holomorphic isometries.

Moreover let $\Hol(D,g)^+$  be a
non-compact semisimple Lie group endowed with an involution $\iota$, which induces by its
differential a Cartan involution on $Lie(\Hol(D,g))$, and $K_{\iota} \subset \Hol(D,g)^+$ be the
subgroup, on which $\iota$ acts as ${\rm id}$. The Hermitian symmetric space $D$ is a
Hermitian symmetric domain, if the isotropy group $K$ of one point $p \in D$ satisfies
$K_{\iota}^+ \subseteq K\subseteq K_{\iota}$. 
\end{definition}

\begin{definition}
A bounded symmetric domain $D$ is an open, bounded, connected
submanifold $D$ of $\C^N$, which has the property that each $p\in
D$ is an isolated fixed point of an involutive holomorphic
diffeomorphism onto itself.
\end{definition}

\begin{Theorem}
Each bounded symmetric domain $D$ can be equipped with a unique
Hermitian metric (called Bergman metric), which turns $D$ into a Hermitian symmetric domain.
Conversely each Hermitian symmetric domain has a holomorphic diffeomorphism onto a bounded
symmetric domain.
\end{Theorem}
\begin{proof}
The correspondence between Hermitian symmetric domains and bounded symmetric domains is given by
\cite{Helga}, Theorem {\bf VIII}, 7.1. The uniqueness of the Bergman metric follows from the fact
that each holomorphic diffeomorphism between bounded symmetric domains is an isometry with respect
to the Bergman metric (see \cite{Helga}, Proposition ${\bf VIII}$, 3.5.).
\end{proof}

\begin{definition}
A Shimura datum $(G,h)$ is given by a reductive $\Q$-algebraic group $G$ and
a conjugacy class of homomorphisms $h : \BS \to G_{\R}$ of algebraic groups
satisfying:
\begin{enumerate}
\item The inner automorphism of $(\ad \circ h)(i)$ on $G_{\R}^{\ad}$ is a Cartan involution.
\item The adjoint group $G^{\ad}$ does not have any direct $\Q$-factor $H$ such the Cartan
involution of $(1)$ is trivial on $H_{\R}$.
\item The representation $(\ad \circ h)(\BS)$ on $Lie(G)_{\C}$ corresponds to a Hodge structure of
the type $(1,-1) \oplus (0,0) \oplus (-1,1)$.
\end{enumerate}
\end{definition}

\begin{example} \label{obehalb}
The connected algebraic group $\GSp_{\Q}(2g)$ is the almost direct product of its central
torus $\G_{m,\Q}$ and its simple derived group $\Sp_{\Q}(2g)$. Hence it is reductive.
By Construction $\ref{1.25}$ and Example $\ref{exi}$, we have a conjugacy class of complex
structures, which corresponds to a conjugacy class of homomorphisms $h:\BS \to \GSp_{\R}(2g) $
satisfying the condition $(1)$ of a Shimura datum. The adjoint
group $\GSp_{2g}(\Q)^{\ad} = \Sp_{2g}(\Q)^{\ad}$ has only one direct simple factor on
which the Cartan involution above is not trivial. Hence condition $(2)$ of the Shimura datum is
satisfied. Since the center of
$\GSp_{\R}(2g)$ is given by $\G_{m,\R}$ (see \cite{Milne}, page 66), the kernel of the adjoint
representation on $Lie(\GSp_{\R}(2g))$ of any $h(\BS)$ in the conjugacy class is given by
$\G_{m,\R}$. Since $h(a+ib) = aE_{2g}+bJ$, each $g\in \GSp_{2g}(\R)$ commutes with $J$, if and only
if it commutes with each element of $\BS(\R)$. Hence on the complexified
eigenspace $(\mathfrak{p}_0)_{\C}$ with eigenvalue $-1$ with respect to the Cartan involution, $\BS$
acts by the characters $z/\bar z$ and $\bar z/z$.  This corresponds to a Hodge structure of the
type $(1,-1) \oplus (0,0) \oplus (-1,1)$ on $Lie(\GSp_{\R}(2g))$. Hence condition $(3)$ is
satisfied and the conjugacy class of $h_0:\BS \to \GSp_{\R}(2g)$ with $h_0(i) = J_0$ is a Shimura
datum.
\end{example}

\begin{remark}
Let $(G,h)$ be a Shimura datum. Since $G^{\ad} = G/Z(G)$ and $K$ is the centralizer of $h(\BS)$,
one has $G(\R)^+/(K(\R)\cap G(\R)^+) = G^{\ad}(\R)/\ad_{G}(K(\R))$.

The Cartan involution $int(\ad \circ h)(i)$ fixes exactly $\ad_{G}(K)$. By \cite{Sat}, {\bf I},
Corollary 4.5, the subgroups of a connected $\R$-algebraic reductive group on which a Cartan
involution acts as ${\rm id}$ are maximal compact. Hence $\ad_{G}(K(\R))$
is a maximal compact subgroup.
\end{remark}

\begin{remark}
Let $(L,h,Q)$ be a polarized integral Hodge structure of type $(1,0),(0,1)$ with corresponding
complex structure $J\in Sp_{2g}(\R)$, where $Q$ is given by $\eqref{einfach}$. The Cartan involution
corresponding to $J$ leaves $\Hg(L,h)_{\R} \subset \Sp_{\R}(2g)$ stable. Hence by \cite{Sat}
Theorem {\bf I}. $4.2$, the group $\Hg(L,h)_{\R}$ has a Cartan involution and $\Hg(L,h)$ is
reductive.
\end{remark}

Next we need to recall the definition of a variation of Hodge structures ($VHS$):

\begin{definition} \index{variation of Hodge structures}
Let $D$ be a complex manifold and $R$ be a ring such that $\Z\subseteq R \subseteq \R$. A
variation $\sV$ of $R$-Hodge structures of weight $k$ over $D$ is given by a local system
$\sV_{R}$ of $R$-modules of finite rank and a filtration $\sF^{\bullet }$ of
$\sV_{\sO_{D}}$ by holomorphic subbundles such that:
\begin{enumerate}
\item Griffiths transversality condition holds.
\item $(\sV_{R,p},\sF^{\bullet }_p)$ is an $R$-Hodge structure of weight $k$ for all $p \in D$.
\end{enumerate}

The variation $\sV$ of Hodge structures is polarized, if there is
a flat (i.e. locally constant) bilinear form $Q$ on $\sV_{R}$ such that
$(\sV_{R,p},\sF^{\bullet }_p ,Q_p)$ is a polarized $R$-Hodge structure of weight $k$ for all
$p \in D$.
\end{definition}

\begin{Theorem}
Let $h : \BS \to G$ be a Shimura datum, $t \in \N$ and $K$ denote the centralizer of $h(\BS)$. Then each
connected component $D^+$ of $D = G(\R)/K(\R)$ has a unique structure of a Hermitian symmetric domain.
These domains are isomorphic, where the connected component of the group of holomorphic isometries
is given by the quotient of $G^{\ad}(\R)$ by its direct compact factors. Each representation
$\rho:G_{\R} \to \GL_{\R}(t)$ yields a holomorphic variation $(\R^t,\rho \circ h)_{h \in D}$ of
Hodge structures on $D$.
\end{Theorem}
\begin{proof}
(See \cite{Dat}, 2.1.1.)
\end{proof}

\begin{remark}
The Lie group $\GSp_{2g}(\R)$ has two connected components. One component consists of matrices with
positive determinant and the other consists of matrices with negative determinant. Hence the
corresponding homogeneous space $D$ parametrizing the elements of the conjugacy class has two
connected components. Since $\GSp_{2g}(\R)^+$ is a product of $\Sp_{2g}(\R)$ and $\G_{m}(\R)^+$,
where $\G_{m}(\R)^+$ is contained in the stabilizers of all points, the
corresponding connected homogeneous space may be identified with $\mathfrak{h}_g$ such that the
preceding Theorem endows $\fh_g$ with the structure of a Hermitian symmetric domain. By the
natural representation of $\GSp_{\R}(2g)$ on $\R^{2g}$, $\fh_g$ is endowed with the natural
holomorphic variation of Hodge structures of type $(1,0),(0,1)$.
\end{remark}

\section{The construction of Shimura varieties}
In the preceding section we have seen that a Shimura datum yields a bounded symmetric domain.
This is the first step of the construction of a Shimura variety. For completeness we sketch the
construction of a Shimura variety in this section. But later we will only need to use the
language of Shimura data and their associated bounded symmetric domains.

\begin{definition}
Let $G$ be a $\Q$-algebraic group. An arithmetic subgroup $\Gamma$ of $G(\Q)$ is
a group, which is commensurable with $G(\Z)$.

A subgroup $\Gamma$ of a connected Lie group $H$ is arithmetic, if there is
a $\Q$-algebraic group $G$, an arithmetic subgroup $\Gamma_0$ of $G(\Q)$ and a surjective
homomorphism $\eta : G(\R)^+ \to H$ of Lie groups with compact kernel such that
$\eta(\Gamma_0) = \Gamma$.
\end{definition}

The second step of the construction of a Shimura variety is given by the following theorem:

\begin{Theorem}
[of Baily and Borel] Let $D$ be a bounded symmetric domain, and $\Gamma$ be an
arithmetic subgroup of $\Hol(D)^+$. Then the quotient $\Gamma \backslash D$ can be endowed with a
structure of a complex quasi-projective variety. This structure is unique, if $\Gamma$ is
torsion-free.
\end{Theorem}
\begin{proof}
(see \cite{Dat}, 2.1.2. (or \cite{BB} for the construction of the structure of a complex variety))
\end{proof}

Next one needs the ring of finite ad\`eles\footnote{One reason for the introduction of ad\`ele
rings is given by the fact that one wants to have canonical models of Shimura varieties over
number fields in number theory. We do not use this here, but we write it down for completeness},
which is given by
$$\A^f = \Q \otimes_{\Z} \prod\limits_p \Z_p,$$
where $p$ runs over all prime numbers. Hence $\A^f$ is the subring
of $\prod \Q_p$ consisting of the $(a_p)$ such that $a_p \in \Z_p$
for almost all $a_p$. Now let $(G,h)$ be a Shimura datum, which gives the bounded symmetric domain
$D^+$ by a connected component of the conjugacy class $D$ of $h$ and
$K$ be a compact open subgroup of $G(\A^f)$.

\begin{definition}
Let $G$ be a $\Q$-algebraic group. A principal congruence subgroup of $G(\Q)$ is 
$$\Gamma(n) := \{g \in G(\Z) | g \equiv E_g \mbox{ mod } n\}$$
for some $n \in \N$. A congruence subgroup of $G(\Q)$ is a
subgroup $\Gamma$ containing $\Gamma(n)$ such that
$[\Gamma:\Gamma(n)] < \infty$ for some $n \in \N$.
\end{definition}

\begin{lemma}
Let $K$ be a compact open subgroup of $G(\A^f)$. Then $\Gamma := K\cap G(\Q)$ is a congruence
subgroup of $G(\Q)$.
\end{lemma}
\begin{proof}
(see \cite{Milne}, Proposition 4.1)
\end{proof}

The Shimura variety ${\rm Sh}_{K}(G,h)$ \index{Shimura variety} is given by the double quotient
$${\rm Sh}_{K}(G,h) := G(\Q)\backslash D \times G(\A^f)/K
:= G(\Q)\backslash (D \times (G(\A^f)/K)).$$

\begin{proposition}
Let $K$ be a compact open subgroup of $G(\A^f)$, $C := G(\Q)\backslash G(\A^f)/K,$ and
$\Gamma_{[g]} = gKg^{-1}\cap G(\Q)^+$ for some $[g] \in C$. Then one has
$${\rm Sh}_K(G,h)= \bigsqcup\limits_{[g] \in C} \Gamma_{[g]}\backslash D^+.$$
\end{proposition} \begin{proof}
(see \cite{Milne}, Lemma 5.13)
\end{proof}

Hence the preceding proposition and the Theorem of Baily and Borel endow ${\rm Sh}_{K}(G,h)$
with the structure of an
algebraic variety. By \cite{Milne}, Proposition 3.2, the surjection
$G \to G^{\ad}$ maps a congruence subgroup of $G$
onto an arithmetic subgroup of $G^{\ad}$. Now we consider compact open subgroups with the
property that the resulting arithmetic subgroups on $G^{\ad}(\R) = \Hol(D^+,g) = \Hol(D^+)$ are
torsion-free. Recall that the structure
of a complex quasi-projective variety on the quotient of a bounded symmetric domain by a
torsion-free arithmetic group is unique. If $K' \subset K$, we have a natural morphism
\begin{equation} \label{morry}
{\rm Sh}_{K'}(G,h) \to {\rm Sh}_{K}(G,h). 
\end{equation}
By the projective limit running over all compact open $K \subset G(\A^f)$ proving a torsion-free
arithmetic group on $G^{ad}(\R)$, which is given via $(\ref{morry})$, we obtain the Shimura
variety\footnote{Some authors denote only ${\rm Sh}(G,h)$ as Shimura variety.}
$${\rm Sh}(G,h) = \lim\limits_{\longleftarrow }  {\rm Sh}_{K}(G,h).$$

\section{Shimura varieties of Hodge type}
Now we know how to construct a Shimura variety. Hence next we construct the Shimura varieties
resp., Shimura data, which we will need.

\begin{definition} \index{Shimura variety!of Hodge type}
A Shimura datum $(G,h)$ is of Hodge type, if there is an embedding
$\rho :G \hookrightarrow \GSp_{2g, \Q}$ such that one has the Shimura datum of Example
$\ref{obehalb}$ by
$$\BS \stackrel{h}{\hookrightarrow } G_{\R} \stackrel{\rho_{\R}}{\hookrightarrow} \GSp_{2g, \R}.$$
A Shimura variety $SH$ is of Hodge type, if it is obtained by a Shimura datum $(G,h)$ of Hodge
type. 
\end{definition}

From now on we will write $K$, if mean the algebraic subgroup of $G_{\R}$ given by the
centralizer of $h(\BS)$ resp., $h(S^1)$. Moreover for simplicity we will write $K$ instead of
$K(\R)$ by an abuse of notation. In the respective situation the respective meaning will be clear.
 
\begin{construction}
Let $(V,h,Q)$ be a polarized $\Q$-Hodge structure of type $(1,0),(0,1)$.
The conjugacy class of the representation $h:\BS \to \GSp_{2g, \R}$
is the Shimura datum of Example $\ref{obehalb}$. Hence the adjoint representation of $\BS$
on $Lie(\MT(V,h))_{\C} \subset Lie(\GSp(V,E))_{\C}$ induces a Hodge structure of the
same type (or of the type $(0,0)$). Moreover the same arguments imply that the inner automorphism
corresponding to $(\ad \circ h)(i)$ is a Cartan involution on $\MT_{\R}^{\ad}(V,h)$. Hence
$\MT(V,h)$ is reductive. Thus it remains to show that $\MT(V,h)^{\ad}$ does not
have any non-trivial direct $\Q$-factor $H$ on which the Cartan involution is trivial:

Let $H$ be a simple direct $\Q$-factor of $\MT(V,h)^{\ad}$ with trivial Cartan
involution. We have a surjection
$$s:\MT(V,h) \stackrel{\ad}{\longrightarrow}  \MT(V,h)^{\ad} \stackrel{pr_H}{\longrightarrow} H,$$
which is obviously a homomorphism of $\Q$-algebraic groups. Hence the kernel
$\tilde K$ of $s$ is a $\Q$-algebraic group. The complex structure $J$, which satisfies that
$\ad(J)$ is the Cartan involution, satisfies that all elements of the adjoint group $H_{\R}$
commute with $\ad(J)$. Hence $J$ is
contained in $\tilde K_{\R}$. Thus $Lie(H)_{\C}$ is contained in the Lie sub-algebra of
$Lie_{\C}(\MT(V,h))$ on which $\BS$ acts by the character 1. Hence one has
$h(\BS) \subset \tilde K_{\R}$, which implies $\tilde K = \MT(V,h)$ resp., $H = \{e\}$.

Hence we obtain a Shimura datum $h:\BS \to MT(V,h)_{\R}$ of Hodge type.
\end{construction}

\begin{lemma} \label{idhggr}
$$\Hg(V,h) = (\MT(V,h) \cap \SL(V))^0$$
\end{lemma}
\begin{proof}
By the natural multiplication, we have a morphism
$$m: \Hg(V,h) \times \G_{m,\Q} \to \MT(V,h)$$
with finite kernel. The Zariski closure $Z$ of $m(\Hg(V,h) \times \G_{m,\Q})$ in $\MT(V,h)$ is
an algebraic subgroup of $\MT(V,h)$. Moreover one has that $h(\BS) \subset Z_{\R} \subset
\MT_{\R}(V,h)$. Hence $Z = \MT(V,h)$.

Since all homomorphisms $f: G \to G'$ of algebraic groups over algebraically closed fields
satisfy $f(G) = \overline{f(G)}$ (see \cite{Abra}, Satz 2.1.8), we have the equality
$$\Hg_{\bar \Q}(V,h) \cdot \G_{m,\bar \Q} = Z_{\bar \Q} = \MT_{\bar \Q}(V,h).$$
Now let $M \in \MT(V,h)(\bar \Q) \cap \SL(V)(\bar \Q)$. It is given by a product $N \cdot M_1$
with $N \in \G_{m}(\bar \Q)$ and $M_1 \in \Hg(V,h)(\bar \Q)$. Since $\Hg(V,h)(\bar \Q) \subset
\SL(V)(\bar \Q)$, one concludes $N \in \G_{m}(\bar \Q) \cap \SL(V)(\bar \Q) = \mu_n(\bar \Q)$,
where $\dim V = n$. If
and only if $N \in Hg(V,h)(\bar \Q)$, one has $M \in \Hg(V,h)(\bar \Q)$. Hence by the fact that
$\mu_n(\bar \Q)$ is finite, one obtains the statement. 
\end{proof}

\begin{remark}
For a polarized Hodge structure of weight 1 of a curve of genus $g$, we have a natural
embedding $\Hg(V,h) \subset \Sp_{\Q}(2g)$. Since $\mu_{2g}(\bar \Q)$ is not a subgroup of
$Sp_{2g}(\bar \Q)$ for $g > 1$ and for $g = 1$ one has $\mu_{2}\subset h(S^1)$, we obtain the
equality
$$\Hg(V,h) = \MT(V,h) \cap \SL(V)$$ only in the case of a genus one curve. 
\end{remark}

\begin{remark}
Now assume that $(V,h,Q)$ is a polarized $\Q$-Hodge structure of type $(1,0),(0,1)$. Since $\MT(V,h)$
is reductive and $\MT(V,h)^{\der}$ is semisimple in this case, one concludes by Lemma
$\ref{idhggr}$ that $\MT(V,h)^{\der} = \Hg(V,h)^{\der}$. Thus one has that
$\MT(V,h)^{\ad}(\R) = \Hg(V,h)^{\ad}(\R)$. Hence by the preceding construction
$\Hg(V,h)^{\ad}(\R)$ is the identity component of the holomorphic isometry
group of a Hermitian symmetric domain. The isotropy group of a point is given by a maximal
compact subgroup of $\Hg(V,h)^{\ad}(\R)$ fixed by the Cartan involution on $\Hg(V,h)^{\ad}_{\R}$
of the corresponding complex structure $J \in \Hg(V,h)(\R)$. Hence one can consider
$h|_{S^1}: S^1 \to \Hg(V,h)_{\R}$ as Shimura datum, too.
\end{remark}

Now we
construct the holomorphic family of Jacobians over $\Hg(V,h)(\R)/K$ corresponding to the
$VHS$ induced by the embedding $\Hg(V,h) \hookrightarrow \Sp_{\Q}(2g)$, where $(V,h)$ is of type
$(1,0),(0,1)$.\footnote{This construction
and the rest of this section are similar to \cite{Mum2}, $\S 1$ with some technical differences.}

\begin{construction} \label{5.50}
Let $(L,h,Q)$ be a polarized $\Z$-Hodge structure of type $(1,0),(0,1)$ with $V := L_{\Q}$ as
before and $\{v_1, \ldots, v_g, w_1, \ldots, w_g\}$ be a symplectic basis of $L$
with respect to $Q$. For example it may be given on $L := H^1(C,\Z)$, where $C$ is a curve of
genus $g$. One has that
$\Hg(V,h) \subset \Sp(V,Q)$. Let $K \subset \Hg(V,h)(\R)^+$ be the centralizer of $h(S^1(\R))$.
Thus $\Hg(V,h)(\R)^+/K$ is a Hermitian symmetric domain as we have seen.
Consider the linearly independent set $B = \{[w_1], \ldots, [w_g]\} \subset H^{0,1}$, which
generates the real subvector space $W$. Now $iW$ is obviously generated by
$\{[Jw_1], \ldots, [Jw_g]\}$. 
The principal polarization $H$ of the Abelian variety $A = H^{0,1}/L$ is given
by the corresponding alternating form $E = -Q$ as in the proof of Theorem $\ref{ralfsiegel}$.
Since $E$ vanishes on  $W$, the principal polarization $H$ given by $H = E(i.,.)+iE(.,.)$ vanishes
on the complex vector space $W \cap iW$, too. Hence $W \cap iW= 0$. Thus the fact that
$Span_{\R}(v, Jv)$ is mapped to
$Span_{\C}([v])$ implies that $B$ is a $\C$-basis of $H^{0,1}$. Hence the period matrix of the
corresponding Abelian variety may be given by $(Z,E_g)$, where
the columns of $Z$ are given by the $[v_i]$ in their coordinates with respect to $B$.

Thus the
embedding $H^{1,0} \hookrightarrow V_{\C}$ is given by the matrix $(-E_g,Z^t)^t$. Since we have a
holomorphic variation of Hodge structures, this matrix varies holomorphically. Thus the period
matrices of the corresponding Abelian varieties vary holomorphically, too. Hence the corresponding
action of $L$ on $H^{0,1} \times \Hg(V,h)(\R)/K$ is holomorphic and we obtain a holomorphic
family of Abelian varieties over $\Hg(V,h)(\R)/K$.
\end{construction}

Now recall that our main interest is not the theory of Shimura varieties, but families with a
dense set of $CM$ points defined below: 

\begin{definition} \index{$CM$ point}
Let $D$ be a complex manifold and $\sV$ be a holomorphic variation of rational Hodge structures. A
point $p \in D$ is a $CM$ point with respect to $\sV$, if $\sV_p$ has a commutative Hodge group.

Let $\sX \to D$ be a holomorphic family of complex manifolds. A point $p \in D$ is a $CM$ point
with respect to $\sX$, if $\sX_p$ is a $CM$ fiber resp., $\sX_p$ has a complex multiplication.
\end{definition}

Now we give criteria for dense sets of $CM$ points, which imply that the family of Abelian
varieties over $\Hg(V,h)(\R)/K$ of Construction $\ref{5.50}$ has a dense set of $CM$ fibers:

\begin{lemma} \label{tija}
Let $(G,h)$ denote a Shimura datum. If the connected component $G(\R)/K$ contains a $CM$
point with respect to a $VHS$ induced by some closed embedding $G \to \GL(W)$ for some $\Q$-vector
space $W$, then the set of $CM$ points of the same type with respect to the same $VHS$ is dense in
$D$.
\end{lemma}
\begin{proof}
We have two cases. Assume that $G$ is a $\Q$-algebraic torus. In this case $G(\R)/K$ consists
of one point. The fact that we have a closed embedding $G \hookrightarrow \GL(W)$ implies that the
Hodge group of the Hodge structure over this point is a subtorus of the torus $G$.

In the other case $G$ is not a $\Q$-algebraic torus. By the assumptions, we
have a $CM$ point in $G(\R)/K$ with respect to the $VHS$ induced by some closed embedding
$G \to \GL(W)$. This implies that
$G$ contains a $\Q$-algebraic torus $T$ such that the conjugacy class of
$h:\BS \to G_{\R}$ contains an element, which factors through $T_{\R}$. By our preceding
construction,
the stabilizer of the $CM$ point $[s_0]_K \in G(\R)/K$ is given by $s_0Ks_0^{-1}$. Thus one
can replace $K$ by $s_0Ks_0^{-1}$. In this case the fact that the $VHS$ is unduced by an
embedding $G \hookrightarrow \GL(W)$ implies that the Hodge group of the Hodge structure over
$[e]$ is a subtorus of $T$. Hence $[e]$ is a $CM$ point with respect to this $VHS$, and any
$s \in G(\Q) \subset G(\R)$ has the property that it
is mapped to a $CM$ point, too. By the Real Approximation Theorem, $G(\Q)$ lies dense in the
manifold $G(\R)$ for all connected affine $\Q$-algebraic groups $G$. Since the quotient map is
continuous, the set of $CM$ points in $G(\R)/K$ is dense.
\end{proof}

\begin{Theorem} \label{commul}
Let $(G,h)$ denote a Shimura datum. The set of $CM$ points with respect to the $VHS$ induced
by some closed embedding $G \to \GL(W)$ for some $\Q$-vector space $W$ is dense in $G(\R)/K$.
\end{Theorem}
\begin{proof}
By the preceding lemma, we have only to show that there exists one $CM$ point on $G(\R)/K$. By
the closed embedding $G \to \GL(W)$, each $\Q$-algebraic torus of $G$ can be identified with a
$\Q$-algebraic torus of $\GL(W)$. Thus the existence of a $CM$ point is
equivalent to the statement that there is a $h: \BS \to G_{\R}$ in this $VHS$, which factors
through a $\Q$-algebraic torus of $G$.

Now let $T$ be a
maximal ($\Q$-algebraic) torus of $G$. The centralizers of the maximal tori (resp., the
Cartan subgroups) of a reductive group are the maximal tori (see
\cite{Bor}, IV. 13.17.). The torus $T_{\R}$ is contained in a maximal torus $T_M$
of $G_{\R}$, which has the property that each point of $T_M$ is contained in the
centralizer of $T_{\R}$ resp., in the centralizer of $T$. Thus the torus $T_{\R}$ is in fact
maximal in $G_{\R}$.

The Cartan subgroups, i.e. the centralizers of the maximal tori, which are the maximal tori in our
 case, are conjugate (see \cite{Bor}, IV. 12.1.).  The stabilizer of the point given by $h$ in
$G_{\R}/K$ is the centralizer $K$ of $h(\BS)$. The center of
$K$, which is a torus contained in a maximal torus $T_1$, contains
obviously $h(\BS)$ resp., we have a maximal torus $T_1$ containing
$h(\BS)$, where $T_1 \subset K$. Thus by the fact that $T_1$ is
conjugate to $T_{\R}$ by some element $s_0$ and our preceding construction, the Hodge group of
$s_0 \circ K \in G(\R)/K$ is a subtorus of $T$. Hence $s_0 \circ K$ is a $CM$ point.
\end{proof}

\chapter{Cyclic covers of the projective line}
\section{Description of a cyclic cover of the projective line}

Let us first repeat some known facts about Galois covers of $\bP^1$.

\begin{definition} \label{ZykUeb}
Let $T_1$, $T_2$, and $S$ be topological spaces resp., complex manifolds resp., algebraic
varieties. The coverings $f_1: T_1 \to S$ and $f_2 : T_2 \to S$, which are morphisms in the
respective category, are called equivalent, if there is an isomorphism $g : T_1 \to T_2$ 
in the respective category such that $f_1 = f_2\circ g$.
\end{definition}

\begin{proposition} \label{loch}
Let $G$ be a finite group, and $S :=\{a_1, \ldots, a_{n}\} \subset
\A^1 \subset \bP^1$. There is a correspondence between the
following objects:
\begin{enumerate}
\item The isomorphism classes of Galois extensions of $\C(\bP^1) = \C(x)$ with Galois group $G$
and branch points contained in $S$.
\item The equivalence classes of (non-ramified) Galois coverings $f : R \to \bP^1 \setminus S$
of topological spaces with deck transformation group isomorphic to
$G$.
\item The normal subgroups in the fundamental group $\pi_1(\bP^1\setminus S)$ with quotient
isomorphic to $G$.
\end{enumerate}
\end{proposition}
\begin{proof}
(see \cite{Volk}, Theorem 5.14)
\end{proof}

\begin{remark} \label{adda}
We will need to understand the correspondence of the preceding Proposition. The correspondence
between (1) and (2) is given by the facts that a Galois covering $f : R \to \bP^1 \setminus S$ (of
topological spaces) yields a covering $f : \bar R \to \bP^1$ of compact
Riemann surfaces, and any morphism of compact Riemann surfaces corresponds
to an embedding of their function fields.

The correspondence
between (2) and (3) is given by the path lifting properties of
coverings of Hausdorff spaces. Take $b \in R$. Let $p = f(b)$, and
$\gamma \in \pi_1(\bP^1 \setminus S,p)$, and $f^*(\gamma(0)) = b$.
Then $f^*(\gamma(1)) = g\cdot b$ for some $g \in G \cong
{\rm Deck}(R/(\bP^1\setminus P))$. This induces a homomorphism $\Phi_b:
\pi_1(\bP^1 \setminus S,p) \to G$ and hencefore a kernel of this
homomorphism, which is a normal subgroup.
\end{remark}

\begin{remark} \label{loopy}
Let $f:R \to \bP^1$ be a Galois covering with branch points $a_1,
\ldots, a_{n}$. One can choose  $\gamma_1, \ldots, \gamma_n\in \pi_1(\bP^1\setminus P)$ such that
each $\gamma_k$ is given by a loop running counterclockwise "around" exactly one $a_k$. Hence one
has that
$$\gamma_{n} = \gamma_{1}^{-1} \ldots \gamma_{n-1}^{-1}$$
and we conclude that
$$\Phi_b(\gamma_{n}) = \Phi_b(\gamma_{1})^{-1} \ldots
\Phi_b(\gamma_{n-1})^{-1}.$$
\end{remark}

From now on we consider only irreducible cyclic covers of $\bP^1$. An irreducible cyclic cover
can be given by a prime ideal
$$(y^m - (x-a_1)^{d_1} \cdot \ldots \cdot (x-a_{n})^{d_{n}}) \subset \C[x,y].$$
First this ideal defines only an affine curve in $\A^2$, which has singularities, if there
are some $d_k >1$. But there exists a unique smooth projective curve birationally equivalent to
this affine curve. By the natural projection onto the $x$-axis, one obtains a cyclic cover of the
birationally equivalent projective smooth curve onto $\bP^1$.

\begin{remark} \label{monodisc}
Let us consider the cover given by
$$y^m = (x-a_1)^{d_1} \cdot \ldots \cdot (x-a_{n})^{d_{n}},$$
and fix a $k_0 \in  \{1, \ldots n\}$. By an automorphism of $\bP^1$,
one can put $a_{k_0}$ onto $0$. Let $\mu_{k_0} = \frac{d_{k_0}}{m}\in \Q$,
and $D$ a small disc centered in 0, which does not contain any
other $a_k$ with $k \neq k_0$. Take any point $p \in \partial D$ and
remove the line $[0,p]$. The topological
space $D \setminus [0,p]$ is simply connected. Hence one can define root functions
$z \to z^{\mu_{k_0}}$ on this space.

These functions on $D \setminus [0,p]$ are given by:
$$z^{\mu_{k_0}} = |z|^{\mu_{k_0}}\exp(\frac{2\pi i t d_{k_0}}{m}+2\pi i\frac{\ell}{m}) \  \
(\mbox{with}\  \ \ell = 0, 1, \ldots, m-1 \  \ \mbox{and} \  \ z = |z|\exp(2\pi i t))$$
Since the
cover is given by $y^m = x^{d_{k_0}}$ resp., $y = x^{\mu_{k_0}}$ over a small disc around $0$, we
may lift a closed path around 0 to some path with starting point $(z, z^{\mu_{k_0}} )$ and ending
point $(z,e^{2\pi i \mu_{k_0}}z^{\mu_{k_0}})$.
\end{remark}

\begin{definition}\index{local monodromy datum}
Let $e^{2\pi i \mu_{k_0}}$ and $d_{k_0}$ be given by Remark $\ref{monodisc}$. Then
$e^{2\pi i \mu_{k_0}}$ is the local monodromy datum of $d_{k_0}$.
\end{definition}

\begin{lemma} \label{indexe}
Assume that $d_1, \ldots, d_{n} < m$. Let the (non-singular projective) curve $C$
be given by
$$
y^m = (x-a_1)^{d_1} \cdot \ldots \cdot (x-a_{n})^{d_{n}}.
$$
Then the Galois group $G$ is $\Z/(m)$, and the covering $C \to \bP^1$ is given by the kernel of
the homomorphism given by
$\gamma_k \to d_k \in \Z/(m)$. If and only if $m$ does not divide $\sum\limits_{k=1}^n
d_k$, the point $\infty$ is a branch point and
$$\gamma_{\infty} \to -\sum\limits_{k=1}^n d_k \  \ \mbox{mod} \ \ m.$$
\end{lemma}
\begin{proof}
The last statement of the lemma follows by the preceding rest of the
lemma and the Remark $\ref{loopy}$.

The Galois group and $\Z/(m)$ are obviously isomorphic. Let us
remove the ramification points of $C$. Then we obtain a Riemann surface
$R$. Now take a small loop $\gamma_k$ around $p_k$, which starts and ends in $p
\in \bP^1$. Now take a point $b \in R$ with $f(b) = p$.
The definition of $R$ and Remark $\ref{monodisc}$ imply that the
lifting $f^*(\gamma_k)$ of the path $\gamma_k$ starting in $b$ ends
in the point $d_k \cdot b$. Hence the statement follows from Proposition
$\ref{loch}$ resp., Remark $\ref{adda}$.
\end{proof}

Let $d \in \Z$ and $1 < m \in \N$. The residue class of $d$ in $\Z/(m)$ is denoted by
\index{$[d]_m$} $[d]_m$.

\begin{remark} \label{drehen}
Let $G = \Z/(m)$, and $[d]_m \in \Z/(m)^*$. We consider the kernels of
the monodromy representations of the covers
locally given by
$$y^m = (x-a_1)^{d_1} \cdot \ldots \cdot (x-a_{n})^{d_{n}}$$
and
$$y^m = (x-a_1)^{[dd_1]_m} \cdot \ldots \cdot (x-a_{n})^{[dd_{n}]_m}.$$
By the preceding lemma, these kernels coincide. Hence we conclude that both covers are equivalent.
\end{remark}

\section{The local system corresponding to a cyclic cover}
Now let us assume that our cover $\pi: C \to \bP^1$ is given by
$$y^m = (x-a_1)^{d_1} \cdot \ldots \cdot (x-a_{n})^{d_{n}},$$
where $m$ divides $d_1 + \ldots +d_n $ and $\infty$ is not a branch point. Moreover let
$$S := \{a_1, \ldots, a_n\}.$$
First let us consider the construction of a cyclic cover of an arbitrary
algebraic manifold:

\begin{definition} \index{branch index}
Let $X$ be a complex algebraic manifold, $\sL$ an invertible sheaf on $X$ and
$$D = \sum b_kD_k$$
a normal crossing divisor on $X$, where $\sL^m = \sO(D)$ and $0<b_k < m$
for each $k$. Then by
$\sL$ and $D$, one can construct a cyclic cover of degree $m$ onto $X$ (see \cite{EV1}, $\S3$).
The number $b_k$ is called the branch index of $D_k$ with respect to this cyclic cover.
\end{definition}

\begin{example}
In the case of
$$X = \bP^1, \  \ D = \sum\limits_{k=1}^n d_ka_k, \  \ \sL =
\sO_{\bP^1}(\frac{1}{m}\sum\limits_{k=1}^n d_k),$$
the cyclic cover of the preceding definition is given by
$$y^m = (x-a_1)^{d_1} \cdot \ldots \cdot (x-a_{n})^{d_{n}}.$$
\end{example}

Next we describe the local system $\pi_*(\C)|_{\bP^1 \setminus S}$ and its monodromy.

\begin{lemma} \label{monod}
Let $V$ be a $\C$-vector space of dimension $n$, and $X$ be an
arcwise connected and locally simply connected topological space
with $x \in X$. Then the monodromy representation provides a bijection between the set of
isomorphism classes of local systems of stalk $V$ and the set of
representations
$$\pi_1(X,x) \to \GL_n(\C),$$
modulo the action of ${\rm Aut}_{\C}(V)$ by conjugation.
\end{lemma}
\begin{proof}
(see \cite{Voi}, Remark 15.12)
\end{proof}

Since $\GL_1(\C) \cong \C^*$ is commutative, we can conclude:

\begin{corollary}
The monodromy yields a bijection between the set of isomorphism classes of rank
one local systems on $\bP^1 \setminus S$ and the set of
representations
$$\pi_1(\bP^1 \setminus S) \to \GL_1(\C).$$
\end{corollary}

The Galois group of our covering curve is isomorphic to
$\Z/(m)$ and generated by a map $\psi$, which is given by $(x,y) \to
(x,e^{2\pi i \frac{1}{m}}y)$ with respect to the above affine curve contained in $\A^2$, which is
birationally equivalent to the covering curve. Hence a character $\chi$ of this group is
determined by
$\chi(\psi)$ with $\chi(\psi) \in \{e^{2\pi
i \frac{j}{m}}| j = 0, 1, \ldots, m-1\}$. Thus the character group
is isomorphic $\Z/(m)$ and we identify the character, which maps $\psi$ to
$e^{2\pi i \frac{j}{m}}$, with $j \in \Z/(m)$.\footnote{These two identifications with $\Z/(m)$
are obviously not canonical, but useful for the description of $\pi_*\C_{C}|_{\bP^1 \setminus S}$
by using our explicit equation for $\pi: C \to \bP^1$ as we will see a little bit later.}

Let $D$ be an arbitrary disc contained in $\bP^1 \setminus S$. The preimage of $D$ is given by
the disjoint union of discs $D_r$ with $r = 0, 1, \ldots, m-1$ such that
$\psi(D_r) = D_{[r+1]_m}$. The vector space $\pi_*\C_{C}|_{\bP^1 \setminus S}(D)$ has the basis
$\{v_j|j = 0, 1,\dots, m-1\}$, where
$$v_j := (e^{\frac{2\pi j(m-1)}{m} }  , \ldots,  e^{\frac{2\pi j}{m} }, 1),$$
and the $r$-th. coordinate denotes the value of the corresponding section of $\pi^{-1}(D)$ on
$D_r$. By the push-forward action, each $v_j$ is an eigenvector with respect to
the character given by $j$. Since $D$ is arbitrary, one can glue the local eigenspaces, and obtain
an eigenspace decomposition
$$\pi_*\C_{C}|_{\bP^1 \setminus S} = \bigoplus\limits_{j = 0}^{m-1}\LL_j$$
into rank 1 local systems, where $\LL_j$ is the eigenspace with respect to the character given by
$j\in \Z/(m)$. Hence the monodromy representation $\rho : \pi_1(\bP^1 \setminus
S) \to GL_{m}(\C)$ has the corresponding decomposition
$$\rho = (\rho_0, \rho_1, \ldots, \rho_{m-1}):\pi_1(\sX) \to
\prod\limits_{i = 0}^{m-1} \GL_1(\C),$$
where
$$\rho_j:\pi_1(\bP^1 \setminus S) \to  \GL_1(\C)$$
is the monodromy representation of $\LL_j$ for all $j = 0, 1,\ldots, m-1$.

Let us recall that our cyclic cover $C$ is given by
$$y^m = (x-a_1)^{d_1} \ldots (x-a_n)^{d_n},$$
where $\infty$ is not a branch
point.  Now let $x\in \bP^1 \setminus S$, and $x\in D$, where $D$ is a sufficiently small open
disc as above. Take a counterclockwise loop $\gamma_k$ around $a_k$ with
$\gamma_k(0) = \gamma_k(1) = x$ and cover the loop with a finite
number of (sufficiently) small discs. The 
continuation of $\tilde s$ on the unification of these discs leads
to a multi-section. By Remark $\ref{monodisc}$, the possible liftings
$\gamma_k^{(r)}$ of the loop $\gamma_k$ are paths with starting
point $\gamma_k^{(r)}(0) = y_r$, where $y_r \in D_r$ and ending point $\gamma_k^{(r)}(1) =
y_{[d_k+r]_m}$. This implies that the monodromy representation of $\LL_j$ maps $\gamma_k$ to
$e^{\frac{2\pi jd_k}{m}}$. Hence we conclude:

\begin{Theorem} \label{monorep}
Let the cyclic cover $\pi: C \to \bP^1$, which is not branched over $\infty$, be given by
\begin{equation} \label{cover}
y^m = (x-a_1)^{d_1} \ldots (x-a_n)^{d_n}.
\end{equation}
Then the local system $\pi_*\C|_{\bP^1 \setminus S}$ is given by the monodromy
representation
$$\gamma_k \to \{ (x_j)_{j=0,1 \ldots, m-1} \to (e^{\frac{2\pi i
jd_k}{m}}x_j)_{j=0,1 \ldots, m-1}\}.$$
\end{Theorem}

\begin{remark}
One can consider $\pi_*(\Q(e^{2\pi i \frac{1}{m}}))|_{\bP^\setminus
S}$, too. Since a generator $\psi$ of $\Gal(C;\bP^1)$ satisfies $\psi^m=1$, the minimal polynomial
of its action on $\pi_*(\Q(e^{2\pi i \frac{1}{m}}))|_{\bP^\setminus S}$ decomposes into linear
factors contained in $\Q(e^{2\pi i\frac{1}{m}})[x]$. Hence the eigenspace decomposition is defined
over $ \Q(e^{2\pi i\frac{1}{m}}) $.
\end{remark}

Each local system $L$ of $\C$-vector spaces on any topological space $X$ has a dual local system
$L^{\vee }$ given by the sheafification of the presheaf 
$$U \to {\rm Hom}_{\C}(L,\C).$$

\begin{proposition} \label{bar}
One has
$$\LL_j^{\vee} = \bar \LL_j.$$
Furthermore the monodromy
representation $\mu_{\LL_j^{\vee}}$ of $\LL_j^{\vee}$ is given by
$\mu_{\LL_j^{\vee}}(\gamma_s) = \overline{\mu_{\LL_j}(\gamma_s)}$
for all $s \in S$.
\end{proposition}
\begin{proof}
(see \cite{Doran}, Proposition 2)
\end{proof}

Hence by the respective monodromy representations, we obtain for all $j = 1, \ldots, m-1$:

\begin{corollary}
$$\LL_j^{\vee} = \LL_{m-j}$$
\end{corollary}

Let $r|m$. We consider the $\C$-algebra endomorphism $\Phi_r $ of
$\C[x,y]$ given by $x \to x$ and $y \to y^r$. The (non-singular)
curve $C$ is birationally equivalent to the affine variety given by
${\rm Spec}(\C[x,y]/I)$, where
$$I = (y^m - (x-a_1)^{d_1} \ldots (x-a_n)^{d_n}).$$
By $\Phi_r$, we obtain the prime ideal
$$\Phi_r^{-1}(I) = (y^{\frac{m}{r}} - (x-a_1)^{d_1} \ldots
(x-a_n)^{d_n}).$$
Let $C_r$ \index{$C_r$} be the irreducible projective non-singular curve birationally
equivalent to the affine variety given by ${\rm Spec}(\C[x,y]/\Phi_r^{-1}(I))$.

\begin{remark} \label{cr}
By the equation above, we have a cover $\pi_r: C_r \to \bP^1$ of
degree $\frac{m}{r}$. The homomorphism $\Phi_r$ induces a cover
$\phi_r: C \to C_r$ of degree $r$ such that
$$\pi = \pi_r \circ \phi_r.$$
\end{remark}

\begin{proposition} \label{hti}
$$(\pi_r)_*\C_{C_r}|_{\bP^1 \setminus S} = \bigoplus\limits_{j =
0}^{\frac{m}{r}-1} \LL_{r\cdot j} \subset \pi_*\C_C|_{\bP^1 \setminus S}.$$
\end{proposition}
\begin{proof}
Let $m_0 := \frac{m}{r}$. By Theorem $\ref{monorep}$, the monodromy representation of the local
system $(\pi_r)_*\C_{C_r}|_{\bP^1 \setminus \{a_1, \ldots, a_n\}}$ is given by
$$\gamma_k \to \{ (x_j)_{j=0,1 \ldots, \frac{m}{r}-1} \to
(e^{\frac{2\pi i jd_k}{m_0}}x_j)_ {j=0,1 \ldots, \frac{m}{r}-1} =
(e^{\frac{2\pi i jrd_k}{m}}x_j)_{j=0,1 \ldots, \frac{m}{r}-1}\}.$$
By the respective monodromy representations of the local systems $\LL_j$, this yields the
statement.
\end{proof}

\section{The cohomology of a cover}
In this section we discuss some known facts about the eigenspace decomposition of the Hodge
structure of a curve $C$ with respect to a cyclic cover $\pi: C \to \bP^1$.
The main reference for this section is given by $\S 3$ of the
book \cite{EV1} of H. Esnault and E. Viehweg. Section 2 of the
essay \cite{DM} of P. Deligne and G. D. Mostow contains additional information about
our case. 

Let $\pi: C \to \bP^1$ be given by
$$y^m = (x-a_1)^{d_1} \cdot \ldots \cdot (x-a_{n})^{d_{n}}$$
such that $\infty$ is not a branch point, $$S = \{a_1, \ldots, a_n\}, \   \ D = d_1a_1+ \ldots
+ d_na_n \  \ \mbox{and} \  \
\sL^{(j)} := \sO_{\bP^1}(j\frac{d_1 + \ldots + d_n}{m}-\sum\limits_{k =
1}^{n+3} [\frac{j}{m}\cdot d_k]).$$
Moreover let the generator $\psi$ of the Galois group of $\pi$ be given by $(x,y) \to
(x,e^{2\pi i \frac{1}{m}}y)$ with respect to the explicit
equation above, which yields $\pi$.

We fix some new notation: Let $[q]_1 := q-[q]$ \index{$[q]_1$} for all $q \in \Q$. Moreover we
define
$$S_j := \{a \in S| [j\mu_a]_1 \neq 0\}. \index{$S_j$}$$

\begin{proposition} \label{fer}
The sheaves $\pi_*(\sO)$ and $\pi_*(\omega)$ have a decomposition into
eigenspaces with respect to the Galois group representation, which
are given by the sheaves $\sL^{(j)^{-1}}$ and
$$\omega_j := \omega_{\bP^1}({\rm log } D^{(j)} ) \otimes
\sL^{(j)^{-1}} \  \ \mbox{with} \  \ D^{(j)} :=\sum\limits_{a \in S_j} a$$
for $j = 0, 1, \ldots, m-1$ such that $\psi$ acts via pull-back by the character
$e^{2\pi i\frac{j}{m}}$ on $\sL^{(j)^{-1}}$ resp., $\omega_j$.
\end{proposition}
\begin{proof}
The eigenspace decomposition of $\pi_*(\sO)$ follows by \cite{EV1}, Corollary 3.11. Moreover
\cite{EV1}, Lemma 3.16, d) yields the decomposition of $\pi_*(\omega)$ into the claimed
sheaves. Since $\sL^{(j)^{-1}}$ is an eigenspace with respect to the Galois group representation,
$\omega_j$ is an eigenspace of the same eigenvalue.
\end{proof}

\begin{remark}
One has obviously $h^0(\omega_0) =0$. By \cite{EV1}, $2.3,c)$, one concludes that
$$\omega_{\bP^1}({\rm log }D^{(j)}) = \omega_{\bP^1}(D^{(j)})$$
for $j = 1, \ldots, m-1$. Hence for $j = 1, \ldots, m-1$ we obtain
$$h^0(\omega_j) = h^0(\sO_{\bP^1}(-2+deg(D^{(j)})-j\frac{d_1 + \ldots +
d_{n+3}}{m}+ \sum\limits_{k = 1}^{n+3}[\frac{j}{m}\cdot d_k]) )$$
$$= -1 + |S_j| + \sum\limits_{a \in S_j}(-j\mu_a+[j\mu_a])=
-1+ \sum\limits_{a \in S_j}(1-[j\mu_a]_1) .$$
\end{remark}

But here we want to determine our eigenspaces on $\pi_*(\omega_C)$ with respect to the
push-forward action. Thus we put $\omega^{(j)} := \omega_{[m-j]_m}$, and we obtain
\index{$H^{1,0}_j(C)$}
$$h^{1,0}_j(C) := h^0(\omega^{(j)}) =h^0(\omega_{[m-j]_m}) =
-1+ \sum\limits_{a \in S_j}(1-[(m-j)\mu_a]_1)
= -1+ \sum\limits_{a \in S_j}[j\mu_a]_1.$$
Moreover let $H^{0,1}_j(C)$ \index{$H^{0,1}_j(C)$} denote the vector space of antiholomorphic
$1$-forms on $C$ with respect to the corresponding character of the Galois group action. Since
the push-forward action of the Galois group respects the alternating form of the
polarization of the Hodge structure on $H^1(C,\Z)$, one concludes that $H^{0,1}_{[m-j]_m}(C)$ is
the dual of $H^{1,0}_j(C)$. Thus:

\begin{proposition} \label{hodgi}
We have the eigenspace decomposition
$$H^1(C,\C) = \bigoplus_{j = 1}^{m-1} H^1_j(C,\C) \  \ \mbox{with} \  \  H^{1,0}_j(C)
\oplus H^{0,1}_j(C) = H^1_j(C,\C). \index{$H^1_j(C,\C)$}$$
\end{proposition}

Moreover by $h^{0,1}_{j}(C) = h^{1,0}_{[m-j]_m}(C)$ and the preceding calculations, one concludes: 

\begin{proposition} \label{1.27}
We have
$$h^{1,0}_j(C) = \sum\limits_{s \in S_j}
[j\mu_s]_1 - 1, \mbox{ and } h^{0,1}_j(C) =
\sum\limits_{s \in S_j} (1-[j\mu_s]_1) - 1. $$
\end{proposition}

The preceding two propositions imply:

\begin{corollary} \label{rangi}
$$h^{1}_j(C,\C) = |S_j| - 2$$
\end{corollary}

\section{Cyclic covers with complex multiplication}

Let us now search for examples of covers of $\bP^1$ with complex multiplication. The
family given by
$$\bP^2 \supset V(y^m - x_1(x_1-x_0)(x_1-a_1x_0) \ldots (x_1-a_{m-3}x_0))$$
$$\to (a_1, \ldots, a_{m-3})\in (\A^1\setminus \{0,1\})^{m-3} \setminus \{a_i = a_j | i \neq j\}$$
has obviously a fiber isomorphic to the Fermat curve $\F_m$, which is
given by $V(y^m+x^m+1)$ and has complex multiplication (see
\cite{Rohr} and \cite{Kohr}). For another family with a fiber with
complex multiplication, we must work a little bit.

\begin{lemma}\label{7.1}
If $(V,h_1)$ and $(W,h_2)$ are two $\Q$-Hodge structures of weight $k$, then
$$\Hg(V\oplus W,h_1 \oplus h_2) \subset
\Hg(V,h_1)\times \Hg(W,h_2) \subset \GL(V)\times \GL(W) \subset \GL(V\oplus
W),$$ and the projections
$$\Hg(V\oplus W)\to \Hg(V), \  \ \mbox{and} \  \ \Hg(V\oplus W)\to \Hg(W)$$
are surjective.
\end{lemma}
\begin{proof}
(see \cite{VZ5}, Lemma 8.1)
\end{proof}

\begin{lemma} \label{nachbarin}
Let $V \subset W$ be a rational sub-Hodge structure of a polarized
Hodge structure $W$. Then we have a direct sum decomposition
$$W = V \oplus V',$$
where $V'$ is also a rational sub-Hodge structure of $W$.
\end{lemma}
\begin{proof}
(see \cite{Voi}, Lemma 7.26)
\end{proof}

\begin{lemma} \label{jojo}
A curve $C$, which is covered by the Fermat curve $\F_m$ given by $V(x^m+ y^m+ z^m) \subset \bP^2$
for some $1 \leq m \in \N$, has complex multiplication.
\end{lemma}
\begin{proof} A covering $\F_m \to C$ yields an injective vector space homomorphism
$$H^1(C,\Q) \to H^1(\F_m,\Q),$$
which extends to an embedding of Hodge structures (see \cite{Voi},
7.3.2 for more details). This embedding induces a direct sum
decomposition into two rational sub-Hodge structures of $H^1(\F_m,\Q)$ (see
Lemma $\ref{nachbarin}$). Hence by Lemma $\ref{7.1}$ and the
fact that $\F_m$ has complex multiplication, one obtains the
statement.
\end{proof}

\begin{Theorem} \label{geilomat}
Let $0 <d_1, d < m$, and $\xi_k$ denote a primitive $k$-th. root of
unity for all $k \in \N$. Then  the curve $C$, which is given by
$$y^m = x^{d_1}\prod\limits_{i = 1}^{n-2} (x-\xi_{n-2}^i)^d,$$
is covered by the Fermat curve $\F_{(n-2)m}$ given by $V(y^{(n-2)m}+x^{(n-2)m}+1)$ and has complex
multiplication.
\end{Theorem}
\begin{proof}
Let  $C$ be the curve, which is given by
$$y^m = x^{d_1}\prod\limits_{i = 1}^{n-2} (x-\xi_{n-2}^i)^d,$$
and $\phi: \A^2 \to \A^2$ be the morphism, which is given by $y \to
yx^{d_1}$ and $x\to x^m$. By a little abuse of notation, we denote by
$C \cap \A^2$ the singular affine curve given by the equation above,
which is birationally equivalent to $C$. The corresponding
homomorphism $\phi^*: \C[x,y] \to \C[x,y]$ sends the ideal, which
defines $C \cap \A^2$, to the ideal generated by
$$y^mx^{m\cdot d_1} - x^{m\cdot d_1}\prod\limits_{i = 1}^{n-2}(x^m- \xi_{n-2}^i )^d.$$
This is contained in the ideal generated by
\begin{equation} \label{c_1}
y^m - \prod\limits_{i = 1}^{n-2} (x^m-\xi_{n-2}^i)^d.
\end{equation}
Let $m_0 := \frac{m}{gcd(m, d)}$, and $d_0 := \frac{d}{gcd(m,
d)}$. It is obvious that
$$
y^m - \prod\limits_{i = 1}^{n-2} (x^m-\xi_{n-2}^i)^d = \prod\limits_{j=
0}^{gcd(m, d)-1}(y^{m_0} - \xi_{gcd(m, d)}^j\prod\limits_{i =
1}^{n-2} (x^m-\xi_{n-2}^i)^{d_0}).
$$
Now we take the curve $C_1$, which is given by
$$y^{m_0} = \prod\limits_{i = 1}^{n-2} (x^m- \xi_{n-2}^i )^{d_0}.$$
By the definitions of $m_0$ and $d_0$, and Remark $\ref{drehen}$, the curve $C_1$ is given by
$$y^{m_0} = \prod\limits_{i = 1}^{n-2} (x^m-\xi_{n-2}^i),$$
too. Hence this curve irreducible, and $\phi$ induces a cover $C_1 \to C$
resp., $\phi^*$ induces  a $\C$-algebra homomorphism $\C[C\cap \A^2]
\to \C[C_1\cap \A^2]$. By $x \to x$ and $y \to y^{n-2\frac{m}{m_0}}$, we get a
cover of the Fermat curve $\F_{(n-2)m}$ given by $V(y^{(n-2)m}+x^{(n-2)m}+1)$ onto $C_1$. Now we
use the composition of these covers $\F_{(n-2)m} \to C_1$ and $C_1 \to C$, and Lemma
$\ref{jojo}$. This yields the statement.
\end{proof}

\chapter{Some preliminaries for families of cyclic covers}

\section{The theoretical foundations}

We want to study the variations of Hodge structures ($VHS$) of the families of cyclic
covers onto $\bP^1$, which will be constructed in the next section. Hence let us first make
some general observations about the relation between their monodromy groups and Hodge
groups resp., Mumford-Tate groups, which will be needed for the calculation (of the derived group)
of the generic Hodge group defined below.

\begin{proposition} \label{genmt}
Let $F$ be a totally real number field, $W$ be a complex connected algebraic manifold, $\sA \to W$
be a family of Abelian varieties and $\sV$ be its polarized variation of $F$-Hodge structures of
weight $1$ on $W$. Then there is a countable
union $W' \subset W$ of subvarieties such that all $\MT(\sV_p)$ coincide (up to conjugation by
integral matrices) for all (closed)
$p \in W \setminus W'$. Moreover one has $\MT(\sV_{p'}) \subset \MT(\sV_p)$ for all (closed)
$p' \in W'$ and $p \in W \setminus W'$.
\end{proposition}
\begin{proof}
(see \cite{MVZ}, Subsection 1.2)\end{proof}

The preceding Proposition motivates the definition of the generic Mumford-Tate group
\index{Mumford-Tate group!generic} of a 
polarized variation $\sV$ of $\Z$-Hodge structures of weight $k$ on a non-singular connected
algebraic variety $W$ given by $\MT(\sV) = \MT(\sV_p)$ for all (closed) $p \in W \setminus W'$.

Since the image of the embedding $\SL(\sV_{\Q,p}) \hookrightarrow \GL(\sV_{\Q,p})$ is independent
with respect to the chosen coordinates on $\sV_{\Q,p}$, Lemma $\ref{idhggr}$ allows us to
define the generic Hodge group $Hg(\sV):=(\MT(\sV)\cap \SL(\sV))^0$ \index{Hodge group!generic}
such that $\Hg(\sV) = \Hg(\sV_p)$ for all (closed) $p \in W \setminus W'$.

\begin{definition}
Let $\Q \subseteq K\subseteq \R$ be a field and $\sV = (\sV_{K},\sF^{\bullet } ,Q)$ be a polarized
variation of $K$ Hodge structures on a connected complex manifold $D$. Then $\Mon^0_K(\sV)_p$
denotes the connected component of identity of the
Zariski closure of the monodromy group in $\GL((\sV_{K})_p)$ for some $p \in D$. For simplicity we
write $\Mon^0(\sV)_p$ instead of $\Mon^0_{\Q}(\sV)_p$
\end{definition}

\begin{Theorem} \label{monmtder}
Keep the assumptions and notations of Proposition $\ref{genmt}$. One has that $\Mon_F^0(\sV)_p$ is a
subgroup of $\MT_F^{\der}(\sV_p)$ for all $p \in W\setminus W'$. Moreover for a variation of $\Q$
Hodge structures one has that $\Mon^0(\sV)_p$ is a normal subgroup of $\MT^{\der}(\sV_p)$ and
$$\Mon^0(\sV)_p = \MT^{\der}(\sV_p)$$
for all $p \in W\setminus W'$, if $\sV_{\Q}$ has a $CM$ point.
\end{Theorem}
\begin{proof}
(see \cite{Mo}, Theorem 1.4 for the statement about the variations of $\Q$ Hodge structures and
\cite{MVZ}, Properties $7.14$ for the statement about the variations of $F$ Hodge structures)
\end{proof}

\begin{corollary} \label{ssmon}
Keep the assumptions of Theorem $\ref{monmtder}$. Then the group $\Mon^0(\sV)$ is semisimple.
\end{corollary}
\begin{proof}
The Lie subalgebra $Lie(\Mon_{\Q}^0(\sV)_{\R})$ of $Lie(\MT_{\Q}^{der}(\sV)_{\R})$ is an ideal.
Hence the algebra $Lie(\Mon_{\Q}^0(\sV)_{\R})$ consists of the direct sum of simple
subalgebras of $Lie(\MT_{\Q}^{\der}(\sV)_{\R})$. Thus $\Mon_{\Q}^0(\sV)_{\R}$ and hence
$\Mon^0(\sV)$ is semisimple.
\end{proof}

\section{Families of covers of the projective line}
Let $S$ be some $\C$-scheme. Recall that the covers $c_1: V_1 \to \bP^1_S$ and $c_2: V_2 \to
\bP^1_S$ are equivalent, if there is a $S$-isomorphism $j:V_1 \to V_2$ such that
$c_1 = c_2 \circ j$.

In this section we construct a family of cyclic covers onto $\bP^1$ such that all equivalence
classes of covers with a fixed number of branch points with fixed branch
indeces are represented by some of its fibers. For us it is sufficient to start with a space,
which is not a moduli scheme, but
whose closed points ``hit'' all equivalence classes of covers of $\bP^1$ with Galois group
$G = (\Z/m,+)$ and a fixed number of branch points with fixed branch indeces.

We can start with the space \index{$\sP_n$}
$$(\bP^1)^{n+3} \supset  \sP_n := (\bP^1)^{n+3}\setminus \{z_i= z_j|
i\neq j\},$$
which parameterizes the injective maps $\phi: N \to \bP^1$, where $N:= \{s_1, \ldots, s_{n+3}\}$.
Thus a point $q \in \sP_n$
corresponds to an injective map $\phi_q: N \to \bP^1$.\footnote{The set $N$ is some arbitrary
finite set, where the set $S$ of the preceding chapter is a concrete set $S \subset \bP^1$ given
by $S = \phi_q(N)$ for some $q \in \sP_n$.} One
can consider $\sP_n$ as configuration space of $n+3$ ordered points,
too.

We endow the points $s_k \in N$ with some local monodromy
data $\alpha_k =e^{2\pi i \mu_k}$, where
$$\mu_k \in \Q, \ \  0 < \mu_k < 1 \  \ \mbox{ and } \  \ \sum\limits_{k = 1}^{n+3} \mu_k \in \N.$$
Now we construct a family of covers of $\bP^1$ by these local monodromy data:

\begin{construction} \label{hiercr} \index{$\sC \to \sP_n$}
Let $m$ be the smallest integer such that $m\mu_k \in \N$ for $k = 1, \ldots, n+3$, and
$D_k \subset \bP_{\sP_n}:= \bP^1 \times \sP_n$ be the prime divisor given by 
$$D_k= \{(a_k,a_1, \ldots, a_k, \ldots, a_{n+3})\}.$$
Let $D$ be the divisor
$$D := \sum\limits_{k=1}^{n+3} m\mu_k D_k\sim m D_0 \ \ \mbox{with} \ \ D_0 :=
(\sum\limits_{k = 1}^{n+3} \mu_k) \cdot (\{0\} \times \sP_n).$$
By the sheaf $\sL := \sO_{\bP_{\sP_n}}(D_0)$ and the divisor $D$, we obtain an irreducible
cyclic cover $\sC$ of degree $m$ onto $\bP_{\sP_n}$ as in \cite{EV1}, $\S3$ (where irreducible
means that the covering variety is irreducible). By
$\pi :\sC \to \bP^1 \times \sP_n \stackrel{pr_2}{\to} \sP_n$, this cyclic cover yields a family of
irreducible cyclic covers of degree $m$ onto $\bP^1$.

Suppose that $r$ divides $m$. By taking the quotient of the subgroup of order $r$ of the Galois
group of the cyclic cover $\sC \to \bP^1 \times \sP_n$, one gets a family $\pi_r: \sC_r \to \sP_n$
\index{$\sC_r$} of cyclic covers of degree $\frac{m}{r}$ onto $\bP^1$. Let $\phi_r : \sC \to \sC_r$
denote the quotient map. One has
$$\pi = \pi_r \circ \phi_r.$$
\end{construction}

\begin{remark}
Without loss of generality one may assume that $q := (a_1,\ldots, a_{n+3}) \in \sP_n$ is contained
in $\A^{n+3}$, too. Thus the fiber $\sC_q$ is given by the equation
$$y^m = (x-a_1)^{d_1}\cdot \ldots  \cdot (x-a_{n+3})^{d_{n+3}}$$
with $d_k = m\mu_k$. By Remark $\ref{monodisc}$, the local monodromy datum $\alpha_k$ describes
the lifting of a path $\gamma_k$ around $a_k \in \bP^1$.\footnote{This circumstance explains the
term ``local monodromy datum''.} One checks easily that each equivalence class of cyclic covers of
degree $m$ with $n+3$ branch points and fixed branch indexes $d_1, \ldots, d_{n+3}$ is represented
by some fibers of $\sC$. Moreover for $C = \sC_q$ the quotient $C_r$ of Remark $\ref{cr}$ is given
by the fiber $(\sC_r)_q$.
\end{remark}

A family of smooth algebraic curves over $\C$ determines a proper submersion $\tau:X \to Y$ in the
category of differentiable manifolds (follows by \cite{Voi}, Proposition 9.5). By the Ehresmann
theorem, we obtain that over any contractible submanifold $W$ of $Y$ the family is
diffeomorphic to $X_0 \times W$, where $X_0$ is the fiber of some point $0 \in W$. This fact has
some consequences for the monodromy representation of its variation of integral Hodge structures.

It is a well-known fact that $R^1\tau_*(\Z)$ is the sheaf associated to the presheaf
$$V \to H^1(\tau^{-1}(V) , \Z|_{\pi^{-1}(V)}) \  \ (\forall \  \ V
\in Top(\sP_n)).$$
Moreover we have
$$H^1(X_0,\Z) = H^1(X_W,\Z) = (R^1\tau_*(\Z))(W)$$
for some contractible $W \subset \sP_n$ with $0 \in W$, which implies that
$R^1\tau_*(\Z)$ is a local system (see \cite{Voi}, 9.2.1).

By using these facts, one can easily ensure that the monodromy group of the $VHS$ of a family of
curves can be calculated over any arbitrary field $\Q \subseteq K \subseteq \C$:

\begin{lemma} \label{KeineAhnung}
Let $K$ be a field with $char(K) = 0$. Moreover let $\tau : X \to Y$ be a holomorphic family of
curves. Then we obtain
$$R^1\tau_*( K) = R^1\tau_*(\Z) \otimes_{\Z} K.$$
\end{lemma}
\begin{proof}
By \cite{hart}, {\bf III}, Proposition 8.1, the sheaf $R^1\tau_*(K)$ is given by the sheafification
of the presheaf
$$V \to H^1(\tau^{-1}(V), K|_{\tau^{-1}(V)}) \ \ (\forall \ \ V \in Top(Y)).$$
Hence by the description of the cohomology of a compact manifold by \v{C}ech complexes
(see \cite{Voi}, 7.1.1), this presheaf is given by
$$V \to H^1(\tau^{-1}(V), \Z |_{\tau^{-1}(V)})\otimes_{\Z} K \ \ (\forall \ \ V \in Top(Y)).$$
By the fact that a local section of $\Z$ or $K$ on a connected
component of $V$ resp., $\tau^{-1}(V)$ is constant, one does not need to differ between the locally
constant sheaves given by $\Z$ resp.,
$K$ on $X$ or $Y$ for the computation of $R^1\tau_*(K)$. Hence by using \cite{hart}, {\bf III},
Proposition 8.1 for $R^1\tau_*(\Z)$, one obtains the desired identification.
\end{proof}

By the fact that the integral cohomology of a curve does not have torsion, one concludes:

\begin{corollary} \label{hehehe}
Keep the assumptions of Lemma $\ref{KeineAhnung}$. Then the monodromy
representations of $R^1\tau_*(\Z)$ and $R^1\tau_*(K)$ coincide.
\end{corollary}

\begin{remark} \label{slj}
Recall the we have an eigenspace decomposition of
$$H^1(\sC_0,\C) = H^1(\sC_0,\Z) \otimes \C$$
with respect to the Galois group action. By $H^1(\sC_0,\C) = (R^1\pi_*(\C))(W)$ for some
contractible $W \subset \sP_n$ with $0 \in W$, we obtain an eigenspace decomposition of
$(R^1\pi_*(\C))(W)$. Since we
have this decomposition over all contractible $W \subset \sP_n$, we can glue these eigenspaces,
which yields a decomposition of the whole sheaf $R^1\pi_*(\C)$ into eigenspaces with
respect to the Galois group action.

Recall that we have an identification between the
characters of the Galois group of some fiber and the elements $j \in \Z/(m)$. This identification
allows a compatible identification between the characters of the Galois group of the family and the
elements $j \in \Z/(m)$. Let $\sL_j$ \index{$\sL_j$} denote the eigenspace of $R^1\pi_*(\C)$ with
respect to the character $j$.
\end{remark}

\begin{remark} \label{difint}
Let $0 \in \sP_n$. We have a monodromy action $\rho_{\sC}$ by diffeomorphisms on the
fiber $\sC_0$, which is induced by the glueing diffeomorphisms of the locally constant family
of manifolds given by $\sC$. Since these glueing differeomorphisms induce the
glueing homomorphisms of $R^1\pi_*(\Z)$ in the obvious natural way, the monodromy representation
$\rho$ of $R^1\pi_*(\Z)$ is given by
$$\rho(\gamma)(\eta) = (\rho_{\sC}(\gamma))_*(\eta) \ \ (\forall \  \ \eta \in H^1(\sC_0,\Z)).$$
\end{remark}

\begin{remark} \label{dromy}
Since each glueing diffeomorphism of the locally constant family of manifolds corresponding to
$\sC$ respects intersection form, Remark $\ref{difint}$ implies that the monodromy group of
$R^1\pi_*(\C)$ respects the polarization of the Hodge structures. Assume that
$H^1_j(\sC_q,\C) = (\sL_j)_q$ satisfies that $H^{1,0}_j(\sC_q) = n_1$ and $H^{0,1}_j(\sC_q)_2$.
This means that the variation of integral polarized Hodge structure endows $(\sL_j)_q$ with an
Hermitian form with signature $(n_1,n_2)$. Hence the monodromy group of this eigenspace is
contained in $U(n_1,n_2)$. In this sense we say that $\sL_j$ \index{$\sL_j$} is of type
$(n_1,n_2)$\index{$(n_1,n_2)$ type of $\sL_j$}.
\end{remark}

\section{The homology and the monodromy representation}
In this section we study the monodromy representation of $\pi_1(\sP_n)$ on the dual of
$R^1\pi_*(\C)$ given by the complex homology. This will yield corresponding statements for the
monodromy representation of $R^1\pi_*(\C)$.

In the case of the configuration space $\sP_n$ of $n+3$ points, we make a difference between these
different points. One says that the points are ``colored'' by different
``color''. Moreover one can identify its fundamental group with the subgroup
of the braid group on $n+3$ strands in $\bP^1$, which is given by the braids leaving
the strands invariant (see \cite{Hansen}, Chapter {\bf I}. 3.). This subgroup of the braid group
is called the colored braid group. An element of this group is for example
given by the Dehn twist $T_{k_1,k_2}$ with $1 \leq  k_1 < k_2 \leq n+3$. The
Dehn twist \index{Dehn twist} $T_{k_1,k_2}$ is given by leaving $a_{k_2}$ "run"
counterclockwise around $a_{k_1}$.

Now we consider a fiber $C  = \sC_q$ of $\sC$. Recall that $C$ is a cyclic cover of $\bP^1$
described in Chapter 2.

Consider the eigenspace $\LL_j$, which can be extended from a local system on $\bP^1 \setminus S$
to a local system on $\bP^1 \setminus S_j$
with $S_j = \{a_1, \ldots, a_{n_j+3}\}$. For simplicity one may without loss of generality assume 
that $a_{n_j+3} = \infty$ and $a_k \in \R$ such that $a_k < a_{k+1}$ for all
$k = 1, \ldots, n_{j+2}$. Here we assume that $\delta_{k}$ is the oriented path from $a_k$ to
$a_{k+1}$ given by the straight line.

\begin{construction}
Let $\zeta$ be a path on $\bP^1$. Assume without loss of generality that $\zeta((0,1))$ is
contained in a simply connected open subset $U$ of $\bP^1 \setminus S$. Otherwise we decompose
$\zeta$ into such paths. It has $m$ liftings $\zeta^{(0)}, \ldots, \zeta^{(m-1)}$ to $C$ such that
$\psi(\zeta^{(\ell)}) = \zeta^{([\ell-1]_m)}$. By the tensorproduct of $\C$ with the free Abelian
group generated by the paths on $C$, one obtains the vector space of $\C$-valued paths on $C$. Now
let
$c \in \C$ and take the linear combination of $\C$-valued paths on $C$ given by
$$\hat \zeta = c\zeta^{(0)}+ \ldots+ce^{2\pi i\frac{jr}{m}}\zeta^{(r)} + \ldots+
ce^{2\pi i\frac{j(m-1)}{m}} \zeta^{(m-1)}.$$
By the assumptions, one verifies easily that $\psi(\hat \zeta) = e^{2\pi i\frac{j}{m}}\hat \zeta$.
Moreover by the local sections given by $c, \ldots, ce^{2\pi i\frac{jr}{m}}, \ldots,
ce^{2\pi i\frac{j(m-1)}{m}}$ on the corresponding sheets over $U$ containing the different
$\zeta^{(\ell)}((0,1))$, one obtains a corresponding section $\tilde c \in \LL_j(U)$. In this
sense we have a $\LL_j$-valued path $\tilde c\cdot \zeta$ on $\bP^1$.
\end{construction}

\begin{remark} \label{homolog}
Consider the (oriented) path $\delta_k$. Let $e_k$ be a non-zero local section of
$\LL_j$ defined over an open set containing $\delta_k((0,1))$. The $\LL_j$-valued path
$e_k \cdot \delta_k$ yields an
element $[e_k \cdot \delta_k]$ of the homology group of $H_1(C, \C)$, which is represented by the
corresponding linear combination of paths in $C$ lying over $\delta_k$. It has
the character $j$ with respect to the Galois group representation. Let $H_1(C,\C)_j$
denote the corresponding eigenspace.
\end{remark}

\begin{definition} \index{stable partition}
A partition of $S_j$ into some disjoint sets $S^{(1)} \cup  \ldots \cup  S^{(\ell)} = S_j$ is stable
with respect to the local monodromy data $\mu_k$ of $\LL_j$, if
$$\sum\limits_{a_k \in S^{(1)}}\mu_k \notin \N, \ldots, \sum\limits_{a_k \in S^{(\ell)}}\mu_k
\notin \N.$$
\end{definition}

\begin{Theorem}
Assume that $S_j = \{a_i: i = 1, \ldots, n_j+3\}$ has the stable partition
$\{a_1, \ldots, a_{\ell+1}\}, \{a_{\ell+2},\ldots, a_{n_j+3}\}$ for some
$1 \leq \ell \leq n_j+1$. Then the eigenspace
$H_1(C,\C)_j$ of the complex homology group of $C$ has a basis given by
$$ \sB = \{[e_k\delta_k]: k = 1, \ldots, \ell\} \cup
\{[e_k\delta_k]: k = \ell +2, \ldots, n_j+2\}.$$
\end{Theorem}
\begin{proof}
By \cite{Loo}, Lemma $4.5$, one has that $\{[e_k\delta_k]: k = 1, \ldots, n_j+1\}$ is a
basis of $H_1(C,\C)_j$. Hencefore $\{[e_k\delta_k]: k = 1, \ldots, n_j+2\}$ is linearly dependent.

One can compute a non-trivial linear combination, which yields 0, in the following way:
Choose a non-zero section of $\LL_j$ over
$$U=\bP^1\setminus(\bigcup\limits_{k= 1}^{n_j+2}\delta_k).$$ This
yields a linear combination of the sheets over $U$, on which $\psi$ acts by $j$. By the boundary
operator $\partial$, one gets the desired non-trivial linear combination of $\LL_j$-valued paths,
which is equal to 0. Now let $\alpha_k$ denote the local monodromy datum of $\LL_j$ around
$a_k \in S_j$ for all $k = 1, \ldots, n_j+3$. By the local monodromy data, one can easily compute
this linear combination. This computation yields that
$\{\delta_1, \ldots, \delta_{\ell}\} \cup \{\delta_{\ell+2},\ldots, \delta_{n_j+2}\}$ is linearly
independent, if and only if $\{a_1, \ldots, a_{\ell+1}\},\{a_{\ell+2},\ldots, a_{n_j+3}\}$ is
a stable partition.
\end{proof}

Let $\alpha_k$ denote the local monodromy datum of $\LL_j$ around $a_k \in S_j$ for all
$k = 1, \ldots, n_j+3$. One has up to a certain normalization of the basis vectors
$[e_1\delta_1], \ldots [e_1\delta_{n_j+1}]$ the following description of the monodromy
representation:

The Dehn twist $T_{k,k+1}$ leaves obviously $\delta_{\ell}$ invariant for all $|k-\ell| >1$.
Moreover by following a path representing $T_{k,k+1}$, one sees that the monodromy action of
$T_{k,k+1}$ on $H_1(C,\C)_j$ (induced by push-forward) is given by
$$[e_{k-1}\delta_{k-1}] \to [e_{k-1}\delta_{k-1}] + \alpha_k(1-\alpha_{k+1})[e_k\delta_{k}],$$
$$[e_k\delta_{k}] \to \alpha_k \alpha_{k+1} [e_k\delta_{k}]$$
$$\mbox{and} \  \ [e_{k+1}\delta_{k+1}] \to [e_{k+1}\delta_{k+1}] + (1-\alpha_{k})[e_k\delta_{k}].$$
Hence the monodromy representation is given by:

\begin{proposition} \label{repro}
The monodromy representation of $T_{\ell,\ell+1}$ on
$H_1(C,\C)_j$ is given with respect to the basis $\{[e_k\delta_k]| k = 1, \ldots n_j+1\}$ of
$H_1(C,\C)_j$ by the matrix with the entries:
$$M_{\ell,\ell+1}(a,b)
= \left\{\begin{array}{r@{\quad:\quad}l}
1 & a= b\quad\mbox{and}\quad a \neq \ell\\
\alpha_{\ell} \alpha_{\ell+1} & a= b = \ell\\
\alpha_{\ell}(1- \alpha_{\ell+1}) & a= \ell\quad\mbox{and}\quad b = \ell-1\\
1- \alpha_{\ell} & a= \ell\quad\mbox{and}\quad b = \ell+1\\
0 & \mbox{elsewhere} \end{array}  \right. $$
\end{proposition}

\begin{remark}
The monodromy representation of Proposition $\ref{repro}$ corresponds to an eigenspace in the
local system given by the direct image of the complex homology. By integration over $\C$-valued
paths, this eigenspace is the dual local
system of $\sL_{m-j}$. By the cup-product, $\sL_j$ is the dual of $\sL_{m-j}$, too. Hence
Proposition $\ref{repro}$ yields the monodromy representation of $\sL_j$.
\end{remark}

\chapter{The Galois group decomposition of the Hodge structure}
In this chapter we make some general observations of the $VHS$ of $\sC \to \sP_n$ and its generic
Hodge group. Moreover we will give an upper bound for the Hodge group and a sufficient criterion
for dense sets of complex multiplication fibers.

\section{The Galois group representation on the first cohomology} Let $\pi:C \to \bP^1$ be a
cyclic cover of degree $m$. The elements of $\Gal(\pi)$ act as $\Z$-module automorphisms on
$H^1(C,\Z)$. This induces a faithful representation
\begin{equation} \label{rho}
\rho^{1} :\Gal(\pi) \to \GL(H^1(C, \Q)).
\end{equation}

By the Galois group representation of a cyclic cover of degree $m$, we have the following
eigenspace decomposition:
$$H^1(C, \Q) \otimes \Q(\xi) = H^1(C, \Q(\xi)) = \bigoplus\limits_{i =1}^{m-1} H_j^1(C, \Q(\xi))$$
Recall that $\pi: C \to \bP^1$ is
given by some fibers of a family $\pi:\sC \to \sP_n$. The
monodromy representation of $R^1\pi_*(\C)$ has a decomposition into
subrepresentations on the different eigenspaces. In general there is
not a $\Q(\xi)$ structure on $H^1(C,\Q)$, which turns $H^1(C,\Q)$ into a $\Q(\xi)$-vector space. But
in this section we will see that $H^1(C, \Q)$ has a direct sum
decomposition into sub-vector spaces with different $\Q(\xi^r)$ structures, where $r|m$. Moreover
we will see that the monodromy representation respects the different $\Q(\xi^r)$ structures, which
we will study.

Let $\psi$ denote a generator of $\Gal(\pi)$ as in Chapter 2. The characteristic
polynomial of $\rho^1(\psi)$ decomposes into the product of the
minimal polynomials of the different $\xi^r$, where $r|m$ and $\xi$ is a $m$-th. primitive root of
unity. By \cite{MichKov}, Satz 12.3.1., we have a decomposition of $H^1(C,\Q)$ into subvector
spaces $N^1(C_r,\Q)$ \footnote{In the next section we will see that there is a correspondence
between the covers $C_r$ and the subvector spaces $N^1(C_r,\Q)$, which justifies this notation.}
\index{$N^1(C_r,\Q)$} such that the $\Q$-vector space automorphism $\rho^1(\psi)|_{N^1(C_{r} , \Q)}$
is (up to conjugation)
given by a matrix
$$\left(\begin{array}{ccc}
M & & 0 \\
 & \ddots & \\
0 &  &  M
\end{array} \right),$$
where $M$ is the $k \times k$ matrix given by
$$M = \left(\begin{array}{ccccc}
 0& 0& \ldots & 0 & -p_0\\
 1 & 0 &\ldots & 0 & -p_1\\
 0 & 1 &\ddots & 0 & -p_2\\
\vdots &  & \ddots & \ddots & \vdots\\
 0 & \ldots & 0 & 1 & -p_{k-1}
\end{array} \right),$$
where $x^k + p_{k-1}x^{k-1}+ \ldots + p_1x+p_0$ is the minimal polynomial of $\xi^r$.
We call a $\Q$-vector space with such an automorphism of the form $\diag(M, \ldots, M)$ a
$\Q(\xi^r)$-structure. By $\xi^r \cdot v := g(v)$, this defines a scalar multiplication of
$\Q(\xi^r)$, which turns $N^1(C_{r} , \Q)$ into a $\Q(\xi^r)$-vector space.  We obtain:

\begin{proposition}
The direct sum decomposition
$$H^1(C, \Q) = \bigoplus\limits_{r|m} N^1(C_{r} , \Q)$$
is a direct sum of $\Q(\xi^r)$ structures on $H^1(C, \Q)$.
\end{proposition}

Next we consider the trace map
$$\tr:H_j^1(C, \Q(\xi)) \to H^1(C,\Q) \mbox{ given by } v \to
\sum\limits_{\gamma \in \Gal(\Q(\xi);\Q)} \gamma v,$$
which will be one of our main tools in this chapter. By the Galois group action, the vector space
$N^1(C_{r},\Q(\xi^r))$ decomposes into eigenspaces $H^1_j(C, \Q(\xi^r)))$ such that
$$H_j^1(C, \Q(\xi)) = H_j^1(C, \Q(\xi^r)) \otimes_{\Q(\xi^r)} \Q(\xi).$$

\begin{lemma}
Let $r|m$ and $r = \gcd(j,m)$. Then  $\tr|_{H_j^1(C, \Q(\xi^r))}$ is a monomorphism.
\end{lemma}
\begin{proof}
Let $f \in H_j^1(C, \Q(\xi^r))\setminus\{0\}$. We need
some Galois theory. By the fact that $\Q(\xi^r)$ is a Galois
extension of $\Q$, the group $\Gamma_r :={\rm Aut}(\Q(\xi);\Q(\xi^r))$ is
a normal subgroup of $(\Z/(m))^*\cong \Gamma := \Gal(\Q(\xi);\Q)$, which is the kernel of the
epimorphism $\Gamma \to \Gal(\Q(\xi^r);\Q)$ given
by $\gamma \to \gamma|_{\Q(\xi^r)}$ for all $\gamma \in
\Gal(\Q(\xi);\Q)$. Hence we obtain that

$$\tr(f) = \sum\limits_{\gamma \in \Gal(\Q(\xi);\Q)} \gamma f =
\sum\limits_{[\gamma] \in \Gamma /\Gamma_r} [\gamma]
\sum\limits_{\gamma \in \Gamma_r} \gamma f = \sum\limits_{\gamma \in
\Gal(\Q(\xi^r);\Q)} \gamma |\Gamma_r|f.$$ 
Since $\psi$ is acts by an integral matrix, one has $\gamma \circ \psi = \psi\circ  \gamma$ for all
$\gamma \in \Gamma$. This implies that
\begin{equation} \label{commy}
\gamma(\xi^r) \gamma(f) = \gamma(\xi^r f) = (\gamma\circ \psi)(f) = \psi(\gamma f).
\end{equation}
Thus $\gamma(f) \in H_{j_0j}^1(C, \Q(\xi))$, where $j_0 \in (\Z/(m))^*$ corresponds to $\gamma$.
By the fact that we have a direct sum of eigenspaces, we conclude that
$$\tr(f) = \sum\limits_{\gamma \in \Gal(\Q(\xi^r);\Q)} \gamma
|\Gamma_r|f \neq 0.$$
\end{proof}

Now we consider the restriction of the trace map to
$$R := \bigoplus\limits_{ r|m} H_{r}^1(C, \Q(\xi^r)).$$

\begin{proposition} \label{isotrace}
The trace map $\tr|_R : R \to H^1(C,\Q)$ is an isomorphism of
$\Q$-vector spaces.
\end{proposition}
\begin{proof}
Let
$$v := \sum\limits_{ r|m} v_r \in R$$
with $v_r \in H_{r}^1(C, \Q(\xi^r))$. By the proof of
the preceding lemma, we know that
$$\tr(v_r) = \sum\limits_{\gamma \in \Gal(\Q(\xi^r);\Q)} \gamma
|\Gamma_r|v \in \bigoplus\limits_{j\in (\Z/(\frac{m}{r}))^*} H_{j}^1(C, \Q(\xi)).$$
These $\xi^{jr}$ with $j\in (\Z/(\frac{m}{r}))^*$ are exactly the
$\frac{m}{r}$-th. primitive roots of unity. Thus they are the
elements with order $\frac{m}{r}$ in the multiplicative group generated by $\xi$. Hence by the
fact that we have a direct sum of eigenspaces, we conclude that
$\tr(v) = 0$ implies that $\tr(v_r) = 0$ for all $r$ with $r|m$. By
the preceding lemma, this implies that $v_r = 0$ for all $r$ with
$r|m$ and hencefore $v = 0$. Hence the map $\tr|_R$ is injective, and we
have only to verify that $\dim_{\Q}(R) = \dim_{\Q}(H^1(C, \Q))$:
$$\dim_{\Q}R =\sum\limits_{ r|m } \dim_{\Q(\xi)}(H_{r}^1(C, \Q(\xi))) \cdot
[\Q(\xi^r);\Q]$$
$$=\sum\limits_{ r|m } \dim_{\Q(\xi)}(H_{r}^1(C, \Q(\xi))) \cdot
\sharp\{\mbox{primitive } \frac{m}{r}\mbox{-th. roots of unity}\}$$
$$=\sum\limits_{j= 1 }^{m-1} \dim_{\Q(\xi)}(H_{j}^1(C, \Q(\xi)))
= \dim_{\Q(\xi)}(H^1(C, \Q(\xi))) = \dim_{\Q}(H^1(C, \Q))$$
\end{proof}

\begin{remark}
We know that the monodromy representation fixes $H^1(C,\Q)$ and each $H_{j}^1(C, \Q(\xi))$
invariant. By the fact that
$$N^1(C_{r} , \Q) = N^1(C_r,\Q(\xi)) \cap H^1(C,\Q),$$
we conclude that the monodromy representation fixes $N^1(C_{r} , \Q))$, too.
\end{remark}

\begin{proposition} \label{label}
The monodromy representation $\rho$ on $N^1(C_r,\Q)$ is given by

$$\rho(\omega) =\left(\begin{array}{ccc}
\gamma_1M_{\omega} & & \\
 & \ddots & \\
 & & \gamma_{k}M_{\omega}
\end{array} \right),$$
where $M_{\omega}$ denotes the image of $\omega$ in the monodromy of $H^1_r(C,\Q(\xi^r))$, and
$\{\gamma_1, \ldots, \gamma_k\} = \Gal(\Q(\xi^r);\Q)$.
\end{proposition}
\begin{proof}
Since $\rho(\gamma)$ leaves the eigenspaces invariant, it acts by $\diag(M_1, \ldots, M_k)$,
where each $M_{\ell}$ with $1 \leq \ell \leq k$ describes the action of $\rho(\omega)$ on
$\gamma_{\ell}H^1_r(C,\Q(\xi^r))$. Let $j_{\gamma} \in (\Z/(\frac{m}{r}))^*$ and $\gamma$ correspond.
The description of the $M_1, \ldots, M_k$ follows from the facts that each $\rho(\omega)$ commutes
with each $\gamma \in \Gal(\Q(\xi^r);\Q)$, and that $\gamma H^1_r(C,\Q(\xi^r)) =
H_{rj_{\gamma}}^1(C,\Q(\xi^r))$ (see $\eqref{commy}$) for all $\gamma \in \Gal(\Q(\xi^r);\Q)$.
\end{proof}

Now let $N_{\omega}$ denote the restriction of $\rho(\omega)$ on $N^1(C_r,\Q)$ and
$v \in N^1(C_r,\Q)$ given by $v = \tr(w)$
for some $w \in H^1_r(C,\Q(\xi^r)) $. By the preceding proposition, we have:
$$N_{\omega}(v) = N_{\omega} ([\Q(\xi);\Q(\xi^r)] \sum\limits_{\gamma \in \Gal(\Q(\xi^r);\Q)}
\gamma w) = [\Q(\xi);\Q(\xi^r)] \sum\limits_{i = 1}^{k} \gamma_i M_{\omega} ( \gamma_i (w))$$
$$=[\Q(\xi);\Q(\xi^r)]\sum\limits_{i = 1}^{k} \gamma_i(M_{\omega}( w)) = tr(M_{\omega} ( w))$$
The trace map $H^1_r(C,\Q(\xi^r)) \to N^1(C_r,\Q)$ is an isomorphism  of $\Q(\xi^r)$-vector
spaces with respect to the $\Q(\xi^r)$ structure on $N^1(C_r,\Q)$. Thus one has:

\begin{proposition}
The monodromy representation on $N^1(C_r,\Q)$ is a representation on a $\Q(\xi^r)$-vector
space given by the $\Q(\xi^r)$ structure, which coincides up to the trace map with the monodromy
representation on $H^1_r(C,\Q(\xi^r))$.
\end{proposition}

We will need a decomposition of $H^1(C,\R)$ into a direct sum of certain sub-vector spaces fixed by the Galois group representation. This decomposition is defined over
$\Q(\xi^j)^+ = \Q(\xi^j) \cap \R$ \index{$\Q(\xi^r)^+$} and given by the sub-vector spaces
$$\Re \V(j)  :=  (H^1_j(C,\Q(\xi)) \oplus H^1_{m-j}(C,\Q(\xi))) \cap H^1(C,\Q(\xi^j)^+).$$
\index{$\Re \V(j)$} Since the monodromy representation fixes
$$H^1_j(C,\Q(\xi)), \  \ H^1_{m-j}(C,\Q(\xi)) \  \ \mbox{and} \ \ H^1(C,\Q(\xi^j)^+),$$
it fixes $\Re \V(j)$, too.

\begin{remark}
One has that $\tr: H^1_{j}(C,\Q(\xi^{j})) \to N^1(C_j, \Q)$ coincides with
the composition
$$H^1_{j}(C,\Q(\xi^{j})) \stackrel{\tr}{\to} \Re \V(j) \stackrel{\tr}{\to}
N^1(C_j, \Q).$$
Hence the latter trace map $\Re \V(j) \stackrel{\tr}{\to}
N^1(C_j, \Q)$ induces a $\Q(\xi^j)^+$-structure on $N^1(C_j, \Q)$, which is compatible with the
$\Q(\xi^j)$-structure via $\Q(\xi^j)^+ \hookrightarrow \Q(\xi^j)$. Thus by the preceding results
about
the monodromy representation on $N^1(C_j, \Q)$, the monodromy representation on $N^1(C_j, \Q)$ is
a $\Q(\xi^j)^+$-vector space representation with respect to the $\Q(\xi^j)^+$-structure.
\end{remark}

\begin{remark}
In the case of $H^1_{\frac{m}{2}}(C,\Q(\xi^{\frac{m}{2}})$ one gets that $\Q(\xi^{\frac{m}{2}})
= \Q(-1) = \Q$.
In other terms: The monodromy group on
$H^1_{\frac{m}{2}}(C,\Q(\xi^{\frac{m}{2}})$ is the monodromy group on the rational vector space
$N^1(C_{\frac{m}{2}}, \Q)$.
\end{remark}

\section{Quotients of covers and Hodge group decomposition}

In this section we consider our quotient families $\pi_r: \sC_r \to
\sP_n$ of covers, and their Hodge groups. Moreover we will explain the notation $N^1(C_r,\Q)$
and show that the decomposition of $H^1(C,\Q)$ into these $\Q(\xi^r)$ structures is a
decomposition into rational sub-Hodge structures. Recall that $\sC_r$ is
given by a quotient of the subgroup of order $r$ of the Galois group
of $\sC$ (see Construction $\ref{hiercr}$).

Let $C$ and  $C_r$ denote a fiber of $\sC$ and the corresponding fiber of $\sC_r$ over the same
point. The natural cover $\phi_r: C \to C_r$ induces an
embedding of Hodge structures, which gives a direct sum
decomposition of $H^1(C,\Q)$ into two rational sub-Hodge structures (see \cite{Voi}, 7.3.2. and
\cite{Voi}, Lemma 7.26).

The Hodge structure on $H^1(C_r,\Q)$ is the sub-Hodge structure of $H^1(C,\Q)$ fixed by $\Gal(\phi_r)$. Hence the eigenspaces of $H^1(C_r,\C)$ with respect to the Galois
group $\pi_r$ can be identified with the eigenspaces of $H^1(C,\C)$, on which $\Gal(\phi_r)$
acts trivial. Thus one obtains
$$H^1(C_r,\C) = \bigoplus\limits_{j =1}^{\frac{m}{r}-1}H_{jr}^1(C,\C)
\hookrightarrow \bigoplus\limits_{j = 1}^{m-1} H_j^1(C,\C) = H^1(C,\C).$$
Recall that every eigenspace
$\sL_j$ of $R^1\pi_*(\C)$ is a local system. We consider the
eigenspace $(\sL_j)_{\sC_r}$ of $R^1(\pi_r)_*(\C)$ with the character $j$ and the eigenspace
$\sL_{jr}$ of $R^1\pi_*(\C)$. Proposition $\ref{hti}$ tells us that the local monodromy data of
$(\LL_j)_{C_r}$ and $\LL_{rj}$ coincide. By Proposition $\ref{repro}$, these monodromy data
determine the dual monodromy representations of the eigenspaces of the dual $VHS$ given by the
homology. Thus we obtain:

\begin{proposition} \label{projing}
The local systems $(\sL_j)_{\sC_r}$ and $\sL_{jr}$ coincide.
\end{proposition}

The following statements will explain the notation ``$N^1(C_r,\Q)$''. One has that
$$N^1(C_r,\Q) \otimes_{\Q} \C = \bigoplus\limits_{j \in (\Z/\frac{m}{r})^*}
H_{jr}^1(C,\C).$$
Since each $H_{jr}^1(C,\C) \subset N^1(C_r,\C)$ has a decomposition into
$$H_{j}^{1,0}(C_r) \oplus H_j^{0,1}(C_r), \  \ 
\mbox{where} \  \ \overline{H_{j}^{1}(C_r,\C)} = H_{m-j}^{1}(C_r,\C) \subset N^1(C_r,\C),$$
each $N^1(C_r,\Q)$ is a rational sub-Hodge structure of $H^1(C,\Q)$. Moreover each
$N^1(C_r,\Q)$ is the maximal sub-Hodge structure of $H^1(C_r,\Q)$, which is orthogonal (with
respect to the polarization) to each sub-Hodge structure of $H^1(C_r,\Q)$ given by a quotient
$H^1(C_{r'},\Q)$ with $r< r' < m$, $r|r'$ and $r'|m$. By using Lemma $\ref{7.1}$, we have the
result:

\begin{proposition} \label{hggrcomp}
We have a decomposition
$$H^1(C,\Q) = \bigoplus\limits_{r|m} N^1(C_r,\Q)$$
into rational Hodge structures and a natural embedding
$$\Hg(C) \hookrightarrow \prod\limits_{r|m}\Hg(N^1(C_r,\Q))$$
such that the natural projections
$$\Hg(C) \to \Hg(N^1(C_r,\Q))$$
are surjective for all $r$.
\end{proposition}

\begin{remark}
Note that the preceding section yields a corresponding statement about the Zariski closures
of the monodromy group of $R^1\pi_*(\Q)$ and the restricted representations monodromy
representations on the different $N^1(C_r,\Q)$. These two facts will play a very important role.
\end{remark}

\section{Upper bounds for the Mumford-Tate groups of the direct summands}
The different $N^1(C_r,\Q)$ on the fibers induce a decomposition of $R^1\pi_*(\Q)$ into a direct
sum \index{$\sN^1(\sC_r,\Q)$} \index{direct summand|see{$\sN^1(\sC_r,\Q)$}} of local systems
$\sN^1(\sC_r,\Q)$. Now we consider the induced variations $\sV_r$ of rational Hodge
structures on the local systems $\sN^1(\sC_r,\Q)$. Let $Q_r$ denote the alternating form on
$N^1(C_r,\Q)$ obtained by the restriction of the intersection form $Q$ of the curve $C$. One has
that each element of $\rho(\pi_1(\sP_n))$ commutes with the Galois group. The same holds true for
the image of the homomorphism
$$h : \BS \to \GSp(H^1(C,\R),Q)$$
corresponding to the Hodge structure of an arbitrary fiber.
Since the Galois group respects the intersection form, its representation on $N^1(C_r,\Q)$ is
contained in $\Sp(N^1(C_r,\Q),Q_r)$. Let $C_r(\psi)$ \index{$C_r(\psi)$} denote the
centralizer of the Galois group in $\Sp(N^1(C_r,\Q),Q_r)$ and $GC_r(\psi)$ denote the
centralizer of the Galois group in $\GSp(N^1(C_r,\Q),Q_r)$. One concludes:

\begin{proposition}
The centralizer $GC_r(\psi)$ contains the
generic Mumford-Tate group $\MT(\sV_r)$. Moreover the centralizer $C_r(\psi)$ contains the generic
Hodge group $\Hg(\sV_r)$ and $\Mon^0(\sV_r)$.
\end{proposition}

We write
$$C(\psi) := \prod\limits_{r|m} C_r(\psi).$$ \index{$C(\psi)$}

\begin{remark}
If $r \neq \frac{m}{2}$, the preceding proposition yields some information. In the case
$r = \frac{m}{2}$ the elements of the Galois group act as the multiplication with 1 or
$-1$ on $N^1(C_{\frac{m}{2}},\Q)$. Since ${\id}$ resp., $-\id$ is contained in the center of
$\Sp(N^1(C_{\frac{m}{2}},\Q),Q_{\frac{m}{2}})$, this proposition does not give any new information
in this case.
\end{remark}

Now let us assume that $r \neq \frac{m}{2}$. We describe $C_r(\psi)$ by its
$\R$-valued points. Let $\xi^j$ be a $\frac{m}{r}$-th. primitive root of unity such that
$H^1_j(C,\C) \subset N^1(C_r,\C)$, $v \in H^1_j(C,\C)$ and $M \in C_r(\psi)(\R)$. Then
one gets
$$\psi M(v) = M(\psi v) = M(\xi^jv) =\xi^j M(v).$$
Thus $M$ leaves each $H^1_j(C,\C)$ invariant.

For our description of $C_r(\psi)$ we introduce the
trace map
$$\tr:\GL(H^1_j(C,\C)) \to \GL(\Re\V(j)_{\R})$$
given by
\begin{equation} \label{trace}
\GL( H^1_j(C,\C) ) \ni N \to N \times \bar N \in \GL(H^1_j(C,\C)) \times \GL(H^1_{m-j}(C,\C)),
\end{equation}
where $\bar N$ denotes the matrix, which satisfies that $\bar N \bar v = \overline{Nv}$ for all
$v \in H^1_j(C,\C)$. Recall that we have a fixed complex structure. Thus one checks easily
that $N \times \bar N$ leaves $\Re \V(j)_{\R}$ invariant. Hence we consider it as a real matrix.

For the Hermitian form $H(\cdot,\cdot) :=iE(\cdot,\bar \cdot)$ and
$v,w \in H^1_j(C,\C)$ one obtains
$$H(v,w)= iE(v,\bar w) = iE(M v,M \bar w) = iE(M v,\overline{M w}) = H(Mv,Mw).$$
Thus the matrix
$M|_{\Re\V(j)_{\R}}$ is contained in $\tr(\U(H^1_j(C,\C), H|_{H^1_{m-j}(C,\C)})$.

Assume conversely that $M \in \GL(N^1(C_r,\C))$ satisfies that
$$M|_{\Re\V(j)_{\R}} \in \tr(\U(H^1_j(C,\C), H|_{H^1_j(C,\C)}))$$
for each $\frac{m}{r}$-th. primitive root of unity $\xi^j$. Since $M$ leaves all eigenspaces
$H^1_j(C,\C)) \subset N^1(C_r,\C)$ invariant, $M$ commutes with the Galois group representation on
$N^1(C_r,\R)$. Now let $N \in \GL(H^1_j(C,\C))$ be the matrix with $\tr(N) = M|_{\Re\V(j)_{\R}}$.
One has that
$$iE(v,\bar w) = iE(N v,\overline{N w}) \Leftrightarrow E(v,\bar w) = E(N v,\overline{N w})$$
for all $v, w \in H^1_j(C,\C)$. By the fact that $E$ is an alternating form, one gets
$$E(\bar v, w) = E(\overline{N v},N w),$$
too. Since each element of $\Re\V(j)_{\C}$ can be given by $v_1 + \bar v_2$ and
$w_1 + \bar w_2$ with $ v_1, v_2,w_1, w_2\in H^1_j(C,\C)$, one concludes that
$$E(v_1 + \bar v_2,w_1 + \bar w_2) = E(v_1,\bar w_2) + E(\bar v_2,w_1) = E(N v_1,\overline{N w_2})
+E(\overline{N v_2},N w_1)$$
$$= E(M v_1,M \bar w_2)
+E(M \bar v_2,M w_1) = E(M(v_1 + \bar v_2),M(w_1 + \bar w_2)).$$
Thus $M$ is contained in the symplectic group. Altogether we conclude:

\begin{Theorem} \label{dirprod}
If $r \neq \frac{m}{2}$, the group $C_r(\psi)(\R)$ is isomorphic to the direct product of the Lie
groups given by the $\R$-valued points of unitary groups over the spaces
$\Re\V(j)_{\R} \subset N^1(C_r,\R)$ induced by the trace maps and the unitary groups 
$\U(H^1_j(C,\C), H|_{H^1_j(C,\C)})$.
\end{Theorem}

Recall the definition of the type $(a,b)$ of an eigenspace $\sL_j$ in Remark $\ref{slj}$. If there
is an eigenspace of $N^1(C_r,\C)$ of type $(a,b)$ with $a> 0$ and $b >0$, we call $N^1(C_r,\Q)$
\index{direct summand!general} general. Otherwise we call it
\index{direct summand!special} special. Now assume that $N^1(C_r,\Q)$ is special. In this case
$h(\BS)$ is contained in the center of $GC_r(\psi)_{\R}$, and $h(S^1)$ is contained in the
center of $C_r(\psi )_{\R}$. Thus one concludes:

\begin{remark}
Assume that $N^1(C_r,\Q)$ is special. Then the  center $Z(GC_r(\psi))$ of $GC_r(\psi)$ contains
$\MT(\sV_r)$. Moreover the center $Z(C_r(\psi))$ of $C_r(\psi)$ contains $\Hg(\sV_r)$.
\end{remark}

\begin{remark} \label{nixda}
One has that $C_r(\psi)_{\R}$ consists of $\U(s)^t$ for some $s,t \in \N_0$, if $N^1(C_r,\Q)$ is
special. Thus in this case the monodromy group is a discrete sub-group of the compact group
$\U(s)^t$. Hence it is finite and $\Mon^0(\sV_r)$ is trivial in this case.
\end{remark}

\section{A criterion for complex multiplication}

In this short section we find a sufficient condition for the existence of a dense
set of $CM$ fibers of a family of cyclic covers. By technical reasons,
we do not consider the family $\sC \to \sP_n$, but a family over the space $\sM_n$, \index{$\sM_n$}
which can be considered as the quotient
$$\sM_n= \sP_n/{\rm PGL}_2(\C).$$
But one has an embedding $\iota_{a,b,c}: \sM_n \to \sP_n$, \index{$\iota_{a,b,c}$} too. Its image
is the subspace of $\sP_n$, which parameterizes the maps $\phi: N \to \bP^1$ satisfying
$\phi(a) = 0$, $\phi(b) = 1$ and $\phi(c) = \infty$ for some fixed $a,b,c \in
N$ (compare to \cite{DM}, 3.7).

\begin{remark} \label{mnn}
One can move 3 arbitrary branch points of a fiber of $\sC \to \sP_n$ to 0, 1 and $\infty$. Hence
one has that all fibers of the geometric points of $\sP_n$ occur as fibers of the restricted
family $\sC_{\sM_n} \to \sM_n$, too. Hence the generic Hodge groups and the generic Mumford-Tate
groups of the both families coincide.
\end{remark}

\begin{pkt} \label{japp}
Each curve $C$ with $g(C) > 1$ has at most $84(g-1)$ automorphisms (see
\cite{hart}, {\bf IV}. Exercise 2.5). Thus $C$ can have only finitely many cyclic covers onto
$\bP^1$ with different Galois groups. Moreover, there is
an automorphism $\alpha$ of $\bP^1$, if the Galois groups of the covers of $\sC_{p_1}$ and
$\sC_{p_2}$ can be conjugate by an isomorphism $\iota$ such that the following diagram commutes:
$$\xymatrix{
  {\sC_{p_1}} \ar[rr]^{\iota} \ar[d]  &  & {\sC_{p_2}}\ar[d]\\
  {\bP^1} \ar[rr]^{\alpha} &  & {\bP^1} 
}$$
Thus $C$ occurs only as finitely many fibers of $\sC_{\sM_n}$, if
$g(C) \geq 2$. \end{pkt}

Recall that we have defined the type of an eigenspace $\sL_j$ in Remark $\ref{slj}$.

\begin{definition} \index{pure $(1,n)-VHS$}
Let $\sC \to \sP_n$ be a family of cyclic covers onto $\bP^1$ and $C$ denote an arbitrary
fiber. The family $\sC$ has a pure $(1,n)-VHS$, if it has only one eigenspace $\sL_j$ of type
$(1,n)$ such that $\sL_{m-j}$ is of type $(n,1)$ with respect to the
Galois group representation, and all other eigenspaces
are of type $(a,0)$ or of type $(0,b)$ for some $a, b \in \N_0$.
\end{definition}

\begin{Theorem} \label{jnvz}
Let $\sC_{\sM_n} \to \sM_n$ be a family of cyclic covers onto $\bP^1$ and $C$ be a fiber with
$g(C) \geq 2$ as before. Assume that $\sC$ has a pure $(1,n)-VHS$. Then the family
$\sC_{\sM_n} \to \sM_n$ has a dense set of complex multiplication fibers.
\end{Theorem}
\begin{proof}
We have to show that over an arbitrary open simply connected subset $W$ of $\sM_n(\C)$ there are
infinitely many $CM$ points of the $VHS$ of $\sC_{\sM_n}$.
Let $q_0 \in W$ and $\sL_j$ be the eigenspace of type $(1,n)$. We have a trivialization
$$R^1\pi_*(\C)|_{W} = H^1(\sC_{q_0},\C) \times W \ \ \mbox{such that} \ \ \sL_j|_W \cong 
H_j^1(\sC_{q_0},\C) \times W.$$
Let $q \in W$ and $\varpi_q^{(j)} \in H_j^{1,0}(\sC_q)\setminus \{0\}$. By the holomorphic $VHS$
of the family, one obtains a holomorphic "fractional period" map
$$p: W \to \bP(H_j^1(\sC_{q_0},\C)) \  \ \mbox{via} \  \ q \to [\varpi_q^{(j)}].$$
By the assumptions, the integral Hodge structure depends uniquely on the class
$[\varpi_q^{(j)}] \in \bP(H_j^1(\sC_{q_0},\C))$. Since for each fiber there are only finitely many
isomorphic fibers (see
$\ref{japp}$) and two curves have isomorphic polarized integral Hodge structures, if and only
if they are isomorphic, the fibers of $p$ have the dimension 0. Thus \cite{Nara}, Chapter
{\bf VII}, Proposition 4 and the fact that $\dim W = \dim \bP(H_j^1(\sC_{q_0},\C))$ tell us that
$p$ is open.

The natural embedding $C(\psi) \hookrightarrow \GL(H^1(\sC_{q_0},\C))$ induces a holomorphic
variation of Hodge structures over the bounded symmetric domain associated with $C(\psi)(\R)/K$.
This $VHS$ depends uniquely on the fractional $VHS$ on the eigenspace $H_j^1(\sC_{q_0},\C)$ of type
$(1,n)$. Hencefore this $VHS$ yields a holomorphic injection $\varphi : C(\psi)(\R)/K \to
\bP(H_j^1(\sC_{q_0},\C))$.

Note that $C(\psi)(\R)/K$ parameterizes the integral Hodge structures of type $(1,0), (0,1)$ on
$H^1(\sC_{q_0},\C)$, whose Hodge group is contained in $C(\psi)$.
Hence altogether the map $\varphi^{-1} \circ p$, which assigns to each fiber $\sC_q$ its integral
Hodge structure, is open. Since the set of $CM$ points on $C(\psi)(\R)/K$ is dense (see
Theorem $\ref{commul}$), this yields the desired statement.
\end{proof}

\chapter{The computation of the Hodge group}

In this chapter we try to compute the derived group of the generic Hodge group of a family
$\sC \to \sP_n$. For infinitely many examples we will not be able to do this. But we will get many
information and in infinitely many examples we will obtain
$$\MT^{\der}(\sV) = \Hg^{\der}(\sV) = \Mon^0(\sV) = C^{\der}(\psi).$$

Recall that $\sP_n$ is the configuration space of $n+3$ points and $\sM_n = \sP_n/{\rm PGL}_2(\C)$.
Finally we will see that a family $\sC \to \sM_1$ induces an open period map
$$p: \sM_1(\C) \to \MT^{\ad}(\sV)/K,$$
if and only if it has a pure $(1,1)-VHS$.

\section{The monodromy group of an eigenspace}
Let $j \in \{1, \ldots, m-1\}$. Then we have an eigenspace $\sL_j$ in the variation
of Hodge structures of a family $\sC \to \sP_n$ of cyclic degree $m$ covers onto $\bP^1$.
There are $p, q \in \N_0$ such that the eigenspace $H^1_j(C,\C)$ of an arbitrary fiber $C$ is
of type $(p,q)$, where $(p,q)$ is the signature of the restricted polarization of the latter
eigenspace.
The type of $\sL_j$ is given by the type of $H^1_j(C,\C)$. The embedding
$\R \hookrightarrow \C$ allows to consider $H^1_j(C,\C)$ as $\R$-vector space. Let $\Mon^0(\sL_j)$
denote the identity component of the Zariski closure of the monodromy group of $\sL_j$ in
${\rm GL}_{\R}(H^1_j(C,\C))$.

We show in this section:

\begin{Theorem} \label{nundenn}
Let $\sL_{j}$ be of type $(p,q)$ with $p, q \geq 1$. Moreover assume that $j \neq \frac{m}{2}$ or
$p = q = 1$. Then
$$\Mon^0(\sL_j) = \SU(p,q).$$
\end{Theorem}

If $p = 0$ or $q = 0$, the statement of the preceding theorem does not hold true in general as one
can conclude by Remark $\ref{nixda}$.

We give a proof of Theorem $\ref{nundenn}$ by induction over the integer given by $p+q$.

By the following lemma, we start the proof of Theorem $\ref{nundenn}$:

\begin{lemma} \label{infini}
If $\sL_{j}$ is of type $(1,1)$, its monodromy group contains infinitely many elements.
\end{lemma}
\begin{proof}
There are two cases: In the first case there are some local monodromy data $\alpha_1$ and
$\alpha_2$ of the eigenspace $\LL_j$ in $(\pi_q)_*(\C_C)|_{\bP^1 \setminus S_j}$ for the fiber
$C := \sC_q$ of some arbitrary $q \in \sP_n$ such that $\alpha_1\alpha_2 = 1$. In this case the
Dehn twist $T_{1,2}$ yields a unipotent triangular matrix (follows by Proposition $\ref{repro}$)
and we are done.

Otherwise each Dehn twist $T_{k,\ell}$ provides a semisimple matrix, where its eigenvalues are
given by $1$ and a $m$-th. root of unity. Note that the matrices induced by the Dehn twists
$T_{1,2}$ and $T_{2,3}$ do not commute. In the considered case $\{a_1,a_2\}, \{a_3,a_4\}$ is a
stable partition. Hence one can choose the basis
$\sB = \{[e_1\gamma_1], [e_3\gamma_3]\}$ of $H_1^j(C,\C)$.
By the fact that these two cycles do not intersect each other, this basis
is orthogonal with respect to the Hermitian form induced by the intersection form. Hence by
normalization, this basis is orthonormal with respect to the Hermitian form such that the Hermitian
form is without loss of generality given by $\diag(1,-1)$ with respect to $\sB$. The matrix induced
by $T_{1,2}$ is given by $\diag(\xi,1)$ with respect to $\sB$, where $\xi$ is a $m$-th. root of
unity. Since the matrix $A$ of $T_{2,3}$ with respect to $\sB$ does not commute with
$\diag(\xi,1)$, it is not a diagonal matrix. Now we compute the commutator
$$K = A\cdot  \diag(\xi,1)\cdot  A^{-1}\cdot  \diag(\bar\xi,1).$$
One can replace $A$ by a non-diagonal matrix in $\SU(1,1)$ and the matrix $\diag(\xi,1)$
by $\diag(e,\bar e) \in \SU(1,1)$, where $e^2 = \xi$, for the computation of $K$. By \cite{Sat}, page
59, one has a description of the matrices in $\SU(1,1)(\R)$ such that
$$K =\left(\begin{array}{cc}
a & b\\ \bar b & \bar a 
\end{array} \right) \diag(e,\bar e) \left(\begin{array}{cc}
\bar a & -b\\
-\bar b & a 
\end{array} \right) \diag(\bar e,e) = \left(\begin{array}{cc}
 a \bar a -e^{-2} b \bar b &  a  b -e^{2}  a  b \\
\bar a \bar b -e^{-2} \bar a \bar b & a \bar a - e^2b \bar b
\end{array} \right).$$ 
Hence
$$tr(K) -2 = 2a \bar a - 2\Re(e^2) b \bar b -2= 2a \bar a - 2\Re(e^2) b \bar b - a \bar a+
b \bar b -1$$
$$\geq (a \bar a - |\Re(e^2)|b \bar b) + (b \bar b - |\Re(e^2)|b \bar b)-1\geq
a \bar a - |\Re(e^2)|b \bar b -1 \geq 0.$$
If the eigenvalues of $K$ would be roots of unity (if it is not unipotent), one would have
$|tr(K)| <2$. Hence by the fact that $tr(K) \geq 2$, one concludes that $K$ is unipotent or has
eigenvalues $v$ with $|v| \neq 1$. In both cases $K$ has infinite order.
\end{proof}

For the proof of Theorem $\ref{nundenn}$ we need to recall some facts about complex simple Lie
algebras. The complex simple Lie algebra $\fsl_n(\C)$ will be very important:

\begin{remark}
The Lie algebra $\fsl_n(\C)$ is given by
$$\fsl_n(\C) = \{ M \in M_{n \times n} (\C): {\rm tr}(M) = 0\}.$$
The Cartan subalgebra of $\fsl_n(\C)$ is given by
$$\fh=\{\diag(a_1, \ldots, a_n): \sum\limits_{i = 1}^n a_i = 0\}.$$
Each root space is given by the matrices $(a_{i,j})$, which have
exactly one entry $a_{i_0,j_0} \neq 0$ for a fixed pair $(i_0,j_0)$ with $i_0 \neq j_0$.
\end{remark}

We want to show a statement about unitary groups, and not about special linear groups. The
fact, which makes $\fsl_n(\C)$ interesting for us, is given by the following remark:

\begin{remark} \label{unisimple}
We can obviously embed $\fsu_{p,q}(\R)$ into $\fsl_{p+q}(\C)$, since
$\SU(p,q)(\R)$ is a Lie subgroup of $\SL_{p+q}(\C)$. Moreover
$i\fsu_{p,q}(\R)$ is a subvector space of $\fsl_{p+q}(\C)$ (considered
as real vector space). One has that
$$\fsu_{p,q}(\C) = \fsu_{p,q}(\R) \oplus i\fsu_{p,q}(\R) = \fsl_{p+q}(\C).$$
(see \cite{FH}, page 433)
\end{remark}

Moreover we need to compare the monodromy group of $\sL_j$ with the monodromy groups of some of
its restrictions over certain subspaces of $\sP_n$.

\begin{remark}
Consider some embedding $\iota_{a,b,c}:\sM_n \hookrightarrow \sP_n$. By the holomorphic
diffeomorphism
$$\PGL_2(\C)\times \iota_{a,b,c}(\sM_n)(\C) \ni M \times q \to M(q) \in \sP_n(\C),$$
we have that
$$\PGL_2(\C)\times \sM_n  \cong \sP_n \  \ \mbox{and} \  \
\pi_1(\PGL_2(\C)) \times \pi_1(\sM_n) \cong \pi_1(\sP_n),$$
where $\pi_1(\PGL_2(\C)) \cong \Z/(2)$ (compare \cite{DM}, $3.7$). 
\end{remark}

For technical reasons, we need to introduce an additional subspace of $\sP_n$:
\index{$\sP_n^{(a_k)}$}
$$\sP_n^{(a_k)} = \{q \in \sP_n|\phi_q(a_k) = \infty\}$$

Let $G_T$ denote the group of triangular matrices given by
$$G_T = \{\left(\begin{array}{cc}
a& 0\\
b & 1
\end{array} \right) \in M_{2\times 2}(\C)|a \neq 0\}.$$
We have obviously an embedding $\iota_{a,b,c}: \sM_n \hookrightarrow \sP_n^{(a_{n+3})}$
such that we get a holomorphic diffeomorphism
$$G_T\times \iota_{a,b,c}(\sM_n)(\C) \ni M \times q \to M(q) \in \sP^{(a_{n+3})}_n(\C).$$
Hence we have that
$$G_T\times \sM_n  \cong \sP^{(a_{n+3})}_n \  \ \mbox{and} \  \
\pi_1(G_T) \times \pi_1(\sM_n) \cong \pi_1(\sP^{(a_{n+3})}_n),$$
where $\pi_1(G_T) \cong \Z/(2)$.

The space $\sP_n^{(a_{n+3})}$ has a natural interpretation as configuration space of $n+2$ points
on $\R^2$. Its fundamental group is the colored braid group on $n+2$ strands in $\R^2$.

\begin{lemma} \label{hansen}
The fundamental group of the configuration space of $n+2$ points
on $\R^2$ is generated by the Dehn twists $T_{k_1,k_2}$ with
$1\leq k_1 < k_2 \leq n+2$.
\end{lemma}
\begin{proof}
(see \cite{Hansen}, Chapter {\bf I}. $4$)
\end{proof}

\begin{pkt}
By the preceding results, the monodromy groups of $\sL_j$, $(\sL_j)_{\sM_n}$ and
$(\sL_j)_{\sP_n^{(a_{n+3})}}$ are commensurable. Hencefore their $\R$-Zariski closures have the
same connected component of identity. Thus we do not need to distinguish between them and we will
call them simply $\Mon^0(\sL_j)$.
\end{pkt}

Again assume that $\sL_j$ is of type $(1,1)$. By Lemma $\ref{hansen}$, the monodromy group
$\rho_j(\pi_1(\sP_1^{(a_{4})} ))$ of $(\sL_j)_{\sP_1^{(a_{4})}}$ is generated by the
matrices $\rho_j(T_{k,\ell})$ for $k,\ell \in \{1,2,3\}$. For each
Dehn twist $T$ one can choose a suitable numbering of the branch points such that $T = T_{1,2}$.
Hence by Proposition $\ref{repro}$, one concludes that the generators of the monodromy group
are contained in the group given by
$$\{M \in GL_{2}(\C)|{\rm det}(M)^m = 1\}.$$
Since $\Mon^0(\sL_j)$ is contained in
$\U(1,1)$, one concludes that $\Mon^0(\sL_j) \subseteq \SU(1,1)$. Thus the
complexification of the
Lie algebra of $\Mon^0(\sL_j)$ is contained in $\fsl_{2}(\C)$. Note that the real Zariski closure
$\Mon^0(\Re\V(j)_{\R})$ is isomorphic to $\Mon^0(\sL_j)$ and $\Mon^0(\Re\V(j)_{\R})$ is a
quotient of the semisimple group $\Mon_{\R}^0(\sV_r)$. Thus by the kernel, which is semisimple, we
have an exact sequence of algebraic groups. This yields an exact sequence of semisimple Lie
algebras such that $\Mon^0(\sL_j)$ must be semisimple. One has that
$\Mon_{\C}^0(\sL_j) \subseteq \SU_{\C}(1,1)$. Since
$\fsu_{1,1}(\C) = \fsl_{2}(\C)$ is the smallest semisimple non-trivial complex Lie algebra (see
\cite{FH}, $\S14.1$, Step 3) and $\Mon^0(\sL_j)$ is infinite by Lemma $\ref{infini}$, one
concludes:

\begin{proposition}
If $\sL_j$ is of type $(1,1)$, then $\Mon^{0}(\sL_j) = \SU(1,1)$.
\end{proposition}

Recall that we want to give a proof of Theorem $\ref{nundenn}$ by induction. The following
construction explains our method to compare the monodromy groups of eigenspaces of different
type, which we will need for the induction:

\begin{construction}[Collision of points] \label{cop} \index{collision of points}
Let $\LL_j$ be an eigenspace in the cohomology of a fiber $C = \sC_q$ with the local monodromy
data $\alpha_k$ on $a_k$. Now let
$$b := \{a_{n_j+2},a_{n_j+3}\} \  \  \mbox{and} \  \  P =
\{\{a_1\}, \ldots,\{a_{n_j+1}\},  b \}$$
be a stable partition of $N = \{a_1, \ldots, a_{n_j+3}\}$. Let $\phi_P: P \to \bP^1$ be some
embedding and the local system $\LL(P)_j$ on
$\bP^1 \setminus \phi_P(P)$ have the local monodromy data 
$$\alpha_{b} = \alpha_{a_{n_j+2}}\alpha_{a_{n_j+3}} \  \ \mbox{and otherwise}
\    \ \alpha_{\{a_k\}} = \alpha_{a_k}.$$
By Construction $\ref{hiercr}$, these monodromy data allow the construction of a family of cyclic
covers
$$\pi(P): \sC(P) \to \sP_{n_j-1}.$$
The higher direct image sheaf $R^1 \pi(P)_*(\C)$ has an
eigenspace with respect to the character given by 1, which we denote by $\sL(P)_j$.\footnote{This
definition may seem to be a little bit odd. But it is motivated by some reasons, which
should become clearer by Remark $\ref{gugu}$.} By the description of the respective monodromy
representations in Proposition $\ref{repro}$, we can identify the monodromy group of
$(\sL(P)_j)_{\sP_{n_j-1}^{(b)}}$ with the subgroup of the monodromy group of
$(\sL_j)_{\sP_n^{(a_{n_3})}}$ generated by the Dehn twists $T_{a_{k_1},a_{k_2}}$ with
$k_1, k_2 \leq n_j+1$.
\end{construction}

\begin{remark} \label{gugu}
The local system $\sL(P)_j$ is in general not the $j$-th. eigenspace of a family of irreducible
covers of degree $m$ obtained by a collision of two branch points of a family of irreducible
covers of degree $m$. The problem is given by the irreducibility of the resulting family obtained
by collision. Take for example the family $\sC \to \sP_2$ with generic fibers given by
$$y^4 = (x-a_1)(x-a_2)(x-a_3)^2\cdot \ldots \cdot (x-a_5)^2.$$
By the collision of $a_1$ and $a_2$, one does not obtain an irreducible family of degree 4 covers.
But the resulting local system $\sL(P)_1$ is the eigenspace with respect to the character 1 on the
higher direct image sheaf of the family $\sC(P) \to \sP_1$ with generic fibers given by
$$y^2 = (x-a_1)\cdot \ldots \cdot (x-a_4).$$
\end{remark}

Now let $\sL_j$ be of type $(p,q)$ with $p,q >0$. By the collision of two points
and Proposition $\ref{1.27}$, one gets an eigenspace of type $(p,q-1)$ or of type $(p-1,q)$, if
there is a suitable corresponding stable partition. A little bit later we will see that this
construction yields an induction
step such that the statement of Theorem $\ref{nundenn}$ for local systems of type $(p,q-1)$ (if
$p, q-1 \geq  1$) and of type $(p-1,q)$ (if $p-1, q \geq  1$) implies
the statement of Theorem $\ref{nundenn}$ for local systems of type $(p,q)$.

For the application of the step of induction we will need a pair of stable partitions such that
the resulting two eigenspaces satisfy the assumptions of Theorem $\ref{nundenn}$. Moreover
one can assume that for each fiber $S_j$ contains at least 5 different points. Otherwise
$\sL_j$ is of type $(1,1)$ or unitary. By the following technical lemma, we start to show that
there exists a suitable pair of stable partitions, if the assumptions of Theorem $\ref{nundenn}$
are satisfied and if $S_j$ contains at least 5 points:

\begin{lemma} \label{pop}
Assume that $j \neq \frac{m}{2}$. Then there is an $a_k\in S_j$ with $\mu_k \neq \frac{1}{2}$.
\end{lemma}
\begin{proof}
Assume that all $a_k\in S_j$ satisfy  $\mu_k = \frac{1}{2}$ and $j \neq \frac{m}{2}$. One has that
$\sC_{r}$ (with $r = \gcd(m,j)$) is a family of
irreducible cyclic covers onto $\bP^1$ of degree $\frac{m}{r} > 2$ given by
$\mu_1,\ldots, \mu_{n+3}$ in the sense of Construction $\ref{hiercr}$. By the assumption that all
$a_k\in S_j$
satisfy $\mu_k = \frac{1}{2}$, each branch point has the same branch index $\frac{m}{2r}$, which
divides the degree $\frac{m}{r}$. Since we assume that $j\neq \frac{m}{2}$, one concludes that the
branch indeces given by $\frac{m}{2r}$ are not 1. Thus $\sC_{r}$ is not a family of irreducible
cyclic covers. Contradiction!
\end{proof}

Next we show that a $\mu_k \neq \frac{1}{2}$ yields two stable partitions:

\begin{lemma} \label{cases}
Assume that $S_j$ contains at least 5 different points such that there is an $a_k\in S_j$
with $\mu_k \neq \frac{1}{2}$. Then there are some pairwise different $\mu_h, \mu_i, \mu_s, \mu_t
\in S_j$ such that
$$\mu_h + \mu_i \neq 1, \ \ \mbox{and} \  \ \mu_s + \mu_t \neq 1.$$
\end{lemma}
\begin{proof}
Assume that each pair $h, i' \in \{1,\ldots, n+3\}$ with $h \neq i'$ satisfies
$\mu_h+ \mu_{i'} = 1$. This implies that $\mu_h = \mu_{i'} = \frac{1}{2}$ for each pair $h, i'$.
But this contradicts the assumptions of this lemma. Hence by the assumptions, there must be a pair
$(h,i')$ such that $\mu_h+ \mu_{i'} \neq 1$.

Now consider $S_j' := S_j \setminus \{a_h,a_i'\}$. Let us assume that each pair
$a_{s'}, a_{t'} \in S_j'$ with $s' \neq t'$ satisfies $\mu_{s'}+ \mu_{t'} = 1$. Since $|S_j'| \geq
3$, one concludes that $\mu_{s'} = \mu_{t'} = \frac{1}{2}$. Since $\mu_h = \frac{1}{2}$ or
$\mu_{i'}= \frac{1}{2}$ would contradict the assumptions in this case, one
concludes that $\mu_h, \mu_{i'} \neq \frac{1}{2}$. Hence put $i:=s', s :=i', t := t'$, and we are
done in this case.

If there are $a_{s'}, a_{t'} \in S_j'$ with $s' \neq t'$ and $\mu_{s'}+ \mu_{t'} \neq 1$, we
put $i := i', s := s', t := t'$, and we are done.

\end{proof}

By Lemma $\ref{pop}$ and Lemma $\ref{cases}$, one concludes immediately:

\begin{corollary} \label{1}
Assume that $S_j$ contains at least 5 different points and $j \neq \frac{m}{2}$. Then there are
some pairwise different $\mu_h, \mu_i, \mu_s, \mu_t \in S_j$ such that
$$\mu_h + \mu_i \neq 1, \ \ \mbox{and} \  \ \mu_s + \mu_t \neq 1.$$
\end{corollary}

\begin{remark} \label{2}
The condition that
$$\mu_h + \mu_i \neq 1, \ \ \mbox{and} \  \ \mu_s + \mu_t \neq 1$$
implies that
$$(\mu_h \neq \frac{1}{2} \  \ \mbox{or} \  \ \mu_i \neq \frac{1}{2}) \  \ \mbox{and} \  \ 
(\mu_s \neq \frac{1}{2} \  \ \mbox{or} \  \ \mu_t \neq \frac{1}{2}).$$
Hencefore the resulting eigenspace obtained by the collision of $a_h$ and $a_i$ resp.,
$a_s$ and $a_t$ satisfies that there is a local monodromy datum $\mu_k \neq \frac{1}{2}$. Hence
the resulting eigenspace is not a middle part $\sL_{\frac{m}{2}}$ of the $VHS$ of the family
obtained by the respective collision of two points. One must only ensure that the resulting
eigenspaces are not of type $(a,0)$ resp., $(0,b)$ in order to satisfy the assumptions of Theorem
$\ref{nundenn}$ in this case.
\end{remark}

\begin{pkt}
Assume that $\sL_j$ is of type $(1,n)$ with $n > 1$. By Proposition $\ref{1.27}$, one calculates
that
$$\sum\limits_{i = 1}^{n+3} \mu_i = 2$$
in this case. One can choose the indeces such that $$\mu_1 \leq \ldots \leq \mu_{n+3}.$$
Hence one has
$$\mu_1 + \mu_3 \leq \mu_2 + \mu_4 \leq \mu_3 + \mu_5.$$
By the fact that
$$(\mu_2 + \mu_4) + (\mu_3 + \mu_5) < 2 \   \ \mbox{and} \    \
\mu_2 + \mu_4 \leq \frac{1}{2}((\mu_2 + \mu_4) + (\mu_3 + \mu_5)),$$
one has
$$\mu_1 + \mu_3 \leq \mu_2 + \mu_4 < 1.$$
Since the local systems with respect to the corresponding stable partitions of the collision of
$a_1$ and $a_3$ resp., the collision of
$a_2$ and $a_4$ are of type $(1, n-1)$ as one can calculate by Proposition
$\ref{1.27}$, one can apply the induction hypothesis for these partitions.
\end{pkt}

Now let $\sL_{j}$ be of type $(n,1)$. Then the monodromy representation of $\sL_{j}$
is the complex conjugate of the monodromy representation of $\sL_{m-j}$, which is of type
$(1,n)$ in this case. Hence first the induction step yields the statement for all $\sL_j$ of
type $(1,n)$. Then we have the statement for all $\sL_j$ of type $(n,1)$, too. 

Assume that $\sL_{j}$ is of type $(p,q)$ with $p,q\geq 2$ and satisfies the assumptions of Theorem
$\ref{nundenn}$. By Corollary $\ref{1}$, one has a pair of stable partitions. Remark $\ref{2}$ and
the fact that $p,q\geq 2$ imply that the corresponding eigenspaces satisfy the assumptions of
Theorem $\ref{nundenn}$, too.

Now we must only prove and explain the step of induction:

One has without loss of generality the stable partitions
$$P_1 =  \{\{a_1\}, \ldots, \{a_{n+1}\}, \{a_{n+2}, a_{n+3}\}\}, \  \ \mbox{and} \   \
P_2 = \{\{a_{1}, a_{2}\}, \{a_3\}, \ldots,  \{a_{n+3}\}\}.$$
Here we assume without loss of generality that $a_k \in \R$ and $a_k < a_{k+1}$ such that
$\delta_k$ is the oriented path from $a_k$ to $a_{k+1}$ given by the straight line.

Let $q\in \sP_n$. We consider the monodromy representation with
respect to the basis $\sB$ of $(\sL_j)_q$ given by 
$$\sB = \{ [e_1\delta_1], \ldots, [e_n\delta_{n}], [e_{n+2}\delta_{n+2}]\}.$$
One has obviously that $\Mon^0(\sL_j(P_1))$ leaves $\langle[e_1\delta_1],\ldots, [e_n\delta_{n}]
\rangle$ invariant and fixes all vectors in $\langle[e_{n+2}\delta_{n+2}]\rangle$. Now let $U_1$
be a small open neighborhood of the identity in $\Mon^0(\sL_j(P_1))(\R)$ such that the ``inverse''
$$\log: U_1 \to Lie(\Mon^0(\sL_j(P_1)))$$
of the exponential map is defined on $U_1$. By Remark $\ref{unisimple}$ and the induction
hypothesis, $\log(U_1)$ generates a Lie algebra, whose complexification $L_1$ is with respect to
$\sB$ given by the matrices
$$\left(\begin{array}{cccc}
a_{1,1} & \ldots & a_{1,n} & 0\\
\vdots & & \vdots & \vdots\\
a_{n,1} & \ldots & a_{n,n} & 0\\
0 & \ldots & 0 & 0
\end{array} \right), \  \ \mbox{where} \  \ N := \left(\begin{array}{ccc}
a_{1,1} & \ldots & a_{1,n} \\
\vdots & & \vdots \\
a_{n,1} & \ldots & a_{n,n}
\end{array} \right)$$
is an arbitrary $n\times n$ matrix with $\tr(N) = 0$. Note that $\Mon^0(\sL_j(P_2))$
fixes all vectors in $\langle[e_1\delta_1]\rangle$ and leaves
$\langle[e_3\delta_3],\ldots, [e_{n+2}\delta_{n_2}]\rangle$ invariant. Hence in
a similar way $\log(U_2)$ ($e \in U_2 \subset \Mon^0(\sL_j(P_2))(\R)$) generates a Lie algebra.
Its complexification $L_2$ is given by the matrices
$$\left(\begin{array}{cc}

0& v\\
0 & N
\end{array} \right),$$
where $N$ is again an arbitrary $n\times n$ matrix with  $\tr(N) = 0$ and
$$v = (v_1, \ldots, v_n)$$
is a vector depending on $N$. It is easy to see that $L_1$ and $L_2$ generate $\fsl_{n+1}(\C)$.

Since $\Mon^0(\sL_j)$ is contained in $\SU(p,q)$ and $\fsu_{p,q} \otimes \C \cong \fsl_{n+1}(\C)$,
the group $\Mon^0(\sL_j)$ is isomorphic to $\SU(p,q)$.

\section{The Hodge group of a general direct summand}
The $VHS$ of a family $\sC \to \sP_n$ has a decomposition into rational
subvariations $\sV_r$ of Hodge structures, which where introduced in Section $4.3$. Recall that
$\sV_r$ is general, if its monodromy group is infinite. Otherwise we call it special. Let
$r \neq \frac{m}{2}$, $\sV_r$ be general and $\sL_j \subset \sV_r$ in this section. Moreover
recall that $\Mon^0_{\R}(\sV_r)_q$ denotes the connected component of identity of the
Zariski closure of the monodromy group in $\GL(((\sV_r)_{\R})_q)$ for some $q \in \sP_n$. Since
$\Mon^0_{\R}(\sV_r)_{q_1}$ and $\Mon^0_{\R}(\sV_r)_{q_2}$ are conjugated, we write
$\Mon^0_{\R}(\sV_r)$ \index{$\Mon_{\R}^0(\sV_r)$} instead of $\Mon^0_{\R}(\sV_r)_q$ for simplicity.

\begin{remark}
The group
$\Mon_{\R}^0(\sV_r)$ does not need to be equal to $\Mon^0(\sV_r)\times_{\Q} \R$. But it satisfies
$\Mon_{\R}^0(\sV_{r})\subseteq \Mon^0(\sV_r)\times_{\Q} \R$. Hence $\Mon_{\R}^0(\sV_r)$
yields a lower bound for $\Mon^0(\sV_r)$. Thus one obtains $C^{\der}_r(\psi)  = \Mon^0(\sV_r)$, if
$C^{\der}_r(\psi)_{\R}  = \Mon_{\R}^0(\sV_r)$.
\end{remark}

By the preceding section,
we know that $\Mon_{\R}^0(\sV_r) \to \Mon^0(\sL_j)$ can be considered as the projection onto
some $\SU(a,b)$, if $\sL_j$ is of type $(a,b)$ with $a, b > 0$. Otherwise one can use induction
with the corresponding stable partitions again. We only consider the start of induction:

Assume that $S_j = 4$. hence one has without loss of generality $\sC_r \to \sP_1$. By our
assumptions, there is an eigenspace $\sL_{j_2}$ in $\sN^1(\sC_r,\C)$ of type $(1,1)$, whose
monodromy group is infinite. Since the monodromy
group of $\sL_j$ is conjugated to the monodromy of $\sL_{j_2}$ by some $\gamma \in
Gal(\Q(\xi^r);\Q)$, it is infinite, too. One concludes similarly to the preceding section that
$\Mon^0(\sL_j) = \SU(2)$ (since $\fsu_2(\C) = \fsl_2(\C)$ by \cite{FH}, page 433, too).
The rest of the proof is an induction analogue to the induction of the
preceding section.

By the preceding considerations, one has:

\begin{proposition}
Assume that $\sV_r$ is general. Then the image of the natural projection $\Mon_{\R}^0(\sV_r) \to
\GL(\Re\V(j)_{\R})$ is given by the special unitary group induced by the trace map
and the special unitary group $\SU(H_j^1(C,\C), H|_{H_j^1(C,\C)})$ described in Section $4.3$. 
\end{proposition}

Moreover we know that $\Mon_{\R}^0(\sV_r)$ is contained in $C_r^{\der}(\psi)_{\R}$, which is given
by a direct product of
certain groups $\SU(a,b)$. Either $\Mon^0(\sV_r) = C_r^{\der}(\psi)$ or it is given by a proper
subgroup. We want to examine the conditions of the case $\Mon^0(\sV_r) \neq C_r^{\der}(\psi)$.
This will yield information and some criteria for the structure of $\Mon^0(\sV_r)$.

First let us make a simple, but very useful observation: 

\begin{remark} \label{tutty}
Let $G_1, \ldots, G_t$ be connected simple Lie groups and
$N \subset G_1 \times  \ldots \times  G_t =: G$ be a normal connected subgroup.  One has that $Lie(G)$ is a direct sum of the
simple ideals $Lie(G_1), \ldots, Lie(G_t)$, which implies that each ideal is a sum of certain
$Lie(G_i)$ (see \cite{Helga}, {\bf II}. Corollary $6.3$). Since the normal
connected subgroups of $G$ and the ideals of $Lie(G)$ correspond (follows by \cite{FH},
Proposition $8.41$ and \cite{FH}, Exercise $9.2$), one obtains that
$$N = G_1 \times  \ldots \times G_{t_0}\times \{e\} \times \ldots \times \{e\}$$
for some $t_0 \leq t$ with respect to a suitable numbering.
\end{remark}

The decomposition of the rational Hodge structure $N^1(C_r,\Q)$ into the $\Q(\xi^r)^+$-Hodge
structures $\Re\V(j)$ yields a decomposition of the variation $\sV_r$ of rational Hodge structures
into the variations $\Re\sV(j)$ of $\Q(\xi^r)^+$-Hodge structures.

By technical reasons, we consider the semisimple adjoint group $\Mon_{\R}^{\ad}(\sV_r)$ instead of
$\Mon_{\R}^0(\sV_r)$ first. By Remark $\ref{tutty}$, one concludes that $\Mon_{\R}^{\ad}(\sV_r)$ is isomorphic
to the direct product of $\Mon^{\ad}(\Re \sV(j)_{\R})$ and the kernel $K_j$ of the natural projection
$\Mon_{\R}^{\ad}(\sV_r) \to \Mon^{\ad}(\Re \sV(j)_{\R})$. Moreover one has:

\begin{lemma} \label{5.15}
Let $G_1, \ldots, G_t$ be simple adjoint Lie groups and $G$ be a semisimple subgroup of 
$G_1 \times  \ldots \times  G_t$ such that each natural projection
$$G \hookrightarrow G_1 \times  \ldots \times  G_t \stackrel{pr_j}{\longrightarrow} G_j$$
is surjective. One has $G \neq G_1 \times  \ldots \times  G_t$, if and only if there are some
$j_1,j_2 \in \{1, \ldots, t\}$ with $j_1 \neq j_2$ such that $G$ contains a simple subgroup $G'$
isomorphically mapped onto $G_{j_1}$ and $G_{j_2}$ by the
natural projections.
\end{lemma}
\begin{proof}
The "if" part is easy to see. The "only if" part follows by induction.
\end{proof}

Note that we have a natural embedding
$$\Mon_{\R}^{\ad}(\sV_r)\hookrightarrow \prod\limits_{j \in \Z/\frac{m}{r}, j \leq \frac{m}{2}}
\Mon^{\ad}(\Re \sV(j)_{\R}).$$
Thus the preceding lemma and our assumption that $\Mon^0(\sV_r) \neq C_r^{\der}(\psi)$ imply that
there is a direct simple factor of $\Mon_{\R}^{\ad}(\sV_r)$, which isomorphically mapped onto
$\Mon^{\ad}( \Re \sV(j_1)_{\R} )$ and $\Mon^{\ad}( \Re \sV(j_2)_{\R} )$ for some $j_1$ and $j_2$
with $j_2 \neq j_1$ and $m-j_1$. By Remark $\ref{tutty}$, $\Mon_{\R}^{\ad}(\sV_r)$ is a direct
product of the kernel of the both projections and this direct simple factor.

Thus the natural projections onto
$\Mon^{\ad}(\Re \sV(j_1)_{\R})$ and $\Mon^{\ad}(\Re \sV(j_2)_{\R})$ yield an isomorphism 
$$\alpha^{\ad} : \Mon^{\ad}(\Re \sV(j_1)_{\R})\to \Mon^{\ad}(\Re \sV(j_2)_{\R}).$$
Moreover one concludes that the
image $\Mon^{\ad}(\Re \sV(j_1,j_2)_{\R})$ of the projection 
$$\Mon_{\R}^{\ad}(\sV_r) \to \Mon^{\ad}( \Re \sV(j_1)_{\R} ) \times \Mon^{ad}(\Re \sV(j_2)_{\R}) $$
is given by the graph of $\alpha^{\ad}$.

\begin{pkt}
For the image $\Mon^{0}( \Re \sV(j_1,j_2)_{\R} )$ \index{$\Mon^{0}( \Re\sV(j_1,j_2)_{\R} )$} of the
projection
$$\Mon_{\R}^{0}(\sV_r) \to \Mon^{0}(\Re \sV(j_1)_{\R}) \times \Mon^{0}(\Re \sV(j_2)_{\R})$$
this implies that the natural projections
$$p_1: \Mon^{0}(\Re\sV(j_1,j_2)_{\R}) \to \Mon^{0}(\Re \sV(j_1)_{\R})$$
and
$$p_{2}: \Mon^{0}(\Re\sV(j_1,j_2)_{\R}) \to \Mon^{0}(\Re \sV(j_2)_{\R})$$
are isogenies. Since
$$\Mon^{0}(\Re \sV(j_1)_{\R})(\C) = \Mon^{0}(\Re \sV(j_2)_{\R})(\C) = \SL_{a+b}(\C),$$
where $(a,b)$ is the type of $\sL_{j_1}$, and the Lie group $\SL_{a+b}(\C)$ is simply connected
(see \cite{FH}, Proposition $23.1$), the induced isogenies of Lie groups of $\C$-valued points are
isomorphisms. Hence the isogenies $p_1$ and $p_2$ are isomorphisms.

Hence our assumption implies the existence of an isomorphism
$$\alpha : \Mon^{0}(\Re \sV(j_1)_{\R})\to \Mon^{0}(\Re \sV(j_2)_{\R}),$$
which satisfies that $\Mon^{0}(\Re \sV(j_1,j_2)_{\R})$ is given by ${\rm Graph}(\alpha)$.
\end{pkt}

\section{A criterion for the reaching of the upper bound}
In this section we give a necessary criterion for the existence of an isomorphism $\alpha$.
This yields a sufficient condition that $\Mon^0(\sV_r)$ reaches the upper bound
$C^{\der}_r(\psi)$. In addition we will see that $\Mon^0(\sV) = \Mon^0(\sV_1)$ reaches the upper
bound, if the degree $m$ of the covers given by the fibers of $\sC \to \sP_n$ is a prime number
$>2$.\footnote{For $m = 2$ we will later see that $\Mon^0(\sV)$ reaches the upper bound as well.}

We say that a Dehn twist $T$ is semisimple \index{semisimple Dehn twist} (with respect to
$\sV_r$), if the monodromy representation $\rho_j$ of one (and hence of all) $\sL_j \subset \sV_r$
yields a semisimple matrix $\rho_j(T)$. By the trace map (see $\eqref{trace}$),
we can identify 
$\Mon^{\ad}(\Re \sV(j)_{\R})$ and $\Mon^{\ad}( \sL_j)$. Thus $\Mon^{0}(\Re \sV(j_1,j_2)_{\R})$
is equal to ${\rm Graph}(\alpha)$, if and only if one has a corresponding isomorphism
$\alpha^{\ad}:\Mon^{\ad}(\sL_1) \to \Mon^{\ad}(\sL_2)$ such
that $\Mon^{\ad}(\sL_{j_1} \oplus \sL_{j_2})$ is given by ${\rm Graph}(\alpha^{\ad})$. By an abuse
of notation, we will write $\alpha$ instead of $\alpha^{\ad}$ from now on.

First let us formulate a sufficient criterion for the non-existence of $\alpha$ in the case
$\sC \to \sP_1$:

\begin{proposition} \label{uekur}
Let $\sV_r$ be general and $\sL_{j_1}, \sL_{j_2} \subset \sV_r$ be of type $(a,b)$, where
$a+b = 2$. Then there does not exist an isomorphism $\alpha: \Mon^{\ad}(\sL_1) \to
\Mon^{ad}(\sL_2)$ such that $P\rho_{j_2} = \alpha \circ P\rho_{j_1}$, if there is a semisimple
Dehn twist $T$ such that the non-trivial eigenvalue $z_2$ of $\rho_{j_2}(T)$ is not contained in
$\{z_1,\bar z_1\}$, where $z_1$ denotes the non-trivial eigenvalue of $\rho_{j_1}(T)$.
\end{proposition}
\begin{proof}
Assume that $\Mon^{\ad}(\sL_1)$ and $\Mon^{ad}(\sL_2)$ are isomorphic and $T$ satisfies the
assumptions of this proposition. Thus $\rho_{j_1}(T)$ generates a finite commutative subgroup $FT$ of
$\Mon^{\ad}(\sL_{j_1})$, which is contained in a maximal torus $G$ of $\Mon^{\ad}(\sL_{j_1})$. Our
assumption that $a+b = 2$ implies that $\Mon^{\ad}(\sL_{j_1}) \cong \Mon^{\ad}(\sL_{j_2})$ is
isomorphic to $\PU(1,1)$ or $\PU(2)$. Note that the maximal tori of reductive groups are conjugate
and in the case of $\PU(1,1)$ resp., $\PU(2)$ the $\R$-valued points of $G(\R)$ can (up to
conjugation) be given by the diagonal matrices in $\PU(1,1)$ resp., $\PU(2)$. Thus one checks
easily that $G$ is unique and
that the Lie group $G(\R)$ is isomorphic to $S^1(\R)$. Hence one can identify $FT(\R)$ with some
$\langle\xi^{s}\rangle \subset S^1(\R)$. Now let $1 \neq \zeta \in \langle\xi^{s}\rangle$ satisfy
the property that there is a closed interval on $S^1(\R)$ with end points $1$ and $\zeta$,
which does not contain any other element of $\langle\xi^{s}\rangle$. Hence
there is a closed interval $I$ on $T$ with ending points $[\diag(1,1)]$ and
$[\diag(\zeta ,1)] \in FT$, which does not contain any other element of $FT$.

Now assume such an isomorphism $\alpha$ exists. Note that we have an identification
$\alpha(G)(\R) = S^1(\R)$, too. But our assumptions imply that
$$\alpha(\diag(\zeta ,1)) \notin
\{\diag(\zeta ,1), \diag(\bar \zeta ,1)\}.$$
Hence by our identification $\alpha(G)(\R) = S^1(\R)$, one obtains that
$$\alpha(\zeta) \notin \{\zeta ,\bar \zeta \}.$$
Thus $\alpha(I) \subset \alpha(G)(\R)$ is not a connected interval, which does not contain any
other element of $\langle\xi^{s}\rangle$ except of $1$ and $\alpha(\zeta)$. But $\alpha$ must be
a homeomorphism on the $\R$-valued points. Contradiction!
\end{proof}

By the preceding proposition, we can use certain semisimple Dehn twists for the study of the
generic Hodge group. Hence we make some observations about the orders and the existence of
semisimple Dehn twists:

\begin{lemma} \label{buaha}
Let $j \neq \frac{m}{2}$ and $v|\frac{m}{r}$, where
$$1 \neq v, \  \ r := \gcd(m,j) \  \ \mbox{and} \  \  1, 2 \neq \frac{m}{rv}.$$
Then there exists a Dehn twist $T \in \pi_1(\sP_n)$ such that
$\rho_j(T) \in \Mon(\sL_j)$ is semisimple and $|\langle\rho_j(T)\rangle|$ does not divide $v$.
\end{lemma}
\begin{proof}
One can replace $\sC$ by $\sC_r$ and choose a suitable collection of local monodromy data for $\sC$
such that $j= 1$. By an isomorphism $\langle\xi\rangle \cong \Z/(m)$, the non-trivial eigenvalues
of the semisimple Dehn twists $T_{k_1,k_2}$ correspond to some elements $[b_{k_1,k_2}]\in \Z/(m)$,
where $b_{k_1,k_2} := d_{k_1}+d_{k_2}$ and $d_{k_1}$ and $d_{k_2}$ denote the branch indeces of
$a_{k_1}$ and $a_{k_2}$.

Assume that each semisimple Dehn twist satisfies that its order divides some $v$ with $v|m$.
This implies that $\frac{m}{v}|b_{k_1,k_2}$ for all $b_{k_1,k_2}$. Hence for all
$k = 1, \ldots, n+3$ one has that $\frac{m}{v}$ divides
$$2d_{k} = (d_k+d_{k_1}) +(d_k + d_{k_2}) -(d_{k_1} + d_{k_2}) = b_{k,k_1}+b_{k,k_2}-b_{k_1,k_2}.$$
Since there does not exist any integer $N\neq 1$, which divides each $d_k$, one has that
$\frac{m}{v}$ divides $2$. This implies that $\frac{m}{v}= 1$ or $\frac{m}{v}= 2$.
\end{proof}

For the formulation of our criterion in the higher dimensional case we need the following lemma:

\begin{lemma} \label{5.19}
Let $q \in \sP_n$. Assume that we have a stable partition
$$P :=\{\{a_1\}, \{a_2\}, \{a_3\}, \{a_4, \ldots, a_{n_j+3}\}\}$$
with respect to the local monodromy data of $(\sL_j)_q$ such that we can define the
eigenspace $\sL_j(P)$ over $\sP_1$ with 
$b = \{a_4, \ldots, a_{n_j+3}\}$ as in Construction $\ref{cop}$. Then the monodromy group
$\rho_j(P)(\pi_1(\sP_1))$ of $\sL_j(P)$ has a subgroup of finite index generated by
$\rho_{j}(T_{1,2})$ and $\rho_{j}(T_{2,3})$.
\end{lemma}
\begin{proof}
The stability of the partition ensures that
$\alpha_b = \alpha_{a_4} \ldots \alpha_{n_j+3} \neq 1.$ It is a well-known fact that
$\pi_1(\sM_1(\C))$ is generated by the two loops around 0 and 1, where we identify
$\A^1\setminus\{0,1\} = \sM_1$. By the embedding $\sM_1 \to \sP_1$ given by
$$a_1 = 0, \  \ a_3 = 1, \  \ a_4 = \infty,$$
we can identify the generators of $\pi_1(\sM_1(\C))$ with the Dehn twists $T_{1,2}$ and $T_{2,3}$. The
statement follows from the fact that the monodromy group of $\sL_j(P)|_{\sM_1}$ has finite index
in the monodromy group of $\sL_j(P)$.
\end{proof}

\begin{proposition} \label{ula}
Let $\sL_{j_1}, \sL_{j_2} \subset (\sV_1)_{\C}$ with $j_1 \neq j_2$ and $j_1 \neq m-j_2$.
Assume that we have a stable partition
$$P :=\{\{a_1\}, \{a_2\}, \{a_3\}, \{a_4, \ldots, a_{n+3}\}\}$$
such that the monodromy group of $\sL_{j_1}(P)$ or $\sL_{j_2}(P)$ is infinite. Let
$\Mon^0(\sL_{j_1}(P))$ and
$\Mon^0(\sL_{j_2}(P))$ be not isomorphic or $T_{k,\ell}$ be a semisimple Dehn twist with
$k, \ell \in \{1,2,3\}$ such that the non-trivial eigenvalue $z_2$ of $\rho_{j_2}(T_{k,\ell})$ is
not contained in $\{z_1,\bar z_1\}$, where $z_1$ denotes the non-trivial eigenvalue of
$\rho_{j_1}(T_{k,\ell})$. Then
$$\Mon^{0}( \Re\sV(j_1,j_2)_{\R} ) =\Mon^{0}( \Re\sV(j_1)_{\R} )\times
\Mon^{0}( \Re\sV(j_2)_{\R}).$$
\end{proposition}
\begin{proof}
By Lemma $\ref{5.19}$ and the fact that the monodromy group of $\sL_j|_{\sM_n}$ has finite index
in the monodromy group of $\sL_j$, one concludes that the group generated by $\rho_{j_1}(T_{1,2})$
and $\rho_{j_1}(T_{2,3})$ resp., $\rho_{j_2}(T_{1,2})$ and $\rho_{j_2}(T_{2,3})$ has finite index
in the monodromy representation of $\sL_{j_1}(P)$ resp., $\sL_{j_2}(P)$. Hencefore an isomorphism
$$\alpha:\Mon^{0}( \Re\sV(j_1)_{\R} )\to \Mon^{0}( \Re\sV(j_2)_{\R})$$
yields an isomorphism
$$\alpha(P):\Mon^0(\sL_{j_1}(P)) \to \Mon^0(\sL_{j_2}(P)).$$
Thus one only needs to apply Proposition $\ref{uekur}$.
\end{proof}

Now let us first define the condition for the reaching of the upper bound and then write
down the obvious theorem:

\begin{definition} \index{direct summand!very general}
Assume that one has for each $\sL_{j_1}, \sL_{j_2} \subset \sV_r$ with
$j_1 \neq j_2, m-j_2$ and $\Mon_{\R}^0(\sL_{j_1}) \cong \Mon_{\R}^0(\sL_{j_1})$ a stable
partition
$$P :=\{\{a_1\}, \{a_2\}, \{a_3\}, \{a_4, \ldots, a_{n_j+3}\}\}$$
(with respect to a suitable enumeration) such that the monodromy group of $\sL_{j_1}(P)$ or $\sL_{j_2}(P)$ is
infinite and one of the following conditions is satisfied:
\begin{enumerate}

\item $\Mon^0(\sL_{j_1}(P))$ and $\Mon^0(\sL_{j_2}(P))$ are not isomorphic. 
\item There is a semisimple Dehn twist $T_{k,\ell}$ with
$k, \ell \in \{1,2,3\}$ such that the non-trivial eigenvalue $z_2$ of $\rho_{j_2}(T_{k,\ell})$ is
not contained in $\{z_1,\bar z_1\}$, where $z_1$ denotes the non-trivial eigenvalue of
$\rho_{j_1}(T_{k,\ell})$.
\end{enumerate}
We call $\sV_r$ very general in this case.

A direct summand $\sV_r$ is exceptional, if it is general, but not
\index{direct summand!exceptional} very general.
\end{definition}

By Proposition $\ref{ula}$, one concludes:

\begin{Theorem}
If $\sV_r$ is very general, $\Mon^0(\sV_r)$ reaches the upper bound $C^{\der}(\psi)$.
\end{Theorem}

\begin{Theorem}
If the degree $m$ of the covers given by the fibers of $\sC \to \sP_n$ is a prime number $m > 2$,
$\Mon^0(\sV) = \Mon(\sV_1)$ reaches the upper bound. 
\end{Theorem}
\begin{proof}
By the preceding theorem, we have only to show that $\Mon^0(\sV) = \Mon(\sV_1)$ is very general.
Note that Lemma $\ref{buaha}$ implies that there is a semisimple Dehn twist for $m > 2$.

Assume that we are in the case of a family $\sC \to \sP_1$, and that $j_1 \neq j_2, m-j_2$. Since
$\Z/(m)$ is a field in our case, one has that each semisimple
Dehn twist satisfies that the non-trivial eigenvalue of $\rho_{j_2}(T)$ is not contained in
$\{z_1,\bar z_1\}$, where $z_1$ denotes the non-trivial eigenvalue of $\rho_{j_2}(T)$. Thus in
this case the statement follows from Proposition $\ref{uekur}$.

Otherwise we have to find a stable partition $P$ as in Proposition $\ref{ula}$. One has without
loss of generality the semisimple Dehn twist $T_{1,2}$. Moreover assume without loss of generality
that $d_1 + d_2 = m-1$. One has two cases: Either there is some $a_3$ such that
$$P = \{\{a_1\},\{a_2\}, \{a_3\}, \{a_4, \ldots, a_{n+3}\}\}$$
is the desired stable partition or one has that
$$d_3 = \ldots = d_{n+3} = 1.$$
Since in the case $m = 3$ there is nothing to show, one can otherwise assume that $m > 3$ and take
the stable partition
$$P = \{\{a_3\},\{a_4\}, \{a_5\}, \{a_1,a_2,a_6, \ldots, a_{n+3}\}\}.$$
\end{proof}

\section{The exceptional cases}
At this time the author does not see a possibility to calculate the monodromy group of the $VHS$
of an arbitrary family $\sC \to \sP_n$. Hencefore we consider mainly a family $\sC \to \sP_1$.

\begin{pkt} 
Let $\rho_{j_1}$ and $\rho_{j_2}$ denote the monodromy representations of
$\sL_{j_1}, \sL_{j_2} \subset \sV_r$. Proposition $\ref{repro}$ yields a description of
$\rho_{j_1}(T)$ and $\rho_{j_2}(T)$ for some Dehn twist $T$. By this description, the entries of the
matrices $\rho_{j_1}(T)$ and $\rho_{j_2}(T)$ differ by some $\gamma \in \Gal(\Q(\xi^r);\Q)$. By
its action on $\langle\xi^r\rangle \cong \Z/(\frac{m}{r})$, each $\gamma$ can be identified with
some $[v] \in (\Z/(\frac{m}{r}))^*$ such that $[\frac{j_1}{r}v]_{\frac{m}{r}} =
[\frac{j_2}{r}]_{\frac{m}{r}}$. One has a subgroup $H_1(\gamma)$ of $\langle\xi^r\rangle$
consisting of roots of unity fixed by $\gamma$ and a subgroup $H_2(\gamma)$ of
$\langle\xi^r\rangle$ consisting of roots of unity, on which $\gamma$ acts by
complex conjugation. Since $j_1 \neq j_2, m-j_2$, one has that
$\gamma$ is neither given by the complex conjugation nor by the identity. Thus $H_1(\gamma)$
resp., $H_2(\gamma)$ is
given by $\{1\}$ or some proper subgroup of $\langle\xi^r\rangle$ generated by
$\xi^{rt_1(\gamma)}$ resp., $\xi^{rt_2(\gamma)}$, where $1 \neq t_{1}(\gamma)$ and
$1 \neq t_{2}(\gamma)$ divide $\frac{m}{r}$.
\end{pkt}

For the rest of this section we consider only families $\sC \to \sP_1$ of degree $m$ with an
exceptional part $\sV_r$. Assume without loss of generality that $\sV_1$ is exceptional and $j_1
=1$. Let $\gamma$ correspond to $v$. For simplicity we write $t_1$ and $t_2$ instead of
$t_1(\gamma)$ and $t_2(\gamma)$, and $H_1$ and $H_2$ instead of $H_1(\gamma)$ and $H_2(\gamma)$.

\begin{lemma} \label{oui}
Let $\sC \to \sP_1$ be a family of degree $m$ covers such that $\sV_1$ is exceptional. Then one is
without loss of generality in one of the following cases:
\begin{enumerate}
\item (Complex case) \index{direct summand!exceptional!complex}
$t_1|d_1+d_2$, $t_1|d_2+d_3$ and $t_2|d_1+d_3$, where
$t_1$ does not divide $d_1+d_3$. \item (Separated case) $t_1 = 2$ and 2 divides $d_1+d_2$,
$d_2+d_3$ and $d_1+d_3$. \index{direct summand!exceptional!seperated}
\end{enumerate}
\end{lemma}
\begin{proof}
If $\sV_1$ is exceptional, then $d_1+d_2$, $d_2+d_3$ and $d_1+d_3$ are divided by $t_1$ or
$t_2$. Assume that $t_1$ (resp., $t_2$) divides $d_1+d_2$, $d_2+d_3$ and $d_1+d_3$. Hence 
one has $t_1= 2$ (resp., $t_2= 2$) as in the proof of Lemma $\ref{buaha}$. Otherwise one has only
to choose a suitable enumeration such that one is in the complex case.
\end{proof}

\begin{remark}
It can occur that one is in the complex case and the separated case with respect to the
same eigenspaces (up to complex conjugation). Consider the family $\sC \to \sP_1$ of degree 12
covers given by
$$d_1 = 5, \ \ d_2 =1, \ \ d_3 = 11, \  \ d_4 = 7.$$
Let $v = 5$. Then one has $t_1 = 3$ and $t_2 = 2$ such that $3|d_1+d_2$, $3|d_2+d_3$ and
$2|d_1+d_3$. Now let $v = 7$. In this case one has $t_1 = 2$ and 2 divides $d_1+d_2$, $d_2+d_3$ and
$d_1+d_3$. By $\ref{nabel}$, we will see that there is an isomorphism $\alpha:
\Mon^0(\Re\sV(1))_{\R} \to \Mon^0(\Re\sV(5))_{\R}$.

On the other hand consider the family $\sC \to \sP_1$ of degree 12 covers
given by
$$d_1 = 11, \ \ d_2 =1, \ \ d_3 = 11, \  \ d_4 = 1.$$
Again by the same arguments, we are in the complex case and the separated case at the same time.
But in this case the existence of a suitable isomororphism $\alpha:
\Mon^0(\Re\sV(1))_{\R} \to \Mon^0(\Re\sV(5))_{\R}$ is not known to the author at this time.
\end{remark}

\begin{pkt} \label{punkt}
Assume that the direct summand $\sV_1$ is separated with respect to $[v]_m \in (\Z/(m))^*$ for a
family $\sC \to \sP_1$ of degree $m$ covers. One has $[v2] = [2]$ in each separated case. This
implies that $[2][v-1] = [0]$. Hencefore one has $[v] = [\frac{m}{2}+1] \in (\Z/(m))^*$ in each
separated case.
Hence $v\in (\Z/(m))^*$ is an involution. The fact that $[v] = [\frac{m}{2}+1] \in (\Z/(m))^*$
implies that $\frac{m}{2}+1$ is odd. Hence $4$ divides $m$. In the separated case $r_1 = 2$
divides each $d_k+d_{\ell}$. Thus $\sV_1$ is separated, if and only if $4|m$ and each $d_k$ is
odd.

Hencefore there are infinitely many cases of families $\sC \to \sP_1$ such that $\sV_1$ is
separated.  At this time the author can not give an isomorphism $\alpha:\Mon^{\ad}(\sL_1) \to
\Mon^{\ad}(\sL_{\frac{m}{2}+1})$ for each separated example.
\end{pkt}

By the preceding point we have classified and described all examples $\sC \to \sP_1$ such that
$\sV_1$ is separated. Hence we consider only the case of a family $\sC \to \sP_1$ such that
$\sV_1$ is complex for the rest of this section.

\begin{lemma} 
Assume that $\sV_1$ is complex. Then one has:
$$\ell := \lcm(t_1,t_2) = \left\{ \begin{array}{r@{\quad:\quad}l}
m & m \  \ \mbox{is odd} \\
\frac{m}{2} & m \  \ \mbox{is even}
\end{array} \right.$$
\end{lemma}
\begin{proof}
If $m$ is odd, $H_1\cap H_2= \{1\} = \{\xi^m\}$. If $m$ is even, $H_1\cap H_2 = \{1,-1\} =
\langle\xi^{\frac{m}{2}}\rangle$. \end{proof}

\begin{lemma} \label{H1H2}
Assume that $\sV_1$ is complex. Then one has that $t_1t_2 =m$ or $t_1t_2 =\frac{m}{2}$. Moreover
one has that $t_1t_2 =m$, if $m$ is odd, and $t_1t_2 =\frac{m}{2}$, if $2|m$, but 4 does not
divide $m$.
\end{lemma}
\begin{proof}
If $m$ is odd, one has $\ell = \lcm(t_1,t_2) = m$. Hence one has $t_1't_2'g = m$ for
$g := \gcd(t_1,t_2)$ and $t_i = gt_i'$. Hence $|H_1| = t_2'$ and $|H_2| = t_1'$. If $g >2$, there
is a semisimple Dehn twist, whose order does not divide $t_1't_2'$ (follows from Lemma
$\ref{buaha}$). But this can not occur by our assumption that $\sV_1$ is complex. Hence $g = 1$,
since $g = 2$ is not possible for $m$ odd.

If $m$ is even, one has $\ell = \lcm(t_1,t_2) = \frac{m}{2}$. Hence one has $t_1't_2'g =
\frac{m}{2}$ for $g := \gcd(t_1,t_2)$ and $t_i = gt_i'$. If
$g >2$, there is a semisimple Dehn twist, whose order does not divide $t_1't_2'$. Hence one has
$g = 1$ or $g = 2$. Thus $t_1t_2 =m$ or $t_1t_2 =\frac{m}{2}$. 

Now assume that $2|m$, but 4 does not divide $m$. Then one has that $\frac{m}{2}= \lcm(t_1,t_2)$
is odd. Hence one can not have that $g = 2$ in this case. Thus $g = 1$ and $t_1t_2 = \frac{m}{2}$.
\end{proof}

\begin{example}
In the case $4|m$ both $t_1t_2 =m$ and $t_1t_2 =\frac{m}{2}$ can occur. Let $m = 24$ and
take $v = 5$ for the corresponding automorphism of $\Q(\xi)$. In this case one has
$t_1 = 6$ and $t_2 = 4$ such that $t_1t_2 = 24 = m$.

Now let $m = 24$ and take $v = 7$. In this case one has $t_1 = 4$ and $t_2 = 3$ such that $t_1t_2
= 12 = \frac{m}{2}$.
\end{example}
 
\begin{proposition}
Assume $\gamma \in  Gal(\Q(\xi);\Q)$ yields an example of a complex case. Then $\gamma$ is an
involution.
\end{proposition}
\begin{proof}
Let $[v] \in \Z/(m)^*$ correspond to $\gamma$. One has that $t_1t_2 = m$ or
$t_1t_2 = \frac{m}{2}$. Since one has that $[vt_1]_m = [t_1]_m$ and $[vt_2]_m = -[t_2]_m$, one
gets that
$$(v-1)t_1 \in (m) \  \ \mbox{and} \  \ (v+1)t_2 \in (m).$$
This implies that $t_2|(v-1)$ and $t_1|(v+1)$ or (if $t_1t_2 = \frac{m}{2}$) that $2t_2|(v-1)$ and
$2t_1|(v+1)$. Hence in each case one obtains that
$$v^2-1 = (v-1)(v+1) \in (m).$$
\end{proof}

\begin{Theorem} \label{piu}
Let $\sC \to \sP_1$ be a family of degree $m$ covers. Then
$\sV_1$ is complex, if and only if the fibers of $\sC$ have the branch indeces $d_1, \ldots, d_4$
with $2m = d_1+ \ldots +d_4$ such that
$$[vd_2]_m = [d_1+d_2+d_3]_m, \  \ [vd_1]_m = [-d_3]_m, \  \ [vd_3]_m = [-d_1]_m$$
or
$$[vd_2]_m = [d_1+d_2+d_3 +\frac{m}{2} ]_m, \  \ [vd_1]_m = [-d_3+\frac{m}{2}]_m,
\  \ [vd_3]_m = [-d_1+\frac{m}{2}]_m$$
for some $v$ with $[v^2]_m = [1]_m$ and $[v]_m \notin \{[1]_m, [m-1]_m\}$.
\end{Theorem}
\begin{proof}
The condition $2m = d_1+ \ldots +d_4$ ensures that $\sV_1$ is not special.

By an abuse of notation, each integer $z$ denotes the residue class $[z]_m$ in this proof. Assume
that $\sV_1$ is complex. Hence by Lemma $\ref{oui}$, one has that
$$2vd_2 = v((d_1+d_2)-(d_1+d_3) + (d_2+d_3)) = (d_1+d_2)+(d_1+d_3) + (d_2+d_3) =2(d_1+d_2+d_3),$$
$$2vd_1 = v((d_1+d_2)+(d_1+d_3) - (d_2+d_3)) = (d_1+d_2)-(d_1+d_3) - (d_2+d_3)) = -2d_3,$$
$$2vd_3 = v(-(d_1+d_2)+(d_1+d_3) + (d_2+d_3)) = -(d_1+d_2)-(d_1+d_3) + (d_2+d_3)) = -2d_1.$$
Hence one has two cases:
$$vd_2 = d_1+d_2+d_3 \  \ \mbox{or} \  \  vd_2 = d_1+d_2+d_3 +\frac{m}{2}$$
In the first case (resp., the second case) the fact that $v(d_1 + d_2) = d_1 + d_2$ implies that
$vd_1 = -d_3$ (resp., $vd_1 = -d_3 + \frac{m}{2}$). Moreover in the first case (resp., the second
case) the fact that $v(d_2 + d_3) = d_2 + d_3$ implies that $vd_3 = -d_1$ (resp., $vd_3 = -d_1 +
\frac{m}{2}$). Hence we have obtained the claimed equations.

Assume conversely that the family $\sC \to \sP_1$ satisfies one of the two systems of equations of
this theorem. Then one can easily calculate that $\sV_1$ is complex. 
\end{proof}

\begin{pkt} \label{nabel}
Let $\sC \to \sP_1$ be a family of degree $m$ covers. Assume that $d_1,d_2, d_3$ satisfy the
first system of equations of Theorem $\ref{piu}$ with respect to some $v$ with $[v^2] = [1]_m$,
which satisfies
that $[v]_m \notin \{[1]_m, [m-1]_m\}$. Moreover let $j \in (\Z/(m))^*$ such that $\sL_j \subset
\sV_1$ with monodromy representation $\rho_j$. Now we calculate that $\Mon_{\Q(\xi)^+}^0(\sV_1)$
does not reach the upper bound $C^{\der}_1(g)_{\Q(\xi)^+}$ in this case.

Let $a_1 = 0$, $a_3 = 1$ and $a_4 =\infty$. The fundamental group of the corresponding copy of
$\sM_1$ is generated by $T_{1,2}$ and $T_{2,3}$. One obtains that
$$\rho_j(T_{1,2}) =\left(\begin{array}{cc}
\xi^{jd_1+jd_2} & 1-\xi^{jd_1} \\
0 & 1
\end{array} \right), \ \
\rho_j(T_{2,3}) =\left(\begin{array}{cc}
1 & 0\\
\xi^{jd_2}-\xi^{jd_2+jd_3} & \xi^{jd_2+jd_3}
\end{array} \right).$$
Let $\gamma_v \in \Gal(\Q(\xi);\Q)$ denote the automorphism corresponding to $[v]$. The monodromy
representation of $\sL_{jv}$ is given by
$$\rho_{jv}(T_{1,2}) =\left(\begin{array}{cc}
\xi^{jd_1+jd_2} & 1-\xi^{-jd_3} \\
0 & 1
\end{array} \right), \ \
\rho_{jv}(T_{2,3}) =\left(\begin{array}{cc}
1 & 0\\
\xi^{jd_1+jd_2+jd_3}-\xi^{jd_2+jd_3} & \xi^{jd_2+jd_3}
\end{array} \right).$$

One calculates easily that
$$\frac{1-\xi^{jd_1}}{1-\xi^{-jd_3}} \cdot
\frac{ \xi^{jd_2}-\xi^{jd_2+jd_3} }{ \xi^{jd_1+jd_2+jd_3}-\xi^{jd_2+jd_3} } =
\frac{\xi^{jd_2}-\xi^{jd_2+jd_3} -\xi^{jd_1+jd_2}+\xi^{jd_1 + jd_2+jd_3}}
{\xi^{jd_1+jd_2+jd_3}-\xi^{jd_2+jd_3} -\xi^{jd_1+jd_2}+\xi^{jd_2}}= 1.$$
Hence there is a $z \in \Q(\xi)$ such
that $\gamma_v|_{<\rho_j(T_{1,2}),\rho_j(T_{2,3})>}$
coincides with $\alpha|_{<\rho_j(T_{1,2}),\rho_j(T_{2,3})>}$, where $\alpha$ is given by
$$\left(\begin{array}{cc}
a & b\\
c & d
\end{array} \right) \to \left(\begin{array}{cc}
a & zb\\
z^{-1}c & d
\end{array} \right).$$
Thus by Lemma $\ref{5.15}$, the group $\Mon_{\Q(\xi)^+}^0(\sV_1)$ does not attain its upper bound
in this case. In addition one calculates easily that $\alpha$ is given by
$$\left(\begin{array}{cc}
a & b\\
c & d
\end{array} \right)
\to \left(\begin{array}{cc}
\sqrt{z} & 0\\
0 & \sqrt{z^{-1}}
\end{array} \right) \left(\begin{array}{cc}
a & b\\ c & d
\end{array} \right)
\left(\begin{array}{cc}
\sqrt{z^{-1}} & 0\\
0 & \sqrt{z}
\end{array} \right).$$
Thus the monodromy representations of $\sL_j$ and $\sL_{jv}$ coincide up to
conjugation such that $\sL_j$ and $\sL_{jv}$ are isomorphic for each $j \in (\Z/(m))^*$.
\end{pkt}

\begin{corollary}
There are infinitely many families $\sC \to \sP_1$ such that $\sV_1$ is complex and
$\Mon_{\Q(\xi)^+}^0(\sV_1)$ does not reach its upper bound.
\end{corollary}
\begin{proof}
Let $p, q \in \N$ such that $\gcd(p,q) = 1$ with $p,q \notin \{1,2\}$ and $m := pq$. Hence
$\Z/(m) = \Z/(p) \times \Z/(q)$. Let $v < m$ correspond to $(1,-1)\in \Z/(p) \times \Z/(q)$. Thus
we get $[v^2] = [1]_m$ and $[v]_m \notin \{[1]_m, [m-1]_m\}$. One has that
$$d_1 = v, \  \ d_2 = 1, \  \ d_3 = m -1$$
satisfies the first system of equations of Theorem $\ref{piu}$, which quaranties by $\ref{nabel}$
that $\Mon_{\Q(\xi)^+}^0(\sV_1)$ does not
reach its upper bound. Since there are infinitely many
possible choices for $p, q \in \N$ such that $\gcd(p,q) = 1$ with $p,q \notin \{1,2\}$, one
obtains infinitely many families $\sC \to \sP_1$ such that $\sV_1$ is exceptional.
\end{proof}

\section{The Hodge group of a universal family of hyperelliptic curves}
If the middle part \index{middle part} $\sV_{\frac{m}{2}}$ is of type $(1,1)$, one obtains
$\Mon^0(\sV_{\frac{m}{2}}) = \Sp_{\Q}(2)$, since $\Sp_{\R}(2) \cong \SU(1,1)$, and
$\Mon_{\R}^0(\sV_{\frac{m}{2}}) = \SU(1,1)$ as one has by Theorem $\ref{nundenn}$.

By using \cite{Zar} Theorem $10.1$ and Remark $10.2$, one can conclude that the Hodge group
$\Hg(\sV_{\frac{m}{2}})$ of an arbitrary middle part $\sV_{\frac{m}{2}}$ coincides
with $\Sp(\sV_{\frac{m}{2}},Q_{\sV_{\frac{m}{2}}})$. For completeness we give an elementary proof.
We use the the fact that
$$\Mon^0(\sV_{\frac{m}{2}}) \subseteq \Hg(\sV_{\frac{m}{2}}) \subseteq
\Sp(\sV_{\frac{m}{2}},Q_{\sV_{\frac{m}{2}}})$$
and show by explicite calculations that the dimensions of the Lie algebras of $\Mon^0(\sV_{\frac{m}{2}})$
and $\Sp(\sV_{\frac{m}{2}},Q_{\sV_{\frac{m}{2}}})$ coincide.

By Proposition $\ref{repro}$, each Dehn twist $T_{\ell,\ell+1}$ yields a unipotent subgroup of
$\Mon^0(\sV_{\frac{m}{2}})$ isomorphic to $\G_a$.
Its corresponding subvector space of the Lie algebra is generated by
$$A_{\ell,\ell+1}(a,b)
= \left\{\begin{array}{r@{\quad:\quad}l}
-1 & a= \ell\quad\mbox{and}\quad b = \ell-1\\
1 & a= \ell\quad\mbox{and}\quad b = \ell+1\\
0 & \mbox{elsewhere} \end{array}  \right. .$$

Now we consider the middle part of type $(2,2)$. Hence we are in the case of the genus 2 curves.
For $\ell =1, \ldots, 4$ the matrices $A_{\ell,\ell+1}$ generate a 4 dimensional vector
space. Moreover by $[A_{i,i+1},A_{i+1,i+2}]$ for $i = 1,2,3$, we get the 3 additional linearly
independent matrices 
$$\left(\begin{array}{cccc}
-1 & 0 & 1 & 0\\
0 & 1 & 0 &0\\
0 & 0 & 0 & 0\\
0& 0 & 0 & 0
\end{array} \right), \  \
\left(\begin{array}{cccc}
0 & 0 & 0 & 0\\
0 & -1 & 0 & 1\\
-1 & 0 & 1 & 0\\
0& 0 & 0 & 0
\end{array} \right), \  \ \mbox{and} \  \
\left(\begin{array}{cccc}
0 & 0 & 0 & 0\\
0 & 0 & 0 & 0\\
0 & 0 & -1 & 0\\
0& -1 & 0 & 1
\end{array} \right).$$
By
$$[A_{2,3},[A_{3,4},A_{4,5}]] \  \ \mbox{resp.,} \  \ [[A_{1,2},A_{2,3}],A_{3,4}],$$
we obtain the two further linearly independent matrices
$$\left(\begin{array}{cccc}
0 & 0 & 0 & 0\\
0 & 0 & -1 & 0\\

0 & 0 & 0 & 0\\
-1 & 0 & 1 & 0
\end{array} \right) \ \ \mbox{and} \  \
\left(\begin{array}{cccc}
0 & -1 & 0 & 1\\
0 & 0 & 0 & 0\\
0 & 1 & 0 & 0\\
0 & 0 & 0 & 0
\end{array} \right).$$
Thus the Lie algebra has at least dimension 9. Moreover one checks easily
that $$[[A_{1,2},A_{2,3}],[A_{3,4},A_{4,5}]]
=\left(\begin{array}{cccc}
0 & 0 & -1 & 0\\
0 & 0 & 0 & 0\\
0 & 0 & 0 & 0\\

0 & 1 & 0 & 0
\end{array} \right).$$
is a tenth linearly independent matrix. Thus the well-known fact that $Sp_{\Q}(4)$ has dimension
10 implies:

\begin{proposition} \label{5.39}
If $\sV_{\frac{m}{2}}$ is of type $(2,2)$, then $\Mon^0(\sV_{\frac{m}{2}}) \cong
\Sp(\sV_{\frac{m}{2}},Q_{\sV_{\frac{m}{2}}})$.
\end{proposition}

Note that the quotient of $\Sp_4(\R)$ by its maximal compact subgroup is Siegel's upper half plane
$\fh_2$, which has dimension 3. Since $\sM_3$ has dimension 3, one concludes for
the restricted family $\sC_{\sM_3}\to \sM_3$ of genus 2 curves:

\begin{corollary}
The family $\sC_{\sM_3}\to \sM_3$ of genus 2 curves has a dense set of $CM$ fibers.
\end{corollary}
\begin{proof}
One has (similarly to the proof of Theorem $\ref{jnvz}$) that the holomorphic period map $p:\sM_3
\to \fh_2$ has fibers of dimension 0. Since $\dim(\fh_2) = \dim(\sM_3) = 3$, one concludes that
$p$ is open. Hence the statement follows from the fact that $\fh_2$ has a dense set of $CM$ points.
\end{proof}

We will use Proposition $\ref{5.39}$ and the calculations, which yield this proposition, to show
the following theorem by induction:

\begin{Theorem} \label{middlep}
If $\sV_{\frac{m}{2}}$ is of type $(g,g)$, then $\Mon^0(\sV_{\frac{m}{2}}) \cong
\Sp(\sV_{\frac{m}{2}},Q_{\sV_{\frac{m}{2}}})$.
\end{Theorem}
 
\begin{corollary}
$$\Hg(\sV_{\frac{m}{2}}) = \Sp(\sV_{\frac{m}{2}},Q_{\sV_{\frac{m}{2}}}) \  \ \mbox{and} \ \
\MT(\sV_{\frac{m}{2}}) = \GSp(\sV_{\frac{m}{2}},Q_{\sV_{\frac{m}{2}}})$$
\end{corollary}

It is a well-known fact that $\dim(\Sp_{\Q}(2g)) = 2g^2+g$.\footnote{Otherwise one has a
description of $\fsp_{2g}(\C)$ in \cite{FH}, page 239. By this description, one can easily
determine its dimension.} Hence one gets
$$\dim(\Sp_{\Q}(2g+1)) = 2(g+1)^2 + g+1 = (2g^2 + g)+(4g+3).$$

We will show by induction that for each $g \in \N$ the matrices $A_{\ell, \ell+1}$ generate a Lie
algebra, which has at least the same dimension as $\mathfrak{sp}_{2g}(\Q)$. This yields Theorem
$\ref{middlep}$. Since we have shown the statement for $g = 1,2$, we will only give the induction
step:

Recall that we have defined $\LL_j$-valued paths $[e_k \delta_k]$ in Section $3.3$. We consider
a middle part of type $(g+1,g+1)$ with respect to the basis $\sB = \{[e_1\delta_1],\ldots,
[e_{2g+2}\delta_{2g+2}]\}$. The Dehn twists $T_{\ell,\ell+1}$ for $\ell = 1, \ldots, 2g$ yield the monodromy
group $G_1$ of a middle part of type $(g,g)$. Hencefore by the induction hypothesis, they yield a
group isomorphic to $\Sp_{2g}(\Q)$.

\begin{remark} \label{joh}
One has the obvious embedding of $G_1 \hookrightarrow \GL(N^1(C_{\frac{m}{2}},\Q))$ with respect
to the basis $\sB_1 := \{[e_1\delta_1], \ldots, [e_{2g}\delta_{2g}], [e_{2g+2}\delta_{2g+2}],
[e_{2g+3}\delta_{2g+3}]\}$ such that
$$G_1 \ni A \to  \left(\begin{array}{ccc}
A &  & \\
 & 1 & 0\\
 & 0 & 1
\end{array} \right) \in \GL(N^1(C_{\frac{m}{2}},\Q)).$$
Moreover this embedding of $G_1$ into $\GL(N^1(C_{\frac{m}{2}},\Q))$ is given by
$$G_1 \ni A \to  \left(\begin{array}{ccc}
A & v & \\
 & 1 & 0\\
 & 0 & 1
\end{array} \right) \in \GL(N^1(C_{\frac{m}{2}},\Q)),$$
with respect to the basis $\sB$, where
$v^t = (v_1, \ldots, v_{2g})$ is a vector depending on $A$.
\end{remark}

Since we consider the embedding with respect to the latter basis, we want to understand $v$, which
is possible, if we understand the base change between the bases of the preceding remark.

\begin{lemma}
Let $C \to \bP^1$ be a hyperelliptic curve of genus $g+1$. One has (up to a suitable normalization)
$$\sum\limits_{i= 0}^{g+1}[e_{2i+1}\delta_{2i+1}] = 0.$$
\end{lemma}
\begin{proof}
Let $\zeta \in H_2(C,\C)$ be a nontrivial linear combination of the closures of the sheets of
$\bP^1 \setminus S$, on which $\psi$ acts via push-forward by the character $1 \in \Z/(2)$.
One has that $\partial \zeta$ represents a linear combination of
$[e_1\delta_1], \ldots, [e_{2g+1}\delta_{2g+3}] \in H_1(C,\C)_1$, which is equal to zero. Recall
that over $\delta_1 \cup \ldots \cup \delta_{2g+3}$ the glueing of these sheets depends on the
local monodromy data determined by the branch indeces of the branch points $a_k$.
Since each $a_k$ has the local monodromy datum $-1$, this linear combination is (up to a suitable
normalization of $[e_1\delta_1], \ldots, [e_{2g+1}\delta_{2g+3}]$) given by
$$\sum\limits_{i= 0}^{g+1}[e_{2i+1}\delta_{2i+1}] = 0.$$
\end{proof}

\begin{pkt} \label{hoehoeo}
By the preceding lemma, the matrices of base change between the bases $\sB$ and $\sB_1$ are given
by
$$M^{\sB_1}_{\sB}({\rm id}) = \left(\begin{array}{cccccc}
1&  & &  &  & -1\\
& \ddots & &  &  & \vdots\\
& & 1&  &  & -1\\
& & & 1 &  & 0\\
& & &  & 0 & -1\\
& & &  & 1 & 0
\end{array} \right) \  \ \mbox{and} \  \ M^{\sB}_{\sB_1}({\rm id}) = \left(\begin{array}{cccccc}
1&  & &  & -1 & \\
& \ddots & & & \vdots & \\
& & 1&  &  -1 & \\
& & & 1 & 0 & \\
& & &  & 0 & 1\\
& & &  & -1 & 0
\end{array} \right)$$
such that
$$\left(\begin{array}{ccc}
A & v & \\
 & 1 & 0\\
 & 0 & 1
\end{array} \right) = M_{\sB}^{\sB_1}({\rm id})\cdot \left(\begin{array}{ccc}
A &  & \\
 & 1 & 0\\
 & 0 & 1
\end{array} \right)\cdot M^{\sB}_{\sB_1}({\rm id}).$$
Thus one calculates easily that $v_1= 0$, if $a_{1,1} = 1$ and $a_{1,j} = 0$ for
$2 \leq j \leq 2g$ and $A = (a_{i,j})$. The exponential map $\exp$ is a diffeomorphism on a
neighborhood of $0$. Hence by the definition
$$\exp(m) = 1 + m + \frac{m^2}{2} + \frac{m^3}{6} + \ldots,$$
one concludes that each $(m_{i,j}) \in Lie(G_1)$ satisfies that $m_{1,2g+1}= 0$, if $m_{1,j} = 0$
for all $j = 1, \ldots, 2g$, which will play a very important role later. Otherwise $\exp$
would yield a matrix with $a_{1,1} = 1$, $a_{1,j} = 0$ for $2 \leq j \leq 2g$ and $v_1 \neq 0$
as one can calculate by the fact that each $(m_{i,j}) \in Lie(G_1)$ satisfies that $m_{i,j}= 0$
for $i > 2g$.
\end{pkt}

\begin{lemma} \label{juchu}
Let $i_0\leq 2g$ and $j_0 < 2g$ be integers such that $i_0-j_0 >0$. In the Lie algebra
$Lie(G_1)$ one finds an element $(x^{(i_0,j_0)}_{i,j})$ with $x^{(i_0,j_0)}_{i_0,j_0} \neq 0$ and
$x^{(i_0,j_0)}_{i,j} = 0$, if $i> i_0$ or $j < j_0$ or $i = 1$.
\end{lemma}
\begin{proof}
Let $k_0 := i_0-j_0 >0$. We show the statement by induction over $k_0$. Each pair $(i_0,j_0)$
with $i_0-j_0 = k_0 = 1$ is given by $(i_0,i_0-1)$. By $A_{i_0, i_0+1}$, such an element is
given for each $(i_0,i_0-1)$.

Now let $(i_0,j_0)$ be a pair with $k_0 := i_0-j_0 > 1$ and assume that the statement is satisfied
for $k_0 -1, \ldots, 1 > 0$. Hence one has $(x^{(i_0,j_0+1)}_{i,j}), A_{j_0+1,j_0+2} \in
Lie(G_1)$. By
$$(x^{(i_0,j_0)}_{i,j}) := [(x^{(i_0,j_0+1)}_{i,j}), A_{j_0+1,j_0+2}],$$
one obtains the desired element of $Lie(G_{1})$, since one has the entry
$$x^{(i_0,j_0)}_{i_0,j_0} =  x^{(i_0,j_0+1)}_{i_0,j_0+1} \cdot (A_{j_0+1,j_0+2})_{j_0+1,j_0}
\neq 0.$$
\end{proof}

Moreover the Dehn twists $T_{2n-1,2n}, \ldots, T_{2g+2,2g+3}$ generate a group $G_2$ isomorphic to
the monodromy group of a middle part of type $(2,2)$, which has dimension 10. One can
easily compare the matrices of $Lie(G_2)$ with the above explicitely given matrices of a middle part of type
$(2,2)$: ``The restriction of the matrices of $Lie(G_2)$ to the lower right
corner looks like the matrices of the Lie algebra of the monodromy group of a middle part of type
$(2,2)$.''

Since the vectors
$$A_{2g-1,2g}, \  \ A_{2g,2g+1} \ \ \mbox{and} \ \  [A_{2g-1,2g}, A_{2g,2g+1}]$$ 
are contained in $Lie(G_1) \cap Lie(G_2)$, both Lie algebras yield together a
$2g^2+g+7$-dimensional vector space of matrices $(x_{i,j})$, whose entries $x_{i,j}$ vanish for
$j < 2g-3$ and $i > 2g$. Hence by using
$$[A_{2g+1,2g+2}, (x^{(2g,j_0)}_{i,j})] \  \ \mbox{and} \  \
[[A_{2g+1,2g+2}, A_{2g+2,2g+3}],x^{(2g,j_0)}_{i,j}]$$
for $j_0 < 2g-3$, one has $4g-6$ additional linearly independent vectors. Thus we have altogether
$(2g^2+g) + (4g+1)$ linearly independent vectors. Hence 2 remaining linearly independent vectors
are to find. Since $x^{(i_0,j_0)}_{i,j} = 0$ for $i = 1$,
in the constructed vector space
of matrices $(m_{i,j})$ the coordinate $m_{1,2g+1}$ depends uniquely on the
vectors in $Lie(G_1)$ such that $m_{1,2g+1} = 0$, if $m_{1,j} =0$ for all $j = 1, \ldots, 2g$
as we have seen in $\ref{hoehoeo}$. Let
$$Lie(G_1)\ni (y_{i,j}) = [A_{1,2},[A_{2,3},[\ldots[A_{2g-1,2g},A_{2g,2g+1}]\ldots]].$$ 
One checks easily that
$$y_{1,2g+1} \neq 0.$$
Now the matrix
$$(y'_{i,j}) = [(y_{i,j}),[A_{2g+1,2g+2}, A_{2g+2,2g+3}] ]$$
satisfies $y'_{1,2g+1} \neq 0$, $y'_{i,j} = 0$ for $i,j \leq 2g$ and $y'_{1,2g+2} = 0$. Thus we
have found a new
vector not contained in the vector space, which we have constructed by $Lie(G_1)$, $Lie(G_2)$ and
some Lie brackets at the present.

Note that all matrices $(x_{i,j})$, which we have found, satisfy $x_{1,2g+2} = 0$. But
$$(z_{i,j}) := [(y_{i,j}),A_{2g+1,2g+2}]$$
satisfies $z_{1,2g+2} \neq 0$. Hencefore we are done.

\section{The complete generic Hodge group}

By this section, we finish our calculation (of the derived group) of the generic Hodge group
and obtain the final result:

\begin{Theorem} \label{hmhmm}
One has
$$\Mon^0(\sV) = \prod\limits_{r|m}\Mon^0(\sV_r)$$
in the following cases:
\begin{enumerate}
\item The degree $m$ of the covers given by the fibers of $\sC \to \sP_n$ is odd.
\item $\sP_n = \sP_1$ and 6 does not divide $m$.
\end{enumerate}
\end{Theorem}

\begin{corollary}
Assume that $\sC \to \sP_n$ satisfies one of the following conditions:
\begin{enumerate}
\item The degree $m$ of the covers given by the fibers of $\sC \to \sP_n$ is odd.
\item $\sP_n = \sP_1$ and 6 does not divide $m$.
\end{enumerate}
Then one has
$$\MT^{der}(\sV) = \Hg^{der}(\sV) \supseteq 
\prod\limits_{r|m}\Mon^0(\sV_r).$$
\end{corollary}

By Theorem $\ref{geilomat}$, one has a $CM$-fiber, if the fibers of $\sC \to \sP_n$ have $n+1$
branch points with the same branch index $d$. Thus by the fact that this implies the equality of
$\Mon^0(\sV)$ and $\MT^{\der}(\sV)$ (see Theorem $\ref{monmtder}$), one concludes:

\begin{corollary}
Let the fibers of $\sC \to \sP_n$ have $n+1$ branch points with the same branch index $d$ and
$\sC \to \sP_n$ satisfy one of the following conditions:
\begin{enumerate}
\item The degree $m$ of the covers given by the fibers of $\sC \to \sP_n$ is odd.
\item $\sP_n = \sP_1$ and 6 does not divide $m$.
 \end{enumerate}
Then
$$\MT^{\der}(\sV) =  \Hg^{\der}(\sV) = \prod\limits_{r|m}\Mon^0(\sV_r).$$
\end{corollary}

Since $C^{\der}_r(g)$ is an upper bound for $\Hg^{\der}(\sV_r)$, one concludes finally:

\begin{corollary}
Assume that $\sC \to \sP_n$ satisfies one of the following conditions:
\begin{enumerate}
\item The degree $m$ of the covers given by the fibers of $\sC \to \sP_n$ is odd.
\item $\sP_n = \sP_1$ and 6 does not divide $m$.
\end{enumerate}
If all $\sV_r$ except of the middle part are very general or special, one has
$$\MT^{\der}(\sV) =  \Hg^{\der}(\sV) = \Mon^0(\sV) = \prod\limits_{r|m}\Mon^0(\sV_r).$$
\end{corollary}

Recall that we search for families $\sC \to \sP_n$
with dense set of complex multiplication fibers. One obtains dense set of complex multiplication fibers, if one has an open (multivalued)
period map
$$p: \sM_n(\C) \to \MT^{\der}(\sV)(\R)/K$$
given by the $VHS$. Hence for our applications we need to know $\MT^{\der}(\sV)$ and the
dimension of $\MT^{\der}(\sV)(\R)/K$, but not $\MT(\sV)$ itself. Let us first prove Theorem
$\ref{hmhmm}$. After this proof we will see that the (multivalued) period map of a family
$\sC \to \sM_1$ onto $\MT^{\der}(\sV)(\R)/K$ is open, if and only if one has a $(1,1)-VHS$. 

For the proof of Theorem $\ref{hmhmm}$ we use the same methods as before. One has that
$\Mon^{\ad}(\sV)$ is the direct product of the kernel of the natural projection
$$p_1: \Mon^{\ad}(\sV) \to \Mon^{\ad}(\sV_{r_1})$$
and an adjoint semisimple group $G_{r_1}$ isomorphic to $\Mon^{\ad}(\sV_{r_1})$. Moreover one has
that
$$\Mon^{\ad}(\sV) = \prod\limits_{r|m}\Mon^{\ad}(\sV_r),$$
if and only if each $G_{r_1}$ is contained in the kernels of the natural projections onto all
$\Mon^{\ad}(\sV_{r_2})$ with $r_1 \neq r_2$.

We give a proof of Theorem $\ref{hmhmm}$ by contradiction. Thus we assume that
$$\Mon_{\R}^0(\sV) \neq \prod\limits_{r|m} \Mon_{\R}^0(\sV_r). \ \ \mbox{This implies} \ \
\Mon_{\R}^{\ad}(\sV) \neq \prod\limits_{r|m} \Mon_{\R}^{\ad}(\sV_r).$$

Hence some $G_{r_1}$ is not contained in the kernel of the projection onto
$\Mon^{\ad}(\sV_{r_2})$ for some $r_2 \neq r_1$. Since all simple direct factors of $G_{r_1}$ resp.,
$G_{r_2}$ project isomorphically onto some $\Mon^{\ad}(\sL_{j_1})$ resp.,
$\Mon^{\ad}(\sL_{j_2})$, one gets an isomorphism
$$\alpha : \Mon^{\ad}(\sL_{j_1}) \to \Mon^{\ad}(\sL_{j_2}),$$
which respects the respective projective monodromy representations. But by the
following proposition, the isomorphism $\alpha$ can not exist, if the assumptions of Theorem
$\ref{hmhmm}$ are satisfied. This yields the proof of Theorem $\ref{hmhmm}$. 

\begin{proposition} \label{tyui}
Assume that $r_1 := \gcd(m,{j_1}) \neq r_2 := \gcd(m,{j_2})$. Moreover assume that one of the
following cases holds true:
\begin{enumerate}
\item $m$ is odd.
\item $\sP_n = \sP_1$ and 6 does not divide $m$.
\end{enumerate}
Then an isomorphism
$$\alpha : \Mon^{\ad}(\sL_{j_1}) \to \Mon^{\ad}(\sL_{j_2}),$$
which respects the respective projective monodromy representations, can not exist.
\end{proposition}
\begin{proof}
Assume without loss of generality that $r_1 < r_2$. This implies $\frac{m}{r _1} > \frac{m}{r_2}$.
There are two cases: Either $2r _1 \neq r_2$ or $2r _1 = r_2$.

If $m$ is odd, one has $\frac{r_1}{2} \neq g := \gcd(\frac{m}{r_1},\frac{m}{r_2})$. Hence
by Lemma $\ref{buaha}$, one finds a Dehn twist $T$ such that $P\rho_{j_1}(T)$ is semisimple and
the order of $P\rho_{j_1}(T)$ does not divide $g$. One has that $P\rho_{j_2}(T)$ is either
unipotent or semisimple. If $P\rho_{j_2}(T)$ is semisimple, its order divides $\frac{m}{r_2}$. But
the order of $P\rho_{j_1}(T)$ does not divide $\frac{m}{r_2}$. If $P\rho_{j_2}(T)$ is unipotent,
its order infinite. But $P\rho_{j_1}(T)$ has finite order.
However $P\rho_{j_1}(T)$ and $P\rho_{j_2}(T)$ do not have the same order. Hence such an
isomorphism $\alpha: \Mon^{\ad}(\sL_{j_1}) \to \Mon^{\ad}(\sL_{j_2})$, which respects the
respective projective monodromy representations, can not exist in this case.

Now assume that we are in the case of a family $\sC \to \sP_1$, where 6 does not divide $m$. There
is a Dehn twist $T$ such that $P\rho_{j_1}(T)$ is semisimple. If $P\rho_{j_1}(T)$ and
$P\rho_{j_2}(T)$ do not have the same order, one can argue as above. Otherwise all semisimple Dehn
twists have the same order. Hence one must have $2r_1 = r_2$. The
nontrivial eigenvalue of $\rho_{j_2}(T)$ is given by the square of the nontrivial
eigenvalue $\xi$ of $\rho_{j_1}(T)$. Note that the corresponding maximal tori are isomorphic to
$S^1$, where $S^1_{\C} \cong \G_{m,\C}$. Thus its character group is isomorphic to $\Z$. Hence the
induced map of the corresponding maximal tori can be an isomorphism, only if one has $\xi^2 =
\xi^{-1} = \bar \xi$. In this case $\xi$ would be a primitive cubic root of unity, which
implies that 3 divides $m$. Since we have that $2r _1 = r_2$, 6 would divide $m$. But by the
assumptions, this is not possible.
\end{proof}

\begin{remark}
If $2r _1 = r_2$, there are many additional cases, in which $\alpha$ can not exist. These obvious
cases are given, if for a Dehn twist $T$ the order of the semisimple matrix $\rho_{r_1}(T)$ does
not divide $\frac{m}{2r_1}$, if $\rho_{r_1}(T)$ is semisimple and $\rho_{j_2}(T)$ is unipotent or
if $\sL_{j_1}$ and $\sL_{j_1}$ are of type $(a_1,b_1)$ and $(a_2, b_2)$ such that
$$(a_1,b_1) \neq (a_2, b_2) \  \ \mbox{and} \  \ (a_1,b_1) \neq (b_2, a_2).$$
But in the case of the family $\sC \to \sP_1$ of degree 6 covers given by the local monodromy data
$$d_1 = d_2 = 1, \  \ d_3 = d_4 = 5$$
nothing of them holds true with respect to $\sL_1$ and $\sL_2$.
\end{remark}

Now let us finish this chapter and state the final result about the period map:

\begin{Theorem}
In the case of a family $\sC \to \sM_1$ the period map
$$p: \sM_1(\C) \to \MT^{\der}(\sV)(\R)/K$$
is open, if and only if one has a pure $(1,1)-VHS$.
\end{Theorem}
\begin{proof}
As we have seen in the proof Theorem $\ref{jnvz}$, the period map is open, if one has a pure
$(1,1)-VHS$.

For the other direction assume that the period map is open and there are up to complex
conjugation at least two different eigenspaces, which are not unitary. 

\begin{lemma} \label{ropo}
Assume that we have a family $\sC_{\sM_1} \to \sM_1$. Only if all $\sV_r$ except for exactly one
$\sV_{r_0}$ are special, the period map
$$p: \sM_1(\C) \to \MT^{\der}(\sV)(\R)/K$$
can be open.
\end{lemma}
\begin{proof}
Assume that $r_1$ and $r_2$ divide $m$ such that $r_1 \neq r_2$ and $\sV_{r_1}$ and $\sV_{r_2}$
are not special. If $2r_1 \neq r_2$ or if there is a Dehn twist, whose finite order
with respect to $\sV_{r_1}$ does not divide $\frac{m}{r_2} = \frac{m}{2r_1}$, the same arguments
as in the proof of Proposition $\ref{tyui}$ imply that
$$\dim(\MT^{\der}(\sV)(\R)/K)>1= \dim(\sM_1).$$
Hencefore the period map can not be open.

Otherwise assume without loss of generality that $r_1 = 1$ and all semisimple Dehn twists
have an order dividing $\frac{m}{2}$. This implies that all $d_k$ are odd and the degree $m$ is
even. Hence $\Mon^0(\sV_{\frac{m}{2}})$ is isomorphic to $\Sp_{\Q}(2)$, where its monodromy
representation sends all Dehn twists to unipotent matrices. Thus $\dim(\MT^{\der}(\sV)(\R)/K) >1$.
\end{proof}

By Lemma $\ref{ropo}$, these two eigenspaces, which are not unitary, must be contained in the same
$\sV_{r_0}$, which must be exceptional. Hence assume
without loss of generality that $\sV_{r_0} = \sV_1$.

In the separated case, the fact that all $d_k$ are odd (compare to $\ref{punkt}$) implies that
$\Mon^0_{\R}(\sV_{\frac{m}{2}}) = \Sp_{\R}(2)$. Hence by Lemma $\ref{ropo}$, we have a
contradiction.

In the complex case Lemma $\ref{oui}$ implies without loss of generality that 
$$t_1|d_1+d_2, \  \ t_1|d_2+d_3, \  \ t_1|d_1+d_4, \  \ t_1|d_3+d_4.$$
This implies that $t_1$ divides each $d_k$ or that $t_1$ does not divide any $d_k$. Thus $t_1$
does not divide any $d_k$. Hence $\sC_{\frac{m}{t_1}}$ is a family
of covers with 4 branch points, where $\rho_{\frac{m}{t_1}}(T_{1,2})$ and
$\rho_{\frac{m}{t_1}}(T_{2,3})$ are unitary. Hence $\sV_{\frac{m}{t_1}}$ has an infinite monodromy
group resp., it is not special. Thus by Lemma $\ref{ropo}$, we have a contradiction.
\end{proof}

\chapter{Examples of families with dense sets of complex multiplication fibers}
\section{The necessary condition $SINT$} 
By Theorem $\ref{jnvz}$, one has a sufficient criterion for a dense set of $CM$ fibers of a family
$\sC_{\sM_n} \to \sM_n$. This criterion is satisfied, if $\sC$ has a pure $(1,n)-VHS$ (i.e. its
$VHS$ contains one eigenspace of type $(1,n)$, a complex
conjugate eigenspace of type $(n,1)$ and otherwise only eigenspaces of the type $(a,0)$ and
$(0,b)$ for some $a,b \in \N_0$).

\begin{remark}
Assume that the family $\sC \to \sP_n$ of cyclic covers of degree $m$ has a pure $(1,n)$-VHS
and that $\sL_{j_0}$ is the eigenspace of type $(1,n)$.
Let $j_0\notin (\Z/(m))^*$. Then we have $ 1 < r_0 := gcd(j_0, m)$. By Section $4.2$,
the family $\sC_{r_0}$ has a pure $(1,n)$-VHS, too.
\end{remark}

\begin{definition}
A pure $(1,n)-VHS$ is primitive, if $j_0\in (\Z/(m))^*$.\index{pure $(1,n)-VHS$!primitive} Otherwise
it is a derived pure $(1,n)-VHS$ \index{pure $(1,n)-VHS$!derived} with the associated primitive
pure $(1,n)-VHS$ induced by $\sC_{r_0}$, where $\sC_{r_0}$ is given by the preceding remark.
\end{definition}

Hence first we search for families with a primitive pure $(1,n)-VHS$. Later we will look for
families with a derived pure $(1,n)-VHS$. It is helpful to have a necessary condition to find the
families with a primitive pure $(1,n)-VHS$. In \cite{DM} P. Deligne and G. D. Mostow have
formulated the following integral condition $INT$:

\begin{definition}
A local system on $\bP^1 \setminus S$ of monodromy $(\alpha_s)_{s \in S}$ with
$\alpha_s = exp(2\pi i \mu_s)$ and $\mu_s \in \Q$ for all $s \in S$ satisfies the condition $INT$,
if:
\begin{enumerate}
\item $0 < \mu_s < 1$ for all $s \in S$.
\item We have for all $s, t \in S$: $(1-\mu_s -\mu_t)^{-1}$ is an integer, if $s\neq t$ and
$\mu_s + \mu_t <1$.
\item $\sum \mu_s = 2$.
\end{enumerate}
\end{definition}

One can identify the local monodromy data, which yield the family $\sC \to \sP_n$ by Construction
$\ref{hiercr}$, with the local monodromy data of the eigenspace $\LL_1$ of some fiber $\sC_q$ for
an arbitrary $q \in \sP_n$. Hence one can formulate the condition $INT$ for the local monodromy data
of the family. For the latter data we give a corresponding stronger integral condition $SINT$:

\begin{definition} \index{$SINT$}
A family $\sC \to \sP_n$ of cyclic covers onto $\bP^1$ given by the local monodromy data
given by $\mu_k \in \Q$ around $s_k \in N$ satisfies $SINT$, if we have:
\begin{enumerate}
\item $\mu_{k_1} + \mu_{k_2} = 1$ or $(1-\mu_{k_1} -\mu_{k_2})^{-1} \in \Z$ for all
$s_{k_1}, s_{k_2} \in N$ with $s_{k_1} \neq s_{k_2}$.
\item $\sum\limits \mu_s = 2$.
\end{enumerate} 
\end{definition}

\begin{remark}
The reader checks easily that for a family $\sC \to \sP_1$ the conditions $INT$ and $SINT$ are
equivalent. Moreover by the list on \cite{DM}, page 86, each family $\sC \to \sP_n$ with
$n\geq 2$, which satisfies $INT$, satisfies $SINT$, too.

At the present the author can not explain this fact. We use $SINT$ instead of $INT$, since this
yields a stronger and hencefore a more helpful condition.
\end{remark}

By the following theorem, we have our helpful necessary condition for families $\sC$, which
have a primitive pure $(1,n)-VHS$:

\begin{Theorem} \label{sint}
If the family $\sC \to \sP_n$ has a primitive pure $(1,n)-VHS$, its local monodromy data can
be given rational numbers satisfying $SINT$.
\end{Theorem}

For the proof of Theorem $\ref{sint}$ we first reduce the situation to the case of a family
$\sC \to \sP_1$ of covers with only 4 branch points. That means we will consider a pair of branch
points of a fiber of $\sC \to \sP_n$, where $\sC$ has a primitive pure $(1,n)-VHS$, as a pair of
branch points with the same branch indeces of a fiber of a family $\sC(P) \to \sP_1$, which has a
primitive pure $(1,1)-VHS$. The following lemma will make it possible in almost all cases:

\begin{lemma} \label{lbc}
Assume that $\sC$ is given by local monodromy data on at least 5 points, where one does not have
$\mu_3 = \ldots = \mu_{n+3} = \frac{1}{2}$. Then there exists a
stable partition $P$ with $\{a_1\}, \{a_2\} \in P$ such that $|P| = 4$.\footnote{Since the
assumptions of this lemma are sufficient, we do not restrict to the interesting case of a family
with a primitive pure $(1,n)-VHS$.}
\end{lemma}  
\begin{proof}
One can without loss of generality assume that $\mu_1 + \mu_2 \leq 1$.
Otherwise we take the local monodromy data of $\sL_{m-1}$.

Now assume that such a stable partition $P$ with $\{a_1\},\{a_2\} \in P$
does not exist. Hence one must have $\mu_{1}+ \mu_{2}+\mu_k = 1$ for all $3 \leq k \leq n+3$.
Otherwise one obtains the stable partition
$$P = \{\{a_1\},\{a_2\},\{a_{k}\}, \{a_3, \ldots, a_{k-1}, a_{k+1},\ldots, a_{n+3}\}\}$$
Thus one must have
$$\mu:=\mu_{3} = \ldots = \mu_{n_j+3}.$$
Since 
$$P = \{\{a_1\},\{a_2\},\{a_{3}, a_{4}\}, \{a_{5},\ldots, a_{n_j+3}\}\}$$
is not a stable partition by our assumption, too, one has
$$2\mu = \mu_3+ \mu_4 = 1. \ \ \mbox{Hence} \ \ \mu=\frac{1}{2}.$$
\end{proof}

\begin{pkt} \label{excep}
The family of irreducible cyclic covers of $\bP^1$ given by the local monodromy data
$$\mu_1 = \mu_2 = \frac{1}{4}, \  \ \mu_3 = \mu_4 = \mu_5 = \frac{1}{2}$$
has a primitive pure $(1,2)-VHS$. Moreover it is easy to calculate that this family satisfies $SINT$.

But this is the only example of a family $\sC \to \sP_n$ with a primitive pure $(1,n)-VHS$ with
$n >1$, which
does not satisfy the assumptions of Lemma $\ref{lbc}$: It is very easy to see that this is the
only degree 4 example with a primitive pure $(1,n)-VHS$ for $n>1$, which contradicts the
assumptions of Lemma $\ref{lbc}$. If $m > 4$, $\sL_3$ must be unitary. But in this case the
condition that
$$n+3 > 4 \ \ \mbox{and}  \ \ [3\mu_3]_1 = \ldots = [3\mu_{n+3}]_1 = \frac{1}{2}$$
and Proposition $\ref{1.27}$ imply that
$$h_3^{1,0}(C) \geq (\sum\limits_{k \geq 3} [3\mu_k]_1) -1
= (\sum\limits_{k \geq 3} \frac{1}{2}) -1 > 0$$
and
$$h_3^{0,1}(C) \geq \sum\limits_{k \geq 3} (1-[3\mu_k]_1) -1 =
(\sum\limits_{k \geq 3} \frac{1}{2}) -1 > 0.$$
Thus $\sL_3$ is not unitary.
\end{pkt}

\begin{pkt}
Assume that $\sC\to \sP_n$ has a primitive pure $(1,n)-VHS$. Hence $\sL_1$ is without loss of
generality the eigenspace of type $(1,n)$. For our application of Lemma
$\ref{lbc}$ we must check that the collision of Lemma $\ref{lbc}$ resp., its corresponding
stable partition yields a family $\sC(P)\to \sP_1$, which has a primitive pure $(1,1)-VHS$.
The family $\sC(P)$ is given by $N=P$ with the local monodromy data
$$\alpha_{\{a_k,\ldots, a_{\ell} \}} = \alpha_k \cdot \ldots \cdot \alpha_{\ell} \  \
(\forall \ \ \{a_k,\ldots, a_{\ell} \} \in P)$$
as in Construction $\ref{hiercr}$. The fibers of $\sC(P)$ have the degree $m'$, where $m'$ divides
$m$. For $j = 1, \ldots, m'-1$ and $q \in \sM_1$, the eigenspace $\LL_j(P)$ in the Hodge structure of
$\sC(P)_q$ with the character $j$ is given by the local monodromy data
$$[j\mu_1]_1, \  \ [j\mu_2]_1, \  \ [j\mu_3+ \ldots+j\mu_k]_1,
\  \ [j\mu_{k+1}+ \ldots+j\mu_{n+3}]_1.$$
If the eigenspace $\sL_j$ in the $VHS$ of $\sC$ is of type $(0,a)$, Proposition $\ref{1.27}$
implies that its local monodromy data satisfy
$$[j\mu_1]_1 + \ldots + [j\mu_{n+3}]_1 = 1.$$
Hence one has that
$$[j\mu_1]_1 + [j\mu_2]_1 + [j\mu_3+ \ldots+j\mu_k]_1 + [j\mu_{k+1}+ \ldots+j\mu_{n+3}]_1 = 1,$$
too. Thus by Proposition $\ref{1.27}$, $\LL_j(P)$ is of type $(0,a')$.

If $\sL_j$ is of type $(a,0)$, $\sL_{m-j}$ is of type $(0,a)$. The dual eigenspace
$\LL_j(P)^{\vee}$ of $\LL_j(P)$ is given by
$$[(m-j)\mu_1]_1, \  \ [(m-j)\mu_2]_1, \  \ [(m-j)\mu_3+ \ldots+(m-j)\mu_k]_1,
\  \ [(m-j)\mu_{k+1}+ \ldots+(m-j)\mu_{n+3}]_1.$$
The same arguments as above tell us that $\LL_j(P)^{\vee}$ is of type $(0,a')$. Thus $\LL_j(P)$ is
of type $(a',0)$.

The restricted family $\sC_{\sM_1}(P) \to \sM_1$ of cyclic covers
with 4 different branch points has a non-trivial variation of Hodge structures. This follows from
the fact that each fiber of $\sC_{\sM_n} \to \sM_n$ is isomorphic to only finitely many other
fibers (compare to $\ref{japp}$). Hencefore the eigenspaces $\LL_{1}(P)$ and $\LL_{m'-1}(P)$
are of type $(1,1)$. In addition one concludes that $m' = 2$ or $m' = m$.
\end{pkt}

Now we are without loss of generality in the case of a family $\sC \to \sP_1$ with a primitive
pure $(1,1)-VHS$. For the proof of Theorem $\ref{sint}$ we need the following lemma:

\begin{lemma} \label{jaaa}
Let $C$ and $C'$ be curves and $\gamma : {\rm Jac}(C) \to {\rm Jac}(C')$ be an isomorphism of
principally polarized Abelian varieties. Then there exists a unique isomorphism $f: C \to C'$
such that
$$\pm \gamma \circ \alpha_p = \alpha_{f(p)} \circ f$$
for each $p \in C$, where $\alpha_p$ and $\alpha_{f(p)}$ denote the respective Abel-Jacobi maps.
\end{lemma}
\begin{proof}
By \cite{milne2}, Theorem 12.1, for each $p \in C$ and $p' \in C'$ there is a unique isomorphism
$f: C \to C'$ and a unique $c \in {\rm Jac}(C')$ such that
$$\pm \gamma \circ \alpha_p + c = \alpha_{p'} \circ f.$$
Since $(\gamma \circ \alpha_p)(p) = 0 \in {\rm Jac}(C')$ and
$(\alpha_{p'} \circ f)(p) = [f(p)-p'] \in {\rm Jac}(C')$, one has $c = 0$ for $p' = f(p)$.\end{proof}

By the next proposition, we will apply Lemma $\ref{jaaa}$ for our proof of Theorem $\ref{sint}$:

\begin{proposition} \label{uebers}
Let $q_1, q_2 \in \sP_n$ and $\sC \to \sP_n$ be a family of cyclic covers. Assume there is an
isomorphism between the polarized integral Hodge structures of the fibers $\sC_{q_1}$ and
$\sC_{q_2}$, which respects the eigenspace decompositions of
$H^1(\sC_{q_1},\C)$ and $H^1(\sC_{q_1},\C)$. Then there is an isomorphism
$\iota : \sC_{p_1} \to \sC_{p_2}$ and an isomorphism $\alpha : \bP^1 \to \bP^1$ such that the
following diagram commutes:
$$\xymatrix{
  {\sC_{p_1}} \ar[rr]^{\iota} \ar[d]  &  & {\sC_{p_2}}\ar[d]\\
  {\bP^1} \ar[rr]^{\alpha} &  & {\bP^1} 
}
$$
\end{proposition}
\begin{proof}
Let $\gamma$ be an isomorphism of polarized Hodge structures respecting the eigenspace
decompositions of $H^1(\sC_{q_1}\C)$ and $H^1(\sC_{q_1}\C)$. Then there exists a suitable pair
$(\psi_1,\psi_2)$ of generators of the Galois groups of $\sC_{q_1}$ and $\sC_{q_1}$ such that
$$\gamma \circ (\psi_1)_* = (\psi_2)_* \circ \gamma.$$
For simplicity we write $\psi$ instead of $\psi_1$ and $\psi_2$.

By the exponential exact sequence, an isomorphism
$\gamma: H^1(\sC_{q_1},\Z) \to H^1(\sC_{q_2},\Z)$ of polarized Hodge structures
commuting with the action of $\psi$ on these integral Hodge structures induces an isomorphism
$\gamma':{\rm Jac}(\sC_{q_1}) \to {\rm Jac}(\sC_{q_2})$ commuting with $\psi_*$.
In other terms one has
$$\gamma' \circ \psi_* = \psi_* \circ \gamma'$$
for the Jacobians.

By Lemma $\ref{jaaa}$, one obtains a unique isomorphism
$\sC_{p_1} \stackrel{\iota}{\to} \sC_{p_2}$
such that 
$$\iota \circ \psi = \psi \circ \iota.$$
Thus one obtains the desired automorphism $\alpha$.
\end{proof}

\begin{pkt} \label{pot}
Now assume that $\sC \to \sP_n$ has a primitive pure $(1,n)-VHS$. Moreover one can without
loss of generality assume that $\sL_1$ is the eigenspace of type $(1,n)$. Choose $s_1 , s_2 \in N$.

If $\mu_1 + \mu_2 = 1$, there remains nothing to prove with respect to these two points for
Theorem $\ref{sint}$.

Otherwise we let the branch points collide as in Lemma $\ref{lbc}$, if we are not in the only
exceptional case, which satisfies $SINT$ as we have seen in $\ref{excep}$. Thus we can restrict to
the case $\sC \to \sP_1$. Assume that all 4 branch points of
a fiber of $\sC \to \sP_1$ have pairwise different branch indeces. Hencefore there will not be an
isomorphism $\alpha$ as in Proposition $\ref{uebers}$ for different fibers. Hence Proposition
$\ref{uebers}$ implies that the fractional period map according to $\sL_{1}|_{\sM_1}$ is
injective. Now choose the embedding $\sM_1 \hookrightarrow \bP^1$ corresponding to
$$p_1 = 0, \  \ p_3 = 1, \  \ p_4 = \infty.$$
By \cite{Loo}, Section 4, one can identify the fractional period map concerning $\sL_1$ with
some multivalued map, which is called Schwarz map. The Schwarz map is the composition of
the multivalued map studied by P. Deligne and G. D. Mostow in \cite{DM}, which is defined by some
integrals, with the natural
map $\C^{n+1}\setminus\{0\} \to \bP^n_{\C}$. By \cite{DM}, $9.6$ and the preceding
description of the fractional period map, there exists a sufficiently small neighborhood $U$ of
$0 \in \bP^1 \setminus \sM_1$ such that the fractional period map concerning $\sL_1$ is (up to a
biholomorphic map) given by $x \to x^{1-\mu_1-\mu_2}$ on $U\setminus \{0\}$. Hence the injectivity
of the period map implies that $(1- \mu_1 - \mu_2)^{-1} \in\Z$. This yields $SINT$.
\end{pkt}

\begin{pkt}
Now we have the problem that we can not directly apply Proposition $\ref{uebers}$ as before, if we
assume that there are 4 branch points, where exactly two of them have the same branch index: Let
$p_1$ and $p_4$ have the same branch index and $p_3$ run around $p_2$, where
$$p_1 = 0, \  \ p_2 = 1, \  \ p_4 = \infty.$$
The automorphism $x \to x^{-1}$ interchanges $0$ and $\infty$ and leaves a basis of neighborhoods
of $1\in \bP^1 \setminus \sM_1$ invariant. We have obviously the same problem, if we let $p_1$ run around
$p_4$. But for all other pairs $k_1, k_2 \in \{1,2,3,4\}$ with $k_1 \neq k_2$ and the coordinates
$$k_1 = 0, \  \ k_3 = 1, \  \ k_4 = \infty,$$
Proposition $\ref{uebers}$ implies that the multivalued period map is injective on
$U\setminus \{0\}$, where $U$ is a sufficiently small neighborhood of $0 \in
\bP^1 \setminus \sM_1$. Thus $k_1$ and $k_2$ satisfy the integral condition
$$1 - \mu_{k_1} - \mu_{k_2} = 0 \ \ \mbox{or} \  \ (1 - \mu_{k_1} - \mu_{k_2})^{-1} \in \Z.$$
Hence one must ensure that the remaining pairs satisfy this latter condition, in order to show
that $SINT$ is satisfied:

Let us change the enumeration and assume that
$\mu_1 = \mu_2$. By Proposition $\ref{uebers}$, we can have
$$1-\mu_1 -\mu_2 = \frac{1}{\ell} \ \ \mbox{or} \  \
1-\mu_{1} -\mu_{2} = \frac{2}{\ell} \ \ \mbox{or} \  \ 1-\mu_1 -\mu_2 = 0$$
for some odd $\ell \in \Z$. Note that $1-\mu_1-\mu_2 = -(1-\mu_3-\mu_4)$, if
$\mu_1 + \ldots +\mu_4 = 2$. Hence we only have to exclude the second case
$(1-\mu_{1} -\mu_{2}) = \frac{2}{\ell}$. First assume that $m$ is odd. In this case $m - 2d_1$ is
odd  and the second case can not occur. Hence assume that $m$ is even and let $m = 2^sr$, where
$r= k\cdot \ell$ is odd. If the second case holds true, one has
$$\frac{2^sk\ell - 2 d_1}{2^sk\ell} = \frac{2}{\ell} \Leftrightarrow 
2^{s-1}k \ell - d_1=2^sk \Leftrightarrow d_1 = 2^{s-1}k(\ell-2).$$
If $s\geq 2$, one has that $d_1= d_2$ is even. Since $\sC_2$ must have a trivial $VHS$, one has
without loss of generality that $d_3 = 2^{s-1}k\ell$. Since we have
$$2m = d_1+ \ldots + d_4,$$
which is even, where $d_1, d_2, d_3$ are even, too, $d_4$ must be even. But in this case the cover
is not irreducible. Hence we must have $s= 1$. Since $\sC_2$ must have a trivial $VHS$, one has
without loss of generality $d_3 = k\ell$. Since we have
$$2m = d_1+ \ldots + d_4,$$
which is even, where $d_1, d_2, d_3$ are odd, one must have that $d_4$ is odd, too. But in this
case $\sC_{k\ell}$ is the family of elliptic curves and we do not have a
primitive pure $(1,1)-VHS$. Hencefore the second case is excluded.
\end{pkt}

\begin{remark}
If we have 4 branch points and more than exactly two of them have the same branch index, one can
have the additional simple cases
$$\mu_1 = \mu_2 = \mu_3 \  \ \mbox{or} \  \ \mu_1 = \mu_2, \ \ \mu_3 = \mu_4.$$
For these very simple cases one can directly calculate all occurring examples of families
$\sC\to \sP_1$ with a primitive pure $(1,1)-VHS$. Then one can verify by their local
monodromy data that Theorem $\ref{sint}$ holds true in these cases as we will do now.
\end{remark}

\begin{remark}
One must without loss of generality have
$$d_1 < d_4
\  \ \mbox{resp.,} \  \ d_1 = d_2 = d_3 < d_4 \   \  \mbox{or} \  \ d_1 = d_2 < d_3 = d_4$$
in the simple cases, if $m > 2$. Otherwise we would obtain
$$d_1 = d_2 = d_3 = d_4 = \frac{m}{2},$$
which implies that $\sC$ is not irreducible, if $m > 2$.
\end{remark}

\begin{lemma}
Assume that the family $\sC \to \sP_1$ with the branch indeces $d_1 = d_2 = d_3 \neq d_4$ has a
primitive pure $(1,1)-VHS$. Then the degree $m$ is odd and satisfies $m \leq 9$. Moreover one has
without loss of generality that $d_1 = d_2 + d_3 = 1$.
\end{lemma} 
\begin{proof}
By the assumptions we have that $2m = 3d_1 + d_4$. Hence $g= \gcd(m,d_1) = 1$ divides $d_4$, too,
which implies by the irreducibility of the fibers of $\sC$ that $g=1$. Thus if $m$ is even,
we have that $d_1= d_2 = d_3$ and $d_4$ are odd. But then
$\sC_{\frac{m}{2}}$ would be the family of elliptic curves such that $\sL_{\frac{m}{2}}$ is of
type $(1,1)$. Contradiction! Hence $m$ must be odd.

It remains to show that $m \leq 9$. Since $\gcd(m,d_1) = 1$, the fibers are without
loss of generality given by
$$y^m = x(x-1)(x-\lambda)$$
such that $\sL_{[\frac{m}{2}]}$ is of type $(1,1)$ as one can calculate by Proposition
$\ref{1.27}$. By Proposition $\ref{1.27}$, one can calculate the type of $\sL_{\frac{m-3}{2}}$ by
its local monodromy data, too. For this local system one gets that
$$3[\frac{m-3}{2m}]_1 +[\frac{(m-3)(m-3)}{2m}]_1$$
$$= 3\frac{m-3}{2m}
+ (m-3)\frac{m-3}{2m} -[\frac{(m-3)(m-3)}{2m}]
= \frac{m-3}{2} -[\frac{(m-3)(m-3)}{2m}].$$
Now let us assume that $9< m$. Since $m$ must be odd, we obtain
$$3[\frac{m-3}{2m}]_1 +[\frac{(m-3)(m-3)}{2m}]_1 = \frac{m-3}{2} - [\frac{m-6}{2}+\frac{9}{2m}] =
\frac{m-3}{2} -\frac{m-7}{2} = 2.$$
This result and Proposition $\ref{1.27}$ imply that $\sL_{\frac{m-3}{2}}$ is of type $(1,1)$ in
this case, too. Hence we do not have a pure $(1,1)-VHS$, if $9< m$.
\end{proof}

\begin{remark} \label{3gleich}
In the case of the preceding lemma one obtains all examples of families $\sC \to \sP_1$ with
a primitive pure $(1,1)-VHS$ by $m = 5, 7, 9$, which satisfy $SINT$ as one can calculate easily,
too.
\end{remark} 

\begin{remark} \label{deg3}
If we are in the second simple case $d_1 = d_2 \neq d_3 = d_4$, one obtains
$$d_1 + d_3 = d_2 + d_4 = m.$$
By the fact that
$d_1 \neq d_3$, one concludes that $\mu_1, \mu_3 \neq \frac{1}{2}$. Hence the local monodromy data
of $\sL_2$ satisfy $[2\mu_i]_1 \neq 0$ for all $i = 1, \ldots ,4$. Moreover one has
$$[2\mu_1]_1+[2\mu_3]_1 = [2\mu_2]_1+[2\mu_4]_1 = 1.$$
Hence $\sL_2$ is of type $(1,1)$ and $\sC$ can have a primitive $(1,1)-VHS$, only if $m = 3$. Thus
the only possible case is given by
$$\mu_1 = \mu_2 = \frac{1}{3} \ \ \mbox{and} \ \ \mu_3 = \mu_4 = \frac{2}{3},$$
which satisfies $SINT$ as one can easily verify.
\end{remark}

\section{The application of $SINT$ for the more complicated cases}

In the preceding section we have seen that $SINT$ is a necessary condition for families
$\sC \to \sP_n$ with a primitive pure $(1,n)-VHS$. In addition we have given all examples of
families $\sC \to \sP_1$ with a primitive pure $(1,1)-VHS$, which do not satisfy that at most
two branch points have the same branch index. Here we calculate all examples of families
$\sC \to \sP_1$ with a primitive pure $(1,1)-VHS$, which satisfy that at most two branch points have the same branch index.

By technical reasons, we will sometimes assume $m \geq 4$. Note that the only
possible case of a family $\sC\to \sP_1$ of degree 3 covers with a pure $(1,1)-VHS$ is given
by Remark $\ref{deg3}$, where the only possible case of degree 2 covers is given by the elliptic
curves. Thus this assumption does not provide any restriction for the more complicated cases.

Note that in the case of a family $\sC \to \sP_1$ the condition $SINT$ is equivalent to $INT$.

\begin{remark} \label{1dimballs}
By \cite{DM}, 14.3, one can describe all families of covers $\sC \to \sP_1$, whose local monodromy
data satisfy $INT$, such that there is not any pair $k_1,k_2 \in \{1,2,3,4\}$ with $k_1 \neq k_2$
satisfying $\mu_{k_1} + \mu_{k_2} = 1$, in the following way:
Let $(p,q,r) \in \N^3$ with  $ \frac{1}{p} + \frac{1}{q}+\frac{1}{r} < 1$ and
$1 < p \leq q \leq r < \infty$. Then in the case of 4 branch points these solutions of
$INT$ for covers can be given by:
\begin{eqnarray*}
\mu_1 = \frac{1}{2}(1-\frac{1}{p}-\frac{1}{q}+\frac{1}{r}), \ \
\mu_2 = \frac{1}{2}(1-\frac{1}{p}+\frac{1}{q}-\frac{1}{r}),\\
\mu_3 = \frac{1}{2}(1+\frac{1}{p}-\frac{1}{q}-\frac{1}{r}), \ \
\mu_4 = \frac{1}{2}(1+\frac{1}{p}+\frac{1}{q}+\frac{1}{r})
\end{eqnarray*}
We have that
$$\mu_1 + \mu_2 = 1- \frac{1}{p}, \    \ \mu_1 + \mu_3 = 1- \frac{1}{q}, \     \
\mu_2 + \mu_3 = 1- \frac{1}{r}.$$
Thus $p$ ,$q$, and $r$ divide the degree $m$ of the cover. This fact and the equations, which use
$p$, $q$, and $r$ for the definition of the different $\mu_i$, imply that we have
$$m = \lcm(p,q,r) \    \ \mbox{or} \    \ m = 2 \cdot \lcm(p,q,r).$$
\end{remark}

If we are in the case of a family with a primitive pure $(1,1)-VHS$ such that all local monodromy
data satisfy $\mu_{k_1} + \mu_{k_2} \neq 1$ and at most two branch points have the same index,
we are in the case of Remark $\ref{1dimballs}$ with the additional condition $p < r$. Hence let us
first consider this case. Later we will consider families with at most two branch points with the
same branch index and some $\mu_{k_1} + \mu_{k_2} = 1$, which is the last remaining subcase.

Now let $\ell := \lcm(p,q,r)$.

\begin{lemma} \label{dudu}
Let $\sC \to \sP_1$ be given by $p, q, r$ as in Remark $\ref{1dimballs}$, where $p < r$, and have
a primitive pure $(1,1)-VHS$. Then one has
$$\frac{1}{p} = \frac{1}{q} + \frac{1}{r}.$$
\end{lemma}
\begin{proof}
Since $p|\ell$ resp., $p|m$, we have the family $\sC_p$, which must have a trivial $VHS$. This
implies that there is a $d_{i_0}$ with $\frac{m}{p}|d_{i_0}$, which implies that
$\frac{\ell}{p}|d_{i_0}$. Since
$$d_{i_0} = \ell\pm \frac{\ell }{p} \pm \frac{\ell }{q} \pm \frac{\ell }{r} \  \ \mbox{or} \ \
2d_{i_0} = \ell\pm \frac{\ell }{p} \pm \frac{\ell }{q} \pm \frac{\ell }{r},$$
one concludes that $\frac{\ell}{p} |(\frac{\ell}{q}\pm \frac{\ell}{r})$. From the
fact that $ \frac{\ell}{p} \geq \frac{\ell}{q}$ and $\frac{\ell}{p} >\frac{\ell}{r}$, one obtains
$$\frac{\ell}{p} = \frac{\ell}{q} + \frac{\ell}{r}. \  \ \mbox{Hence} \   \
\frac{1}{p} = \frac{1}{q} + \frac{1}{r}.$$
\end{proof}

\begin{lemma} \label{l3}
Let $\sC \to \sP_1$ be a family with a primitive pure $(1,1)-VHS$,
which is given by $p, q, r$ as in Remark $\ref{1dimballs}$, where $p < r$. Then the
family $\sC$ and the eigenspace $\LL_1$ are given by the local monodromy data
$$\mu_1 = \frac{1}{2}-\frac{1}{q}, \  \ \mu_2 = \frac{1}{2}-\frac{1}{r}, \   \ \mu_3 = \frac{1}{2},
\  \ \mu_4 = \frac{1}{2}+\frac{1}{p}.$$
\end{lemma}
\begin{proof}
By Lemma $\ref{dudu}$, we have
$$\frac{1}{p} = \frac{1}{q} + \frac{1}{r}.$$
This equation and  Remark $\ref{1dimballs}$, this imply that $\sC$ and $\LL_1$ have the local
monodromy data
$$\mu_1 = \frac{1}{2}-\frac{1}{q}, \  \
\mu_2 = \frac{1}{2}-\frac{1}{r}, \   \
\mu_3 = \frac{1}{2}, \  \
\mu_4 = \frac{1}{2}+\frac{1}{p}.$$
\end{proof}

\begin{remark} \label{3bsp}
Let $\sC \to \sP_1$ is a family of covers of degree $m \geq 4$ with a primitive pure
$(1,1)-VHS$ satisfying the assumptions of Lemma $\ref{l3}$. Moreover assume that
$3 \notin (\Z/(m))^*$. Hence the assumption that $\sC \to \sP_1$ has primitive pure $(1,1)-VHS$
implies that the family $\sC_3$ must have a trivial $VHS$. Thus all fibers of $\sC_3$ must be
isomorphic. Hence they are ramified over at most 3 points. By Lemma $\ref{l3}$, one
concludes that
$$0 = [\frac{1}{2}-\frac{3}{q}]_1, \  \
0 = [\frac{1}{2}-\frac{3}{r}]_1 \   \
\mbox{or} \  \
0 = [\frac{1}{2}+\frac{3}{p}]_1.$$
Since $\mu_4 =\frac{1}{2}+\frac{1}{p} < 1$, one concludes that $2 < p \leq q \leq r$. Thus one has
$$p = 6, \  \ q = 6 \  \ \mbox{or} \  \ r = 6.$$
Hence one can determine all examples of families $\sC \to \sP_1$ with a primitive pure $(1,1)-VHS$
in this case as we will do now:
\end{remark}

\begin{pkt} \label{3bsp1}
Keep the assumptions of Remark $\ref{3bsp}$. In the case $p = 6$ one has that $[3\mu_4]_1 = 0$.
One can have $q = 7,8,9,10,11,12$, where
$q = 12$ implies that
$$\frac{1}{6} = \frac{1}{p}, \  \ \frac{1}{q}= \frac{1}{r} = \frac{1}{12},$$
which leads to a family with a primitive pure $(1,1)-VHS$. Now we
verify that $q = 7,8,9,10,11$ do not lead to a family with a primitive pure $(1,1)-VHS$:
One must have that $\LL_5$ is unitary. It has the local monodromy data
$$\mu_3 = \frac{1}{2} \  \ \mbox{and} \  \ \mu_4 = [\frac{1}{2}+\frac{5}{6}]_1 = \frac{1}{3}.$$
Hence one must have that
$$\frac{1}{6} \geq \mu_1 = [\frac{1}{2}-\frac{5}{q}]_1,$$
which is satisfied for $q= 10,11$, but not for $q = 7,8,9$. For $q = 10$ we have that
$$\frac{1}{r} = \frac{1}{p} -\frac{1}{q} = \frac{1}{15}.$$
This leads to a family given by the local monodromy data
$$ \mu_1 = \frac{4}{10}, \  \ \mu_2 = \frac{13}{30}, \  \ \mu_3 = \frac{1}{2}, \  \
\mu_4 = \frac{2}{3}.$$
One calculates easily that the eigenspace $\sL_7$ in the $VHS$ of this family is given by
$$\mu_1 = \frac{4}{5}, \  \ \mu_2 = \frac{1}{30}, \  \ \mu_3 = \frac{1}{2}, \  \
\mu_4 = \frac{2}{3}.$$
Hence this family has not a pure $(1,1)-VHS$.

For $q = 11$ we have that
$$\frac{1}{p} -\frac{1}{q} = \frac{5}{66}.$$
Hence the equation of Lemma $\ref{dudu}$ can not be satisfied in this case.
\end{pkt}

\begin{pkt} \label{3bsp2}
Keep the assumptions of Remark $\ref{3bsp}$. Moreover assume that $q = 6$.
In this case we can have $p = 3, 4, 5$, where $p = 3$ implies
$$\frac{1}{p} = \frac{1}{3}, \  \ \frac{1}{q} = \frac{1}{r} = \frac{1}{6},$$
which yields an example of a family with a primitive $(1,1)-VHS$. For $p= 4$ resp., $p = 5$ Lemma
$\ref{dudu}$ and Lemma $\ref{l3}$ yield a family of covers given by the local monodromy data
$$\mu_1 = \frac{1}{3}, \  \ \mu_2 = \frac{5}{12} \  \  \mu_3 = \frac{1}{2}, \  \
\mu_4 = \frac{3}{4}$$
resp.,
$$\mu_1 = \frac{1}{3}, \  \ \mu_2 = \frac{14}{30} \  \ \mu_3 = \frac{1}{2}, \  \
\mu_4 = \frac{7}{10}.$$
Hence one can easily verify that $\sL_5$ is an eigenspace of type $(1,1)$ in both cases.  Thus
$p = 4, 5$ do not lead to a primitive pure $(1,1)-VHS$.
\end{pkt}

\begin{pkt} \label{3bsp3}
Keep the assumptions of Remark $\ref{3bsp}$. Moreover assume that $r = 6$.
In this case Lemma $\ref{dudu}$ implies that
$$\frac{1}{p} \geq \frac{2}{r} = \frac{1}{3}.$$
Hence one has $p = 2$ or $p = 3$,
where $p = 2$ would imply that $\mu_4 = 1$, and $p = 3$ yields the same example of a family with
a primitive pure $(1,1)-VHS$ as in $\ref{3bsp2}$.
\end{pkt}

Now we have considered the subcase given by $3\notin (\Z/(m))^*$. We start the consideration of
the subcase given by $3\in (\Z/(m))^*$ by the following lemma: 

\begin{lemma} \label{dudud}
Let $\sC \to \sP_1$ be a family with a primitive pure $(1,1)-VHS$,
which satisfies that each $\mu_{k_1}+ \mu_{k_2} \neq 1$. Then one has $m > 4$.
\end{lemma}
\begin{proof}
We know that one must have $m \geq 4$ in the considered case. Thus we must only exclude $m= 4$.
Since for a family $\sC$ of degree 4 covers with a primitive pure $(1,1)-VHS$ the family
$\sC_2$ must have a trivial $VHS$, one has without loss of generally $d_1 = 2$. By the assumption
that each $\mu_{k_1}+ \mu_{k_2} \neq 1$, one concludes that $d_1, d_2, d_3$ are not equal to 2.
Hence $d_1, d_2, d_3$ are odd. But this contradicts our assumptions, which imply that we have
the even sum
$$2m = d_1 + \ldots +d_4.$$
\end{proof}

\begin{remark}
Keep the assumption of Lemma $\ref{l3}$. If $m > 4$, the eigenspace $\LL_3$ is not of type $(1,1)$. Assume
that 3 is a unit in $\Z/(m)$. Thus Lemma $\ref{dudud}$ implies for the local monodromy data of
$\LL_3$ that $\sum \mu_i = 3$ or $\sum \mu_i = 1$. Recall that $\mu_3 = \frac{1}{2}$. Hence
$\sum \mu_i = 3$ implies that
$$\mu_1, \mu_2, \mu_4 > \frac{1}{2}.$$
By Lemma $\ref{l3}$, one concludes that
$$ \mu_1 = \frac{3}{2}-\frac{3}{q} \  \ \mbox{and} \ \ \mu_2 = \frac{3}{2}-\frac{3}{r}.$$
By Lemma $\ref{dudu}$, this implies that
$$\mu_4 = 3 -\mu_1 -\mu_2-\mu_3 = 3 - \frac{3}{2}+\frac{3}{q}-\frac{3}{2}+\frac{3}{r}- \frac{1}{2}
= -\frac{1}{2}+\frac{3}{p}+\frac{3}{q}= -\frac{1}{2}+\frac{3}{p}$$
in this case. This implies that $p,q,r < 6$.

In the case $\sum \mu_i = 1$ one gets $\mu_1, \mu_2, \mu_4 < \frac{1}{2}$. By Lemma $\ref{dudu}$
and Lemma $\ref{l3}$, this implies that $$\mu_1 = \frac{1}{2}-\frac{3}{q}, \  \ 
\mu_2 = \frac{1}{2}-\frac{3}{r}\  \ \mbox{and} \ \ \mu_4 = -\frac{1}{2}+\frac{3}{p}$$
such that $p < 6$ and $q,r > 6$.
\end{remark}

\begin{remark}
The case $p,q,r < 6$ does not yield any example of a family with a primitive pure
$(1,1)-VHS$, since no triple $(p,q,r) \in \N^3$ with $2 \leq p\leq q\leq r < 6$ satisfies
both
$$\frac{1}{p} - \frac{1}{q} = \frac{1}{r} \ \ \mbox{and} \ \
\frac{1}{p} + \frac{1}{q} + \frac{1}{r} < 1$$
as one can check by calculation for each example.
\end{remark}

\begin{pkt} \label{az}
Assume that we are in the case $p < 6$ and $q,r > 6$. Since
$\frac{1}{p} + \frac{1}{q} = \frac{1}{r}$, one has $\frac{1}{p} \leq 2\frac{1}{q}$ such that 
$6 < q \leq 2p$ and $3 < p < 6$. Hence one has two cases: $p = 4$ or $p = 5$. Thus by using
that $\frac{1}{p} - \frac{1}{q} = \frac{1}{r}$, one calculates that only the
examples given by
$$p = 4, \  \ q = r = 8 \  \ \mbox{and} \ \ p = 5, \  \ q = r = 10$$
have a primitive pure $(1,1)-VHS$ in this case.
\end{pkt}

Now we consider the last remaining case of a family $\sC \to \sP_1$ with a primitive pure
$(1,1)-VHS$. In this case there are at most 2 branch indeces equal and one has some
$\mu_{k_1} + \mu_{k_2} = 1$.

\begin{lemma}
Let $\sC \to \sP_1$ a family of cyclic covers. If there are $k_1, k_2 \in \{1,2,3,4\}$ such
that $d_{k_1} + d_{k_2} = m$ with $d_1 \leq d_2 \leq d_3 \leq d_4$, then one has
$$d_1 + d_4 = m \  \ \mbox{and} \   \ d_2 + d_3 = m.$$
\end{lemma}
\begin{proof}
(quite easy to see)
\end{proof}

\begin{remark} \label{noncom}
By the preceding lemma, we have that $d_1 + d_4 = d_2 + d_3 = m$, if there are $k_1, k_2 \in
\{1,2,3,4\}$ such that $d_{k_1} + d_{k_2} = m$ with $d_1 \leq d_2 \leq d_3 \leq d_4$.
Hence if $d_1 + d_3 = m$ resp., $d_3 = d_4$, one gets $d_1 = d_2$, too. But this contradicts the
assumption that at most 2 branch indexes are equal. Hence by $SINT$, one gets
$$\mu_1 + \mu_2 = 1- \frac{1}{p}< 1, \    \ \mu_1 + \mu_3 = 1- \frac{1}{q} < 1, \     \
\mu_2 + \mu_3 = 1$$
with $p, q \in \N$ and $p \leq q$. Hence
one obtains similarly to Remark $\ref{1dimballs}$ with $\frac{1}{p}+\frac{1}{q} < 1$:

\begin{eqnarray*}
\mu_1 = \frac{1}{2}(1-\frac{1}{p}-\frac{1}{q}), \   \
\mu_2 = \frac{1}{2}(1-\frac{1}{p}+\frac{1}{q}),\\
\mu_3 = \frac{1}{2}(1+\frac{1}{p}-\frac{1}{q}), \    \
\mu_4 = \frac{1}{2}(1+\frac{1}{p}+\frac{1}{q})
\end{eqnarray*}
\end{remark}

\begin{lemma}
Assume that the local monodromy data of Remark $\ref{noncom}$ yield a family of degree $m \geq 4$
with a primitive pure $(1,1)-VHS$. Then one has $p = q$ and $m$ is even.
\end{lemma}
\begin{proof} In the case of Remark $\ref{noncom}$ the eigenspace $\LL_2$ is given by the local monodromy data
$$\mu_1 = 1-\frac{1}{p}-\frac{1}{q}, \   \
\mu_2 = [1-\frac{1}{p}+\frac{1}{q}]_1, \ \
\mu_3 = [\frac{1}{p}-\frac{1}{q}]_1, \    \
\mu_4 = \frac{1}{p}+\frac{1}{q}.$$
Thus in this case $\sL_2$ is of type $(1,1)$, if and only if $p < q$. Hence one can obtain a
primitive pure $(1,1)-VHS$, only if $p = q$. Now $p = q$ implies that $\mu_2= \mu_3 = 0$ for the
local monodromy data of $\LL_2$. Hence the family of covers has an even degree.
\end{proof}

\begin{proposition}
Assume that the local monodromy data of Remark $\ref{noncom}$ yield a family of degree $m \geq 4$
with a primitive pure $(1,1)-VHS$. Then
$p = q \leq 6$.
\end{proposition} \begin{proof}
By the preceding lemma, the assumptions imply that $p = q$. Hence by Remark $\ref{noncom}$, we
have:
\begin{equation} \label{boff}
\mu_1 = \frac{p-2}{2p}, \  \ \mu_2 = \mu_3 = \frac{1}{2}, \  \ \mu_4 = \frac{p+2}{2p}
\end{equation}
If $p>6$, then $\LL_3$ has the local monodromy given by
$$\mu_1 = \frac{p-6}{2p}, \  \ \mu_2 = \mu_3 = \frac{1}{2}, \  \ \mu_4 = \frac{p+6}{2p}.$$
Hence Proposition $\ref{1.27}$ implies that $\sL_3$ is of type $(1,1)$ in this case.
\end{proof}

\begin{lemma}
Assume that the local monodromy data of Remark $\ref{noncom}$ yield a family of degree $m \geq 4$
with a primitive pure $(1,1)-VHS$. Then $p$ must be even.
\end{lemma}
\begin{proof}
Assume that $p$ is odd. Since $gcd(p-2,2p)= 1$ in this case, one gets a family of degree $2p$
with branch indeces
$$d_1 = p-2, \  \ d_2 = d_3 = p, \  \ d_4 = p+2.$$
Thus all branch indeces are odd, and $\sC_p$ is a family of elliptic curves such that $\sL_p$ is
of type $(1,1)$. Contradiction!
\end{proof}

\begin{remark} \label{letzt}
Keep the assumptions of the preceding lemma. Since one must have $\mu_1 > 0$, the preceding
proposition and $\eqref{boff}$ imply that
$$3 \leq p \leq 6.$$
Since $p = q$ must be even, one can only have $p= 4$ and $p = 6$.
\begin{enumerate}
\item For $p = 4$ one obtains the example of a family with a primitive pure $(1,1)-VHS$ given by
$$\mu_1 = \frac{1}{4}, \  \ \mu_2 = \mu_3 = \frac{1}{2}, \  \ \mu_4 = \frac{3}{4}.$$
\item If $p = 6$ one has the example of a family with a primitive pure $(1,1)-VHS$ given by
$$\mu_1 = \frac{1}{3}, \  \ \mu_2 = \mu_3 = \frac{1}{2}, \  \ \mu_4 = \frac{2}{3}.$$
\end{enumerate}
\end{remark}

\section{The complete lists of examples}
In this section we give the complete lists of examples of families $\sC \to \sP_n$ with
primitive $(1,n)$-variations or derived pure $(1,n)$-variations of Hodge structures.

By our preceding
calculations, we get the following complete list of families of covers $\sC \to \sP_1$ with a
primitive pure $(1,1)-VHS$, where ``ref'' denotes the number of the preceding remark, lemma,
proposition or point yielding the respective example:
\begin{center}
\begin{tabular}{|c||c|c|c|c|}  \hline 
number & degree & branch points with branch index & genus &  ref \\ \hline \hline
1 & 2 & 1 1 1 1 & 1 & (known) \\ \hline
2 & 3 & 1 2 2 1 & 2 & $\ref{deg3}$ \\ \hline
3 & 4 & 1 2 2 3 & 2 & $\ref{letzt},(1)$ \\ \hline
4 & 5 & 1 3 3 3 & 4 & $\ref{3gleich}$ \\ \hline
5 & 6 & 1 4 4 3 & 3 & $\ref{3bsp2}, \ref{3bsp3}$ \\ \hline
6 & 6 & 2 3 3 4 & 2 & $\ref{letzt}, (2)$ \\ \hline
7 & 7 & 2 4 4 4 & 6 & $\ref{3gleich}$ \\ \hline
8 & 8 & 2 5 5 4 & 5 & $\ref{az}$ \\ \hline
9 & 9 & 3 5 5 5 & 7 & $\ref{3gleich}$ \\ \hline
10 & 10 & 3 6 6 5 & 6 & $\ref{az}$ \\ \hline
11 & 12 & 4 7 7 6 & 7 & $\ref{3bsp1}$ \\ \hline
\end{tabular}
\end{center}
We will later see that each derived pure $(1,n)-VHS$ is in fact
a derived pure $(1,1)-VHS$. In the next section we will verify that we get the following
complete list of families of covers $\sC \to \sP_1$ with a derived pure $(1,1)-VHS$, where
$N_{r_0}$ means the number of $\sC_{r_0}$ in the preceding list, which has the corresponding
primitive pure $(1,1)-VHS$:
\begin{center}
\begin{tabular}{|c|c|c|c|c|} \hline
degree & branch points with branch index & genus & $r_0$ & $N_{r_0}$ \\ \hline \hline
4 & 1 1 1 1 & 3 & 2 & 1\\ \hline
6 & 1 1 1 3 & 4 & 3 & 1\\ \hline
6 & 1 2 2 1 & 4 & 2 & 2\\ \hline
\end{tabular}
\end{center}
Note that any family $\sC \to \sP_n$ with a primitive pure $(1,n)-VHS$ satisfies
$SINT$, which implies $INT$. Hence by consulting the list of \cite{DM} on page 86, which contains
all examples satisfying $INT$ for $n \geq 2$, (and the calculation of the types of the
eigenspaces of the corresponding covers), we have the following complete list of families of
covers with a primitive pure $(1,n)-VHS$ for $n> 1$:
\begin{center}
\begin{tabular}{|c|c|c|c|} \hline
degree & branch points with branch index & genus \\ \hline \hline
3 & 2 1 1 1 1 & 3 \\ \hline
4 & 2 2 2 1 1 & 3 \\ \hline
5 & 2 2 2 2 2 & 6 \\ \hline
6 & 3 3 3 2 2 & 4 \\ \hline
3 & 1 1 1 1 1 1 & 4 \\ \hline
\end{tabular}
\end{center}

In \cite{co} R. Coleman formulated the following conjecture:

\begin{conjecture}
Fix an integer $g \geq 4$. Then there are only finitely many complex
algebraic curves $C$ of genus $g$ such that ${\rm Jac}(C)$ is of
$CM$ type.
\end{conjecture}

\begin{remark}
J. de Jong and R. Noot \cite{DeJN} resp., E. Viehweg and K. Zuo \cite{VZ5} have
given counterexamples of families with infinitely many $CM$ fibers for $g = 4, 6$. In our lists
here we have counterexamples for $g = 5,7$.

J. de Jong and R. Noot resp., E. Viehweg and K. Zuo needed to show first the existence
of one fiber with $CM$ for the proofs that their examples of families have infinitely many $CM$
fibers. In the proof of Theorem $\ref{jnvz}$, which implies that the examples of this section
have dense set of complex multiplication fibers, we did not need to show the existence of one $CM$ fiber first.

But by the fact that our examples $\sC\to \sM_n$ with a dense set of $CM$ fibers satisfy that
$n+1$ branch points have the same branch index, Theorem $\ref{geilomat}$ yields the $CM$-type of
one $CM$ fiber and hencefore by Lemma $\ref{tija}$, the $CM$-type of a dense set of $CM$ fibers.
\end{remark}

\section{The derived variations of Hodge structures}

In this section we determine the families of cyclic covers with a derived pure $(1,n)-VHS$ and
verify that the list of examples in the preceding section is complete.

\begin{remark}
Assume that the family $\sC$ of degree $dm$ covers has a derived pure $(1,n)-VHS$ induced by
$\sC_d$. Let
$$d = p_1^{n_1} \cdot \ldots \cdot p_t^{n_t}$$
be the decomposition of $d$ into its prime factors. Then there exists a family of covers of
degree $p_1m$ with a derived pure $(1,n)-VHS$. Hence there are two cases to consider first:
$d$ is a prime number and divides $m$, or $d$ is a prime number and does not divide $m$.
\end{remark}

\begin{lemma}
Let $p$ be a prime number. Assume that $d$ is a prime number such that
$\gcd(d, p) = 1$. Then a family $\sC$ of covers of degree $p\cdot d$ with a derived pure
$(1,n)-VHS$ induced by $\sC_d$ can not exist, if all Dehn twists yield semisimple matrices with
respect to the monodromy representation of $\sL_d$.
\end{lemma}
\begin{proof}
Since $\sC_p$ must have a trivial $VHS$, there exists a $d_2$ such that $d$ divides $d_2$.
Moreover there is a $d_1$ such that $d$ does not divide $d_1$. Hence $\gcd(d, d_1 + d_2) = 1$. By
the fact that $\sC_d$ has the property that its local monodromy data satisfy
$\mu_1 + \mu_2 \neq 1$, one concludes that $\gcd(p, d_1 + d_2) = 1$, too. Hence $[d_1 + d_2]_{dp}$
is a unit in $\Z/(dp)$. Thus there exists a $d_0 \in (\Z/(dp))^*$ such that
$d_0[d_1 + d_2]_{dp} = 1$. One obtains that the sum of the integers of $\{1, \ldots, p-1\}$
representing $[d_0d_1]_{dp}$ and $[d_0d_2]_{dp}$ is given by $dm+1$. By Proposition $\ref{1.27}$,
one concludes that $\sL_{d_0}$ is not of type $(0,n+1)$. Moreover the fact that the local
monodromy data of $\sL_{d_0}$ satisfy
$$\mu_1 + \mu_2 = \frac{dp+1}{dp}, \  \ \mu_3 \leq \frac{dp-1}{dp}, \  \ \mu_{4} + \ldots +
\mu_{n+3} <n$$
tells us that
$$\mu_{1} + \ldots + \mu_{n+3} <n+2.$$
Hence one concludes by Proposition $\ref{1.27}$ that $\sL_{d_0}$ is not of type $(n+1,0)$, too.
\end{proof}

\begin{lemma}
Let $m = 2^tp$, where $p \neq 2$ is a prime number and $t \geq 1$. Assume that $d$ is a prime
number such that $\gcd(d, m) = 1$. Then a family $\sC$ of degree $m\cdot d$ covers with a derived
pure $(1,n)-VHS$, which is induced by $\sC_{d}$, can not exist.
\end{lemma}
\begin{proof}
Since $\sC_2$ must have a trivial $VHS$, one has $d_1 = \ldots = d_n = 2^{t-1}dp$. By the fact
that $\sC_p$ must
have a trivial $VHS$, we obtain that $2^td$ divides $n$ different branch indeces. Since there must
be at least two different branch indeces, which are not divided by $d$, $d_{n+2}$ and $d_{n+3}$
are not divided by $d$. By the fact that $d_1 = \ldots = d_n = 2^{t-1}dp$ is not divided by $2^td$,
one must have $n = 1$ and that $2^td$ divides $d_2$. Moreover the facts that
$$d_1 +\ldots + d_4 \in (m) = (2^tpd)\   \ \mbox{and}  \   \ 2|d_2$$
imply without loss of generality that $2$ does not divide $d_3$. We have two cases: Either $p|d_3$
or this does not hold true. In the first case one has that 2, $p$ and $d$ do not
divide $d_2+d_3$. Hence $d_2+d_3$ is a unit, and again we use the argument that there is a
$d_0 \in (\Z/(dm))^*$ such that $[d_0d_2 + d_0d_3] = 1$.

In the other case $d_3$ yields a unit of $\Z/(dm)$. Hence we have without loss of generality
$d_3 = 1$. Thus $g:=\gcd(dm,d_1+d_3) \in \{1,2\}$. If $g = 1$, we are done again. Otherwise we
must have $t = 1$, if $g=2$. Hence
$$[(pd-2)(d_1+d_3)]_{dm} = pd + pd-2 = dm - 2$$
such that $\sL_{pd-2}$ is neither of type $(0,n+1)$ nor of type $(n+1,0)$, since the fact that
$2^td$ divides $d_2$ implies that $[(pd-2)d_2]_{dm} \neq [1]_{dm}$.
\end{proof}

\begin{lemma}
Let $p$ be a prime number and $m = p^t$ with $t \geq 2$. Assume that $d$ is a
prime number such that $\gcd(d, p) = 1$. Then there can not be a family $\sC$ of degree $m\cdot d$
covers with a derived pure $(1,n)-VHS$, which is induced by $\sC_{d}$.
\end{lemma}
\begin{proof}
Since $\sC_p$ must have a trivial $VHS$, one concludes without loss of generality that
$dp^{t-1}$ divides $d_1, \ldots, d_n$. Since $d$
and $p$ divide 
$$dp^t = d_1 + \ldots +d_{n+3},$$
too, $p$ resp., $d$ does not divide at least two different elements of
$\{d_{n+1}, d_{n+2}, d_{n+3}\}$. Hence there is an element of $\{d_{n+1}, d_{n+2}, d_{n+3}\}$,
which is not divided by both $d$ and $p$. Without loss of generality $d_{n+1}$ is a unit in
$\Z/(2^td)$. Hence one has without loss of generality $[d_1 + d_{n+1}]_{dm} = [1]_{dm}$.
\end{proof}

There are only few remaining examples, which do not satisfy the assumptions of the preceding
lemmas. One of these examples is considered in the following lemma:

\begin{lemma}
Let $d \neq 3$ be a prime number. There can not be a family of covers of degree $3d$ with a
derived pure $(1,2)-VHS$ induced by $\sC_d$ given by the local monodromy data
$$\mu_1 = \ldots = \mu_4 = \frac{1}{3}, \   \ \mu_5 = \frac{2}{3}.$$
\end{lemma}
\begin{proof}
Let $\gcd(d, 3) = 1$ and $\sC$ be a family of degree $3d$ with a derived pure $(1,2)-VHS$.
Since $\sC_3$ should have a trivial $VHS$, one has with a new enumeration $d|d_1$ and $d|d_2$.
Moreover one has without
loss of generality that $d_3$ and $d_4$ are not divided by $d$. Hence $d$ divides neither
$d_1 + d_3$ nor $d_2 + d_4$. Moreover the local monodromy data of
$\sC_d$ tell us that 3 does not divide $d_1 + d_3$ or $d_2 + d_4$.
Hence without loss
of generality $d_1 + d_3$ is a unit in $\Z/(3d)$ such there is a $d_0\in (\Z/(3d))^*$ with the
property that $[d_0d_1 + d_0d_3]_{3d} = 1$, which implies that $\LL_{d_0}$ is of type $(1,2)$ or
of type $(2,1)$.
\end{proof}

The reader checks easily that all examples of families with a primitive pure $(1,n)-VHS$ satisfy
with two exceptions the assumptions of one of the preceding lemmas. These two exceptions yield
examples of families with a derived pure $(1,n)-VHS$ as we will see now.

\begin{pkt}
Now we consider the case of the elliptic curves. Let $d$ be a prime number with $\gcd(d, 2) =1$ and
$\sC$ be a family of degree $d\cdot 2$ covers with a derived pure $(1,1)-VHS$ induced by $\sC_d$.
Thus $d_1, \ldots, d_4$ must be odd. Without loss of generality we have $d_4 = d$, since $\sC_2$
must have a trivial $VHS$. Since $d_3 = d$
would imply that $\LL_1$ is of type $(1,1)$, one has that $d_1, d_2, d_3 \in (\Z/(2d))^*$.
We have two cases. Either $d_1 = d_2$ or this does not hold true. In the first case we put
$d_1 = d_2 = d-2$. One has
$$2d < d_1 + d_2 + d_4 < 2\cdot 2d$$
such that $\LL_1$ is of type $(1,1)$, if $4 <d$. Thus one can have $d = 3$. In this case one
has a family of degree 6 covers, where $d_4 = 3$. Hence one must have
$$\mu_1 = \mu_2 = \mu_3 = \frac{1}{6}, \  \ \mu_4 = \frac{1}{2}.$$

In the second case, one puts $d_3 = d-2$. This
implies that $d_3 + d_4 = 2d-2$. Since $d_1 \neq d_2$, one can not have $d_1 = d_2 = 1$ such that
$\LL_1$ is of type $(1,1)$ in this case.
\end{pkt}

\begin{pkt}
Now we consider the case number 2 in the list of examples with a primitive pure $(1,1)-VHS$. Let
$d$ be a prime number with $\gcd(d, 3) =1$ and $\sC$ be a family of degree $d\cdot 3$ covers with a
derived pure $(1,1)-VHS$ induced by $\sC_d$. Assume without loss of generality that $d$ divides
$d_1$ and $d_1 + \ldots + d_4 = 3d$. We have 2 cases: Either $d$ divides $d_2$, $d_3$ or $d_4$, or
$d$ does not divide $d_2$, $d_3$ and $d_4$. In the first case one has without loss of generality
that $d$ divides $d_2$. Since $d$ divides $d_1$ and $d_1 + \ldots + d_4 = 3d$, one concludes that
$d_1 = d_2 = d$. This implies that $\sL_2$ is of type $(1,1)$ such that $d=2$. In addition one
concludes that
$$d_1 = d_2 = 2, \  \ d_3 = d_4 = 1.$$

In the second case one has  that 3 does
not divide $d(d_1+d_i)$ for exactly one $k \in \{2,3,4\}$, which
follows by the branch indeces in the case number 2. Hence 3 does not
divide $d_1+d_k$. Moreover $d$ does not divide $d_1+d_k$, too. Hence $d_1
+ d_k \in (\Z/(3d))^*$.
\end{pkt}

\begin{proposition}
Let $d$ be a prime number, which divides $m$ and $\sC$ be a family of covers of degree $md$.
Assume a Dehn twist yields a semisimple matrix of maximal order $m$ with respect to the monodromy
representation of $\sL_{d}$. Then $\sC$ can not have a derived pure $(1,n)-VHS$ induced by
$\sC_{d}$.
\end{proposition}
\begin{proof}
Assume without loss of generality that $\rho_d(T_{1,2})$ yields a matrix of order $m$. In this
case $[d(d_1+d_2)]\in \Z/(dm)$ has the order $m$. Hence the fact that $d$ divides $m$ implies that
$d_1+d_2 \in (\Z/(dm))^*$.
\end{proof}

\begin{remark}
One can easily check that the assumptions of the preceding proposition are satisfied for all
examples of families with a primitive pure
$(1,n)-VHS$ except of the case of elliptic curves. In this case we have in fact an example of a
family of degree 4 covers with a derived pure $(1,1)-VHS$. Without loss of generality we have 
$$d_1 + \ldots + d_4 = 4.$$
Hence the only possibility is given by
$$d_1 = \ldots = d_4 = 1.$$
\end{remark}

\begin{pkt}
In the case of the elliptic curves we have families of degree 6 and degree 4 covers with a
derived pure $(1,1)-VHS$. Hence one must check that there is not a family of degree 8, 12
or 18 covers with derived pure $(1,1)-VHS$ in this case.

First we check that there is not a family $\sC$ of degree $8$ covers with a derived pure
$(1,1)-VHS$. Otherwise one has such a family $\sC$ of degree 8 covers such that $\sC_2$ is the
family of degree 4 covers with a derived pure
$(1,1)-VHS$. This implies that each $d_k$ satisfies $[d_k]_4 = [1]_4$ or each $d_k$ satisfies
$[d_k]_4 = [3]_4$. Moreover one has without loss of generality that $d_1 + \ldots +d_4 = 8$. Hence
it is not possible that each $d_k$ satisfies $[d_k]_4 = [3]_4$. Thus the only possibility is (up
to the numbering) given by
$$d_1 = d_2 = d_3 = 1, \  \ d_4 = 5.$$
But in this case $\sL_3$ is of type $(1,1)$. Thus there can not exist a family of degree $8$
covers with a derived pure $(1,1)-VHS$.

There can not be a family of degree 12 covers with a derived pure $(1,1)-VHS$ induced by $\sC_6$.
Otherwise one has that $\sC_3$ the example of degree 4 covers with a derived pure $(1,1)-VHS$. Thus
one concludes that
$$[d_1]_4 = \ldots = [d_4]_4 = [1]_4 \  \ \mbox{or} \  \  [d_1]_4 = \ldots = [d_4]_4 = [3]_4.$$
Since one has without loss of generality that $d_1 + \ldots + d_4 = 12$, the only possibilities
are given by
$$d_1 = d_2 = 5, \  \ d_3 = d_4 = 1 \  \ \mbox{and} \  \  d_1 = 9, \   \ d_2 = d_3 = d_4 = 1.$$
In the first case $\sL_2$ is of type $(1,1)$ and in the second case $\sL_5$ is of type $(1,1)$.

There can not be a family of degree 18 covers with a derived $(1,1)-VHS$ induced by $\sC_9$.
Otherwise one has that $\sC_3$ is the example of degree 6 covers with a derived pure $(1,1)-VHS$
induced by the elliptic curves. Thus one concludes that
$$[d_1]_6 = \ldots = [d_3]_6 = [1]_6 \ \ \mbox{and} \  \ [d_4]_6 = [3]_6$$
or
$$[d_1]_6 = \ldots = [d_3]_6 = [5]_6 \ \ \mbox{and} \  \ [d_4]_6 = [3]_6.$$
Since one has without loss of generality that $d_1 + \ldots + d_4 = 18$, the only possibilities
are given by:
$$d_1 = 13, \   \ d_2 = d_3 = 1, \ \ d_4 = 3$$
$$d_1 = d_2 = 7, \  \ d_3 = 1, \  \ d_4 = 3$$
$$d_1 = 7, \   \ d_2 = d_3 = 1, \  \ d_4 = 9$$
$$d_1 = d_2 = d_3 = 1, \  \ d_4 = 15$$
$$d_1 = d_2 = d_3 = 5, \  \ d_4 = 3.$$
One has that $\sL_5$ is of type $(1,1)$ in case 1, $\sL_2$ is of type $(1,1)$ in case 2,
$\sL_5$ is of type $(1,1)$ in case 3, $\sL_7$ is of type $(1,1)$ in case 4 and
$\sL_2$ is of type $(1,1)$ in case 5.
\end{pkt}

\begin{pkt}
It remains to show that there can not exist a degree $12$ cover with a derived $(1,1)-VHS$ induced
by the degree 3 cover given by
$$d_1 = d_2 = 1, \  \ d_3 = d_4 = 2.$$
Otherwise one has such a family $\sC$ of degree 12 covers such that $\sC_2$ is the family of
degree 6 covers with a derived pure $(1,1)-VHS$ by the degree 3 example above. Thus one
concludes that

$$[d_1]_6 = [d_2]_6 = [2]_6 \ \ \mbox{and} \  \ [d_3]_6 = [d_4]_6 = [1]_6$$
or
$$[d_1]_6 = [d_2]_6 = [4]_6 \ \ \mbox{and} \  \ [d_3]_6 = [d_4]_6 = [5]_6$$
Since one has without loss of generality that $d_1 + \ldots + d_4 = 12$, the only possibilities
are given by

$$d_1 = 8, \  \ d_2 = 2, \  \ d_3 = d_4 = 1 \  \ \mbox{and} \  \  d_1 = d_2 = 2, \   \ d_3 = 7,
\  \ d_4 = 1.$$
One has that $\sL_5$ is of type $(1,1)$ in the first case and one has that
$\sL_3$ is of type $(1,1)$ in the second case.
\end{pkt}

\chapter{The construction of Calabi-Yau manifolds with complex multiplication}

\section{The basic construction and complex multiplication}
Now we have finished our considerations on Hodge structures of cyclic covers of $\bP^1$. We start
with the second part, which is devoted to the construction of families of Calabi-Yau
manifolds with dense set of complex multiplication fibers.

In the works of C. Borcea \cite{Bc}, \cite{Bc2}, of E. Viehweg and K. Zuo \cite{VZ5} and of C.
Voisin \cite{Voi2} the methods to obtain higher dimensional Calabi-Yau manifolds contain one
common basic construction. In this section we describe this construction and explain how it yields
complex multiplication. For this construction we need Kummer coverings. Let $A-B$ be a principal
divisor with $(f) = A-B$ for some $f \in \C(X)$. The Kummer
covering given by $\C(X)(\sqrt[m]{\frac{A}{B}})$ is nothing but the normalization of $X$ in
$\C(X)(\sqrt[m]{f})$.

Let $V_1$ and $V_2$ be irreducible complex algebraic manifolds and $\sA$ resp., $\sB$ be a bundle
of irreducible algebraic manifolds with universal fiber $A$ resp.,
$B$ over $V_1$ resp., $V_2$. Moreover let $\sZ$ resp., $\Sigma$ be a cyclic
Galois cover of $\sA$ resp., a cyclic Galois cover of $\sB$ of
degree $m$ over $V_1$ resp., $V_2$ ramified over a smooth divisor. We assume that the irreducible
components of these ramification divisors intersect each fiber of $\sZ$ resp., $\Sigma$
transversally in smooth subvarieties of codimension 1. Thus we assume that $\sZ$ and
$\Sigma$ are given by Kummer coverings of the kind
$$\C(W) = \C(X)( \sqrt[m]{\frac{D_1+ \ldots + D_k}{D_0^m}} ),$$
where $D_1, \ldots, D_k$ are (reduced) smooth prime divisors, which do not intersect each other.

\begin{example} \label{datda}
By a cyclic degree 2 cover $S \to R$ of surfaces (or in general algebraic varieties),
one has an involution on $S$. Let us assume that the surface $S$ is a smooth $K3$
surface. Moreover assume that there exists an involution $\iota$ on $S$, which acts
via pull-back by $-1$ on $\Gamma(\omega_S)$. It
has the property that it fixes at most a divisor $D$, whose support consists of smooth curves,
which do not intersect each other (see \cite{Voi2}, 1.1.). Moreover by \cite{Voi2}, 1.1., to
give an involution $\iota$ on $S$, which acts by $-1$ on $\Gamma(\omega_S)$, is the same as to
give a cyclic degree 2 cover $S \to R$ of smooth surfaces. In this case $R$ is rational, if and
only if $D \neq 0$.
\end{example}

We consider the following commutative diagram, which yields the basic
construction:
\begin{equation} \label{diag}
\xymatrix{
  {\sZ\times \Sigma} \ar[rr]^{\gamma }  &  & {\sY'} \ar[rr]^{\alpha }  &  & {\sA \times \sB}
\ar[rr] &  & {V_1 \times V_2} \\
  {\widetilde{\sZ\times \Sigma}} \ar[rr]^{\tilde \gamma } \ar[u]^{\beta} &  & {\tilde  \sY}
\ar[rr]^{\tilde \alpha } \ar[u]^{\delta} &  & {\hat \Pi} \ar[u]^{\zeta}
}
\end{equation}
First we explain the upper line of this diagram: The cyclic covers $\sZ$ and $\Sigma$ can locally
be described by equations of the type
$$y^m = \prod_{i = 1, \ldots, k} f_i(x_1, \ldots, x_j)$$ over any open affine set $\A$ of $\sA$
resp., $\sB$, where $f_i$ is the (reduced) equation of $D_i$
in $\A$. The Galois transformations are given by
$$(y, x_1, \ldots, x_j) \stackrel{g_k}{\to} (e^{2\pi \sqrt{-1} \frac{k}{m}}y, x_1, \ldots, x_j)$$
for some $k \in \Z/(m)$. Hence we have a natural identification between $\Z/(m)$ and the Galois
groups given by $[k]_m \to g_k$. By the describtion of the covers above in terms of Kummer
coverings, this identification is independent of the chosen open affine subset. Now $\gamma$ is
the quotient by
$$G := \langle(1,1)\rangle  \  \ \subset  \  \ G' := \Gal(\sZ;\sA) \times \Gal(\Sigma;\sB),$$
and $\alpha$ is the quotient by $G'/G$. The morphism $\zeta$ is given by the blowing up of the
fiber product
of the supports of the branch divisors of $\sZ$ and $\Sigma$. Moreover $\delta$ is the blowing up
along the singular points of $\sY'$, which is given by the intersection locus of the ramification
divisors, and $\beta$ is the blowing up with respect to the corresponding
inverse image ideal sheaf. Hence $\tilde  \alpha$ and $\tilde \gamma$ are the unique cyclic covers
obtained by the universal property of the blowing up (compare to \cite{hart}, {\bf II}. Corollary
$7.15$). By the construction of $\alpha$, one can easily check that $\tilde \alpha$ is not ramified
over the exceptional divisor. Hence the branch locus of
$\tilde \alpha$ is smooth. This implies that $\tilde \sY$ is smooth, too. The ramification locus
of $\tilde \gamma$ is given by the smooth exceptional divisor of $\beta$, since $G$ leaves
the generators of the inverse image ideal sheaf invariant as one can see by the following remark:

\begin{remark} \label{buahaha}
Now we describe $\widetilde{\sZ\times \Sigma}$. A neighborhood of the
preimage point $p \in \sZ\times \Sigma$ of a singular point can be identified with an open
neighborhood of $0 \in \C^2 \times \B$, where $\B$ is a ball of suitable
dimension and the Galois group acts via $(x_1,x_2) \to
(e^{\frac{2\pi i}{m}} x_1, e^{\frac{2\pi i}{m}} x_2)$ with respect to the coordinates on $\C^2$.
Due to \cite{Bart}, {\bf III}. Proposition 5.3, each singular point of $\sY'$ has an analytic
neighborhood isomorphic to $V(x^m = y^{m-1}z) \times \B$. Hence locally we have the product of a
cover of surfaces with $\B$. One should have $\B$ in mind. But for the description of
$\widetilde{\sZ\times \Sigma}$, it is sufficient
to consider only covers of surfaces.  The inverse image ideal sheaf with respect to this cover is
generated by $\{x_1^{m-i}x_2^i : i = 0, 1, \ldots, m\}$. By the Veronese embedding for relative
projective manifolds, one can easily identify the blowing up with respect to this ideal with the
blowing up with respect to the ideal generated by $\{x_1, x_2\}$. But this is the blowing up of
the origin resp., the preimage point of the singular point. Hence in the general situation
$\widetilde{\sZ\times \Sigma}$ is given by the blowing up of the reduced preimage
$\gamma^{-1}(S)$, where $S$ is the singular locus of $\sY'$.
\end{remark}

Now we have described the basic construction. Next we see that this construction yields
complex multiplication. We use following fact:

\begin{proposition} \label{tensorp}
For all $\tilde a \in \sA$, and $\tilde b \in \sB$, we have the following tensor
product of rational Hodge structures on the fibers:
$$H^n(\sZ_{\tilde a} \times \Sigma_{\tilde b}, \Q) = \bigoplus\limits_{a+b = n}
H^a(\sZ_{\tilde a},\Q) \otimes H^b(\Sigma_{\tilde b}, \Q)$$
such that
$$H^{r,s}(\sZ_{\tilde a} \times \Sigma_{\tilde b}) = \bigoplus\limits_{p+p' = r, q + q' = s}
H^{p,q}(\sZ_{\tilde a}) \otimes H^{p',q'}(\Sigma_{\tilde b})$$
\end{proposition} \begin{proof}
(follows from \cite{Voi}, Theorem 11.38)
\end{proof}

We want to construct higher dimensional varieties with complex multiplication. The first main tool
is:

\begin{proposition}
Let $h_1$ and $h_2$ be rational polarized Hodge structures. Then
$h_3 = h_1 \otimes h_2$ is of $CM$ type, if and only if $h_1$ and $h_2$ are of $CM$ type.
\end{proposition}
\begin{proof}
(see \cite{Bc}, Proposition 1.2)
\end{proof}

By the fact that $\sY'$ is not smooth, but the blowing up $\tilde \sY$ is smooth, $\tilde \sY$
will be our candidate for a family of Calabi-Yau manifolds with dense set of complex multiplication fibers. Hence we must
consider the behavior of the Hodge structures under blowing up:

\begin{lemma} \label{blowlow}
Let $X$ be an algebraic manifold of dimension $n$ and $\tilde X$ be
the blowing up $X$ with respect to some submanifold $Z \supset X$ of
codimension $2$. Then $\Hg(H^k(\tilde X,\Z))$ is commutative, if and
only if $\Hg(H^k(X,\Z))$ and $\Hg(H^{k-2}(Z,\Z))$ are commutative,
too.
\end{lemma}
\begin{proof}
By \cite{Voi}, Theorem 7.31, we have an isomorphism
$$H^k(X,\Z) \oplus H^{k-2}(Z,\Z) \cong H^k(\tilde X, \Z)$$
of Hodge structures, where $H^{k-2}(Z,\Z)$ is shifted by $(1,1)$ in
bi-degree. Since
$$\Hg(H^k(\tilde X, \Z)) = \Hg(H^k(X,\Z) \oplus H^{k-2}(Z,\Z)) \subset
\Hg(H^k(X,\Z)) \times \Hg(H^{k-2}(Z,\Z))$$
such that the natural projections
$$\Hg(H^k(\tilde X, \Z)) \to \Hg(H^k(X,\Z)) \  \ \mbox{and} \ \  \Hg(H^k(\tilde X, \Z)) \to
\Hg(H^{k-2}(Z,\Z))$$
are surjective (compare to \cite{VZ5}, Lemma $8.1$), we obtain the result.
\end{proof}

\begin{corollary} \label{blowblow}
Let $X$ be a smooth surface and $\tilde X$ be the blowing up of some point $p \in X$. Then $X$ has
complex multiplication, if and only if $\tilde X$ has complex multiplication, too.
Moreover we obtain that
$$\Hg(H^2(\tilde X, \Z)) \cong \Hg(H^2(X,\Z)).$$
\end{corollary}

Now we want to consider the behavior of the fibers. Hence for
simplicity we assume now that $V_1 = V_2= {\rm Spec}(\C)$ in diagram
($\ref{diag}$). By the fact that $\tilde \sY$ has the Hodge structure given by the Hodge
sub-structure of $\widetilde{\sZ\times \Sigma}$ invariant under the Galois group, one concludes:

\begin{Theorem} \label{hiercm}
If for all $k$ the groups $\Hg(H^k(\sZ,\Q))$, $\Hg(H^k(\Sigma,\Q))$ and $\Hg(H^k(Z_i,\Q))$ are
commutative,\footnote{One needs in fact the condition that each $\Hg(H^k(Z_i,\Q))$ is commutative.
The argument is similar to the argument in the proof of Proposition $\ref{fixhodge}$.}
then $\Hg(H^k(\tilde \sY,\Q))$ is commutative for all $k$, too.
\end{Theorem}

\begin{remark}
At first sight the condition that for all $k$ the groups $\Hg(H^k(\sZ,\Q))$, $\Hg(H^k(\Sigma,\Q))$
and $\Hg(H^k(Z_i,\Q))$ have to be commutative may seem to be a little bit restrictive. But by the
Hodge diamond of a Calabi-Yau $n$-manifold with $n \leq 3$ or the Hodge diamond of a
Calabi-Yau $n$-manifold given by a projective hypersurface, one sees
that the condition that all its Hodge groups are commutative is
equivalent to the condition that it has complex multiplication. Moreover we will need this
condition for an inductive construction of families of Calabi-Yau manifolds with dense set of
complex multiplication fibers in arbitrary high dimension in the next section.
\end{remark}

\section{The Borcea-Voisin tower}

Recall that we want to construct families of Calabi-Yau manifolds with a dense set of $CM$ fibers.
Hence let us now define Calabi-Yau manifolds:

\begin{definition} \label{defCala}
A Calabi-Yau manifold $X$ of dimension $n$ is a compact K\"ahler
manifold of dimension $n$ such that $\Gamma(\Omega_X^i)= 0$ for all
$i = 1, \ldots, n-1$ and $\omega_X \cong \sO_X$.
\end{definition}

By the construction of the preceding section, which we will use, we need more and we get more than
only complex multiplication. Hence let
us define, which we will get:

\begin{definition}
A $CMCY$ family $\sX \to \sB$ of $n$-manifolds \index{$CMCY$
family of $n$-manifolds} is a (smooth) family of Calabi-Yau manifolds of dimension $n$, which has
a dense set of fibers $\sX_b$ satisfying the property that $\Hg(H^k(\sX_b,\Q))$ is commutative for
all $k$.
\end{definition}

In this section the degree $m$ of all cyclic covers, which will occur, is equal to 2. We apply the
construction of a Calabi-Yau manifold with an involution by two Calabi-Yau manifolds with
involutions by C. Borcea \cite{Bc2}. This yields an iterative construction of $CMCY$ families with
involutions in arbitrary high dimension by $CMCY$ families in lower dimension.\footnote{The
construction of C. Borcea is repeated in Proposition $\ref{trivcan}$. By C. Voisin \cite{Voi2},
the same construction was used to construct Calabi-Yau 3-manifolds by $K3$-surfaces with involutions and
elliptic curves. This is the reason that our construction here is called ``Borcea-Voisin tower''. Here this
construction is introduced as a systematic method to construct Calabi-Yau manifolds with complex multiplication
in an arbitrary dimension.}

\begin{construction} \label{bvtower}
Let $\sZ_1 \to \sM$ be a $CMCY$ family of $n$-manifolds covering the $A$
bundle $\sA$ with ramification locus $R_1$, which satisfies the assumptions for $\sZ$ in diagram
$\eqref{diag}$. Moreover let $\Sigma_i$ be a $CMCY$ family $\Sigma_i \to \sM^{(i)}$
of $n_i$-manifolds covering the $B_i$ bundle $\sB_i$ over $\sM^{(i)}$ with ramification locus
$R^{(i)}$, which satisfies the assumptions for $\Sigma$ in diagram $\eqref{diag}$, for all
$1 <i \in \N$. \index{Borcea-Voisin tower}

Let us assume that there is a dense subset of points $m^{(i)} \in \sM^{(i)}$ resp., $p \in \sM$,
which have the property that each $\Hg(H^k((\Sigma_i)_{m^{(i)}},\Q))$ and each
$\Hg(H^k(R_{m^{(i)}}^{(i)},\Q))$ resp., each $\Hg(H^k((\sZ_1)_p,\Q))$ and each
$\Hg(H^k((R_1)_p,\Q))$ is commutative.

We define an iterative tower of covers
$$\sZ_i \to V^{(i)} :=\sM \times \sM^{(2)} \times \ldots \times \sM^{(i)}$$
given by
$$\sZ_{i} = \tilde \sY_i,$$
where $\tilde \sY_{i}$ is obtained from $\tilde \sY$ in the diagram $\eqref{diag}$ with
$V_1 =V^{(i-1)}$, $V_2 = \sM^{(i)}$, $\Sigma = \Sigma_i$ and $\sZ = \sZ_{i-1}$ for all $i \in \N$.
Let us call such a construction Borcea-Voisin tower.
\end{construction}

The assumption that we have ramification in codimension 1 on the fibers of a family of Calabi-Yau
manifolds leads to the important property that the corresponding involutions act by $-1$ on
the global sections of their canonical sheaves, as we see by the following Lemma:

\begin{lemma} \label{ampel}
Let $C$ be a Calabi-Yau manifold and $\iota$ be an involution on it. Assume that the points fixed
by $\iota$ are given by a non-trivial (reduced) effective smooth divisor $D$. Then $\iota$ acts by
$-1$ on $H^0(C, \omega_C)$.
\end{lemma}
\begin{proof}
By our assumptions, the induced natural cyclic cover $\gamma:C \to C/\iota$ is ramified over a
smooth non-trivial divisor $D$ such that $C/\iota$ is smooth. Hence one has a cyclic cover of
manifolds and one can apply the Hurwitz formula (compare \cite{Bart}, {\bf I}. 16). Since $C$ has a
trivial canonical divisor, the Hurwitz formula implies that
$\sO_{C}(-D) \cong \gamma^*(\omega_{C/\iota})$. This implies that
$\omega_{C/\iota}$ does not contain any global section. Since $\omega_{C/\iota}$ yields the
eigenspace for the character 1 of $\gamma_*(\omega_C)$ (see \cite{EV1}, $\S3$), the character
of the action of $\iota$ on $H^0(C, \omega_{C})$ is not given by 1. Thus it is given by $-1$.
\end{proof}

\begin{proposition} \label{trivcan}
Assume that $\gamma_1: C_1 \to M_1$ and $\gamma_2: C_2 \to M_2$ are
cyclic covers of degree 2 with the involutions $\iota_1$ and $\iota_2$ and ramification
divisors $D_1 \subset C_1$ and $D_2 \subset C_2$, which consist of disjoint smooth
hypersurfaces. Moreover assume that $C_1$ and $C_2$ are Calabi-Yau
manifolds of dimension $n_1$ and $n_2$. Let $\widetilde
{C_1 \times C_2}$ denote the blowing up of $C_1 \times C_2$ with
respect to $D_1 \times D_2$. Then by the involution on $\widetilde
{C_1 \times C_2}$ given by $(\iota_1, \iota_2)$, one obtains a
cyclic cover $\gamma: \widetilde {C_1 \times C_2} \to C$ such that
$C$ is a Calabi-Yau manifold.
\end{proposition}
\begin{proof}
We assume that each $C_i$ is a Calabi-Yau manifold such that $h^{t,0}(C_i) = 0$ for all
$t = 1, \ldots, n_i-1$. By the assumption that one has the ramification divisors $D_1$ and
$D_2$ and Lemma $\ref{ampel}$, the corresponding involution of each $\gamma_i$ acts by $-1$ on
each $\omega_{C_i}$. Thus one concludes that $h^{j,0}(C) = 0$ for all $j = 1, \ldots, (n_1+n_2)-1$.

The canonical divisor $K_{\widetilde{C_1 \times C_2}}$ of $\widetilde{C_1 \times C_2}$ is given by
the exceptional divisor $E$ of
the blowing up $\widetilde{C_1 \times C_2}\to C_1\times C_2$. Moreover the ramification divisor
$R$ of $\gamma$ coincides with $E$. Hence by the Hurwitz formula (\cite{Bart}, ${\bf I}.16$), we
have 
$$\sO_{\widetilde{C_1 \times C_2}}(R)\cong
\sO_{\widetilde{C_1 \times C_2}}(K_{\widetilde{C_1 \times C_2}}) =
\omega_{\widetilde{C_1 \times C_2}} \cong \gamma^*(\omega_C) \otimes
\sO_{\widetilde{C_1 \times C_2}}(R).$$
Thus one concludes that $\gamma^*(\omega_C) \cong \sO$.

Since $\iota_1$ and $\iota_2$ act by the character $-1$, the
involution $(\iota_1, \iota_2)$ on $\widetilde{C_1 \times C_2}$ leaves the global sections of
$\omega_{\widetilde{C_1 \times C_2}}$ invariant. Now recall that
$\gamma_*(\omega_{\widetilde{C_1 \times C_2}})$ consists of a direct sum of invertible sheaves,
which are the eigenspaces with respect to the characters of the Galois group action. By
\cite{EV1}, $\S3$, the eigenspace for the character $1$ is given by $\omega_C$. Thus $\omega_C$
has a non-trivial global section. Hence the canonical divisor of $C$ satisfies (up to linear
equivalence) $K_C \geq 0$. Thus by the fact that $\gamma^*(\omega_C) \cong \sO$, we have the
desired result $K_C \sim 0$.
\end{proof}

Altogether one has the following result:

\begin{Theorem} \label{JMCA}
Each family $\sZ_i \to \sM \times \sM^{(2)} \times \ldots \times \sM^{(i)}$ obtained
by the Borcea-Voisin tower is a $CMCY$ family of $n+n_2 + \ldots + n_{i}$-manifolds.
\end{Theorem}
\begin{proof}
The statement that each $(\sZ_{i})_p$ is a Calabi-Yau manifold
follows fiberwise by induction. By the assumptions, we have the result for $n = 1$.
First by induction, one can show that the ramification loci are given by smooth divisors. By using
this fact and the induction hypothesis, one can apply Lemma $\ref{ampel}$ such that each involution
acts by the character $-1$ on each $\Gamma(\omega)$. Hence the assumptions of Proposition
$\ref{trivcan}$ are satisfied, which provides the induction step.

Next we want to show the statement about the commutativity of all Hodge groups over a dense
subset of the basis. Due to the situation described in
diagram $(\ref{diag})$ the connected components of
the ramification locus $(R_{i+1})_{p \times m^{(i+1)} }$ of $(\sZ_{i+1} )_{p \times m^{(i+1)}}$
over $p \times m^{(i+1)} \in V^{(i)} \times \sM^{(i+1)}$ are given by the connected components of
$(\sZ_i)_p \times R^{(i+1)}_{m^{(i+1)}}$ and by the connected components of
$(R_i)_{p} \times (\Sigma_{i+1})_j$, where $(R_i)_{p}$ is the ramification locus of
$(\sZ_{i})_{p}$.

Hence it is sufficient to use an inductive argument and to show the following Claim:
\end{proof}

\begin{claim}
Assume that for all $k$ the Hodge group $\Hg(H^k((\sZ_i)_p, \Z))$ is commutative and each
connected component $Z$ of the ramification locus $(R_i)_{p}$ satisfies that each
$\Hg(H^k(Z, \Z))$ is commutative. In addition we assume that for all $k$ the Hodge group
$\Hg(H^k(\Sigma_{i+1})_{m^{(i+1)}},\Z))$ is commutative and each connected component $Z_{i+1}$ of
$R^{(i+1)}_{m^{(i+1)}}$ satisfies that each $\Hg(H^k(Z_{i+1}, \Z))$ is commutative. Then for all
$k$ each connected component $\tilde Z$ of $(R_{i+1})_{p \times m^{(i+1)}}$ satisfies that each
$\Hg(H^k(\tilde Z, \Z))$ is commutative and for all $k$
$\Hg(H^k((\sZ_{i+1})_{p\times m^{(i+1)}}, \Z))$ is commutative.
\end{claim}
\begin{proof}
By the assumptions of this claim and the description of $R_{i+1}$
above, one obtains obviously that the connected components $\tilde
Z$ of $(R_{i+1})_{p \times m^{(i+1)}}$ satisfy that each $\Hg(H^k(\tilde
Z, \Z))$ is commutative. Then one must simply use Theorem $\ref{hiercm}$ and one obtains that
each $\Hg(H^k(\sZ_{i+1})_{p \times m^{(i+1)}}, \Z))$ is commutative, too.
\end{proof}

\section{The Viehweg-Zuo tower}
By the Borcea-Voisin tower, one can construct $CMCY$ families of manifolds in arbitrary high
dimension. But one needs $CMCY$ families of manifolds (in low dimension) with a suitable
involution, which can be used to be $\sZ_1$ or some $\Sigma_i$. One way to obtain some suitable
$CMCY$ families of $n$-manifolds (in low dimension) is given by the Viehweg-Zuo tower, which we
introduce now.

E. Viehweg and K. Zuo \cite{VZ5} have constructed a tower of projective algebraic manifolds
starting with a family $\sF_1$ of cyclic covers of $\bP^1$ given by \index{Viehweg-Zuo tower}
$$\bP^2 \supset  V( y_1^5 + x_1(x_1-x_0) (x_1-\alpha  x_0)(x_1-\beta   x_0) x_0)
\to (\alpha ,\beta)\in \sM_2,$$
which has a dense set of $CM$ fibers. This is one example of a family of cyclic covers, which has
a primitive pure $(1,2)-VHS$ as one can easily verify by using Proposition $\ref{1.27}$. Since
each of these covers given by the fibers of the family can be embedded into $\bP^2$, the fibers of
$\sF_1$ are the branch loci of the fibers of a family $\sF_2$ of cyclic covers onto $\bP^2$ of
degree 5. Moreover the fibers of $\sF_2$, which can be embedded into $\bP^3$, are the branch loci
of the fibers of a family $\sF_3$ of cyclic covers onto $\bP^3$, which can be embedded into
$\bP^4$. The family $\sF_3$ is given by
$$\bP^4 \supset V(y_3^5+y_2^5 + y_1^5 + x_1(x_1-x_0) (x_1-\alpha  x_0)(x_1-\beta   x_0) x_0)
\to (\alpha ,\beta)\in \sM_2.$$
Thus the fibers of $\sF_3$ are Calabi-Yau 3-manifolds. By an
inductive argument, this latter family has a dense set of $CM$ points on the basis given by the
dense set of the $CM$ points of the family of curves we have started with (see \cite {VZ5}). Since
only the Hodge group of the Hodge
structure on $H^3(X,\Q)$ of a projective hypersurface $X \subset \bP^4$ can be non-trivial, the
family $\sF_3$ is a $CMCY$ family of 3-manifolds.

\begin{example} \label{3cm3}
By Theorem $\ref{geilomat}$, the fibers of $\sF_1$ isomorphic to
$$V(y_1^5+x_1^5+x_0^5), \ \ V(y_1^5+x_1(x_1^4+x_0^4)), \  \ V(y_1^5+x_1(x_1^3+x_0^3)x_0) \subset \bP^2$$
have $CM$. Thus the fibers of $\sF_3$ isomorphic to
$$V( y_3^5+y_2^5 + y_1^5+x_1^5+x_0^5), \ \ V(y_3^5+y_2^5 +y_1^5+x_1(x_1^4+x_0^4)), \ \
V(y_3^5+y_2^5 +y_1^5+x_1(x_1^3+x_0^3)x_0) \subset  \bP^4$$
have $CM$, too.
\end{example}

\begin{example} \label{qui3f}
We consider the $CMCY$ family $\sF_3$
$$\bP^4 \supset V(y_3^5+y_2^5 + y_1^5 + x_1(x_1-x_0) (x_1-\alpha  x_0)(x_1-\beta   x_0) x_0)
\to (\alpha ,\beta)\in \sM_2$$
constructed by E. Viehweg and K. Zuo. On each fiber $(\sF_3)_p$ the involution $\iota$ given by
$$\iota (y_3:y_2:y_1:x_1:x_0) = (y_2:y_3:y_1:x_1:x_0)$$
leaves the smooth divisor $\sD_p$ given by the equation $y_3 = y_2$ invariant. Moreover one has
that $\sD_p \cong (\sF_2)_p$. Therefore there is a dense set of points $p \in \sM_2$, which
have the property that for all $k$ the Hodge groups of $H^k(\sD_p,\Q)$ and $H^k((\sF_3)_p,\Q)$ are
commutative. Hence one can use $\sF_3$ to be $\sZ_1$ or some $\Sigma_i$ for the construction of a
Borcea-Voisin tower of $CMCY$ families of $n$-manifolds.
\end{example}

\begin{example}
Let $\F_d$ denote the Fermat curve of degree $d>2$. The curve $\F_d$
has complex multiplication (see \cite{Kohr} and \cite{Rohr}). By the
construction of E. Viehweg and K. Zuo in \cite{VZ5}, one concludes
that the Calabi-Yau manifold $H_d$ given by
$$V(\sum\limits_{i = 0}^{d-1}x_i^d) \subset \bP^{d-1}$$
has complex multiplication. Since $H_d$ is a projective hypersurface, this implies that $H_d$ has
only commutative Hodge groups. We have the involution $\iota_a$ given by
$$(x_{d-1}: \ldots: x_2: x_1: x_0) \to (x_{d-1}: \ldots; x_2: x_0: x_1)$$
on $H_d$. If $d$ is even, one has the additional involution $\iota_b$ given by
$$(x_{d-1}: \ldots: x_1: x_0) \to (x_{d-1}: \ldots: x_1: -x_0).$$
The involution $\iota_a$ resp., $\iota_b$ (if it is given on $H_d$) fixes the points of a
smooth divisor on $H_d$, which is isomorphic to
$$V(\sum\limits_{i = 0}^{d-2}x_i^d) \subset \bP^{d-2}.$$
Therefore by the same arguments as in Example $\ref{qui3f}$, one can use $H_d$ to be $\sZ_1$ or
some $\Sigma_i$ with $\sM = {\rm Spec}(\C)$, resp., $\sM^{(i)} = {\rm Spec}(\C)$ for the
construction of a Borcea-Voisin tower of $CMCY$ families of $n$-manifolds.
\end{example}

We want to start the construction of a Viehweg-Zuo tower (of projective hypersurfaces as
in \cite{VZ5} or the construction of a modified version) with a family of cyclic covers $\sC
\to \sM_n$ of $\bP^1$ with a dense set of $CM$ fibers. For the smoothness of the higher
dimensional fibers
of the resulting families, we will need the assumption that the fibers of $\sC$ are given by
\begin{equation} \label{univfam}
V( y^m + x(x-1)(x-  a_1) \ldots (x- a_n)) \subset \A^2,
\end{equation}
where $m$ divides $n+3$ such that all branch indeces coincide.

By our preceding results, we have only the following examples of families of cyclic covers onto
$\bP^1$ with a dense set of $CM$ fibers, which satisfy this assumption:
\begin{center}
\begin{tabular}{|c|c|} \hline
degree $m$ & number of ramification points of the fibers\\ \hline \hline
2 & 4\\ \hline
2 & 6\\ \hline
3 & 6\\ \hline
4 & 4\\ \hline
5 & 5\\ \hline \end{tabular}
\end{center}

\begin{remark}
The case with $m = 2$ and 4 ramification points is the case of elliptic curves, which has been
considered by C. Borcea in \cite{Bc}. The case with $m = 5$ yields the example by E. Viehweg and
K. Zuo in \cite{VZ5}.
\end{remark}

The case with $m = 3$ is one of the examples of a family of covers onto $\bP^1$ with a dense set
of $CM$ fibers by J. de Jong and R. Noot \cite{DeJN}. We must a bit work to
give a suitable modified construction of a Viehweg-Zuo tower for this example. The next
chapter is devoted to this modified construction of a Viehweg-Zuo tower.

In the case of the family $\sC \to \sM_3$ of genus 2 curves the author does not see a possibility
for the construction of a Viehweg-Zuo tower.\footnote{One natural choice for an embedding of the
fibers of the family of genus 2 curves is given by the weighted projective space $\bP(3,1,1)$. But
the canonical divisor of the desingularization of $\bP(3,1,1)$ does not allow a natural construction
of a Viehweg-Zuo tower as in the case of $\bP(2,1,1)$, which we will see in the next chapter for
the degree 3 case.}

The case with $m = 4$ yields the Shimura- and Teichm\"uller curve of M. M\"oller \cite{Mll}, which
provides the example of the next section. 

\section{A new example}

Here we see that the Shimura- and Teichm\"uller curve of M. M\"oller yields an example of a
Viehweg-Zuo tower. Moreover we will see that the resulting $CMCY$ family of 2-manifolds is endowed
with some involutions, which make it suitable for the construction of a Borcea-Voisin tower. In
addition we try to decide, which involutions provide isomorphic quotients resp., isomorphic $CMCY$
families by the construction of a Borcea-Voisin tower.

\begin{proposition} \label{k3fl}
The family $\sC_2 \to \sM_1$ given by
$$\bP^3 \supset V(y_2^4 + y_1^4 + x_1(x_1-x_0)(x_1-\lambda x_0)x_0)
\to \lambda \in \sM_1$$
is a $CMCY$ family of 2-manifolds.
\end{proposition}
\begin{proof}
It is well-known that a hypersurface of $\bP^3$ of degree $4$ is a $K3$-surface.

By \cite{VZ5}, Notation 2.2, and Corollary 8.5, we have that $\lambda_0$ is a $CM$-point of
$\sC_2$, if $\lambda_0$ is a $CM$-point of the family $\sC_1 \to \sM_1$ given by
$$\bP^2 \supset V(y_1^4 + x_1(x_1-x_0)(x_1-\lambda x_0)x_0)
\to \lambda \in \sM_1.$$
Note that $\sC_1$ has in fact a dense set of $CM$ fibers, since it has a derived pure $(1,1)-VHS$
as we have seen. Since only the Hodge group of the Hodge structure on $H^2(X,\Q)$ can be
non-trivial for a $K3$-surface $X$ (follows by definition resp., by the Hodge diamond of a
$K3$-surface), the family $\sC_2$ is a $CMCY$ family of 2-manifolds.
\end{proof}

Now we give some examples of $CM$ fibers of $\sC_2$:

\begin{remark} \label{ellcomp}
Consider the family $\sE \to \sM_1$ of elliptic curves given by
$$\bP^2 \supset V(y^2x_0+x_1(x_1-x_0)(x_1-\lambda x_0)) \to \lambda \in \sM_1.$$
Note that $\sC_1$ has a derived
pure $(1,1)-VHS$, where $\sE$ has the associated primitive pure $(1,1)-VHS$.
Thus the Hodge structure decomposition of Proposition $\ref{hggrcomp}$ tells us
that the fiber $(\sC_1)_{\lambda}$ has $CM$, if the fiber $\sE_{\lambda}$ has $CM$.
In the proof of Proposition $\ref{k3fl}$ we have seen that $(\sC_2)_{\lambda}$ has $CM$,
if $(\sC_1)_{\lambda}$ has $CM$. Thus by the $CM$ fibers of $\sE$, we can determine
$CM$ fibers of $\sC_2$.
\end{remark}

\begin{example} \label{ellcm}
By Remark $\ref{ellcomp}$, the well-known $CM$ curves with $j$ invariant
0 and 1728 yield $CM$ fibers of $\sC_2$ isomorphic to
$$V(y_2^4+y_1^4+(x_1^3-x_0^3)x_0), \  \ V(y_2^4+y_1^4+x_1^4+x_0^4) \subset  \bP^3.$$
Theorem $\ref{geilomat}$ yields the same examples.
\end{example}

\begin{example} \label{ellcm2}
By \cite{hart} {\bf IV}. Proposition $4.18$, one concludes that an elliptic curve has complex
multiplication, if it has a non-trivial isogeny with itself.
The elliptic curve with $j$ invariant 8000 resp., -3375 is given by
$$y^2x_0 = x_1(x_1-x_0)(x_1-(1+\sqrt{2})^2x_0) \  \ \mbox{resp.,} \  \
y^2x_0 = x_1(x_1-x_0)(x_1-\frac{1}{4}(3+i\sqrt{7})^2x_0)$$
and has an isogeny of degree 2 with itself. Moreover the elliptic curve with $j$ invariant 1728 has
an isogeny of degree 2 with itself. This follows from the solution of
\cite{hart}, {\bf IV}. Exercise $4.5$, which we will partially sketch. Thus the $K3$ surfaces given by
$$y_2^4+y_1^4 + x_1(x_1-x_0)(x_1-(1+\sqrt{2})^2x_0)x_0 \  \ \mbox{and} \  \
y_2^4 + y_1^4 + x_1(x_1-x_0)(x_1-\frac{1}{4}(3+i\sqrt{7})^2x_0)x_0$$
have complex multiplication.

We sketch how we obtain the given examples: First note that each degree 2 cover $u:\bP^1 \to \bP^1$ is up to a changement of coordinates
given by $x \to x^2$. This follows from the fact that $u$ has two ramification points
by the Hurwitz formula. Without loss of generality the elliptic curve $E$ is endowed with a degree two
cover $i: E \to \bP^1$ such that there exists a $\lambda$ such $i$ is ramified over
$0,1,\lambda,\infty$ resp., $E$ is locally given by
\begin{equation} \label{loceq}
V(y^2-x(x-1)(x-\lambda)) \subset \A^2.
\end{equation}
Since an isogeny $f: E \to E$ is a morphism of Abelian varieties, one concludes that for
each $(x,y) = f(P) \in E$ one has $f(-P) = -(x,y) = (x,-y)$. Hence one concludes that
there exist the degree 2 covers $u_f: \bP^1 \to \bP^1$ and $h_f:E \to \bP^1$ such that
$$i\circ f  = u_f \circ h_f.$$
It is a very easy exercise to check that $u_f$ can be given by $x \to x^2$ in this
case for some suitable $\lambda$, which yields $E$. Thus one concludes that $h_f$ is ramified
over
$$1, -1, \sqrt{\lambda},-\sqrt{\lambda},$$
which follows from considering the ramification indeces. By a changement of coordinates,
$E$ is given by
$$0,1,\frac{(\sqrt{\lambda}+1)^2}{(\sqrt{\lambda}-1)^2},\infty,$$
too. Note that $\lambda$ and $1-\lambda$ yield the same elliptic curve. We substitute $t = \sqrt{\lambda}$ and resolve the equations
$$t^2 = \frac{(t+1)^2}{(t-1)^2} \  \ \mbox{and} \  \ t^2 = 1-\frac{(t+1)^2}{(t-1)^2}$$
by using the computer algebra program MATHEMATICA in the case of the ground field $\C$. This yields the stated elliptic curves $E$
with an isogeny $f: E \to E$ of degree 2. It remains to prove the completeness of the given
examples, which is a well-known fact.
\end{example}

\begin{example} \label{ellcm3}
Elliptic curves with $CM$ has been well studied by number theorists. In \cite{Sil} Appendix {\bf C}, $\S3$
there is a list of 13 isomorphy classes of elliptic curves with complex multiplication containing all classes represented by the
preceding 4 examples. Two examples of the list, which have the $j$ invariants 54000
and 16581375, are given by the equations
$$y^2 = x^3 - 15x +22, \ \ y^2 = x^3 -595x+5586.$$
The equations allow an explicite determination of involutions on these examples. The given equations
for the 7 remaining isomorphy classes of elliptic curves do not allow an immediate description of
involutions.
\end{example}

As we will see, the family $\sC_2$ has some involutions, which make it suitable for the construction of a
Borcea-Voisin tower. The following lemma is obvious:

\begin{lemma}
Over the basis $\sM_1$ the family $\sC_2$ has three involutions given by
$$\iota_1(y_2:y_1:x_1:x_0) = (-y_2:y_1:x_1:x_0), \  \ 
\iota_2(y_2:y_1:x_1:x_0) = (y_2:-y_1:x_1:x_0),$$
$$\iota_3(y_2:y_1:x_1:x_0) = (-y_2:-y_1:x_1:x_0),$$
which constitute with the identity map a subgroup of the $\sM_1$-automorphism group of
$\sC_2$ isomorphic to the Kleinsche Vierergruppe.
\end{lemma}

\begin{remark} \label{invo}
Over $\sM_1$ there are at least the 4 following additional involutions on $\sC_2$:
$$\iota_4(y_2:y_1:x_1:x_0) = (y_1:y_2:x_1:x_0), \ \
\iota_5(y_2:y_1:x_1:x_0) = (iy_1:-iy_2:x_1:x_0),$$
$$\iota_6(y_2:y_1:x_1:x_0) = (-y_1:-y_2:x_1:x_0),\ \
\iota_7(y_2:y_1:x_1:x_0) = (-iy_1:iy_2:x_1:x_0)$$
\end{remark}

\begin{Theorem} \label{iotata}
By the involutions $\iota_1$ and $\iota_4$, the family $\sC_2$ can be used to be $\sZ_1$ or some
$\Sigma_i$ for the construction of a Borcea-Voisin tower of $CMCY$ families of $n$-manifolds.
\end{Theorem}
\begin{proof}
The divisor of the fiber $(\sC_2)_{\lambda}$, which is fixed by $\iota_1$ resp., $\iota_4$ is
given by $y_2 = 0$ resp., $y_2 = y_1$. Hence both divisors are smooth and isomorphic to the fiber
$(\sC_1)_{\lambda}$ given by
$$\bP^2 \supset V(y_1^4 + x_1(x_1-x_0)(x_1-\lambda x_0)x_0)
\to \lambda \in \sM_1.$$
We use the same arguments as in the proof of Proposition $\ref{k3fl}$: If $(\sC_1)_{\lambda}$ has
complex multiplication, then $(\sC_2)_{\lambda}$ and the divisor fixed by
$\iota_1$ resp., $\iota_4$ have complex multiplication, too. Hence by the fact that $\sC_1$ has a
dense set of complex multiplication fibers, $\sC_2$ and $\iota_1$ resp., $\sC_2$ and $\iota_4$
satisfy the assumptions of Construction $\ref{bvtower}$.
\end{proof}

\begin{remark} \label{etos} 
By the fact that $$\iota_2 = \iota_4\circ \iota_1\circ \iota_4,$$
the involution $\iota_2$ is suitable for the construction of a
Borcea-Voisin tower, too. But according to the construction of C. Voisin \cite{Voi2}, this implies
that $\iota_2$ yields a $CMCY$ family of $3$-manifolds over $\sM_1 \times \sM_1$, which is
isomorphic to the corresponding family obtained by $\iota_1$.

Let $\alpha$ denote the $\sM_1$-automorphism of $\sC_2$ given by
$$(y_2:y_1:x_1:x_0) \to (iy_2:y_1:x_1:x_0).$$
One calculates easily that
$$\iota_{5} = \alpha \circ \iota_4 \circ \alpha^{-1}, \  \ \iota_{6} = \alpha^2 \circ \iota_4
\circ \alpha^{-2}, \  \  \iota_{7} = \alpha^{-1} \circ \iota_4 \circ \alpha.$$
Hence one has that $\sC_2/\iota_{4}, \ldots, \sC_2/\iota_{7}$ resp., the resulting
$CMCY$ families of 3-manifolds obtained by the method of C. Voisin \cite{Voi2} are isomorphic as
$\sM_1$-schemes resp., as $\sM_1 \times \sM_1$-schemes.

Since
$$\iota_3 = \iota_1\iota_2,$$
the involution $\iota_3$ acts by ${\rm id}$ on each
$\Gamma(\omega_{(\sC_2)_{\lambda}})$ such that it can not be used for the construction of a
Borcea-Voisin tower.
\end{remark}

\begin{remark}
By Example $\ref{ellcm}$, Example $\ref{ellcm2}$ and Example $\ref{ellcm3}$, one has 6 explicitely
given elliptic curves with $CM$ and explicitely given involutions, which yield
6 $K3$ surfaces with $CM$. By using the method of C. Voisin \cite{Voi2}, these examples yield
36 explicitely given fibers with $CM$ for each of our resulting $CMCY$ family of 3-manifolds. 
\end{remark}

\begin{remark}
The author does not see a way to conjugate $\iota_1$ into $\iota_4$. Moreover we will
see that the fibers of the resulting $CMCY$ families of 3-manifolds constructed with $\iota_1$ and
$\iota_4$ according to C. Voisin \cite{Voi2} have the same Hodge numbers. This means that the
question for isomorphisms between these two families remains open.
\end{remark}

\chapter{The degree 3 case}

\section{Prelude}
We construct a surface $R^1$ by a desingularization of the weighted projective space $\bP(2,1,1)$
during this section. Our modified construction of a Viehweg-Zuo tower starts with
the family $\sC$ of curves with a dense set of $CM$ fibers given by
\begin{equation} \label{fami}
R^1 \supset V(y^3 = x_1(x_1-x_0)(x_1 - \alpha x_0)(x_1-\beta x_0)(x_1-\gamma x_0)x_0) \to
(\alpha ,\beta ,\gamma) \in \sM_3.
\end{equation}
We use such a weighted projective space, since the degree of these covers onto $\bP^1$ does
not coincide with the sum of branch indeces. In this section we construct a rational
3-manifold $R^2$ with a natural projection onto $R^1$. For each fiber $\sC_q$ this projection
induces a cyclic degree 3 cover of a Calabi-Yau hypersurface of $R^2$ onto $R^1$ ramified over
$\sC_q$. We will later see that these Calabi-Yau hypersurfaces of $R^2$ yield a $CMCY$ family of
2-manifolds suitable for the construction of a Borcea-Voisin tower.

Recall that the usual projective space $\bP^n$ is given by ${\rm Proj}(\C[z_n, \ldots, z_1, z_0])$,
where each $z_j$ (with $j = 0, \ldots, n$) has the weight 1. Our weighted projective space $Q^n$
\index{$Q^n$} is given by ${\rm Proj}(\C[y_n, \ldots, y_1, x_1, x_0])$, where each $y_j$
(with $j = 1, \ldots, n$) has the weight 2, and $x_0$ and $x_1$ have the weight 1.

First we investigate and describe the projective space $Q^n$. The following well-known
Lemma will be very useful here:

\begin{lemma} (Veronese embedding)
Let $R$ be a graded ring. Then we have
$${\rm Proj}(R) \cong {\rm Proj}(R^{[d]}).$$
\end{lemma}

\begin{proposition} \label{qgl}
The weighted projective space $Q^n$ is isomorphic to the irreducible singular
hypersurface in $\bP^{n+2}$ given by the equation $z_1z_3 = z_2^2$. The singular locus of $Q^n$ is
given by $V(z_1,z_2,z_3)$.
\end{proposition}
\begin{proof}
By the Veronese embedding, we have
$$Q^n\cong {\rm Proj}(k[x_0^2,x_0x_1, x_1^2,y_1, \ldots, y_n]).$$

Hencefore we obtain a closed embedding of $Q^n$ into $\bP^{n+2}$ given by
$$x_0^2 \to z_1, \ \ x_0x_1 \to z_2, \ \ x_1^2 \to z_3, \ \ y_1 \to z_4, \  \ \ldots,  \  \
y_n \to z_{n+3}.$$
We have that $Q^n \setminus V(x_0^2)$
is isomorphic to $\A^{n+1}$. Hence $\dim(Q^n) = n+1$, which implies that its projective cone, which
is contained in $\A^{n+3}$, has the dimension $n+2$. By \cite{hart}, {\bf I}. Proposition 1.13,
each irreducible component of dimension $n+2$ of this cone is given by an ideal generated by one
irreducible polynomial. The corresponding polynomial of the unique irreducible component of $Q^n$
is
$$f(z_1,z_2,z_3) = z_1z_3 - z_2^2,$$
since each point $p \in Q^n\subset \bP^{n+2}$ satisfies $f(p)=0$ and $f$ is irreducible. The last
statement about the singular locus follows from calculating the partial derivatives of $f$.
\end{proof}

Let $a_1, \ldots, a_{2m} \in \C$, and $m \in \N \setminus \{1\}$. Then $C_{(n)} \subset  Q^n$ is
the subvariety, which is given by the homogeneous polynomial
$$y_n^m + \ldots + y_1^m + (x_1-a_1x_0) \ldots (x_1-a_{2m}x_0).$$
It is a very easy exercise to check that this polynomial is irreducible.

\begin{proposition}
There exists a homogenous polynomial $G \in \C[z_1,z_2, z_3]$ of
degree $m$ such that $C_{(n)} \subset \bP^{n+2}$ is given by the ideal
generated by $h$ and $f$, where
$$h = z_{n+3}^m + \ldots z_4^m + G.$$
\end{proposition}
\begin{proof}
We can obviously choose a polynomial $G$ such that
$$G(x_0^2, x_0x_1, x_0^2) = (x_1-a_1x_0) \ldots (x_1-a_{2m}x_0).$$
Now let $h = z_{n+3}^m + \ldots z_4^m + G$, and
$$\phi : \C[z_1, \ldots, z_{n+3}]  \to \C[x_0^2,x_0x_1, x_1^2,y_1, \ldots, y_n]$$
be the homomorphism associated to the closed embedding $Q^n \hookrightarrow \bP^{n+2}$, which has
the kernel $ ( f ) $. We obtain
$$\phi(h) = y_n^m + \ldots + y_1^m + (x_1-a_1x_0) \ldots (x_1-a_{2m}x_0).$$

Hence $C_{(n)} \subset \bP^{n+2}$ is given by the prime ideal
$$\phi^{-1}(\sI(C_{(n)})) = (h, f).$$
\end{proof}

\begin{proposition}
The singular locus of $C_{(n)}$ is given by
$C_{(n)} \cap V(z_1,z_2,z_3)$.
\end{proposition}
\begin{proof}
On $Q^n \setminus V(x_0) \cong {\rm Spec}(\C[x_1, y_1, \ldots, y_n])$ the
hypersurface $C_{(n)}$ is given by the equation
$$0= y_n^m + \ldots + y_1^m + (x_1-a_1) \ldots (x_1-a_{2m}).$$
By the partial derivatives of the polynomial on the right hand,
one can easily check that there are no singularities of $C_{(n)}$ in
this affine subset. The same arguments give the same statement for
$Q^n \setminus V(x_1)$. Hence all singularities of $C_{(n)}$ are
contained in $V = V(z_1,z_2,z_3)$. For all $P \in C_{(n)} \cap V$, the
Jacobian matrix of $C_{(n)}$ at $P$ does not have the maximal rank 2,
where this is obtained by explicit calculation of the partial
derivatives of $f$ and $h$.
\end{proof}

\begin{pkt}
The variety $Q^n$ has a natural interpretation as degree 2 cover onto the variety given by
$\{z_2=0\}$ ramified over $\{z_1 = z_2=0\}$ and $\{z_2=z_3=0\}$. Hence by blowing up
$V = V(z_1,z_2,z_3)$, the proper transform $R^n := \tilde {Q}^n_V$ \index{$R^n$} is the natural
degree 2 cover onto the proper transform of $\{z_2=0\}$ ramified over the disjoint proper
transforms of $\{z_1 = z_2=0\}$ and $\{z_2=z_3=0\}$. Thus $R^n$ is non-singular.

Note that the general construction of the blowing up yields a natural embedding of an open
subset of $R^n$ into $\A^{n+2} \times \bP^2$. Hence the Jacobian matrix at each
point of $R^n$ has the maximal rank 3 with respect to this local embedding. The Jacobian matrix of
the proper transform $\tilde C_n$ of $C_{(n)}$ is given by adding the line of the partial
derivatives of $h$ to the
Jacobian matrix of $R^n$. Without loss of generality we are on the open subset $\{y_1 = 1\}$. On
the exceptional divisor $E$ the polynomial $G$ vanishes. Thus all points of $\tilde C_n \cap E$
satisfy
$$y_n^m + \ldots + y_2^n + 1 = 0.$$
Hence for each $p \in \tilde C_n \cap E$ there is a partial derivative
$\partial h/\partial y_i(P)\neq 0$. Since all partial derivatives  of the equations defining $R^n$
with respect to $y_i$ vanish, the Jacobian matrix of $\tilde C_n$ has the maximal rank 4
at each point on the exceptional divisor. Thus $\tilde C_n$ is smooth.
\end{pkt}

\begin{remark} \label{regeli}
Note that $Q^1$ has a natural interpretation as projective closure of the affine cone of a
rational curve of degree 2 in $\bP^2$. By \cite{hart}, {\bf V}. Example 2.11.4, one has that
$R^1$, which is the blowing up of the unique singular point given by the vertex of the cone, is a
rational ruled surface isomorphic to $\bP(\sO_{\bP^1} + \sO_{\bP^1}(2))$, where the exceptional
divisor has the self-intersection number $-2$.
\end{remark}

By \cite{hart}, {\bf II}. Proposition 8.20, one has for $n \geq 1$:
$$\omega_{Q^n\setminus V(z_1,z_2,z_3)} = \omega_{\bP^{n+2} \setminus V(z_1,z_2,z_3)}
\otimes \sI(Q^n\setminus \ V(z_1,z_2,z_3)) \otimes \sO_{Q^n\setminus V(z_1,z_2,z_3)}$$
$$= \sO_{Q^n\setminus V(z_1,z_2,z_3)} (-(n+1)V(z_4))$$
By \cite{artin}, Theorem 2.7 and the fact that the self-intersection number of the exceptional
divisor is $-2$, the pull-back of the canonical divisor of $Q^1$ with
respect to the blowing up morphism is the canonical divisor of $R^1$. Note that the canonical
divisor of $Q^1$ yields the canonical divisor of $Q^1 \setminus\{s\}$, where $s$ denotes the
singular point. Thus:

\begin{corollary} \label{kanne}
The canonical divisor of $R^1$ is given by $-2V(z_4)$.
\end{corollary}

The following lemma describes the construction of this section. One has the following commutative
diagram of closed embeddings:
$$
\xymatrix{
{C_{(0)}} \ar[d]^{} \ar[rr]^{}  &  & {\ldots} \ar[d]^{} \ar[rr]^{}  &  & {C_{(n)}} \ar[d]^{}
\ar[rr]^{}  &  & {C_{(n+1)}} \ar[d]^{} \ar[rr]^{}  &  & {\ldots} \ar[d]^{}\\
{Q^0} \ar[d]^{} \ar[rr]^{}  &  & {\ldots} \ar[d]^{} \ar[rr]^{}  &  & {Q^n} \ar[d]^{} \ar[rr]^{}
&  & {Q^{n+1}} \ar[d]^{} \ar[rr]^{}  &  & {\ldots} \ar[d]^{}\\
{\bP^2} \ar[rr]^{}  &  & {\ldots} \ar[rr]^{}  &  & {\bP^{n+2}} \ar[rr]^{}  &  & {\bP^{n+3}}
\ar[rr]^{} &  & {\ldots}}
$$ 
The ideal sheaf of each blowing up $\tilde C_{n} \to C_{(n)}$ and $R^n \to Q^n$ is generated by
$z_1, z_2, z_3$.
Moreover this ideal sheaf is obviously the inverse image ideal sheaf of the ideal sheaf
generated by $z_1, z_2, z_3$ with respect to all embeddings. Hence we obtain by \cite{hart}, II.
Corollary $7.15$ for $V := V(z_1,z_2,z_3)$:

\begin{lemma} \label{pups} We have the commutative diagram
$$\xymatrix{
  {\tilde C_{0}} \ar[d]^{} \ar[rr]^{}  &  & {\ldots} \ar[d]^{} \ar[rr]^{}  &  & {\tilde C_{n}}
\ar[d]^{} \ar[rr]^{}  &  & {\tilde C_{n+1}} \ar[d]^{} \ar[rr]^{}  &  & {\ldots} \ar[d]^{}\\
  {R^0} \ar[d]^{} \ar[rr]^{}  &  & {\ldots} \ar[d]^{} \ar[rr]^{}  &  & {R^n} \ar[d]^{} \ar[rr]^{}
 &  & {R^{n+1}} \ar[d]^{} \ar[rr]^{} &  & {\ldots} \ar[d]^{}\\
  {\tilde \bP^2_V} \ar[rr]^{}  &  & {\ldots} \ar[rr]^{}  &  & {\tilde \bP^{n+2}_V} \ar[rr]^{}
 &  & {\bP^{n+3}_V} \ar[rr]^{} &  & {\ldots}
}$$
of closed embeddings.
\end{lemma}

\begin{remark}
Note that $C_{(0)} = \tilde C_0$, $C_{(1)} = \tilde C_1$ and $Q^0 = R^0$.
\end{remark}

\begin{Theorem} \label{kanone}
The canonical divisor of $R^n$ is given by $-(n+1)\tilde V(z_4)$ for $n\geq 1$.
\end{Theorem} \begin{proof}
By Corollary $\ref{kanne}$, we have the statement for $n= 1$.

We use induction for higher $n$. Let $E_n$ denote exceptional
divisor of the blowing up $R^n \to Q^n$. The open
subset $R^n \setminus E_n$ is isomorphic to $Q^n\setminus
V(z_1,z_2,z_3)$. We know that $-(n+1)\tilde V(z_4)$ is the canonical
divisor of $Q^n\setminus V(z_1,z_2,z_3)$. Hence we conclude that
$$K_{R^{n+1}} = -(n+2)\tilde V(z_4) + z E_{n+1}$$ for some $z \in \Z$. We have that $R^n \sim
\tilde V(z_4)$ in $Cl(R^{n+1})$. By the induction hypothesis, we have
$$\sO_{R^n}(-(n+1)\tilde V(z_4)) \cong \omega_{R^n} \cong \sO_{R^{n+1}}(\tilde V(z_4))\otimes\omega_{R^{n+1}}
\otimes \sO_{R^n}$$
such that $z = 0$ and $-(n+2)\tilde V(z_4)$ is the canonical divisor of $R^{n+1}$.
\end{proof}

Since we want to construct a family of Calabi-Yau manifolds, we note:

\begin{Theorem}
The hypersurface $\tilde C_{m-1} \subset R^{m-1}$ is a Calabi-Yau manifold.
\end{Theorem}
\begin{proof}
By Theorem $\ref{kanone}$, $-m\tilde V(z_4)$ is the canonical divisor of $R^{m-1}$.
Hence \cite{hart}, {\bf II}. Proposition 8.20 and $\tilde C_{m-1} \sim m\tilde V(z_4)$ imply that
$$\omega_{\tilde C_{m-1}} = \sO_{\tilde C_{m-1}}.$$
By the fact that $h^{q,0}$ is a birational invariant of non-singular projective
varieties (see \cite{hart}, page 190), and $R^{m-1}$ is birationally
equivalent to $\bP^{m}$, we obtain that $h^{q,0}(R^{m-1}) = 0$ for all $1
\leq q\leq m$. By Hodge symmetry and Serre duality, we obtain that
$h^{q}(R^{m-1}, \sO) = 0$ for all $1 \leq q\leq m$ and $h^{q}(R^{m-1},
\omega) = 0$ for all $0 \leq q\leq m-1$. Since the canonical divisor
of $R^{m-1}$ is linearly equivalent to $-\tilde C_{m-1}$, we obtain
the exact sequence
\begin{equation*}
0 \to \omega_{R^{m-1}} \to \sO_{R^{m-1}} \to \sO_{\tilde C_{m-1}} \to 0.
\end{equation*}
This implies that $h^i(\tilde C_{m-1}, \sO) = 0$ for $1 \leq i < m-1 = \dim(\tilde C_{m-1})$.
Hence $\tilde C_{m-1}$ is a Calabi-Yau manifold.
\end{proof}

\begin{pkt} \label{wer}
The projection $\bP^{n+2} \setminus \{(1:0: \ldots: 0)\} \to \bP^{n+1}$ given by
$$(z_{n+3}: \ldots:z_1) \to (z_{n+2} : \ldots: z_1)$$
induces a cyclic cover $C_{(n+1)} \to Q^n$ of degree $m$ ramified over $C_{(n)}$. The Galois group is
generated by
$$(z_{n+3}: z_{n+2}: \ldots: z_1) \to (\xi z_{n+3}: z_{n+2}: \ldots:z_1),$$
where $\xi$ is a primitive $m$-th. root of unity.

Recall the commutative diagram of Lemma $\ref{pups}$. Let $\A^4$ be given by
$\{z_{4} = 1\} \subset \bP^{4}$ and $\A^3$ be given by
$\{z_{4} = 1\} \subset \bP^{3}$. Then the projection above yields a morphism
\begin{equation} \label{lebal}
f: \A^{4} \times \bP^2 \to \A^{3} \times \bP^2.
\end{equation}
Since the blowing up yields natural embeddings of open subsets of $\tilde C_{2}$ and $R^1$ into
the varieties of $(\ref{lebal})$, $f$ induces a rational map $\tilde C_{2} \to R^1$.
Now this rational map $\tilde C_{2} \to R^1$ is again a cyclic cover of degree $m$ with
the Galois group as above (on the open locus of definition). On the complements of the exceptional
divisors it coincides with the cyclic cover $C_{(2)} \to Q^1$ above. Hence by glueing, one has a
cyclic cover $\tilde C_{2} \to R^1$ ramified over $C_{(1)}$.
\end{pkt}

\section{A modified version of the method of Viehweg and Zuo}

The following construction is a modified version of the construction
in \cite{VZ5}, Section 5. That means here we show that $\tilde C_2$ has $CM$, if $C_{(1)}$ has
$CM$. In the next section we will use the construction of the preceding section to define a family
of $K3$-surfaces. In this section we give the argument that this family of $K3$-surfaces will be
a $CMCY$ family of 2-manifolds.

For our application, it is sufficient to
consider the situation fiberwise and to work with $\bP^1$-bundles
over $\bP^1$ resp., with rational ruled surfaces. Let $\pi_n: \bP_n \to \bP^1$ denote
the rational ruled surface given by $\bP(\sO_{\bP^1} \oplus
\sO_{\bP^1}(n))$ and $\sigma$ denote a non-trivial global section
of $\sO_{\bP^1}(6)$, which has the six different zero points represented by a point $q \in \sM_3$.
The sections $E_{\sigma}$, $E_0$ and $E_{\infty}$ of $\bP(\sO \oplus \sO(6))$ are
induced by
$${\rm id}\oplus \sigma:\sO \to \sO\oplus \sO(6), \ \ 
{\rm id}\oplus 0:\sO \to \sO\oplus \sO(6)$$
$$\mbox{and} \  \ 0\oplus {\rm id}:\sO(6) \to \sO\oplus \sO(6)$$
resp., by the corresponding surjections onto the cokernels of these embeddings as described in
\cite{hart}, {\bf II}. Proposition 7.12.

\begin{remark}
The divisors $E_{\sigma}$ and $E_0$ intersect each other transversally over the 6 zero points of
$\sigma$. Recall that ${\rm Pic}(\bP_6)$ has a basis given by a fiber and an arbitrary section.
Hence by the fact that $E_{\sigma}$ and $E_0$ do not intersect $E_{\infty}$, one concludes that
they are linearly equivalent with self-intersection number 6. Since $E_{\infty}$
is a section, it intersects each fiber transversally. Thus one has that
$E_{\infty} \sim  E_0-(E_0.E_0)F$, where $F$ denotes a fiber. Hencefore one concludes
$$E_{\infty}.E_{\infty} = E_{\infty}.(E_0-(E_0.E_0)F) = -(E_0.E_0) = -6.$$
\end{remark}

Next we establish a morphism $\mu:\bP_2 \to \bP_6$ over $\bP^1$. By
\cite{hart}, ${\bf II}$. Proposition 7.12., this is the same as to
give a surjection $\pi_2^*(\sO \oplus \sO(6)) \to \sL$, where $\sL$
is an invertible sheaf on $\bP_2$. By the composition
$$\pi_2^*(\sO \oplus \sO(6)) = \pi_2^*(\sO) \oplus \pi_2^*\sO(6) \hookrightarrow
\bigoplus\limits_{i = 0}^3 \pi_2^*\sO(2i) = Sym^3(\pi_2^*(\sO \oplus \sO(6))) \to
\sO_{\bP_2}(3),$$
where the last morphism is induced by the natural surjection $\pi_2^*(\sO \oplus \sO(2)) \to
\sO_{\bP_2}(1)$ (see \cite{hart}, {\bf II}. Proposition 7.11), we obtain a morphism $\mu^*$ of
sheaves. This morphism $\mu^*$ is not a surjection onto $\sO_{\bP_2}(3)$, but onto its image
$\sL \subset \sO_{\bP_2}(3)$. Locally over $\A^1 \subset \bP^1$ all rational ruled surfaces are
given by ${\rm Proj}(\C[x])[y_1,y_2]$, where $x$ has the weight 0. Hence we have locally that
$\pi_2^*(\sO \oplus \sO(6)) = \sO e_1\oplus \sO e_2$. Over $\A^1$ the morphism $\mu^*$ is given by
$$e_1 \to y_1^{3}, e_2 \to y_2^{3}$$
such that the sheaf $\sL = im(\mu^*) \subset \sO_{\bP_2}(3)$ is invertible.
Thus the morphism $\mu: \bP_2 \to \bP_6$ corresponding to $\mu^*$ is locally
given by the ring homomorphism
$$(\C[x])[y_1,y_2] \to (\C[x])[y_1,y_2] \ \ \mbox{via} \  \ y_1 \to y_1^3 \ \ \mbox{and}
\ \ y_2 \to y_2^3.$$

\begin{construction}\label{4.1}
One has a commutative diagram
$$\xymatrix{
{\sY'} \ar[rr]^{\tau'} &   & {\bP'_2} \ar[rr]^{\mu'} &   & {\bP^1 \times \bP^1}\\
{\hat \sY} \ar[rr]^{\hat \tau} \ar[d]_{\rho} \ar[u]^{\delta} &  & {\hat \bP_2} \ar[rr]^{\hat \mu}
\ar[d]^{\rho_2} \ar[u]_{\delta_2} &  & {\hat \bP_6} \ar[d]^{\rho_6} \ar[u]_{\delta_6}\\
{\sY} \ar[rr]^{\tau}_{\sqrt[3]{\frac{\mu^*E_\sigma}{3\cdot (\mu^*E_0)_{red}}}}
\ar[d]_{\pi} &  & {\bP_2} \ar[rr]^{\mu}_{\sqrt[3]{\frac{E_\infty + 6 \cdot F}{E_0}}} \ar[d]^{\pi_2}
&  & {\bP_6} \ar[d]^{\pi_6}\\
{\bP^1} \ar[rr]^{\rm id} &   & {\bP^1} \ar[rr]^{\rm id} &   & {\bP^1}
}$$
of morphisms between normal varieties with:
\begin{enumerate}
\item[(a)] $\delta$, $\delta_2$, $\delta_6$, $\rho$, $\rho_2$ and $\rho_6$
are birational.
\item[(b)] $\pi$ is a family of curves, $\pi_2$ and $\pi_d$
are $\bP^1$-bundles.
\item[(c)] All the horizontal arrows (except for the ones in the bottom line) are Kummer coverings
of degree $3$.
\end{enumerate}
\end{construction}
\begin{proof}
One must only explain $\delta_6$ and $\rho_6$. Recall that
$E_{\sigma}$ is a section of $\bP(\sO \oplus \sO(6))$, which
intersects $E_0$ transversally in exactly 6 points. The morphism $\rho_6$ is the
blowing up of the six intersection points of $E_0 \cap E_{\sigma}$.
The preimage of the six points given by $q \in \sM_3$ with respect to $\pi_6
\circ \rho_6$ consists of the exceptional divisor $\hat D_1$ and the
proper transform $\tilde D_2$ of the preimage of these six points
with respect to $\rho_6$ given by 6 rational curves with self-intersection number $-1$. The
morphism $\delta_6$ is obtained by blowing down $\tilde D_2$.
\end{proof}

\begin{remark}
The section $\sigma$ has the zero divisor given by some $q \in \sP_3$. Hence one obtains
$\mu^*(E_{\sigma}) \cong \sC_q$, where $\sC \to \sP_3$ denotes the family of cyclic covers onto
$\bP^1$ with a pure $(1,3)-VHS$ of degree 3. Since $\tau$ is the unique cyclic degree 3
covering of $\bP_2 \cong R^1$ ramified over $\mu^*(E_{\sigma}) \cong \sC_q$, the surface $\sY$
is isomorphic to some $K3$-surface $\tilde C_2$ of the preceding section.
\end{remark}

Recall that $\F_n$ denotes the Fermat curve of degree $n$.

\begin{proposition} \label{proposition}
The surface $\sY$ is birationally equivalent to
$\sC_q \times \F_3/\langle(1, 1)\rangle$.\footnote{Similarly to
\cite{VZ5}, Construction 5.2, we show that $\sY'$ is birationally equivalent to
$\sC_q \times \F_3/\langle(1, 1)\rangle$.}
\end{proposition}

\begin{proof}
Let $\tilde  E_{\bullet }$ denote the proper transform of the section $E_{\bullet }$ with respect
to $\rho_6$. Then $\hat \mu$ is the Kummer covering given by
$$\sqrt[3]{\frac{\tilde E_\infty + 6 \cdot F}{\tilde E_0+ \hat D_1}},$$
where $\hat D_1$ denotes the exceptional divisor of $\rho_6$. Thus the morphism $\mu'$ is the
Kummer covering
$$\sqrt[3]{\frac{ (\delta_6)_* \tilde E_\infty + 6 \cdot (\delta_6)_* F}
{(\delta_6)_*\tilde E_0+ (\delta_6)_*\hat D_1}}
=\sqrt[3]{\frac{ \bP^1 \times \{\infty\} + 6 \cdot (P \times \bP^1)}
{\bP^1 \times \{0\} + \Delta  \times\bP^1}},$$
where $\Delta $ is the divisor of the 6 different points in $\bP^1$ given by $q \in \sM_3$
and $P \in \bP^1$ is the point with the fiber $F$. Since $E_0 + E_{\sigma}$ is a normal crossing
divisor, $\tilde E_{\sigma}$ neither meets $\tilde E_{0}$
nor $\tilde D_2$, where $\tilde D_2$ is the proper transform of $\pi_6^*(\Delta)$. Therefore
$(\delta_6)_*\tilde E_{\sigma }$ neither meets
$$(\delta_6)_*\tilde E_0 = \bP^1 \times \{0\} \  \ \mbox{nor} \  \
(\delta_6)_*\tilde E_{\infty} = \bP^1 \times \{\infty\}.$$
Hence one can choose coordinates in $\bP^1$ such that
$(\delta_6)_*\tilde E_{\sigma} = \bP^1 \times \{1\}$.

By the definition of $\tau$, we obtain that $\hat \tau$ is given by
$$\sqrt[3]{\frac{ \rho_2^*\mu^*(E_{\sigma} )}{\rho_2^*\mu^*(E_0)}}
= \sqrt[3]{\frac{ \hat \mu^*(\tilde E_{\sigma} )}{\hat \mu^*(\tilde E_0)}},$$
and $\tau'$ is given by

$$\sqrt[3]{\frac{ \mu'^*(\bP^1 \times \{1\})}
{\mu'^*(\bP^1 \times \{0\})}}.$$
By the fact that the last function is the third root of the pullback of a function
on $\bP^1 \times \bP^1$ with respect to $\mu'$, it is possible to reverse the order of the field
extensions corresponding to $\tau'$ and $\mu'$ such that the resulting varieties obtained by
Kummer coverings are birationally equivalent. Hence we have the composition of $\beta:
\bP^1 \times \bP^1 \to \bP^1 \times \bP^1$ given by
$$\sqrt[3]{\frac{ \bP^1 \times \{1\}}{\bP^1 \times \{0\}}}$$
with

$$\sqrt[3]{\frac{ \beta^*(\bP^1 \times \{\infty\}) + 6 \cdot  (P \times \bP^1)}
{\beta^*(\bP^1 \times \{0\}) + (\Delta \times\bP^1)}},$$
which yields the covering variety isomorphic to $\F_3 \times \sC_q/\langle(1,1)\rangle$.
\end{proof}

Hence $\tilde C_2 \cong \sY$ is birationally equivalent to
the algebraic manifold $\hat \sY$ in the diagram ($\ref{diag}$) with
$\sZ = C_{(1)}$ and $\Sigma = \F_3$. Therefore by Corollary
$\ref{blowblow}$, we obtain:

\begin{corollary} \label{cmcm}
If the curve $\mu^*(E_{\sigma})$ has complex multiplication, then the $K3$-surface $\sY$ has only
commutative Hodge groups.
\end{corollary}

\section{The resulting family and its involutions}
\begin{pkt} \label{8.3.1}
Let us summarize the things we have done. By using the Veronese embedding, the weighted projective
space $Q^2 =\bP_{\C}(2,2,1,1)$ is given by $V(z_1z_3 = z_2^2) \subset \bP^4$. Moreover there
exists a homogeneous polynomial $G_{(a_1,a_2,a_3)} \in \C[z_1,z_2,z_3]$ of degree 3 such that
$$
G_{(a_1,a_2,a_3)}(x^2,x,1) = x(x-1) (x-a_1)(x-a_2)(x-a_3)
$$ for each $(a_1,a_2,a_3) \in \sM_3$. Let $W \hookrightarrow  Q^2 \times \sM_3 \stackrel{pr_2}{\to}
\sM_3$ be the family with the fibers
given by $W_q =V(z_1z_3 - z_2^2,\ \ z_5^3+z_4^3 + G_q)$ for all $q \in \sM_3$. Moreover let
$\sW \hookrightarrow  R^2 \times \sM_3 \to \sM_3$ be the smooth family obtained by the proper
transform of $W$ with respect to the blowing up of $V(z_1,z_2,z_3) \times \sM_3$. Since the
family $\sC \to \sM_3$ given by
$$R^1 \supset V(y^3 - x_1(x_1-x_0) (x_1-a_1x_0)(x_1-a_2x_0)(x_1-a_3x_0)x_0)  \to (a_1,a_2,a_3)
\in \sM_3$$ has dense set of complex multiplication fibers, Corollary $\ref{cmcm}$ implies that
$\sW$ is a $CMCY$ family of 2-manifolds.
\end{pkt}

Next we will find and study involutions on $\sW$ over $\sM_3$ satisfying the assumptions for the
construction of a Borcea-Voisin tower.

\begin{remark} \label{1r}
We have the involutions on $W$ over $\sM_3$ given by
\begin{eqnarray*}
\gamma^{(1)}(z_5:z_4:z_3:z_2:z_1) = (z_4:z_5:z_3:z_2:z_1),\\
\gamma^{(2)}(z_5:z_4:z_3:z_2:z_1) = (\xi z_4:\xi^2 z_5:z_3:z_2:z_1),\\
\gamma^{(3)}(z_5:z_4:z_3:z_2:z_1) = (\xi^2 z_4:\xi z_5:z_3:z_2:z_1),
\end{eqnarray*}
where $\xi$ is a fixed primitive cubic root of unity. For simplicity we write $\gamma$ instead of
$\gamma^{(1)}$, too. Since the ideal sheaf of
$V(z_1,z_2,z_3)\cap W$ coincides with its inverse image ideal sheaf with respect to $\gamma^{(i)}$
(for all $i = 1, 2, 3$), each $\gamma^{(i)}$ induces an involution on $\sW$ over the basis $\sM_3$
denoted by $\gamma^{(i)}$, too. 
\end{remark}

\begin{remark}
We have the $\sM_3$-automorphism $\kappa $ of $W$ given by
$$\kappa (z_5:z_4:z_3:z_2:z_1) = (\xi z_5:z_4:z_3:z_2:z_1) \mbox{ with }$$
$$\kappa^{-1} (z_5:z_4:z_3:z_2:z_1) = (\xi^2 z_5:z_4:z_3:z_2:z_1)$$
such that by the same argument as in Remark $\ref{1r}$, we obtain an automorphism of $\sW$ over
$\sM_3$ denoted by $\kappa $, too. On $W$ and hencefore on $\sW$ one has 
$$\gamma^{(2)} = \kappa \circ \gamma \circ \kappa^{-1} \  \ \mbox{and} \  \  \gamma^{(3)} =
\kappa^{-1} \circ \gamma \circ \kappa.$$
Hence these involutions act by the same character on the global differential forms of the fibers
of $\sW$, and all quotients $\sW/\gamma^{(i)}$ are
isomorphic. Therefore it is sufficient to consider the quotient by $\gamma$.
\end{remark}

\begin{proposition} \label{sw1}
On each fiber of $\sW$ the involution $\gamma$ fixes exactly the points on the divisor given
by $V(z_4 = z_5)$ and one exceptional line over one singular point of the corresponding fiber of
$W$.
\end{proposition}
\begin{proof}
Let $q \in \sM_3$ and let $S$ denote the singular locus of $W_q$. On $W_q \setminus S$ the points
fixed by $\gamma$
are given by the divisor $V(z_4 = z_5)$. Now let us consider the exceptional divisors of the
blowing up, which turns $W$ into the family $\sW$ of smooth $K3$-surfaces. There are exactly 3
points of $S$ given by $z_1 = z_2 = z_3 = 0$ and $z_4^3 + z_5^3 = 0$. The involution $\gamma$
fixes $(1:-1:0:0:0)$ and interchanges the other two singular
points. Since the generators of the ideal of the blowing up are invariant under $\gamma$, one
concludes that each point on the exceptional line over $(1:-1:0:0:0)$ is fixed by $\gamma$.
\end{proof} 

Since the divisor on $W_q$ given by $V(z_4 = z_5)$ is isomorphic to $\sC_q$ and the projective
line providing the fixed exceptional divisor has $CM$, one has by Corollary $\ref{cmcm}$:

\begin{Theorem} \label{sw2}
By the involution $\gamma$, the family $\sW$ can be used to be some $\sZ_1$ or $\Sigma_i$ in the
construction of a Borcea-Voisin tower.
\end{Theorem}

\begin{remark} \label{cm3}
By Example $\ref{ellcm}$, Example $\ref{ellcm2}$ and Example $\ref{ellcm3}$, one has 6
explicitely given elliptic curves with $CM$ and explicitely given involutions. Theorem $\ref{geilomat}$ yields the $K3$ surfaces isomorphic to 
$$V(y_2^3+y_1^3+x_1^6+x_0^6), \  \ V(y_2^3+y_1^3+x_1(x_1^5+x_0^5)),
\  \ V(y_2^3+y_1^3+x_1(x_1^4+x_0^4)x_0) \subset  R^2$$
with complex multiplication. Thus by using the method of C. Voisin \cite{Voi2}, one
obtains 18 explicitely given fibers with $CM$ for the resulting $CMCY$ family of 3-manifolds. 
\end{remark}

\chapter{Other examples and variations}

In this chapter we consider the automorphism groups of our examples of $CMCY$ families. We want to
find some new examples of $CMCY$ families of $n$-manifolds by quotients by cyclic subgroups of
these automorphism groups. By using \cite{EV1}, Lemma $3.16,d)$, one can easily determine
the character of the action of these cyclic groups on the global sections of the canonical sheaves
of the fibers. In this chapter we state this character with respect to the pull-back action.

\section{The degree 3 case}
Let $\xi$ denote a fixed primitive cubic root of unity. In $\ref{8.3.1}$ we have constructed the
$CMCY$ family $\sW \to \sM_3$ given by
$$R^2 := \tilde \bP_{\C}(2,2,1,1) \supset \tilde V(y_2^3+ y_1^3 + x_1(x_1-1) (x_1-a_1x_0)(x_1-a_2x_0)(x_1-a_3x_0)x_0)$$
$$\to(a_1,a_2,a_3)\in \sM_3.$$
First we introduce an $\sM_3$-automorphism group $\G_3$ of the family $\sW$. The elements
$g\in \G_3$ can be uniquely written as a product $g = abc$ with $a \in \langle\alpha\rangle$, $b
\in \langle\beta\rangle$ and $c \in \langle\gamma\rangle$, where:
$$\alpha (z_5:z_4:z_3:z_2:z_1) = (\xi z_5:z_4:z_3:z_2:z_1),$$
$$\beta (z_5:z_4:z_3:z_2:z_1) = ( z_5:\xi z_4:z_3:z_2:z_1),$$
$$\gamma (z_5:z_4:z_3:z_2:z_1) = (z_4:z_5:z_3:z_2:z_1)$$
The group $\G_3$ contains exactly 18 elements. The action of $\G_3$ on the global
sections of the canonical sheaves of the fibers induces a surjection of $\G_3$ onto the
multiplicative group of the 6-th. roots of unity. Its kernel is the cyclic group of order 3
generated by $\alpha \beta^{-1}$.

\begin{remark}
Since $\alpha \beta^{-1}$ is an $\sM_3$-automorphism, one obtains the quotient family
$\sW/\langle\alpha \beta^{-1}\rangle \to \sM_3$. One checks easily that $\alpha \beta^{-1}$ leaves
exactly the sections given by $z_5 = z_4 = 0$
invariant. Let $q \in \sM_3$. The fiber $(\sW/\langle\alpha \beta^{-1}\rangle)_q$ of
$\sW/\langle\alpha \beta^{-1}\rangle$ has quotient singularities of the type $A_{3,2}$
(see \cite{Bart}, {\bf III.}  Proposition 5.3). We blow up the sections of fixed points on
$\sW$ and call the resulting exceptional divisor $E_1$. On each connected component of $E_1$ one
has two disjoint sections of fixed points again. But on a fiber the quotient map sends any fixed
point onto a singularity of the
type $A_{3,1}$.\footnote{For this description consider the corresponding action of the cyclic
group on an analytic open neighborhood of a fixed point.} Hence let us blow up these latter
sections of fixed points with exceptional divisor $E_2$. The canonical divisor of the resulting
fibers $\tilde\sW_q$ is given by
$$K_{\tilde\sW_q} = (\tilde E_1)_q + 2(E_2)_q,$$
where quotient map $\varphi$ induced by $\alpha \beta^{-1}$ has ramification on $E_2$. Thus by the
Hurwitz formula, one calculates that $\varphi^*(\omega_q) = \sO((\tilde E_1)_q)$. Note that
the irreducible components of the exceptional curve $(E_1)_q$ have selfintersection-number $-1$.
Since $(E_2)_q$ is the exceptional
divisor of the blowing up of two points of each irreducible component of $(E_1)_q$, each
irreducible component of $(\tilde E_1)_q$ has selfintersection-number $-3$. By the fact that the
quotient map $\varphi: \tilde\sW_q \to (\tilde \sW/\langle\alpha \beta^{-1}\rangle)_q$ is not
ramified over $\varphi((\tilde E_1)_q)$, the irreducible components of $\varphi((\tilde E_1)_q)$
have selfintersection-number $-1$.
\end{remark}

From now on let $\sX := \tilde \sW/\langle\alpha \beta^{-1}\rangle$.

\begin{proposition} \label{pabu}
One can blow down $\varphi(\tilde E_1)$ such that the blowing down morphism $\phi :\sX \to \sY$
yields a $CMCY$ family $\sY \to \sM_3$ of 2-manifolds.
\end{proposition}
\begin{proof}
By the construction of the projective family, one has an invertible relatively very ample
sheaf $\sA:=\sO_{\sX}(D)$ on $\sX$. Let $P$ denote some connected component of
$\varphi(\tilde E_1)$. Note that $\varphi(\tilde E_1)$ consists of different
copies of $\bP_{{\rm Spec}(R)}^1$ with ${\rm Spec}(R) = \sP_n$ such that each invertible sheaf on
$P$ is uniquely determined by its degree. Thus the intersection number $\mu_P := D_q.P_q$ is
independent of $q \in \sP_n$.  As in the proof of the Castelnuovo
Theorem in \cite{hart}, {\bf V}. Theorem $5.7$ the invertible sheaf
$$\sL := \sA(\sum\limits_{P \subset \varphi(\tilde E_1)}\mu_PP)$$
yields the blowing down morphism on the fibers. Since this $\sP_n$-morphism is globally defined,
one obtains a global blowing down morphism $f$ such that the resulting
family $\sY = f(\sX)$ is smooth.

By the fact that $\alpha\beta$ acts by the character $1$ on $\Gamma(\omega_{\sW_q})$, one
concludes easily that $\sY \to \sM_3$ is a family of $K3$ surfaces.
Since $\sW$ has a dense set of $CM$ fibers, one concludes that $\sX = \tilde
\sW/\langle\alpha\beta\rangle$ and $\sY$ have dense sets of $CM$ fibers, too.
\end{proof}

By the blowing down of $\varphi(\tilde E_1)$, we get the following situation:
$$\xymatrix{
{\tilde E_1 \cup E_2} \ar@{^{(}->}[d]_{} \ar@{->>}[rr]^{\varphi} &   &
{\varphi(\tilde E_1 \cup E_2)} \ar @{^{(}->}[d]_{} \ar@{->>}[rr]^{\phi} &   &
{\phi \circ\varphi(E_2)} \ar@{^{(}->}[d]^{}\\
{\tilde \sW} \ar[rr]_{{\rm mod}\langle\alpha \beta^2\rangle}^{\varphi} &   &
{ \sX} \ar[rr]_{\rm{Bl}(\varphi(\tilde E_1))}^{\phi} &   & { \sY} 
}$$

\begin{proposition} \label{sy1}
The $\sM_3$-automorphism $\gamma$ of $\sW$ yields an involution on $\sY$, which makes it suitable
for the construction of a Borcea-Voisin tower.
\end{proposition}
\begin{proof}
One has the following commutative diagram:
$$\xymatrix{
{\tilde \sW} \ar[d]_{\gamma} \ar[rr]^{\alpha \beta^{-1}} &   & {\tilde \sW} \ar[d]^{\gamma}\\
{\tilde \sW} \ar[rr]^{\alpha^{-1} \beta} &   & {\tilde \sW}
}$$
Thus $\gamma$ yields an involution on $\sX = \tilde \sW/\langle\alpha \beta^{-1}\rangle$. By the
fact that $\gamma(E_1) = E_1$, it induces an involution on the complement of the sections of $\sY$
obtained by blowing down $\varphi(\tilde E_1)$.
Since these sections have codimension 2, the involution extends to a holomorphic involution on
$\sY$ (by Hartogs' Extension Theorem \cite{Voi}, Theorem $1.25$). By the fact that $\gamma$ acts
by $-1$ on $\Gamma(\omega_{\sW_q})$, the same holds true for $\sX_q$ and $\sY_q$.

Let $\sC \to \sM_3$ denote the family of degree 3 covers with a pure $(1,3)-VHS$. We have seen
that $\sW_q$ has $CM$, if $\sC_q$ has $CM$. Hencefore $H^k(\sY_q,\Q)$ has a commutative Hodge group
for all $k$, if $\sC_q$ has $CM$. Thus the following point describes  the ramification divisor of
$\gamma_q$ on $\sY_q$ and ensures that there is a dense set of $CM$ fibers $\sY_q$ such that
the ramification divisor of $\gamma_q$ has $CM$, too.
\end{proof}

\begin{pkt} \label{sy2}
Now we describe the divisor of points of $\sY_q$ fixed by $\gamma$ for some $q \in \sM_3$. Each
point of $\sY_q \setminus (\phi \circ\varphi(E_2))$ can be given by the image $[p]$ of a point
$p \in \sW_q$
with respect to the quotient map according to $\langle\alpha\beta^{-1}\rangle$.
One has that a point $[p] \in \sY_q \setminus (\phi \circ\varphi(E_2))$ is fixed by $\gamma$,
if and only if $\gamma(p) \in \langle\alpha\beta^2\rangle\cdot p$. These points $p \in \sW_q$ are
exactly given by $\langle\alpha\beta^2\rangle\cdot V(y_2 = y_1)$ and the exceptional divisor of
$\sW_q \to W_q$. 

By the fact that $\langle\alpha\beta^2\rangle\cdot V(y_2 = y_1)$ interchanges all 3 irreducible
components of $\langle\alpha\beta^2\rangle\cdot V(y_2 = y_1)$ and all 3 irreducible components of
the exceptional divisor of
$\sW \to W$, one obtains a divisor of fixed points on $\sY_q$ given by $\sC_q$ and one copy of
$\bP^1$. Since $\gamma$ is given by $(y_2:y_1) \to (y_1:y_2)$ on $E_1$ and $\alpha \beta^2$ is
given by $(y_2:y_1) \to (y_2:\xi y_1)$ on $E_1$, $\gamma$ interchanges each two irreducible
components of $E_2$, which intersect the same irreducible component of $\tilde E_1$. Thus the
ramification divisor of $\sY \to \sY/\gamma$ given by a family of rational curves and $\sC$, where
$\sC$ denotes the example of a family of degree 3 covers with a pure $(1,3)-VHS$.
\end{pkt}

\section{Calabi-Yau 3-manifolds obtained by quotients of degree 3}
We have seen that the family $\sW$ of $K3$-surfaces given by
$$R^2 := \tilde \bP_{\C}(2,2,1,1) \supset \tilde V(y_2^3+ y_1^3 + x_1(x_1-1) (x_1-a_1x_0)(x_1-a_2x_0)(x_1-a_3x_0)x_0)$$
$$\to(a_1,a_2,a_3)\in \sM_3$$
has a dense set of fibers $\sW_q$ such that $H^k(\sW_q,\Q)$ has a commutative Hodge group for all
$k$.

Recall that the canonical divisor of $R^1 \cong \bP(\sO \oplus \sO(2))$ is given by $-2\tilde V(z_4)$.
Now we consider the up to isomorphisms unique cyclic  cover of degree 3 given by $\sW_q \to R^1$
ramified over $\sC_q$, whose Galois group is generated by $\alpha$. Moreover consider the cyclic
degree 3 cover $\F_3 \to \bP^1$, where $\F_3 = V(x^3 + y^3 + z^3) \subset \bP^2$ denotes the
Fermat curve of degree 3 and $\alpha_{\F_3}$ given by
$$(x:y:z) = (x: y:\xi z),$$
is a generator of the Galois group, which acts by the character $\xi$ on $\Gamma(\omega_{\F_3})$.

Let $X$ be a singular variety of dimension $n$ such that each irreducible component of its
singular locus $S$ has at least the codimension 2. Then we call $X$ a singular Calabi-Yau
$n$-manifold, if $h^0(X\setminus S,\Omega_{X\setminus S}^k) = 0$ for all $k = 1, \ldots, n-1$ and
$\omega_{X\setminus S} \cong  \sO_{X\setminus S}$. With the notation of diagram $\eqref{diag}$ one
gets: 

\begin{proposition}
The quotient of $\sW \times \F_3$ by $\langle(1,2)\rangle$ yields a family of singular Calabi-Yau
3-manifolds with a dense set of $CM$ fibers.
\end{proposition}
\begin{proof}
Note that the $VHS$ of the family $\sW \times \F_3/\langle(1,2)\rangle$ is the sub-$VHS$ fixed by
$\langle(1,2)\rangle$.\footnote{For a short introduction to such orbifolds and their
Hodge theory see \cite{CK}, Appendix $A.3$.} Since $\F_3$ has complex multiplication, a $CM$ fiber
of $\sW$ yields a corresponding $CM$ fiber of $\sW \times \F_3/\langle(1,2)\rangle$.

Let $\varphi$ denote the quotient map
$$\varphi :\sW \times \F_3\to \sW \times \F_3/\langle(1,2)\rangle$$
and $S$ denote the singular locus of $\sW \times \F_3/\langle(1,2)\rangle$. Over each point, which lies
not in the singular locus given by 3 copies of $\sC$, one does not have ramification. Hence by the
Hurwitz formula, $\varphi^*(\omega_{(\sW \times \F_3/\langle(1,2)\rangle)\setminus S})$ is given
by the structure sheaf. Since
$\langle(1,2)\rangle$ acts on $\Gamma(\omega_{\sW \times \F_3})$ by the character 1, the sheaf
$\omega_{(\sW \times \F_3/\langle(1,2)\rangle)\setminus S}$ has global sections. Hence $$ \omega_{(\sW \times \F_3/\langle(1,2)\rangle)\setminus S} =
\sO_{(\sW \times \F_3/\langle(1,2)\rangle)\setminus S}.$$
In addition the reader checks easily that $\langle(1,2)\rangle$ does not act by the character 1 on
a non-trivial sub-vector space of $H^{1,0}(\sW \times \F_3)$ or $H^{2,0}(\sW \times \F_3)$.
Thus $\sW \times \F_3/\langle(1,2)\rangle$ is a family of singular Calabi-Yau 3-manifolds.
\end{proof}

\begin{pkt} \label{23}
Now consider a fiber $(\sW \times \F_3/\langle(1,2)\rangle)_q$ of
$\sW \times \F_3/\langle(1,2)\rangle$ and its
singularities in the complex analytic setting. For the construction of the blowing up of a complex
submanifold we refer to \cite{Voi}, $3.3.3$. As in \cite{Voi}, $3.3.3$ described, one constructs
first the blowing up over open sets. The global blowing up is given by glueing the local blowing
ups. Here were consider the situation on sufficiently small complex open submanifolds.

The $\sM_3$-automorphism $\alpha$ acts on $y_2$ by $\xi$. On each fiber $\sW_q$
the curve $\sC_q$ defines the ramification locus of $\sW_q \to R^2$, which is fixed
by $\alpha$. A local parameter $p_{\sC_q}$ on $\sC_q$ yields a local parameter on $\sW_q$ fixed
by $\alpha$. By $z$, one has a local parameter for the neighborhoods of the ramification points of
$\F_3$. On a small open subset, which intersects the ramification locus of
$$\varphi_q:(\sW \times \F_3)_q \to (\sW \times \F_3/\langle(1,2)\rangle)_q,$$
one has the three local parameters given by $y_2$, $p_{\sC_q}$ and $z$. By the action of
$\langle(1,2)\rangle$ on these three local parameters, the singularities
$\sW \times \F_3/\langle(1,2)\rangle$ are
locally given by the product of the 1-ball $\B_1$ with a surface, which has a singularity of the
type $A_{3,2}$ (with the notation in \cite{Bart}, {\bf III}. Subsection 5). Let us blow up the family of fixed curves on
$\sW \times \F_3$ with respect to $\langle(1,2)\rangle$ and let $E_1$ denote the exceptional
divisor. On each connected component of $E_1$ one has two disjoint families of fixed curves with
respect to the action of $\langle(1,2)\rangle$ again. Again this follows from the consideration of
the action of $\langle(1,2)\rangle$ on local parameters of a small open subset. On a fiber the
quotient map sends any neighborhood of a point on
these latter curves onto the product of the 1-ball $\B_1$ with a surface with a singularity of the
type $A_{3,1}$. Hence let us blow up these latter two families of curves with exceptional
divisor $E_2$. The canonical divisor of the resulting fibers $(\widetilde{\sW \times \F_3})_q$ is
given by
$$K_{(\widetilde{\sW \times \F_3})_q} = (\tilde E_1)_q + 2(E_2)_q,$$
where quotient map $\varphi$ by $\langle(1,2)\rangle$ is ramified over $E_2$. Thus by the
Hurwitz formula, one calculates that $\varphi^*(\omega_q) = \sO((\tilde E_1)_q)$.\footnote{The
author has searched for an opportunity of a smooth blowing down similar to Proposition
$\ref{pabu}$. He considered a fiber $\sW_q$, which is a family of curves given by
$$\sW_q \to R^2 \to \bP^1.$$
But here we do not blow up sections of $\sW_q \times \F_3 \to \bP^1$. Hence here one can not
formulate a relative version of the Castelnuovo Theorem as in Proposition $\ref{pabu}$.}
\end{pkt}

On the other hand $\alpha \beta$ acts by the character $\xi^2$ on
$\Gamma(\omega_{\sW_q})$ for all $q \in \sM_3$. Moreover we have a Galois
cover $\F_3 \to \bP^1$ of degree 3 with a
generator $\alpha_{\F_3}$ given by
$$(x:y:z) = (x: y:\xi z),$$
which acts by the character $\xi$ on $\gamma(\omega_{\F_3})$.
Hence $\alpha_2 := (\alpha \beta ,\alpha_{\F_3})$ leaves $\Gamma(\omega_{\sW_q \times \F_3})$
invariant.

The automorphism $\alpha_2$ fixes a finite number of points on $\sW_q \times \F_3$ given by
$$\{z_5 = z_4 = 0\}\times \{z= 0\},$$
and $\alpha_2$ fixes in addition the points on the curves given by the fiber product of $\{z= 0\}$
with the
exceptional divisor of the blowing up $\sW_q \to W_q$. The latter statement about the exceptional
divisor of $\sW_q \to W_q$ follows from the fact that $\alpha\beta$ fixes the generators
of the corresponding ideal sheaf of the blowing up  and the singular points of $W_q$ given by
$$ (1:-1:0:0:0), \  \ (1:-\xi :0:0:0) \ \ \mbox{and} \  \ (1:-\xi^2 :0:0:0).$$

\begin{pkt}
Now we determine the action of $\alpha \beta$ on the local parameters, whose zero-loci
are given by the exceptional divisor $E_{\sW_q}$ of $\sW_q \to W_q$. The
action of $\alpha\beta$ on $W_q \subset \bP^4$ is given by
$$(z_5:z_4:z_3:z_2:z_1) \to (\xi z_5: \xi z_4:z_3:z_2:z_1) \  \ \mbox{resp.},$$
$$(z_5:z_4:z_3:z_2:z_1) \to
(z_5:z_4:\xi^{-1}z_3:\xi^{-1}z_2:\xi^{-1}z_1).$$
By using the explicit equations for $W_q$ in $\ref{8.3.1}$, one can very easily calculate that
$\alpha\beta$ acts by $\xi^{-1}$ on these local parameters.\footnote{The singular locus of $W_q$
is contained in $W_q \cap \{z_5 = 1\}$. Thus one can calculate the desingularization with the usual
equations $z_it_j = z_jt_i$ for $i,j = 1,2,3$. On $\{t_i =1\}$ the zero locus of the local
parameter $z_i$ yields the exceptional divisor. The local parameter fixed by $\alpha\beta$ can
be given by $t_1/t_i$ or $t_3/t_i$.}

Hence the singularities of $\sW_q \times \F_3/\langle\alpha_2\rangle$, which result by the
exceptional divisor of $\sW_q \to W_q$, are locally given by the product of $\B_1$ with a
singularity of the type $A_{3,2}$.
\end{pkt}

Now we construct a desingularisation of ${\sW \times \F_3}/\langle\alpha_2\rangle$, which is a
$CMCY$ family of 3-manifolds. Let $E_{\sW}$ denote the exceptional divisor of $\sW \to W$. We
start with the blowing up of the family of rational curves
given by the fiberproduct of $E_{\sW}$ with the points on $\F_3$ fixed by $\alpha_{\F_3}$.
This yields the exceptional divisor $E_C$ consisting of 9 rational ruled surfaces. By the same
arguments as in $\ref{23}$, each connected component of $E_C$ contains two families of rational
curves of fixed points. The blowing up $\widetilde{\sW \times \F_3}$ of these latter
families has a quotient
$$\sR: = \widetilde{\sW \times \F_3}/\alpha_2$$
with quotient map given by $\varphi$ such that on the complement of the isolated sections fixed by
$\varphi$
$$\varphi^*\omega_q = \sO((\tilde E_C)_q).$$

\begin{pkt}
Recall that $R^1$ is a rational ruled surface, where the exceptional divisor $E_{R^1}$ of the
blowing up $R^1 \to
Q^1$ is a section of $R^1 \to \bP^1$ (see Remark $\ref{regeli}$). A fiber $\sW_q$ can be
considered as a family
$$\sW_q \stackrel{f}{\to} R^1 \to \bP^1$$
of curves, where $f$ is constructed in $\ref{wer}$. By $\ref{8.3.1}$ and the projection
$R^2 \to R^1$, the morphism $f$ extends to
a morphism $f: \sW \to R^1 \times \sM_3$ such that the exceptional divisor $E_{\sW}$ of the
blowing up $\sW \to W$ is send to the exceptional divisor $E_{R^1\times \sM_3} =
E_{R^1}\times \sM_3$ of the blowing up $R^1\times \sM_3 \to Q^1\times \sM_3$. The following
commutative diagram describes the situation:
$$\xymatrix{
{E_{\sW}} \ar[d]_{f} \ar[rr]^{} &   & {\sW} \ar[d]_{f} \ar[rr]^{}
&   & {W} \ar[d]_{f}\\
{E_{R^1}\times \sM_3} \ar[d]_{} \ar[rr]^{} &   & {R^1\times \sM_3} \ar[d]_{} \ar[rr]^{}
&   & {Q^1\times \sM_3}\\
{\bP^1\times \sM_3} \ar[rr]^{\rm id} &   & {\bP^1\times \sM_3}
}$$
\end{pkt}

\begin{pkt}
Thus
$$g: \sW \stackrel{f}{\to} R^1 \times \sM_3 \to \bP^1 \times \sM_3$$
is a family of curves, which has 3 distinguished sections given by the exceptional divisor
$E_{\sW}$ of $\sW \to W$. Moreover by the description of $f:\sW_q \to R^1$ as degree 3
cover, one can easily see that the fibers of $g$ are given by the Fermat curve of degree
3 or consist of 3 smooth rational curves intersecting each other in exactly one point, which does
not lie on $(E_{\sW})_q$. Over $\bP^1\setminus \{\infty\} \times \sM_3$ and
$\bP^1\setminus \{0\} \times \sM_3$ one can embed the restricted family into some copy of
$\bP^2_{\A^1\times \sM_3}$.

Hencefore we obtain the family
$$\sW \times \F_3 \to \bP^1 \times \sM_3$$
of surfaces, which has sections given by the fiberproduct of the exceptional divisors of
$\sW \to W$ with the points fixed by $\alpha_{\F_3}$, which do not meet any singular point of
a fiber. In addition $\alpha_2$ is a $\bP^1 \times \sM_3$-automorphism of this family.
Hence by the same arguments as in the proof of Proposition $\ref{pabu}$, we can blow down
$\varphi(\tilde E_C)$ over $(\bP^1\setminus \{\infty\}) \times \sM_3$ and
$(\bP^1\setminus \{0\}) \times \sM_3$. By glueing, we obtain the family $\hat \sQ$. Note that the
singular fibers of $\sW \times \F_3 \to \bP^1 \times \sM_3$ are given by 3 copies of
$\bP^1 \times \F_3$. Hence by the restriction of the sheaf, which yields the blowing down morphism,
to the corresponding copies of $\widetilde{\bP^1 \times \F_3}/\langle\alpha_2\rangle$, one obtains
smooth blowing down morphisms on these copies.
\end{pkt}

\begin{construction}
But $\hat \sQ$ has 18 sections of singular points given by the 18 isolated sections fixed by
$\alpha_2$ on $\widetilde{\sW \times \F_3}$. Recall that these sections are given by
$$\{z_5 = z_4 = 0\}\times \{z= 0\}.$$
Let $\sQ \to \hat \sQ$ denote the blowing up of the singular sections of $\hat \sQ$ and
$$\widetilde{\widetilde{\sW \times \F_3}}\to \widetilde{\sW \times \F_3}$$
denote the blowing up of these 18 sections. By the same arguments as in Remark $\ref{buahaha}$,
we obtain the following commutative diagram:
$$\xymatrix{
{\widetilde{\widetilde{\sW}}} \ar[d]_{} \ar[rr]^{\tilde \varphi} &   & {\sQ} \ar[d]_{}\\
{\widetilde{\sW}} \ar[rr]^{\varphi} &   & {\hat \sQ}
}$$
Note that $\tilde \varphi$ is a cyclic cover on the complement of $\tilde E_C$. Thus by the
Hurwitz formula and the fact that $\alpha_2$ acts by the character 1 on
$\Gamma(\omega_{\sW_q \times \F_3})$ for each $q \in \sM_3$, one concludes that $\sQ$ is a family
of Calabi-Yau 3-manifolds.
\end{construction}

\begin{proposition} \label{pubc}
The family $\sQ \to \sM_3$ is a $CMCY$ family of 3-manifolds.
\end{proposition}
\begin{proof}
Note that on each fiber we blow up some points and several copies of $\bP^1$, which have $CM$.
Hence by Theorem $\ref{hiercm}$, we must only
apply the facts that $\F_3$ has $CM$ and $\sW$ has a dense set of fibers $\sW_q$ such
that $\Hg(H^k(\sW_q,\Q))$ is commutative for all $k$.
\end{proof}

\section{The degree 4 case}
Consider the $CMCY$ family $\sC_2 \to \sM_1$ of 2-manifolds given by
$$\bP^3 \supset V(y_2^4 + y_1^4 + x_1(x_1-x_0)(x_1-\lambda x_0)x_0) \to \lambda \in \sM_1$$
of Section $7.4$. In this section we construct quotients of $\sC_2$ by cyclic subgroups of its
group of $\sM_1$-automorphisms, which will be suitable to obtain new $CMCY$ families of
2-manifolds. In the next section we will see that these new examples are endowed with involutions,
which make them suitable for the construction of the Borcea-Voisin tower. Hence by the Hurwitz
formula and some other obvious reasons, one has:

\begin{claim} \label{invos}
Let $C$ be a $K3$ surface and $\alpha$ be an involution on $C$,
which admits a finite set $S$ of fixed
points on $C$. Then the quotient $\tilde C/\alpha$, where $\tilde C$
denotes the blowing up of $C$ with respect to the subvariety given
by $S$, is a $K3$ surface, too. Moreover $\tilde C/\alpha$ has complex multiplication resp., only
commutative Hodge groups, if $C$ has complex multiplication
resp., only commutative Hodge groups.
\end{claim}

Now we introduce a group $\G_4$ of $\sM_1$-automorphisms of the $CMCY$ family $\sC_2 \to \sM_1$.
The elements $g\in \G_4$ can be uniquely written as a product $g = abc$ with
$a \in \langle\alpha\rangle$, $b \in \langle\beta\rangle$, and $c \in \langle\iota_4\rangle$, where:
$$\alpha (y_2: y_1: x_1: x_0) = (iy_2: y_1: x_1: x_0), \ \
\beta  (y_2: y_1: x_1: x_0) = (y_2: iy_1: x_1: x_0),$$
$$\iota_4 (y_2: y_1: x_1: x_0) = (y_1: y_2: x_1: x_0)$$
Therefore the group $\G_4$ contains exactly 32 elements. The action of $\G_4$ on the global
sections of the canonical sheaves of the fibers induces a surjection of $\G_4$ onto the
multiplicative group of the 4-th. roots of unity.

Its kernel $\KK_4$ is a normal subgroup of order 8. It contains the following automorphisms of
order 4:
$$\delta  (y_2: y_1: x_1: x_0) = (-y_1: y_2: x_1: x_0),\ \
\epsilon   (y_2: y_1: x_1: x_0) = (iy_2: -iy_1: x_1: x_0),$$
$$\eta  (y_2: y_1: x_1: x_0) = (iy_1: iy_2: x_1: x_0)$$
One has that
$$\iota_3 = \delta^2 = \epsilon^2 = \eta^2 = (\alpha \beta)^2.$$
Moreover one checks easily that $\KK_4$ is isomorphic to the quarternion group and has the
generators $\delta$, $\epsilon$ and $\eta$. Thus one has
\begin{equation} \label{klein}
\KK_4/\langle\iota_3 \rangle = (\Z/2)^2.
\end{equation}

One can easily calculate that
$$\alpha \langle\delta\rangle \alpha^{-1} = \langle\eta\rangle.$$
By the fact that $\KK_4$ has 2 residue classes with respect to $\langle\delta\rangle$ resp.,
$\langle\epsilon\rangle$ resp., $\langle\eta\rangle$, one concludes that $\langle\delta\rangle$
resp., $\langle\epsilon\rangle$ resp., $\langle\eta\rangle$ is a normal subgroup of
$\KK_4$. Since $[\alpha ]_{\KK_4}$ generates $\G_4/\KK_4$ and
$$\alpha \langle\epsilon\rangle \alpha^{-1} = \langle\epsilon\rangle,$$
$\langle\epsilon\rangle$ is a normal subgroup of $\G_4$.

\begin{pkt} \label{hier}
Recall that $\iota_3$ denotes the involution given by
$$\iota_3(y_2: y_1: x_1: x_0) = (-y_2: -y_1: x_1: x_0).$$
Let $\sC_{\langle\iota_3\rangle}$ be the $CMCY$ family of 2-manifolds given by the quotient
$\tilde \sC_2/\langle\iota_3\rangle$, where $\tilde \sC_2$
denotes the blowing up of $\sC_2$ with respect to the 8 sections fixed by $\iota_3$. Four
sections fixed by $\iota_3$ are given by $(1:\zeta :0:0)$, where $\zeta$ runs through the
primitive 8-th. roots of unity. The other 4 sections are given by
$$(0:0:0:1), \ \  (0:0:1:1), \  \ (0:0:\lambda :1) \  \ \mbox{and} \  \ (0:0:1:0).$$
Since the generators $\alpha$, $\beta$ and $\iota_4$ of $\G_4$ leave the ideal sheaf corresponding
to these 8 sections invariant, all automorphisms of $\G_4$ induce automorphisms on $\tilde \sC_2$.
Note that $\iota_3$ commutes with each $\tau \in \G_4$. For each $\tau \in \G_4$ one finds open
affine subsets invariant under $\langle\tau, \iota_3\rangle$. On these affine sets the global
sections of the structure sheaf invariant under $\langle\tau, \iota_3\rangle$ are contained in
$\sO^{\langle\iota_3\rangle}$, where
$\tau$ leaves $\sO^{\langle\iota_3\rangle}$ invariant. Hencefore $\tau$ induces an
automorphism on $\sC_{\langle\iota_3\rangle}$. One checks easily that $\delta$, $\eta$ and $\epsilon$
yield involutions on $\sC_{\langle\iota_3\rangle}$ leaving only finitely many sections fixed.
Thus by using Claim $\ref{invos}$, these involutions yield the $CMCY$ families of 2-manifolds
$$\sC_{\langle\delta\rangle} \cong \sC_{\alpha \langle\delta\rangle \alpha^{-1}} =
\sC_{\langle\eta\rangle} \  \ \mbox{and} \  \ \sC_{\langle\epsilon\rangle}.$$
\end{pkt}

\section{Involutions on the quotients of the degree 4 example}
In Section $7.4$ we introduced several $\sM_1$-involutions $\iota_1, \ldots, \iota_7$ of $\sC_2$.
We have seen that $\iota_3$ acts by the character $1$ on the global sections of the canonical
sheaves of the fibers. Moreover $\iota_1, \iota_2,\iota_4,\ldots, \iota_7$ act by the character
$-1$ on the global sections of the canonical sheaves of the fibers. Here we show that each for
each $i=1,2,4, \ldots, 7$ the involution $\iota_i$ induces $\sM_1$-involutions on the quotient
families of $\ref{hier}$, which make them suitable for the construction of a Borcea-Voisin tower.

\begin{remark} \label{4fall}
One can use Example $\ref{ellcm}$, Example $\ref{ellcm2}$ and Example $\ref{ellcm3}$ and determine
some explicitely given $CM$ fibers of the new quotient families. By using the method of C. Voisin \cite{Voi2},
these new $K3$ surfaces with complex multiplication and our explicite examples of elliptic curves
with complex multiplication yield new Calabi-Yau 3-manifolds with complex muyltiplication
\end{remark}

We fix some new notation: Let $C_2$ be an arbitrary fiber of $\sC_2$, $p \in C_2$, where $p$ is
not fixed by $\iota_3$, and $F_i$ denote the curve of fixed points on $C_2$ with respect to
$\iota_i$ for all $i = 1, 2, 4,\ldots, 7$.

\begin{pkt} \label{nr1}
The involutions $\iota_1$ and $\iota_2$ induce the same involution on
$\sC_{\langle\iota_3\rangle}$. One has that $\iota_1([p]_{\langle\iota_3\rangle}) =
[p]_{\langle\iota_3\rangle}$, if and only if $p \in F_1 \cup F_2$. The involution $\iota_3$
induces an involution on the curve $F_1$ and on the curve $F_2$. Each of
the covers induced by these involutions has 4 ramification points. Hence by the Hurwitz formula,
$\iota_1$ induces an involution on $\sC_{\langle\iota_3\rangle}$, which has a divisor of fixed
points containing two families of elliptic curves. By \cite{Voi2}, $1.1$, the ramification divisor
of our involution on a fiber of $\sC_{\langle\iota_3\rangle}$ has at most one irreducible
component of genus $g>0$ or consists of two elliptic curves. Thus it consists of two elliptic
curves. It is quite easy to check that by this involution $\iota_1$, the family
$\sC_{\langle\iota_3\rangle}$ is suitable for the construction of a Borcea-Voisin tower.
\end{pkt}

\begin{pkt} \label{nr2}
The involutions $\iota_4$ and $\iota_6$ induce the same involution on
$\sC_{\langle\iota_3\rangle}$. One has that $\iota_4([p]_{\langle\iota_3\rangle}) =
[p]_{\langle\iota_3\rangle}$, if and only if $p \in F_4 \cup F_6$.
The involution $\iota_3$ induces an involution on the curve $F_4$ and on the curve $F_6$. Each of
the covers induced by these involutions have 4 ramification points. Hence by the same arguments as
in $\ref{nr1}$, the involution $\iota_4$ induces an involution on $\sC_{\langle\iota_3\rangle}$,
which has a divisor of fixed points consisting of two families of elliptic curves. It is quite
easy to check that by this involution $\iota_1$, the family $\sC_{\langle\iota_3\rangle}$ is
suitable for the construction of a Borcea-Voisin tower.

Since $\alpha \iota_4 \alpha^{-1} =\iota_5$ and $\alpha \iota_3 \alpha^{-1} =\iota_3$, the
involutions $\iota_5$ and $\iota_7$ induce up isomorphisms the same involution as $\iota_4$ and
$\iota_6$ on $\sC_{\langle\iota_3\rangle}$.
\end{pkt}

Recall the $\sM_1$-automorphisms
$$\delta  (y_2: y_1: x_1: x_0) = (-y_1: y_2: x_1: x_0),\ \
\epsilon   (y_2: y_1: x_1: x_0) = (iy_2: -iy_1: x_1: x_0)$$
of $\sC_2$ of order 4.

\begin{remark}
Now we consider the quotient families $\sC_{\langle\delta\rangle}$ and
$\sC_{\langle\epsilon\rangle}$ in $\ref{hier}$.
One checks easily that $\delta$ and $\epsilon$ act as involutions on the 4 sections given by
$(1:\zeta :0:0)$, where $\zeta$ runs through the primitive 8-th. roots of unity, and leave the
sections given by
$$(0:0:0:1), \ \  (0:0:1:1), \  \ (0:0:\lambda :1), \  \ (0:0:1:0)$$
invariant.

Moreover there does not exist a point $p\in C_2$ such
that $\delta(p) = \iota_3(p)$ or $\epsilon(p) = \iota_3(p)$. This follows from the facts
that $\iota_3 = \delta^{2} = \epsilon^{2}$
and $\delta^{2}(p) = \delta(p)$ resp., $\epsilon^{2}(p)= \epsilon(p)$ would imply that $\delta$
resp., $\epsilon$ is not bijective.

Hencefore either $p$ is contained in one of the 8 sections fixed by $\iota_3$ or
$\langle\delta\rangle \cdot p$ and $\langle\epsilon\rangle  \cdot p$ contain 4 different elements.
For our notation we will assume that $p$ is not fixed by $\iota_3$ as above.
\end{remark}

\begin{pkt} \label{pupu}
The involutions $\iota_1$ and $\iota_2$ commute with $\epsilon$. Thus the same holds true with
respect to the involutions on $\sC_{\langle\iota_3\rangle}$ induced by $\iota_1$, $\iota_2$ and
$\epsilon$. Hence one concludes that $\iota_1$ and $\iota_2$ induce an involution an
$\sC_{\langle\epsilon\rangle}$. Since $\iota_1$ and
$\iota_2$ induce the same involution on $\sC_{\langle\iota_3\rangle}$, the involutions $\iota_1$
and $\iota_2$ induce the same involution on $\sC_{\langle\epsilon\rangle}$.

A point $[p]$ on the fiber $C_{\langle\epsilon\rangle}$ of $\sC_{\langle\epsilon\rangle}$ is fixed by
$\iota_1$, if $\iota_1(p) = \epsilon^i(p)$ for $i = 0, \ldots, 3$. This is exactly satisfied on
$F_1$ and $F_2$ for $i = 0$ or $i = 2$. The automorphism $\epsilon$ yields a quotient of $F_1$
resp., $F_2$ of degree 4 fully ramified over 4 points. Hence by the Hurwitz formula,
$F_1/\langle\epsilon\rangle$ and $F_2/\langle\epsilon\rangle$ are rational curves.

By the definitions of $\iota_1$ and $\epsilon$, one checks easily that their actions coincide
on the exceptional divisor on $\tilde \sC_2$ over the four sections given by $V(y_2,y_1)$.
Moreover by the definitions of $\iota_1$ and $\epsilon$, one checks easily that for each primitive
$8$-th. root $\zeta$ of unity
$$\iota_1(1:\zeta :0:0) = \epsilon(1:\zeta :0:0) = (1:-\zeta :0:0).$$
Both $\sM_1$-automorphisms fix the local parameters $x_1$ and $x_2$.

Thus altogether the involution $\iota_1$ induces an involution on $\sC_{\langle\epsilon\rangle}$,
which has a divisor of fixed points consisting of 8 disjoint families of rational curves. It is
quite easy to check that $\sC_{\langle\epsilon\rangle}$ is suitable for the construction of a
Borcea-Voisin tower by this involution. 
\end{pkt}

\begin{pkt} \label{epi4}
The involutions $\iota_4, \ldots, \iota_7$ do not commute with $\epsilon $. But one has
$\epsilon \iota_i = \iota_i \epsilon^3$ for all $i = 4, \ldots, 7$. Hence $\iota_i$ ($i = 4,
\ldots, 7$) induces an involution on $\sC_{\langle\iota_3\rangle}$. Since
$\iota_5 = \epsilon \iota_4$, $\iota_6 = \epsilon^2 \iota_4$ and $\iota_7 = \epsilon^3 \iota_4$,
these involutions induce the same involution on $\sC_{\langle\epsilon\rangle}$.

A point $[p] \in C_{\langle\epsilon\rangle}$ is invariant under
$\iota_4$, if $\iota_4(p) = \epsilon^i(p)$ for $i = 0, \ldots, 3$. One has that
$\iota_4(p) = (p)$ on $F_4$, $\iota_4(p) = \epsilon^1(p)$ on $F_7$, $\iota_4(p) = \epsilon^2(p)$
on $F_6$ and $\iota_4(p) = \epsilon^3(p)$ on $F_5$. Note that $\epsilon(F_4) = F_6$,
$\epsilon(F_6) = F_4$, $\epsilon^2(F_4)= F_4$ and $\epsilon^2(F_6)= F_6$. Moreover one has
$\epsilon(F_5) = F_7$, $\epsilon(F_7) = F_7$, $\epsilon^2(F_5)= F_5$ and
$\epsilon^2(F_5)= F_5$. The automorphism
$\epsilon^2 = \iota_3$ yields a quotient of $F_4$, $F_5$, $F_6$ resp., $F_7$ of degree 2 ramified
over 4 points, where $F_4$ and $F_6$ resp., $F_5$ and $F_7$ are mapped onto the same quotient by
$\epsilon$. Hence by the Hurwitz formula, the quotient consists of two families of elliptic curves.

By \cite{Voi2}, $1.1$, the ramification divisor of our involution on $C_{\langle\epsilon\rangle}$
has at most one irreducible component of genus $g>0$ or consists of two elliptic curves. Thus
$\iota_4$ induces an involution on $\sC_{\langle\epsilon\rangle}$, which has a divisor of
fixed points consisting of 2 families of elliptic curves. It is quite easy to check
that this involution makes $\sC_{\langle\epsilon\rangle}$ suitable for the construction of a
Borcea-Voisin tower. 
\end{pkt}

\begin{pkt} \label{del1}
The involutions $\iota_4$ and $\iota_6$ do not commute with $\delta $. But one has
$\delta \iota_4 = \iota_4 \delta^3$ and $\delta \iota_6 = \iota_6 \delta^3$. Moreover one has
$$\iota_1  = \delta \circ \iota_4, \ \ \iota_6  = \delta^2 \circ \iota_4, \ \ \mbox{and} \  \
\iota_2  = \delta^3 \circ \iota_4.$$
Hence $\iota_1$, $\iota_2$, $\iota_4$ and $\iota_6$ induce the same involution on
$\sC_{<\delta >}$.

A point $[p] \in C_{\langle\delta\rangle}$ is  invariant under $\iota_4$,
if $\iota_4(p) = \delta^i(p)$. This occurs, if and only if
$$p \in F_1 \cup F_2 \cup  F_4 \cup  F_6.$$
Note that $\delta(F_4) = F_6$ and $\delta(F_1) = F_2$.
Moreover $\delta$ yields a degree 4 quotient of $F_4 \cup F_6$, and a degree 4 quotient of
$F_1 \cup F_2$. Thus the divisor of fixed points contains
two families of elliptic curves.

By the same arguments as in $\ref{epi4}$, the involution $\iota_4$ induces an involution
on $\sC_{\langle\delta\rangle}$, which has a divisor of fixed points consisting of 2 families of
elliptic curves and makes $\sC_{\langle\delta\rangle}$ suitable for the construction of a
Borcea-Voisin tower.  
\end{pkt}

\begin{pkt} \label{del5}
The involution $\iota_5$ commutes with $\delta$. One has that $p = \iota_5(p)$, if $p \in F_5$ and
$\delta^2(p) = \iota_5(p)$, if $p \in F_7$. Note that $\delta$ acts as degree 4 automorphism on
$F_5$ resp., $F_7$. Each of the corresponding quotient maps is fully ramified over 4 points. By
the same arguments as in $\ref{pupu}$, the $\sM_1$-automorphisms $\iota_5$ and $\delta$ act in the
same way on the exceptional divisor of $\tilde \sC_2$. Thus $\iota_5$ induces an
involution on $\sC_{\langle\delta\rangle}$, which fixes a divisor consisting of 8 families of
rational curves. Moreover it is quite easy to check that this
involution makes $\sC_{\langle\delta\rangle}$ suitable for the construction of a Borcea-Voisin
tower.
\end{pkt}

\begin{pkt}
Since $\alpha \iota_1 \alpha^{-1} = \iota_1$ and $\alpha \delta  \alpha^{-1} = \eta$, one
concludes that the involution induced by $\iota_1$ on $\sC_{\langle\eta\rangle}$ coincides up to
an isomorphism with the involution induced by $\iota_1$ on $\sC_{\langle\delta\rangle}$.

Since $\alpha \iota_5 \alpha^{-1} = \iota_6$ and $\alpha \delta  \alpha^{-1} = \eta$, one
concludes that the involution induced by $\iota_6$ on $\sC_{\langle\eta\rangle}$ coincides up to
an isomorphism with the involution induced by $\iota_5$ on $\sC_{\langle\delta\rangle}$.

\end{pkt}

\section{The extended automorphism group of the degree 4 example}
The group $\G_4$ of $\sM_1$-automorphisms of $\sC_2$ does not contain all
$\sM_1$-automorphisms of $\sC_2$. In this section we give an additional group $\E_4$ of
$\sM_1$-automorphisms such that $\G_4$ and $\E_4$ generate an extended $\sM_1$-automorphism group
$\bar \G_4$. Moreover we will make some remarks about $\bar \G_4$ and $\E_4$.

We obtain due to \cite{HS}, Proposition 9 and the notations of \cite{HS}, Section 2:

\begin{proposition} \label{ff4}
The family $\sC_2$ has a group $\E_4$ of $\sM_1$-automorphisms consisting of 16 different
automorphisms given by $(\alpha\beta)^{\nu}$ with $\nu = 0,\ldots,3$ and:
$$\alpha_{\zeta}(y_2:y_1:x_1:x_0) = (\zeta y_2:\zeta y_1:x_1-\lambda x_0:x_1-x_0), \  \
\zeta^4 = (1-\lambda)^2$$
$$\beta_{\varsigma}(y_2:y_1:x_1:x_0) =
(\varsigma  y_2:\varsigma  y_1:x_1-x_0:\frac{1}{\lambda}x_1-x_0), \ \
\varsigma^4 = (1-\frac{1}{\lambda})^2$$
$$\gamma_{\kappa}(y_2:y_1:x_1:x_0) = (\kappa y_2:\kappa y_1:\lambda x_0:x_1), \  \
\kappa^4 = \lambda^2$$
The involutions of $\E_4$ are given by $(\alpha \beta )^{\nu}$, $\alpha_{\zeta}$,
$\beta_{\varsigma}$ and $\gamma_{\kappa}$ for $\nu= 2$, $\zeta^2 = 1-\lambda$, $\varsigma^2 =
1-\frac{1}{\lambda}$ and $\kappa^2 = \lambda$. The group $\E_4$ has a subgroup isomorphic to the
quarternion group given by $(\alpha \beta )^{\nu}$, $\alpha_{\zeta}$,
$\beta_{\varsigma}$ and $\gamma_{\kappa}$ for $\nu= 0,2$, $\zeta^2 = -1+\lambda$, $\varsigma^2 =
-1+\frac{1}{\lambda}$ and $\kappa^2 = -\lambda$. 
\end{proposition}

One can ask for the character of the action of the involutions of $\E_4$ on
$\Gamma(\omega_{(\sC_2)_q})$ for each $q \in \sM_1$ and the possibilities to use these involutions
for the construction of Borcea-Voisin towers. For example one has:

\begin{example} \label{examp}
One checks easily that $\gamma_{\sqrt{\lambda}}$ resp., $\gamma_{-\sqrt{\lambda}}$ fixes the
family curves on $\sC_2$ given by
$$x_1 = \sqrt{\lambda} x_0 \  \ \mbox{resp.,} \  \ x_1 = -\sqrt{\lambda}x_0.$$
This family of curves is isomorphic to the constant
family with universal fiber given by the Fermat curve $\F_4$ of degree 4, which has the genus 3.
Thus it acts by the character $-1$ on $\Gamma(\omega_{(\sC_2)_q})$ for each $q \in \sM_1$. Since
$\F_4$ has complex multiplication, $\gamma_{\sqrt{\lambda}}$ and $\gamma_{-\sqrt{\lambda}}$ make
$\sC_2$ suitable for the construction of a Borcea-Voisin tower.

The following claim implies that $\gamma_{\sqrt{\lambda}}$ and $\gamma_{-\sqrt{\lambda}}$ yield
isomorphic families by the Borcea-Voisin tower:
\end{example}

\begin{claim}
One can conjugate $\gamma_{\sqrt{\lambda}}$ and $\gamma_{-\sqrt{\lambda}}$ in $\E_4$.
\end{claim}
\begin{proof}
There exists some $g$ of order 4 contained in the quarternion subgroup of $\E_4$ such that
$$\gamma_{\sqrt{\lambda}}= (\alpha \beta) g \  \ \mbox{and} \ \  \gamma_{-\sqrt{\lambda}}
= (\alpha \beta)^3 g = (\alpha \beta )(\alpha \beta)^2 g =(\alpha \beta) g^{-1}.$$
It is a well-known fact that there is a $g_2$ contained in the quarternion group such that
$$g^{-1} = g_2\circ g \circ g_2^{-1}.$$
Since $(\alpha \beta)$ is contained in the center of $\E_4$, one obtains the result.
\end{proof}

Finally the question for isomorphy between $\sC_2/\iota_1$ and $\sC_2/\iota_4$ resp., the
corresponding $CMCY$ families of 3-manifolds constructed by the method of C. Voisin \cite{Voi2}
remains open, since we have:

\begin{remark}
By the description of $\E_4$ in Proposition $\ref{ff4}$, one checks easily that the generators
$\alpha, \beta, \iota_4$ of $\G_4$ commute with each element of $\E_4$. Hence each element of
$\bar \G_4$, which is the group generated by $\G_4$ and $\E_4$, can be written as $\kappa \tau$
with $\kappa \in \E_4$ and $\tau \in \G_4$. Thus for each $\sigma\in \G_4$ one obtains
\begin{equation} \label{conju}
(\kappa \tau)^{-1}  \sigma (\kappa \tau)  = \tau^{-1}  \sigma \tau.
\end{equation}
Hence the fact that $\iota_1$ and $\iota_4$ are not conjugate in $\G_4$ implies that $\iota_1$ and
$\iota_4$ are not conjugate in $\bar \G_4$.

Moreover $\eqref{conju}$ implies that $\gamma_{\sqrt{\lambda}}$ is not conjugate to $\iota_1$ or
$\iota_4$ in $\bar \G_4$.
\end{remark}

\begin{remark}
One may search for additional involutions in $\bar \G_4$ and try to determine the character
of the actions of all involutions on $\Gamma(\omega_{(\sC_2)_q})$ for each $q \in \sM_1$. In
addition one can try to
determine the involutions, which are suitable for the construction of a Borcea-Voisin tower and
try to repeat the construction of the preceding section for arbitrary induced involutions on
suitable quotients by cyclic subgroups of $\bar \G_4$. 
\end{remark}

\section{The automorphism group of the degree 5 example by Viehweg and Zuo}

We consider the $CMCY$ family $\sF_3$
$$
\bP^4 \supset V(y_3^5+y_2^5 + y_1^5 + x_1(x_1-x_0) (x_1-\alpha  x_0)(x_1-\beta   x_0) x_0)
\to (\alpha ,\beta)\in \sM_2
$$
of 3-manifolds constructed by E. Viehweg and K. Zuo. Let $\xi$ denote a fixed primitive 5-th. root
of unity. We introduce an $\sM_2$-automorphism group $\G_5$ of the family $\sF_3 \to \sM_2$. The
elements $g\in \G_5$ can be uniquely written as a product $g = abcd$ with $a \in
\langle\alpha\rangle$, $b \in \langle\beta\rangle$, $c \in \langle\gamma\rangle$ and $d \in S_3$,
where:
$$\alpha (y_3: y_2: y_1: x_1: x_0) = (\xi y_3: y_2: y_1: x_1: x_0),$$
$$\beta  (y_3: y_2: y_1: x_1: x_0) = (y_3: \xi y_2: y_1: x_1: x_0),$$
$$\gamma  (y_3: y_2: y_1: x_1: x_0) = (y_3: y_2: \xi y_1: x_1: x_0),$$
$$d (y_3: y_2: y_1: x_1: x_0) = (y_{d(3)}: y_{d(2)}: y_{d(1)}: x_1: x_0)$$
Therefore the group $\G_5$ contains exactly $5\cdot 5 \cdot 5 \cdot 6 = 750$ elements. The
action of $\G_5$ on the global sections of the canonical sheaves of the fibers induces a
surjection of $\G_5$ onto the multiplicative group of the 10-th. roots of unity.\footnote{Note
that $S_3$ is generated by the involutions given by the cycles $(1,2)$ and
$(2,3)$, which act by the character $-1$ on the global sections of the canonical sheaves of
the fibers.}

Its kernel $\KK_5$ is a normal subgroup of order 75. It contains the subgroup
$\langle\alpha \beta^{-1}, \beta \gamma^{-1}\rangle$ of automorphisms of order 5. Moreover it
contains the cyclic group of order 3 given by the permutations of $A_3$. Therefore all elements
of $\KK_5$ are determined.

\begin{pkt}
Let us consider all cyclic groups $\langle g\rangle \subset \KK_5$ with $g = abc \neq e$ as above.
If $a = e$
or $b = e$ or $c = e$, $\langle g\rangle$ is given by $\langle\alpha \beta^{-1}\rangle$,
$\langle\beta \gamma^{-1}\rangle$ or $\langle\alpha \gamma^{-1}\rangle$. These groups are
conjugate by $(1,2),(1,3),(2,3)\in S_3$.

Now consider the cyclic group $\langle g\rangle \subset \KK_5$ with $g = abc$ and $a, b, c \neq e$.
One has that $\langle g\rangle$ contains an element $\alpha \beta^b \gamma^{4-b}$ with $b = 1,
2, 3$. Hence by $e \in S_3$ or $(2,3) \in S_3$, it is conjugate to $\langle\alpha \beta
\gamma^3\rangle$ or $\langle\alpha \beta^2 \gamma^2\rangle$. By the cycle $(1,3) \in S_3$,
these both groups are conjugate. By the fact that $\langle\alpha \beta \gamma^3\rangle$
leaves only finitely many points invariant on each fiber, but $\langle\alpha \beta^{-1}\rangle$
leaves a curve invariant on each fiber, both groups can not be conjugate.

Hencefore we have two conjugacy classes of cyclic subgroups $\langle g\rangle \subset \KK_5$ with
$g = abc \neq e$ represented by $\langle\alpha \beta^{-1}\rangle$ and $\langle\alpha \beta
\gamma^3\rangle$.
\end{pkt}

\begin{claim}
Any automorphism $\tau \in \KK_5$, which is not given by
$$\tau (y_3: y_2: y_1: x_1: x_0) = (\xi^s y_3: \xi^t y_2: \xi^{5-s-t} y_1: x_1: x_0)$$
for some $s, t \in \Z$, satisfies $\tau^3 = \id$.
\end{claim}

\begin{proof}
If $\tau$ satisfies the assumptions of the Claim, then $\tau$ or $\tau^{-1}$
is given by
\begin{equation} \label{tau}
(y_3: y_2: y_1: x_1: x_0) \to (\xi^s y_1: \xi^t y_3: \xi^{5-s-t} y_2: x_1: x_0)
\end{equation}
for some $s, t \in \Z$. Hence assume without loss of generality that $\tau$ is given by
($\ref{tau}$) and verify the statement by calculation:
$$\tau^3 (y_3: y_2: y_1: x_1: x_0) = \tau^2 (\xi^s y_1: \xi^t y_3: \xi^{-s-t} y_2: x_1: x_0)$$
$$= \tau (\xi^{-t} y_2: \xi^{s+t} y_1: \xi^{-s} y_3: x_1: x_0)= (y_3: y_2: y_1: x_1: x_0)$$
\end{proof}

For each $\tau$ as in ($\ref{tau}$) one can easily calculate that
$\alpha^{-s}\beta^{-s-t} \circ \tau \circ \alpha^{s}\beta^{s+t}$ is given by
$$(y_3: y_2: y_1: x_1: x_0) \to (y_1: y_3: y_2: x_1: x_0).$$

Therefore all cyclic subgroups of $\KK_5$ are up to conjugation determined. Hence:

\begin{proposition}
The family $\sF_3$ has up to isomorphisms the following quotient families of Calabi-Yau orbifolds
with dense sets of $CM$ fibers:
$$\sF_3/\langle\alpha \beta^4\rangle, \  \ \sF_3/\langle\alpha \beta \gamma^{3}\rangle, \  \
\sF_3/\langle(1,2,3)\rangle$$
\end{proposition}
\begin{proof}
The existence of dense sets of $CM$ fibers follows, since the $VHS$ of a quotient family of
$\sF_3$ is a sub-$VHS$ of $\sF_3$.
\end{proof}

\chapter{Examples of $CMCY$ families of 3-manifolds and their invariants}

\section{The {\em length} of the Yukawa coupling}

First let us construct the Yukawa coupling. A little bit later in this short section we will give
a motivation to consider it and describe how to calculate its {\em length} for our examples of $CMCY$
families of 3-manifolds.

\begin{construction} \label{yuka} \index{Yukawa coupling}
Assume that $U$ is a quasi projective variety and $\sV$ is a complex polarized variation of Hodge
structures of weight $n$ on $U$. It is a well-known fact that there exists a
suitable finite cover of $U$ such that the pullback of $\sV$ has local unipotent monodromy. We
replace $U$ by this finite cover. There exists a  smooth projective compactification $Y$ of $U$
such that $S := Y \setminus U$ is a normal crossing divisor.
Then one can construct the Deligne extension $\sH$ of $\sV \otimes_{\C}\sO_U$ (i. e., the
unique extension such that the Gau\ss-Manin
connection yields the structure of a logarithmic Higgs bundle $(F,\theta)$ on the associated
graded bundle and the real components of eigenvalues of the residues are contained in $[0,1)$).
The graduation gives a decomposition of $F$ into locally free sheaves $E^{p,n-p}$ and the
Gau\ss-Manin connection induces an $\sO_Y$-linear morphism
$$E^{p,n-p} \to E^{n-1,n-p+1} \otimes \Omega^1_Y({\rm log} S),$$
called Higgs field. The Yukawa coupling $\theta_i$ (for $i \leq n$) is defined by the composition
$$\theta_i: E^{n,0} \stackrel{\theta_{n,0}}{\longrightarrow } E^{n-1,1} \otimes
\Omega^1_Y({\rm log} S)
\stackrel{\theta_{n-1,1}}{\longrightarrow } E^{n-2,2} \otimes {\rm Sym}^2\Omega^1_Y({\rm log} S)
\stackrel{\theta_{n-2,2}}{\longrightarrow } \ldots$$
$$\stackrel{\theta_{n-i+1,i-1}}{\longrightarrow }
E^{n-i,i} \otimes {\rm Sym}^i\Omega^1_Y({\rm log} S).$$
\end{construction}

\begin{definition} \index{$\zeta(f)$} \index{Yukawa coupling!{\em length} of the|see{$\zeta(f)$}}
Let $f: V \to U$ be a family with fibers of dimension $n$ as in Construction $\ref{yuka}$. The
{\em length} $\zeta(f)$ of the Yukawa coupling is given by
$$\zeta(f) := {\rm min}\{i\geq 1;\theta_i = 0\}-1.$$
We say that the Yukawa coupling has {\em maximal length}, if $\zeta(f) = n$.

The family $f: V \to U$ is rigid, if there does not exist a non-trivial deformation of $f$ over
a nonsingular quasi-projective curve $T$.
\end{definition}

The following proposition yields our motivation to consider the {\em length} of the Yukawa coupling:

\begin{proposition}
If the Yukawa coupling has {\em maximal length}, the family is rigid.
\end{proposition}
\begin{proof}
(see \cite{VZ7}, Section 8)
\end{proof}

The statements of the following lemma, which allow the computation of {\em length} of the Yukawa
couplings of our examples of $CMCY$ families of 3-manifolds by their construction, are well-known:

\begin{lemma} \label{bem}
For two variations of Hodge structures $\V$ and $\W$ on a holomorphic manifold one has
$$\zeta (\V \otimes \W) = \zeta (\V) + \zeta (\W) \  \ \mbox{and} \ \
\zeta (\V \oplus \W) = {\rm max}\{\zeta (\V), \zeta (\W)\}.$$
\end{lemma}

\section{Examples obtained by degree 2 quotients}

Let $\sZ_1 \to \sM$ be one of the examples of a $CMCY$ family of 2-manifolds, which we have
constructed in the preceding chapters, with a suitable involution $\iota$ such that it satisfies
the assumptions for $\sZ_1$ in the construction of a Borcea-Voisin tower. Here we list all
examples of $CMCY$ families $\sZ_2$ of 3-manifolds obtained by the Borcea-Voisin tower starting
with such a family $\sZ_1$ and $\Sigma_2$ given by the family $\sE \to \sM_1$ of
elliptic curves endowed with its natural involution. By the definition of Calabi-Yau manifolds,
Serre duality and Hodge symmetry, all Hodge numbers of the fibers of the resulting $CMCY$ family
$\sZ_2$ of 3-manifolds are determined by $h^{1,1}$ and $h^{2,1}$.

\begin{claim} \label{doppeln}
Keep the assumptions above. Let $(\sZ_1)_p \to (\sZ_1)_p/\iota$ be ramified over $N$ curves with
genus $g_1, \ldots, g_N$ for all $p \in \sM$. Then the fibers of $\sZ_2$ have the Hodge numbers
$$h^{1,1} = 11 +5N-N' \mbox{ and }
h^{2,1} = 11+5N'-N, \  \ \mbox{where} \  \
N' = \sum g_i.$$
\end{claim}
\begin{proof}
(see \cite{Voi2}, Corollaire 1.8)
\end{proof}

Hence for our examples of $CMCY$ families of 3-manifolds obtained by using the Borcea-Voisin tower
and $CMCY$ families of 2-manifolds with suitable involutions, we have the following table:
\begin{center}
\begin{tabular}{|c|c|c||c|c|c|c|c||c|} \hline
family $\sZ_1$ & basis $\sM$ & involution $\iota$ & $N$ & $N'$ & $h^{1,1}$ &
$h^{2,1}$ & $\zeta$ & reference \\ \hline \hline
$\sC_2$ & $\sM_1$ & $\iota_1$ & 1 & 3 & 13 & 25 & 2 & $\ref{iotata}$ \\ \hline
$\sC_2$ & $\sM_1$ & $\iota_4$ & 1 & 3 & 13 & 25 & 2 & $\ref{iotata}$ \\ \hline
$\sC_2$ & $\sM_1$ & $\gamma_{\sqrt{\lambda}}$, $\gamma_{\sqrt{-\lambda}}$ & 1 & 3 & 13 & 25 & 2 & $\ref{examp}$ \\ \hline
$\sC_{ \langle \iota_3 \rangle }$ & $\sM_1$ & $\iota_1$ & 2 & 2 & 19 & 19 & 2 & $\ref{nr1}$ \\ \hline
$\sC_{\langle\iota_3 \rangle}$ & $\sM_1$ & $\iota_4$ & 2 & 2 & 19 & 19 & 2 & $\ref{nr2}$ \\ \hline
$\sC_{\langle\epsilon \rangle}$ & $\sM_1$ & $\iota_1$ & 8 & 0 & 51 & 3 & 2 & $\ref{pupu}$ \\ \hline
$\sC_{\langle\epsilon\rangle}$ & $\sM_1$ & $\iota_4$ & 2 & 2 & 19 & 19 & 2 & $\ref{epi4}$ \\ \hline
$\sC_{\langle\delta\rangle}$ & $\sM_1$ & $\iota_1= \iota_4$ & 2 & 2 & 19 & 19 & 2 & $\ref{del1}$ \\ \hline
$\sC_{\langle\delta\rangle}$ & $\sM_1$ & $\iota_5$ & 8 & 0 & 51 & 3 & 2 & $\ref{del5}$ \\ \hline
$\sW$ & $\sM_3$ & $\gamma$ & 2 & 4 & 17 & 29 & 2 & $\ref{sw1}, \ref{sw2}$ \\ \hline
$\sY$ & $\sM_3$ & $\gamma$ & 2 & 4 & 17 & 29 & 2 & $\ref{sy1}, \ref{sy2}$ \\ \hline
\end{tabular}
\end{center}

\section{The Example obtained by a degree 3 quotient and its maximality}
In this section we determine the Hodge numbers of the $CMCY$ family $\sQ$ of 3-manifolds obtained
by Proposition $\ref{pubc}$.

\begin{remark}
In the case of the $CMCY$ family of Proposition $\ref{pubc}$ one has
$\zeta = 1$ for the {\em length} of the Yukawa coupling as one concludes by its construction and
using Lemma $\ref{bem}$.
\end{remark}

Let $X$ be a complex manifold and $\gamma$ an automorphism of $X$ of order $m$. Then
$H^k(X,\C)_{\ell}$ denotes the eigenspace of $H^k(X,\C)$, on which $\gamma$ acts via pullback by
the character $e^{2 \pi i \frac{\ell}{m}}$. For the calculation of the Hodge numbers of this
family we will need the following proposition:

\begin{proposition} \label{fixhodge}
Let $X$ be a K\"ahler manifold of dimension 3. Moreover let $\varphi$ be an automorphism of $X$
fixing a finite set of some isolated points $Z_0$ and a finite set $Z_1$ of disjoint curves such
that $\varphi^m = \id$ for some $m\in \N$. Then one has the following eigenspaces:
$$H^2(\tilde X_{Z_1\cup Z_0}, \Z)_0\cong H^2(X,\Z)_0 \oplus H^{0}(Z_1, \Z)\oplus H^{0}(Z_0, \Z),$$
$$H^3(\tilde X_{Z_1\cup Z_0}, \Z)_0 \cong H^3(X,\Z)_0 \oplus  H^{1}(Z_1, \Z)$$
\end{proposition}
\begin{proof}
Let $Y$ be a K\"ahler manifold and $Z$ be a submanifold of codimension $r$. Then the Hodge
structure of the blowing up $\tilde Y_Z$ along $Z$ is given by
$$H^k(Y,Z) \oplus \bigoplus\limits_{i = 0}^{r-2} H^{k-2i-2}(Z, \Z) \cong H^k(\tilde Y_Z, \Z),$$
where $H^{k-2i-2}(Z, \Z)$ shifted by $(i+1,i+1)$ in bi-degree (see \cite{Voi},
Th\'eor$\grave{\rm e}$me
$7.31$).

Thus one has:
$$H^2(\tilde X_{Z_1\cup Z_0}, \Z) \cong H^2(X,\Z) \oplus  H^{0}(Z_1, \Z)\oplus  H^{0}(Z_0, \Z),$$
$$H^3(\tilde X_{Z_1\cup Z_0}, \Z) \cong H^3(X,\Z) \oplus  H^{1}(Z_1, \Z)$$
Hence it remains to show that $H^{0}(Z_1, \Z)$, $H^{0}(Z_0, \Z)$ and $H^{1}(Z_1, \Z)$ are
invariant as sub-Hodge structures by $\varphi$. Hencefore one considers the proof of \cite{Voi},
Th\'eor$\grave{\rm e}$me $7.31$. These sub-Hodge structures are given by the image of
$j_* \circ(\pi|_{Z_1\cup Z_0})^*(H^0(Z_1\cup Z_0,\Z))$ and
$j_* \circ(\pi|_{Z_1\cup Z_0})^*(H^1(Z_1\cup Z_0,\Z))$, where $j$ denotes the embedding of the
exceptional divisor $E$ of the blowing up morphism $\pi : \tilde X_{Z_1\cup Z_0}
\to X$.\footnote{In general one has $\bigoplus_{i=0}^{r-2}j_* \circ h^i \circ
(\pi|_{Z_1\cup Z_0})^*$ instead of $j_* \circ (\pi|_{Z_1\cup Z_0})^*$ for $i = 0, \ldots, r-2$ in
\cite{Voi}, Th\'eor$\grave{\rm e}$me $7.31$, where $h$ denotes the cup-product with $c_1(\sO_E(1))$
and the sheaf $\sO_E(1)$ of the projective bundle $E$ is described in \cite{Voi}, Subsection
$3.3.2$. But here the weight of the Hodge structures is to small for $i > 0$.} One has the
following commutative diagram:
$$\xymatrix{
{\tilde X_{Z_1\cup Z_0}} \ar[rr]^{\varphi} &   & {\tilde X_{Z_1\cup Z_0}} \\
{E} \ar[rr]^{\varphi} \ar[d]^{\pi|_E} \ar[u]^{j} &  & {E} \ar[d]^{\pi|_E} \ar[u]^{j}\\
{Z_1\cup Z_0} \ar[rr]^{\varphi} &   & {Z_1\cup Z_0}
}$$
Since $\varphi$ acts as the identity on $Z_1\cup Z_0$, the same holds true for the Hodge
structures on $Z_1\cup Z_0$. Hence by the commutative diagram, the same holds true for the
sub-Hodge structures on $\tilde X$ given by $j_* \circ(\pi|_{Z_1\cup Z_0})^*$.
\end{proof}

\begin{proposition} \label{eigeleb}
For all $q \in \sM_3$ the action of the cyclic group $\langle\alpha \beta \rangle$ on $\sW$
yields an eigenspace decomposition of $H^{1,1}(\sW_q)$ of the dimensions
$$h^{1,1}(\sW_q)_0 = 14, \  \ h^{1,1}(\sW_q)_1 = 3, \  \ h^{1,1}(\sW_q)_2 = 3.$$
\end{proposition}
\begin{proof}
Let $\tilde \sW \to \sW$ be the blowing up of the six sections fixed by $\alpha \beta $. By
the same arguments as in the proof of the preceding proposition, each fiber $\tilde \sW_q$ has the
Hodge numbers
$$h^{2,0} = 1, \  \ h^{1,1} = 26, \  \ h^{0,2} = 1.$$
Let $M := \tilde \sW_q/\langle\alpha \beta \rangle$. Now we consider the quotient morphism
$\varphi : \tilde \sW_q \to M$. By the Hurwitz formula, one concludes that
$$\varphi^*(K_M) = -2 E -E^{(2)},$$
where $E$ is the exceptional divisor of $\sW_q \to W_q$ given by three $-2$ curves and $E^{(2)}$
is the exceptional divisor of $\tilde \sW_q \to \sW_q$.
From \cite{Voi}, Proposition $21.14$, we have that $3\cdot K_M^2 = (\varphi^*(K_M))^2$. Since
$$(\varphi^*(K_M))^2 = (-2 E -E^{(2)})^2 = 4\cdot(-6) -6 =-30$$
and $c_1(M)^2 = K_M^2$ (see \cite{hart}, Appendix A, Example $4.1.2$), one obtains
$$c_1(M)^2 = K_M^2 = -10.$$
By the Noether formula (compare to \cite{hart}, Appendix A, Example $4.1.2$ and \cite{Voi},
Remarque 23.6), one has
$$\chi(\sO_M) = \frac{1}{12}(c_1(M)^2 + c_2(M)) \  \ \mbox{with}  \  \ c_2(M)-2 = b_2(M)$$
in our case. From the fact that $\chi(\sO_M) = 1$, one calculates that
$$h^{1,1}(\tilde \sW_q)_0 = b_2(M) = 20.$$
By the fact that the blowing up morphism $\tilde \sW_q \to \sW_q$ has an exceptional divisor
consisting of 6 rational curves, we conclude similar to Proposition $\ref{fixhodge}$ that
$$h^{1,1}(\sW_q)_0 = h^{1,1}(\tilde \sW_q)_0 - 6 = 20 -6 = 14.$$
Since the $K3$ surface $\sW_q$ has the Hodge number
$$h^{1,1}(\sW_q) = 20 \  \ \mbox{and}  \  \ h^{1,1}(\sW_q)_1 =  h^{1,1}(\sW_q)_2,$$
one concludes that
$$h^{1,1}(\sW_q)_1 =  h^{1,1}(\sW_q)_2 = 3.$$
\end{proof}

\begin{proposition}
For all $q \in \sM_3$ one has
$$h^{1,1}(\sQ_q) = 51.$$
\end{proposition}
\begin{proof}
Since
$$h^{0,0}(\sW_q)_0 = h^{0,0}(\F_3)_0 = h^{1,1}(\F_3)_0 = 1, \  \ b_1(\sW_q) = 0$$
and Proposition $\ref{eigeleb}$ tells us that
$$h^{1,1}(\sW_q)_0 = 14,$$
one concludes that $h^{1,1}(\sW_q \times\F_3)_0 = 15$. Note
that $\alpha_2$ fixes $6 \cdot 3 = 18$ points. Moreover we have an additional exceptional
divisor consisting of $3\cdot 3\cdot 3 = 27$ rational ruled surfaces. In the construction of $\sQ$
we blow down 9 of these families of ruled surfaces. Hence by
Proposition $\ref{fixhodge}$,
$$h^{1,1}(\sQ_q) = 15+ 18+27-9 = 51.$$
\end{proof}

Recall that $\alpha\beta$ acts by the character $e^{2\pi i\frac{2}{3}}$ on the global sections of
$\omega_{\sW_q}$ for all $q \in \sP_n$ and $\alpha_{\F_3}$ acts by the character
$e^{2\pi i\frac{1}{3}}$ on the global sections of $\omega_{\F_3}$. Hence one obtains
$$h^{1,0}(\F_3)_1 = h^{0,1}(\F_3)_2 = h^{2,0}(\sW_q)_2 = h^{0,2}(\sW_q)_1 = 1$$
and
$$h^{1,0}(\F_3)_2 =h^{0,1}(\F_3)_1 = h^{2,0}(\sW_q)_1 = h^{0,2}(\sW_q)_2 = 0.$$
Note that $b_1(\sW_q) = b_3(\sW_q) = 0$, $h^{1,1}(\sW_q)_0 = 14$ and $h^{1,1}(\sW_q)_1 =
h^{1,1}(\sW_q)_2 = 3$. Thus
$$H^3(\sW_q \times\F_3,\C)_0 = \bigoplus\limits_{t = 0}^2 H^2(\sW_q,\C)_t \otimes
H^1(\F_3,\C)_{[3-t]_3}.$$
Hence one concludes that
$$H^3(\sW_q \times\F_3,\C)_0 = (H^{2,0}(\sW_q)_2 \oplus H^{1,1} (\sW_q)_2) \otimes
H^{1,0}(\F_3)_1$$
$$\oplus (H^{1,1} (\sW_q)_1 \oplus H^{0,2}(\sW_q)_1) \otimes H^{0,1}(\F_3)_2.$$
This implies that
$$H^{2,1}(\sW_q \times\F_3)_0 = H^{1,1} (\sW_q)_2 \otimes H^{1,0}(\F_3)_1 \  \ \mbox{such that}
\ \ h^{2,1}(\sW_q \times\F_3)_0 = 3.$$
Hence by Proposition $\ref{fixhodge}$ and the fact that $b_1(\bP^1) = 0$, one concludes easily:

\begin{proposition}
For all $q \in \sM_3$ one has
$$h^{1,2}(\sQ_q) = h^{2,1}(\sQ_q) = 3.$$
\end{proposition}

Next we show that $\sQ$ is a maximal family of Calabi-Yau manifolds. First let us define
maximality. For this definition recall:

\begin{proposition}
Each Calabi-Yau manifold $X$ has a local universal deformation $\sX \to B$, where
$$\dim(B) = h^{2,1}(X).$$
\end{proposition}
\begin{proof}
(see \cite{Voi}, $10.3.2$)
\end{proof}

\begin{definition} \index{maximal family}
A family $\sF \to Y$ of Calabi-Yau manifolds is maximal in $0 \in Y$, if the universal
property of the local universal deformation $\sX \to B$ of $\sF_0$ yields a surjection of a
neighborhood of $0$ onto $B$. The family $\sF \to Y$ is maximal, if it is maximal in all
$0 \in Y$. 
\end{definition}

\begin{remark}
If the family $\sF \to Y$ of Calabi-Yau manifolds is maximal in some $0 \in Y$, its restriction to
the complement of a closed analytic subvariety of $Y$ is maximal.
\end{remark}

\begin{remark}
Since $\sW_q$ is birationally equivalent to $\F_3 \times \sC_q/\langle(1,1)\rangle$ (see
Proposition $\ref{proposition}$), one has 
$$H^{2,0}(\sW_q) = H^{1,0}(\F_3)_1 \otimes H^{1,0}(\sC_q)_2,$$
where $\sC$ denotes the family of degree 3 covers with a pure $(1,3)-VHS$. Thus by our former
notation with respect to the push forward action, the $VHS$ of $\sW$ depends uniquely on the
fractional $VHS$ of the eigenspace $\sL_1$ of the $VHS$ of $\sC$.

In Section $9.2$ we have seen that $\sQ$ is birationally equivalent to a quotient of
$\sW \times\F_3$. It differs by some blowing up morphism with respect to some families of rational
curves and some isolated sections. Thus by similar arguments, the $VHS$ of $\sQ$ depends on the
$VHS$ of $\sW$. Hence the $VHS$ of $\sQ$ depends uniquely on the fractional $VHS$ of $\sL_1$. Thus
the period map of $\sQ$ can be considered as a multivalued map to the ball $\B_3$.
\end{remark}

The preceding remark tells us the period map of the family $\sQ \to \sM_3$ is locally injective.
Hence by the Torelli theorem for Calabi-Yau manifolds, one concludes:

\begin{Theorem}
The family $\sQ \to \sM_3$ is maximal.
\end{Theorem}

\section{Outlook onto quotients by cyclic groups of high order}
Recall that we used $K3$ surfaces $S$ and elliptic curves $E$ with cyclic degree $m$ covers
$S \to R$ and $E \to \bP^1$ to construct
Calabi-Yau 3-manifolds by a quotient, where $m = 2,3$. In this chapter we give an outlook on the
possibilities to use of cyclic groups of higher order for the construction of Calabi-Yau
3-manifolds by an elliptic curve and a $K3$-surface.

First the following Lemma shows that there are only finitely many elliptic
curves with an action of a cyclic group with order $m >2$, which could be suitable:

\begin{lemma}
Let $E$ be an elliptic curve, and $f: E \to \bP^1$ be a cyclic cover. Then one obtains
$$m : = deg(f) = 2, \  \ 3, \  \ 4 \   \  \mbox{or} \   \ 6.$$
For each $m >2$ there is at most only one elliptic curve having a cyclic cover
$f:E \to \bP^1$ of degree $m$.\footnote{The well-educated reader knows the automorphism group of
the abelian variety given by one elliptic curve. But the quotient map by a cyclic subgroup of this
automorphism group is fully ramified at the zero-point. There may be cyclic covers, which are
not fully ramified over all branch points. Hence for the proof of this lemma, it is not
sufficient to know the automorphism group of this Abelian variety.}
\end{lemma}
\begin{proof}
We use Proposition $\ref{1.27}$ and Corollary $\ref{rangi}$. Let $f: E \to \bP^1$ be be a cyclic
cover of degree $m > 2$. Moreover if $f$ has $n$ branch points, then $\LL_1$ is of type
$(p,q)$ with $p+q = n -2$. Thus there must be at least 2 branch points. If there are 2 branch
points, we are in the case of the cover $\bP^1 \to \bP^1$ given by $x \to x^m$. Since $\LL_1$ is
of type $(p,q)$ with $p+q = n -2$, $C$ can  be an elliptic curve for $m > 2$, only if $n = 3$.

For $n=3$ and $m > 2$ we have that $\LL_1$ is of type $(p,q)$ with $p+q = 1$. Without loss of
generality we assume that $p = 0$ and $q = 1$. Hence by Proposition $\ref{1.27}$, one concludes
that
$$\mu_1 + \mu_2 + \mu_3 = 1.$$
If $m = 3$, one has only the case of the Fermat curve of degree $3$ given by
$$\mu_1 = \mu_2 = \mu_3 = \frac{1}{3}.$$
If $m > 3$, $\LL_2$ must be of type $(0,0)$, which implies without loss of generality that
$\mu_1 = \frac{1}{2}$. Hence for $m = 4$ we have only the case of the cover given by
$$\mu_1 = \frac{1}{2}, \   \ \mu_2 = \mu_3 = \frac{1}{4}.$$
If $m > 4$, $\LL_2$ and $\LL_3$ must be of type $(0,0)$, which implies without loss of generality
that $\mu_1 = \frac{1}{2}$ and $\mu_2 = \frac{1}{3}$. Hence we obtain the only additional case
given by the degree $6$ cover with the local monodromy data
$$\mu_1 = \frac{1}{2}, \   \ \mu_2 = \frac{1}{3}, \   \ \mu_3 = \frac{1}{6}.$$
\end{proof}

Let $S$ be a $K3$-surface, $E$ be an elliptic curve and the cyclic groups
$\langle\gamma_S\rangle$ and $\langle\gamma_E\rangle$ of order $m >1$ acting on $S$ and $E$ with
the loci $F_S$ and $F_E$ of fixed points such that $\gamma_S$ and $\gamma_E$ act by $-1$ on the
global sections of the respective canonical sheaves. The aim is the construction of a Calabi-Yau
3-manifold by a desingularisation of $S \times E/\langle(\gamma_S,\gamma_E)\rangle$. The following
proposition tells us that there are singularities on $S\times E/\langle(\gamma_S,\gamma_E)\rangle$,
if $m >2$:

\begin{proposition}
Let $m >2$. Then $\gamma_S$ must fix some points.
\end{proposition}
\begin{proof} 
If $\gamma_S$ does not fix any point, one concludes by the Hurwitz formula that
$\varphi_S^*\omega = \sO$. Thus the quotient has a canonical sheaf $\omega$ with
$\omega^{\otimes m} = \sO$ for $m > 2$. Moreover it has the Betti number $b_1 = 0$. In addition it
must be a minimal model, since a rational $-1$ curve would (up to linear equivalence) be in the
support of the canonical divisor $K$ and forbid any torsion of $K$. But by the Enriques-Kodaira
classification (compare to \cite{Bart}, {\bf VI}), such a minimal model does not exist.
\end{proof}

\begin{remark}
The branch points of the degree 4 resp., the degree 6 cover $E \to \bP^1$ have different
branch indeces. Hence for the degree 4 and degree 6 case this yields some problems to find a
desingularisation of
$$S \times E/\langle(\gamma_S,\gamma_E)\rangle,$$
which is a Calabi-Yau manifold. \end{remark}

\chapter{Maximal families of $CMCY$ type}
In this chapter we use the classification of involutions on $K3$ surfaces by V. V. Nikulin
\cite{Niku}. We will see that certain involutions on the integral cohomology of $K3$ surfaces
yield a possibility to construct $CMCY$ families of 3-manifolds with maximal variations of
Hodge structures. For each $n \in \N$ with $n \leq 11$ we will obtain a holomorphic maximal $CMCY$
family over a basis of dimension $n$.

\section{Facts about involutions and quotients of $K3$-surfaces}
In this section we collect some known facts about $K3$ surfaces and their
involutions, which we will need in the sequel.

\begin{pkt}
The integral cohomology $H^2(S, \Z)$ is a lattice of rank 22. We have the cup-product
$(\cdot,\cdot)$ on $H^2(S, \Z)$. Let $L := (H^2(S, \Z),(\cdot ,\cdot ))$. It is a well-known fact
that one has the orthogonal direct sum decomposition
$$L \cong (-E_8) \oplus (-E_8) \oplus H \oplus H\oplus H,$$
where $-E_8$ consists of $\Z^8$ endowed with a certain negative definite integral bilinear form
and $H$ denotes the hyperbolic plane, i. e. $H = (\Z^2,\langle\cdot ,\cdot \rangle)$, where
$\langle\cdot ,\cdot \rangle$ is given by the matrix
$$\left(\begin{array}{cc}
0 & 1 \\
1 & 0
\end{array} \right)$$
(see \cite{Bart}, ${\bf VIII}.$ 1 and also \cite{Bart},
${\bf I}$. Examples $2.7$ for details).
\end{pkt}

\begin{remark}
Let $S$ be a $K3$-surface and $L = H^2(S,\Z)$, where $L$ is endowed with an involution $\iota$.
Assume that $\iota$ corresponds to an involution on $S$, which acts by the character $-1$
on $\Gamma(\omega_S)$. Then the involution induces a degree $2$ cover $\gamma :S \to R$ onto a
smooth surface $R$. Moreover the divisor of fixed points, which yields the ramification divisor
of $\gamma$, consists of a disjoint union of smooth curves or it is the zero-divisor. Moreover
$\iota$ yields integral sub-Hodge structures $H^2(S,\Z)_0$ and $H^2(S,\Z)_1$ of $H^2(S,\Z)$
such that $\iota$ acts by $(-1)^i$ on $H^2(S,\Z)_i$. Since $\iota$ acts by $-1$
on $\Gamma(\omega_S)$ and
$$H^2(R,\Q) = H^2(S,\Q)_0,$$
one has that
$$H^{2,0}(S),H^{0,2}(S) \subset H^2(S,\C)_1.$$
Moreover the intersection form has the signature $(2, r)$ on $H^2(S,\Z)_1$ (compare to
\cite{Voi2}, $\S 1$ and \cite{Voi2}, $2.1$).
\end{remark}

\begin{remark}
Let
$$D = \{[\omega] \in \bP(H^2(S,\C)_1)|(\omega, \omega) = 0, (\omega, \bar \omega) >0\}.
\index{$D$}$$
By the Torelli theorem, each marked $K3$ surface $(S',\phi_{S'})$ endowed with an involution,
which yields the the same involution $\iota$ on his cohomology lattice, yields a unique one
dimensional vector space $H^{2,0}(S') \subset H^2(S,\C)_1$ corresponding to some $p\in D$.
\end{remark}

\section{The associated Shimura datum of $D$}
The Hodge structure of a $K3$ surface $S$ with a cyclic degree 2 cover onto a rational surface
resp., Enriques surface $R$
has a decomposition into two rational Hodge structures $H^2(S, \Q)_1$ and $H^2(S, \Q)_0$. We
consider $H^2(S, \Q)_1$, since the variation of Hodge structures given by $H^2(S, \Q)_0$ is
trivial.

The Hodge decomposition of $H^2(S,\C)$ is orthogonal with respect to the Hermitian form
$(\cdot ,\bar \cdot)$. Hencefore the corresponding embedding
$$h : S^1 \to \SL(H^2(S, \R)_1)$$
factors trough the special orthogonal group $\SO(H^2(S,\R)_1)$ with respect to the
symmetric form given by the cup product pairing, where $\SO(H^2(S,\R)_1)$ is isomorphic to
$\SO(2,r)_{\R}$. Let $\omega \in \omega_S \setminus \{0\}$, 
$$\Re \omega := \frac{1}{2}(\omega + \bar \omega), \  \ \Im \omega :=
\frac{i}{2}(\omega - \bar \omega)$$
and $\{v_1, \ldots v_r\}$ be a basis of $H^{1,1}(X,\R)_1$. One has the basis
$$\{\Re \omega, \Im \omega, v_1, \ldots, v_r\}$$
of $H^{1}(X,\R)_1$ such that the intersection form is without loss of generality given by the
matrix $\diag(1,1,-1, \ldots, -1)$ with respect to this basis. The subgroup, whose elements
are invariant under
$$g \to h(i) g h(i^{-1}),$$
is given by ${\rm S}({\rm O}(2) \times {\rm O}(r))$, where
$$h(i) = h(i^{-1}) = \diag(-1,-1,1, \ldots, 1).$$
Since $h^2(i) = h(-1) = \diag(1, \ldots, 1)$, 
the action of $i$ is an involution. This
implies that one has a decomposition of $\fso_{2,r}(\R)$ into 2 eigenspaces with respect to the
eigenvalues 1 and $-1$. Hence $h(\sqrt{i})$ yields a complex structure on the
eigenspace with eigenvalue $-1$. The eigenspace for the eigenvalue 1 is given by the Lie algebra
of ${\rm S}({\rm O}(2) \times {\rm O}(r))$.  Thus we have a decomposition
$$\fso_{2,r}(\C) = \fh_+ \oplus \fh_0 \oplus \fh_-$$
such that $S^1$ acts by the characters $z/\bar z$, 1 and $\bar z/z$ on the respective complex
sub-vector spaces.

We continue our consideration of the involution $\iota$ given by
$$\iota(g) = h(i)g h^{-1}(i).$$
The matrices $M_1 \in \SO(2,r)(\C)$ with $\bar M_1  = \iota(M_1)$ satisfy
that
$$\bar M_1 = \diag(-1,-1,1, \ldots, 1) \cdot M_1 \cdot \diag(-1,-1,1, \ldots, 1)$$
$$= \diag(1,1,-1, \ldots, -1) \cdot M_1 \cdot \diag(1,1,-1, \ldots, -1).$$
Since $\SO(2,r)(\C)$ is given by the matrices $M$ satisfying
$$M^t \cdot \diag(1,1,-1, \ldots, -1) \cdot M = \diag(1,1,-1, \ldots, -1)$$
$$\Leftrightarrow M^{-1} = \diag(1,1,-1, \ldots, -1) \cdot M^t \cdot \diag(1,1,-1, \ldots, -1),$$
each matrix $M_1$ satisfies
$$M_1^{-1} = \bar M_1^t.$$
Thus $M_1$ is contained in the compact group $\SU(2+r)$, and one concludes:

\begin{proposition}
Our morphism
$$h: S^1 \to \SO(H^2(S,\R)_1)_{\R}$$
yields a Shimura datum.
\end{proposition}

\begin{remark} \label{zx}
Note that the simple Lie group $\SO(2,r)(\R)$ consists of two connected components (see \cite{FH},
Exercise $7.2$). Since the Lie group $\SO(2+r)(\C) \cong \SO(H^2(S,\R)_1)(\C)$ is connected (see
\cite{Helga}, {\bf IX}. Lemma $4.2$), the algebraic group $\SO(H^2(S,\R)_1)$ is connected, too.
Recall that all Cartan involutions of the simple algebraic group $\SO(H^2(S,\R)_1)$ are conjugate.
The action of $S^1$ on $H^2(S,\R)_1$ is given by its action on
$\langle\Re \omega, \Im \omega\rangle$ and $S^1$ fixes all vectors of $H^{1,1}(S,\R)_1$. This
implies that all morphisms
$$h: S^1 \to \SO(H^2(S,\R)_1),$$
which yield the Hodge structure of a $K3$ surface, satisfy that their images $h(S^1)$ are
conjugate. The definition of the Hodge structure on $H^2(S,\R)_1$ implies that the $\R$-valued
points of the kernel of $h$ are given by $\{1,-1\} \in S^1(\R)$. 
Let $\iota_{S^1}: S^1 \to S^1$ be the involution given by $x \to x^{-1}$. For each morphism $h_1$
in the conjugacy class of $h$, there exists exactly one other morphism $h_2$ with
$h_1(S^1) = h_2(S^1)$ and kernel given by $\{1,-1\} \in S^1(\R)$, which is given by
$h_2 = h_1 \circ \iota_{S^1}$. The conjugation by
$\diag(-1,1,-1,1,\ldots,1)$ yields an inner automorphism $\varphi$ of $\SO(H^2(S,\R)_1)$ such that
$h_2 = \varphi \circ h_1$. Thus each Hodge structure of a $K3$ surface obtained by some $p \in D$
is obtained by some element of the conjugacy class of our morphism $h: S^1 \to \SO(H^2(S,\R)_1)$.
Moreover note that the holomorphic $VHS$ over the
bounded symmetric domain associated with $\SO(H^2(S,\R)_1)(\R)^+/K$, which is induced by the natural
embedding $\SO(H^2(S,\Q)_1) \to \GL(H^2(S,\Q)_1)$, is uniquely determined by the variation of the
subbundle of rank 1 given by $H^{2,0}$. Since
$$r = \dim(D) = \dim(\SO(H^2(S,\R)_1)(\R)/K),$$
this $VHS$ yields a biholomorphic map from the bounded symmetric domain associated with
$\SO(H^2(S,\R)_1)(\R)^+/K)$ onto $D^+$.
\end{remark}

The preceding remark and Theorem $\ref{commul}$ imply:

\begin{Theorem}\label{densecm}
There is a dense set of $CM$ points on $D$ with respect to the $VHS$ on $D$ obtained by Remark
$\ref{zx}$.
\end{Theorem}

\section{The examples}

First we construct a holomorphic family of marked $K3$-surfaces with a global involution over its
basis: 

\begin{construction}
There exists a universal family $u:\sX \to B$ of marked analytic $K3$-surfaces, whose basis is not
Hausdorff (see \cite{Bart}, {\bf VIII}. 12). Let $\phi$ denote the global marking of the family
$\sX \to B$. We consider an involution $\iota$ on a marked $K3$ surface $(S,\phi)$, which acts by
$-1$ on $H^{2,0}(S)$. This involution yields an involutive isometry $\iota$ on the
lattice $L$. Thus the involution $\iota$ endows $\sX \to B$ with a new marking $\iota \circ \phi$.
By the universal property of the universal family, this new marking yields an involution of the
family:
$$\xymatrix{
{\sX} \ar[d]_{u} \ar[rr]^{\iota_{\sX}} &   & {\sX} \ar[d]_{u} \\
{B} \ar[rr]^{\iota_B} &   & {B}\\
}$$
Let $\Delta : B \to B \times B$ denote the diagonal embedding. We define
$$B_{\iota} = {\rm Graph}(\iota_B) \cap \Delta(B) \subset B \times B.$$
Note that each point $b \in B_{\iota}$ has an analytic neighborhood $U \subset B$ such that $\sX_U
\to U$ is given by the Kuranishi family and yields an injective period map for $U$. Thus on
$U \times U$ the diagonal $\Delta(U)$ and ${\rm Graph}(\iota_B|_U)$ are closed analytic
submanifolds. Hence $B_{\iota}$ has the structure of an analytic variety, which is not necessarily
Hausdorff, and can have singularities. The
composition $\Delta \circ u$ allows to consider $\Delta(B)$ as basis of the universal family of
the marked $K3$ surfaces. By the restricted family $\sX_{B_{\iota}} \to B_{\iota}$, we obtain a
holomorpic family with a global involution over the basis $B_{\iota}$. For simplicity we
write $\sX_{\iota} \to B_{\iota}$ \index{$\sX_{\iota} \to B_{\iota}$} instead of $\sX_{B_{\iota}}
\to B_{\iota}$. 
\end{construction}

\begin{remark}
The fibers of $\sX_{{\iota}} \to B_{\iota}$ have by the involution $\iota$ a cyclic covering
onto a projective surface (compare to \cite{Voi2}, $2.1$). Thus the fibers of
$\sX_{\iota} \to B_{\iota}$ are algebraic.
\end{remark}

\begin{proposition} \label{nix}
Assume that for all $b \in B_{\iota}$ the involution $\iota_{\sX_b}$ on $\sX_b$ has a locus of
fixed points consisting of rational curves. Then the holomorphic family $\sX_{{\iota}} \to
B_{\iota}$ is due to its global involution suitable for the construction of a holomorphic
Borcea-Voisin tower.
\end{proposition}
\begin{proof}
Let $b_0 \in B_{\iota}$ and $U \subset B_{\iota}$ be a small open neighborhood of $b_0$. The
eigenspace decomposition with respect to $\iota$ yields a variation of Hodge structures on the
eigenspace with respect to $-1$. The corresponding period map yields an open injection of $U$ into
$D$. By the fact that $D$ has a dense set of $CM$ points, the family $\sX_{\iota} \to B_{\iota}$
has a dense set of $CM$ fibers. Since the locus of fixed points with respect to $\iota_{\sX_b}$
consists of rational curves, this locus of fixed points has complex multiplication, too. Hence
$\sX_{\iota} \to B_{\iota}$ can be used for the construction of a holomorphic Borcea-Voisin tower. 
\end{proof}

Assume that $\sX_{\iota} \to B_{\iota}$ satisfies the assumptions of Proposition $\ref{nix}$.
Then let $\mathfrak{X_{\iota}} \to B_{\iota} \times \sM_1$ denote the family obtained by the
holomorphic Borcea-Voisin tower from $\sX_{\iota} \to B_{\iota}$ and $\sE \to \sM_1$ denote the
family of elliptic curves.

\begin{definition} \index{maximal family}
A family $\sF \to Y$ of Calabi-Yau manifolds is maximal in $0 \in Y$, if the universal
property of the local universal deformation $\sX \to B$ of $\sF_0$ yields a surjection of a
neighborhood of $0$ onto $B$. The family $\sF \to Y$ is maximal, if it is maximal in all
$0 \in Y$. 
\end{definition}

\begin{Theorem} \label{jaaah}
The family $\mathfrak{X_{\iota}}$ is maximal.
\end{Theorem}
\begin{proof}
By the following lemma, we start to prove Theorem $\ref{jaaah}$:

\begin{lemma} \label{wups}
$$H^3((\fX_{\iota})_{p\times q}) = H^2( (\sX_{\iota})_p ,\Q)_1 \otimes H^1(\sE_q,\Q)$$
\end{lemma}
\begin{proof}
Due to Proposition $\ref{fixhodge}$ and the fact that the exceptional divisors consist
of some rational curves, one only needs to determine $H^3((\sX_{{\iota}})_p\times \sE_q,\Q)_0$.
Since $b_1((\sX_{\iota})_p) = b_3((\sX_{\iota})_p)= 0$ and $H^1(\sE_q,\Q)=H^1(\sE_q,\Q)_1$, we are
done.
\end{proof}

By using the preceding lemma, we prove the following proposition.

\begin{proposition}
One has that $\dim(B_{\iota} \times \B_1) $ and $h^{2,1}((\fX_{\iota})_{p\times q})$ coincide.
\end{proposition}
\begin{proof}
By Proposition $\ref{wups}$,
$$H^3((\fX_{\iota})_{p\times q}) = H^2((\sX_{\iota})_p,\Q)_1 \otimes H^{1,0}(\sE_q,\Q) \oplus 
H^2((\sX_{\iota})_p,\Q)_1 \otimes H^{0,1}(\sE_q,\Q).$$
Hencefore
$$h^{2,1}((\fX_{\iota})_{p\times q}) = h^{1,1}((\sX_{\iota})_p,\Q)_1 \cdot h^{1,0}(\sE_q,\Q) +
h^{2,0}((\sX_{\iota})_p,\Q)_1 \cdot h^{0,1}(\sE_q,\Q)$$
$$= h^{1,1}((\sX_{\iota})_p,\Q)_1 +  h^{2,0}((\sX_{\iota})_p,\Q)_1 = h^{1,1}((\sX_{\iota})_p,\Q)_1
+ 1.$$
Recall that $D^+$ is the bounded symmetric domain obtained by $\SO(2,r)^+(\R)$, where $r =
h^{1,1}((\sX_{\iota})_p,\R)_1$. By \cite{Helga}, {\bf IX}. Table {\bf II}, $D$ has the
complex dimension $r$.\footnote{By \cite{Helga}, {\bf IX}. Table {\bf II}, $D$ has the dimension
$2r$ as real manifold.} Since the period map $p: B_{\iota} \to D$ of $\sX_{\iota} \to B_{\iota}$
is locally bijective, one concludes
$$h^{1,1}((\sX_{\iota})_p,\Q)_1 = r =\dim(D) = \dim(B_{\iota}),$$
which yields the result.
\end{proof}

By the following proposition, we finish the proof of Theorem $\ref{jaaah}$:
\end{proof}

\begin{proposition}
The period map yields a multivalued map from $\sM_1 \times B_{\iota}$ to the period domain, which
is locally injective. 
\end{proposition}
\begin{proof}
Let $\B$ be a small open subset of $\sM_1 \times B_{\iota}$ and let $x_1, x_2 \in \B$. Note tha
the period map $p$ on $\sM_1 \times B_{\iota}$ yields different image points $p(x_1)$ and $p(x_2)$,
if the classes of $H^{3,0}((\fX_{\iota})_{x_1})$ and $H^{3,0}((\fX_{\iota})_{x_2})$ in
$\bP(H^3((\fX_{\iota})_{x_1},\C))$ do not coincide. The respective period maps on
$B_{\iota}$ and $\sM_1$ are locally injective and depend only on $\omega_{\sE_q}$ and
$\omega_{(\sX_{\iota})_p}$. Since
$$H^{3,0}((\fX_{\iota})_{p\times q}) \subset H^3((\fX_{\iota})_{p\times q}) =
H^2( (\sX_{\iota})_p ,\Q)_1 \otimes H^1(\sE_q,\Q)$$
is given by $H^{2,0}((\sX_{\iota})_p) \otimes H^{1,0}(\sE_q)$, the period map concerning
$\fX_{\iota}$ is locally injective, too.
\end{proof}

It remains to classify the possible involutions $\iota$ on $L$, which provide our families
$\sX_{\iota} \to B_{\iota}$ with a global involution.

\begin{remark}
The involutions on $L$, which yield involutions on certain $K3$ surfaces, are characterized by the
triples of the following integers (compare to \cite{Niku}):
\begin{itemize} \item The integer $t$ is the rank of the sublattice ${\rm Pic}(S)_0$ of the Picard lattice of an
arbitrary fiber $S$ of $\sX_{\iota}$, which is invariant under the global involution. 
\item By the intersection pairing, one obtains a homomorphism ${\rm Pic}(S)_0 \to
{\rm Pic}(S)_0^{\vee}$. The integer $a$ is given by $(\Z/(2))^a \cong
{\rm Pic}(S)_0^{\vee}/{\rm Pic}(S)_0$.
\item By the morphism ${\rm Pic}(S)_0 \to
{\rm Pic}(S)_0^{\vee}$, the intersection form on ${\rm Pic}(S)_0$ yields a quardratic form $q$ on
${\rm Pic}(S)_0^{\vee}$ with values in $\Q$. The integer $\delta$ is 0, if $q$ has only values in
$\Z$ and 1 otherwise.
\end{itemize}
\end{remark}

For a fixed triple $(t,a,\delta)$ we write $\sX_{(t,a,\delta)} \to B_{(t,a,\delta)}$ instead of
$\sX_{\iota} \to B_{\iota}$ and $\fX_{(t,a,\delta)}$ instead of $\fX_{\iota}$.

\begin{remark}
The ramification locus of the fibers with respect to the involution on $\sX \to \sB_{t,a,\delta}$
is given by two elliptic curves, if $(t,a,\delta) = (10,8,0)$, is empty, if
$(t,a,\delta) =(10,10,0)$, and otherwise given by $C_{N'} + E_1 + \ldots + E_{N-1}$, where
$C_{N'}$ is a curve of genus
$$N' = \frac{1}{2}(22-t-a), \  \ \mbox{and} \  \ N = \frac{1}{2}(t-a)+ 1.$$
(compare to \cite{Niku})

Hencefore the triples
$$(t,a,\delta) =(10,10,0) \ \ \mbox{and} \  \ (t,a,\delta) \ \ \mbox{with} \  \  t+a = 22$$
yield the examples of families $\sX_{{(t,a,\delta)}} \to B_{(t,a,\delta)}$ with global involutions
over the basis, whose locus of fixed points consists at most of families of rational curves. Hence
by Proposition $\ref{nix}$, these triples yield maximal holomorphic $CMCY$ families of 3-manifolds.
\end{remark}

\begin{pkt}
By \cite{Niku}, Figure 2, one gets the following complete list of holomorphic maximal $CMCY$
families $\fX_{(t,a,\delta)} \to B_{(t,a,\delta)} \times \sM_1$ of 3-manifolds obtained by this
method. By Claim $\ref{doppeln}$, we obtain the Hodge numbers $h^{1,1}$ and $h^{2,1}$ of the fibers of
$\fX_{(t,a,\delta)}$.
$$
\begin{tabular}{|c|c|c||c||c|c|} \hline
$t$ & $a$ & $\delta$ & $N$ & $h^{1,1}$ & $h^{2,1}$ \\ \hline \hline
10 & 10 & 0 & 0 & 11 & 11 \\ \hline
11 & 11 & 1 & 1 & 16 & 10 \\ \hline
12 & 10 & 1 & 2 & 21 & 9 \\ \hline
13 & 9 & 1 & 3 & 26 & 8 \\ \hline
14 & 8 & 1 & 4 & 31 & 7 \\ \hline
15 & 7 & 1 & 5 & 36 & 6 \\ \hline
16 & 6 & 1 & 6 & 41 & 5 \\ \hline
17 & 5 & 1 & 7 & 46 & 4 \\ \hline
18 & 4 & 1 & 8 & 51 & 3 \\ \hline
18 & 4 & 0 & 8 & 51 & 3 \\ \hline
19 & 3 & 1 & 9 & 56 & 2 \\ \hline
20 & 2 & 1 & 10 & 61 & 1 \\ \hline
\end{tabular}\\[0.7cm]
$$
\end{pkt}

\begin{remark}
C. Borcea \cite{Bc} has constructed Calabi-Yau manifolds of dimension 3 with $CM$ by using
3 elliptic curves with involutions. This construction yields a $CMCY$ family of 3-manifolds over
$\sM_1 \times \sM_1 \times \sM_1$. The fibers have the Hodge numbers $h^{1,1} = 51$ and $h^{2,1}
= 3$. By similar arguments as in Theorem $\ref{jaaah}$, this family is maximal. The associated
period domain is given by $\B_1 \times \B_1 \times \B_1$.

As we have seen in Section $10.3$, the family $\sQ \to \sM_3$ is a maximal $CMCY$ family of
3-manifolds, whose fibers have the same Hodge numbers $h^{1,1} = 51$ and $h^{2,1} = 3$. The
associated period domain is given by $\B_3$

Moreover by Theorem $\ref{jaaah}$ and the preceding point, we have two additional holomorphic
maximal $CMCY$ families of 3-manifolds, whose fibers have the same Hodge numbers $h^{1,1} = 51$ and
$h^{2,1} = 3$. The associated period domain is given by $\B_1 \times D$, where $D$ denotes the
bounded domain given by $\SO(2,2)(\R)/K$

Hence there exist 4 maximal $CMCY$ families of 3-manifolds, whose fibers have the Hodge numbers
$h^{1,1} = 51$ and $h^{2,1} = 3$. One can easily check that the example of \cite{Bc} has a
Yukawa coupling of {\em length} 3, where the Yukawa coupling of the family $\sQ \to \sM_3$
constructed in Section $9.2$ has the {\em length} 1. Hence there are not any open sets of the
respective bases, which allow a local identification of these two families.

By using the involutions on elliptic curves, one gets a local identification between
$\sE \times \sE/\langle(\iota_{\sE},\iota_{\sE})\rangle \to \sM_1 \times \sM_1$, which yields the
example of \cite{Bc}, with one of our
examples $\sX_{(t,a,\delta)} \to B_{(t,a,\delta)}$ with $t = 18$ and $a = 4$. This implies a
local identification between the resulting $CMCY$ families of 3-manifolds obtained by the
Borcea-Voisin tower.
\end{remark}

\begin{remark} \label{cm20}
By Example $\ref{ellcm3}$, there are 13 explicite examples of elliptic curves with $CM$.
Thus for the $CMCY$ family of C. Borcea \cite{Bc}, which we have discussed in the preceding remark, one obtains up to
birational equivalence 455 different examples of $CM$ fibers. For 6 of these 13 elliptic curves, we
have an explicitely given involution. Thus we can at least describe the 56 Calabi-Yau 3-manifolds,
which are obtained by some of the latter 6 elliptic curves, by local equations.
\end{remark}

\begin{remark}
It would be interesting to consider the following question: Is the maximal $CMCY$ family
$\fX_{(10,10,0)}$ its own mirror family?
 
Let $S$ denote a $K3$ surface with an involution, which acts by $-1$ on $\Gamma(\omega_S)$. In
\cite{Voi2} the triples $(t,a,\delta)$, which yield our families
$\sX_{{(t,a,\delta)}} \to B_{(t,a,\delta)}$ satisfying the assumptions of Proposition $\ref{nix}$,
do not satisfy the assumptions of the technical Lemma \cite{Voi2}, Lemme $2.5$. This Lemma guarantees
the existence of a hyperbolic plane $H \subset H^2(S,\Z)_1$, which is needed for the mirror
construction in \cite{Voi2}. Hence these triples $(t,a,\delta)$
do not satisfy the assumptions of the Mirror Theorem \cite{Voi2}, Th\'eor\`eme $2.17$. But by
\cite{CK}, Lemma $4.4.4$, there is a hyperbolic plane $H \subset H^2(S,\Z)_1$ for these triples,
too.

In her construction of a Calabi-Yau 3-manifold (\cite{Voi2}, Lemme $1.3$)
C. Voisin assumes that the involution on the $K3$ surface is not given by the triple $(10,10,0)$,
since it is easy to see that the resulting 3-manifold is not simply connected
in this case. But by Proposition
$\ref{trivcan}$ the resulting 3-manifold satisfies our definition of a Calabi-Yau manifold
(Definition $\ref{defCala}$) in this case, too.

The mirror of a fiber of
$\fX_{(10,10,0)}$ must have the same Hodge numbers $h^{1,1} = h^{2,1} = 11$. By Claim
$\ref{doppeln}$, this implies for an involution on a $K3$ surface:
$$ 5N-N'= 5N'-N = 0$$
Hence one calculates easily that $N = N' = 0$. Thus by V. V. Nikulins \cite{Niku} classification
of involutions on $K3$ surfaces, the Voisin-Borcea Mirror (in the notation of \cite{CK}) of
a fiber of $\fX_{(10,10,0)}$ should be obtained by the triple $(10,10,0)$, too. Hence the author has
the impression that one can consider the maximal $CMCY$ family $\fX_{(10,10,0)}$ of 3-manifolds as
its own mirror family, but one must check the details.
\end{remark}

\newpage
\addcontentsline{toc}{chapter}{Index}
\printindex

\end{document}